\documentclass[11pt,letterpaper]{amsart}
\usepackage{graphicx}
\usepackage{amsmath}
\usepackage{mathtools}
\usepackage{amssymb}
\usepackage{amsfonts}
\usepackage{amsthm}
\usepackage{cancel}
\usepackage[all]{xy}
\usepackage[neveradjust]{paralist}
\usepackage{enumitem}
\usepackage{multicol}
\usepackage{mathrsfs}
\usepackage{mathdots}
\usepackage[pagebackref,colorlinks=true,linkcolor=blue,citecolor=blue]{hyperref}
\usepackage{tikz-cd}
\usepackage{tikz}
\usepackage{blkarray}
\usepackage{scalerel,stackengine}

\UseRawInputEncoding

\stackMath
\newcommand\reallywidehat[1]{%
\savestack{\tmpbox}{\stretchto{%
  \scaleto{%
    \scalerel*[\widthof{\ensuremath{#1}}]{\kern-.6pt\bigwedge\kern-.6pt}%
    {\rule[-\textheight/2]{1ex}{\textheight}}
  }{\textheight}%
}{0.5ex}}%
\stackon[1pt]{#1}{\tmpbox}%
}

\tikzset{
  symbol/.style={
    draw=none,
    every to/.append style={
      edge node={node [sloped, allow upside down, auto=false]{$#1$}}}
  }
} 

\usepackage{color}
\usepackage{blkarray}

\newcommand{\N}{\mathbb{N}}
\newcommand{\Z}{\mathbb{Z}}

\newcommand{\R}{\mathbb{R}}
\newcommand{\BC}{\mathbb{C}}

\newcommand{\inv}{^{-1}}

\newcommand{\SL}{\mathrm{SL}}
\newcommand{\GL}{\mathrm{GL}}
\newcommand{\SO}{\mathrm{SO}}
\newcommand{\Sp}{\mathrm{Sp}}
\newcommand{\St}{\mathrm{St}}
\newcommand{\ord}{\mathrm{ord}}

\newcommand{\Rep}{\underline{\mathrm{Rep}}}

\newcommand{\Jord}{\mathrm{Jord}}

\newcommand{\half}[1]{\frac{#1}{2}}

\newcommand{\comment}[1]{}
\newcommand{\EE}{\mathcal{E}}
\newcommand{\FF}{\mathcal{F}}

\newtheorem{thm}{Theorem}[section]
\newtheorem{cor}[thm]{Corollary}
\newtheorem{lemma}[thm]{Lemma}
\newtheorem{prop}[thm]{Proposition}
\newtheorem {conj}[thm]{Conjecture}

\newtheorem {ques/conj}[thm]{Question/Conjecture}
\newtheorem {ques}[thm]{Question}
\newtheorem{defn}[thm]{Definition}
\newtheorem{remark}[thm]{Remark}

\newtheorem{exmp}[thm]{Example}

\newtheorem{algo}[thm]{Algorithm}

\newtheorem{problems}[thm]{Problem}

\newtheorem{goal}[thm]{Goal}

\DeclareMathOperator{\supp}{supp}

\DeclareMathOperator{\Hom}{Hom}
\DeclareMathOperator{\abs}{abs}
\DeclareMathOperator{\exsupp}{ex.supp}
\numberwithin{equation}{section}

\let\oldbullet\bullet
\renewcommand{\bullet}{{\vcenter{\hbox{\tiny$\oldbullet$}}}}

\begin{document}
\renewcommand{\theequation}{\arabic{equation}}
\numberwithin{equation}{section}

\title[Intersection of local Arthur packets]
{On the intersection of local Arthur packets for classical groups and applications}

\author{Alexander Hazeltine
}
\address{Department of Mathematics\\
University of Michigan\\
Ann Arbor, MI 48109, USA}
\email{ahazelti@umich.edu}

\author{Baiying Liu}
\address{Department of Mathematics\\
Purdue University\\
West Lafayette, IN, 47907, USA}
\email{liu2053@purdue.edu}

\author{Chi-Heng Lo}
\address{Department of Mathematics\\
Purdue University\\
West Lafayette, IN, 47907, USA}
\email{lo93@purdue.edu}


\date{\today}



\thanks{The research of the first-named and the third-named authors are partially supported by the NSF Grants DMS-1848058. 
The research of the second-named author is partially supported by the NSF Grants DMS-1702218, DMS-1848058, and by start-up funds from the Department of Mathematics at Purdue University}

\begin{abstract}
In this paper, for symplectic and split odd special orthogonal groups, we develop an account of theory on the intersection problem of local Arthur packets. Specifically, following Atobe's reformulation on M{\oe}glin's construction of local Arthur packets, we give a complete set of operators on the construction data, based on which, we provide algorithms and Sage codes to determine whether a given representation is of Arthur type. Furthermore, for any representation $\pi$ of Arthur type, we give a precise formula for the set
$$ \Psi(\pi)=\{ \text{local Arthur parameter }\psi \ | \ \text{the local Arthur packet } \Pi_{\psi} \text{ contains } \pi\}.$$
 
Our results have many applications, including the precise counting of tempered representations in any local 
Arthur packet, specifying and characterizing ``the" local Arthur parameter in $\Psi(\pi)$ for $\pi$, especially when $\pi$ belongs to several local Arthur packets but 
does not belong to any local $L$-packet of Arthur type. 
\end{abstract}

\maketitle

\tableofcontents

\section{Introduction}

Let $F$ be a non-Archimedean local field of characteristic zero. Let $\mathrm{G}_n$ denote the split groups $\Sp_{2n}, \SO_{2n+1}$ and let $G_n=\mathrm{G}_n(F)$. The Langlands dual groups are given by
$\widehat{\mathrm{G}}_n(\BC) = \SO_{2n+1}(\BC), \Sp_{2n}(\BC)$, respectively. 
In the fundamental work \cite{Art13},
to characterize the local components of discrete automorphic representations, Arthur introduced  local Arthur packets which are finite sets of representations of $G_n$, parametrized by local Arthur parameters. 
Local Arthur parameters are defined as $\widehat{\mathrm{G}}_n(\BC)$-conjugacy classes of
direct sum of irreducible representations
$$\psi: W_F \times \SL_2(\mathbb{C}) \times \SL_2(\mathbb{C}) \rightarrow \widehat{\mathrm{G}}_n(\BC),$$
\begin{equation}\label{lap}
  \psi = \bigoplus_{i=1}^r \phi_i|\cdot|^{x_i} \otimes S_{a_i} \otimes S_{b_i},  
\end{equation}
satisfying the following conditions: 
\begin{enumerate}
    \item [(1)]$\phi_i(W_F)$ is bounded and consists of semi-simple elements, and $\dim(\phi_i)=d_i$;
    \item [(2)] $x_i \in \R$ and $|x_i|<\half{1}$;
    \item [(3)]the restrictions of $\psi$ to the two copies of $\SL_2(\mathbb{C})$ are analytic, $S_k$ is the $k$-dimensional irreducible representation of $\SL_2(\mathbb{C})$, and 
    $$\sum_{i=1}^r d_ia_ib_i = N:= 
\begin{cases}
2n+1 & \text{ when } G_n=\Sp_{2n}(F),\\
2n & \text{ when } G_n=\SO_{2n+1}(F).
\end{cases}
$$ 
\end{enumerate}
The first copy of $\SL_2(\mathbb{C})$ is called the Deligne-$\SL_2(\mathbb{C})$, denoted by $\SL_2^D(\mathbb{C})$. The second copy of $\SL_2(\mathbb{C})$ is called the Arthur-$\SL_2(\mathbb{C})$, denoted by $\SL_2^A(\mathbb{C})$. 
A local Arthur parameter $\psi$ given in \eqref{lap} is called {\it generic} if $b_i=1$ for $i=1, \ldots, r$; and is called {\it tempered} if additionally $x_i=0$ for $i=1, \ldots, r$. 
A local Arthur packet $\Pi_{\psi}$ is called generic (resp. tempered) if the corresponding local Arthur parameter $\psi$ is so. 
We let $\Psi^{+}(G_n)$ be the set of local Arthur parameters of $G_n$ and $\Psi(G_n)$ be the subset of $\Psi^+(G_n)$ consisting of local Arthur parameters $\psi$ whose restriction to $W_F$ is bounded. In other words, $\psi$ is in $\Psi(G_n)$ if and only if $x_i=0$ for $i=1,\dots, r$ in the decomposition \eqref{lap}. 

For each local Arthur parameter $\psi$, Arthur associated a local $L$-parameter $\phi_{\psi}$ as follows
\begin{equation}\label{apequ1}
\phi_{\psi}(w, x) = \psi\left(w, x, \begin{pmatrix}
        |w|^{\frac{1}{2}} & 0 \\
        0 & |w|^{-\frac{1}{2}}\\
\end{pmatrix}\right),
\end{equation}
and showed that the map $\psi \mapsto \phi_{\psi}$ is injective. The {\it diagonal restriction} of $\psi$ is another $L$-parameter associated to $\psi$ which is defined as 
\begin{align}\label{def diag rest}
\begin{split}
    \psi^{\Delta}:  W_F \times \SL_2(\BC) & \rightarrow \widehat{\mathrm{G}}_n(\BC)\\
    (w,x) & \mapsto \psi(w,x,x).
\end{split}
\end{align}

Given a local Arthur parameter $\psi$ as in \eqref{lap}, the local Arthur packet $\Pi_{\psi}$ defined in \cite[Theorem 2.2.1 and formula (1.5.1)]{Art13} is a finite multi-set of irreducible representations of $G_n$, satisfying certain (twisted) endoscopic character identities, and the local $L$-packet $\Pi_{\phi_{\psi}}$ is contained in $\Pi_{\psi}$ (\cite[Proposition 7.4.1]{Art13}). 
Arthur showed that 
$\Pi_{\psi}$ consists of unitary representations when $\psi \in \Psi(G_n)$ (\cite[Theorem 1.5.1]{Art13}) and conjectured that $\Pi_{\psi}$ also consists of unitary representations when $\psi \in \Psi^+(G_n)$ (\cite[Conjecture 8.3.1]{Art13}). 
We say that a local $L$-parameter $\phi$ is {\it of Arthur type} if $\phi=\phi_\psi$ for some local Arthur parameter $\psi.$
We also say that a representation $\pi$ is {\it of Arthur type} if $\pi\in\Pi_\psi$ for some local Arthur parameter $\psi$ and let $\Pi_A(G_n)$ denote the set of representations of $G_n$ of Arthur type.

In a series of papers (\cite{Moe06a, Moe06b, Moe09a, Moe10, Moe11}), M{\oe}glin explicitly constructed each local Arthur packet $\Pi_{\psi}$ and showed that it is multiplicity free.
In particular, M{\oe}glin (\cite[Corollaire 4.2]{Moe09a}) showed $\Pi_{\psi_1} \cap \Pi_{\psi_2} \neq \emptyset$ only if $\psi_1^{\Delta}=\psi_2^{\Delta}$. 
Then, Xu (\cite{Xu21b}) gave an algorithm to determine whether the representations in M{\oe}glin's construction are nonzero. In recent work (\cite{Ato20b}), 
Atobe gave a reformulation on  M{\oe}glin's construction, using the derivatives (see \S \ref{sec: der and socle}) considered in \cite{Jan14, AM20}, which provides a way to compute the $L$-data. 

While local $L$-packets are disjoint, local Arthur packets may have nontrivial intersections. Hence, the following is a very fundamental question.

\begin{ques}\label{que intro}
When do two given local Arthur packets have a nontrivial intersection?
\end{ques}

The intersection problem of local Arthur packets is considered as a key step towards the local non-tempered Gan-Gross-Prasad problem, see \cite[Conjecture 7.1, Remark 7.3]{GGP20}. For certain results on the intersection problem of local Arthur packets over archimedean local fields, see \cite{MR17} and \cite{MR21}. 

Fix any two local Arthur packets, one can certainly check the intersection by M{\oe}glin's construction or by Atobe's reformulation. The difficulty is to find a systematic way to determine all the local Arthur packets which have nontrivial intersections with a given one.

\begin{ques}\label{que intro 2}
Fix a local Arthur parameter $ \psi$, how do we systematically construct the set 
$$\{\psi' \ | \ \Pi_{\psi} \cap \Pi_{\psi'} \neq \emptyset\}?$$
\end{ques}

By the work of M{\oe}glin (see Theorem \ref{thm red from nu to gp}), to consider Question \ref{que intro 2}, 
we can restrict ourselves to the case of the local Arthur parameters of {\it good parity}, i.e., every summand $\rho \otimes S_{a_i} \otimes S_{b_i}$ is self-dual and of the same type as $\psi.$ 
In \cite{Ato20b}, Atobe showed that for any such local Arthur parameter, the corresponding local Arthur packet is exactly parametrized by  certain extended multi-segments (see Definition \ref{def multi-segment}). More explicitly, given an extended multi-segment $\EE$, there is an associated local Arthur parameter $\psi_{\EE}$ and a representation $\pi(\EE)$ which is irreducible or zero. Then, \cite[Theorem 3.4]{Ato20b} states that
\[ \Pi_{\psi}= \{\pi(\EE)\ | \ \psi_{\EE}=\psi,\ \pi(\EE) \neq 0\}.\]

In this paper, following Atobe's reformulation on M{\oe}glin's construction  (\cite{Ato20b}), we develop an account of theory towards Question \ref{que intro 2}. More precisely, for any irreducible representation $\pi$ of good parity (see Definition \ref{def good parity reps}), we give algorithms to achieve the following goals. 

\begin{goal}\label{goal}\ 
    \begin{enumerate}
        \item [(1)]To determine whether there exists an extended multi-segment $\EE$ such that $\pi(\EE)=\pi$ (see Algorithm \ref{alg Arthur type}). 
        \item [(2)] In the affirmative case, to determine all extended multi-segments $\EE$ such that $\pi(\EE)=\pi$ (see Theorem \ref{main thm intro} below).
    \end{enumerate}
\end{goal}
Moreover, given an extended multi-segment $\EE$ such that $\pi(\EE)\neq 0$, we give a precise formula to compute the set $ \{ \psi' \ | \ \pi(\EE) \in \Pi_{\psi'}\}$ (see Theorem \ref{thm exhaustion of symbol}). Thus, running through $\pi(\EE) \in \Pi_{\psi}$, we obtain
$$\{\psi' \ | \ \Pi_{\psi} \cap \Pi_{\psi'} \neq \emptyset\}.$$

Our algorithms depend on the following four operators, their inverses and compositions:
\begin{enumerate}
    \item[(i)] $R_k$, row exchange of extended segments (see Definition \ref{def row exchange}); 
    \item[(ii)] $ui_{i,j}$, union-intersection of extended segments (see Definition \ref{def ui});
    \item[(iii)] $dual \circ ui_{i,j} \circ dual$, where $dual$ is the Aubert-Zelevinsky dual on extended multi-segments (see Definition \ref{dual segment});
 \item[(iv)] $dual_k$ (partial dual, only applies in the case of $a_i+b_i$ being odd) (see Definition \ref{def partial dual}).
\end{enumerate}

Here are our main results.

\begin{thm}[{Propositions \ref{prop basic operators}, \ref{prop partial dual}, Theorems \ref{thm integer}, \ref{thm half integer}, and \ref{thm exhaustion of symbol}}]\label{main thm intro}
Let $\EE_1$ and $\EE_2$ be two extended multi-segments for $G_n$. Suppose that $\pi(\EE_1) \in \Pi_{\psi}$. 
Then the followings hold. 
\begin{enumerate}
    \item Let $T$ be any of the four operators $R_k$, $ui_{i,j}$, $dual \circ ui_{i,j} \circ dual$, $dual_k$ or their inverses. We have 
    $T(\EE_1)$ is also an extended multi-segment for $G_n$, and 
    $$\pi(T(\EE_1))=\pi(\EE_1).$$
    \item  We have $\pi(\EE_1) \cong \pi(\EE_2)$ if and only if 
$\EE_2$ can be obtained from $\EE_1$
by a finite chain of the four operators and their inverses in Part (1). 
\item There is a precise formula/algorithm to compute the set $\{ \psi' \ | \ \pi(\EE_1) \in \Pi_{\psi'}\}$.
\end{enumerate}
\end{thm}

Independently in  \cite{Ato22}, Atobe  gave different algorithms towards Goals \ref{goal}. 
Atobe's algorithms depend on following three operators, their inverses and compositions:
\begin{enumerate}
    \item[(i)] $R_k$, row exchange of extended segments;
    \item[(ii)] $ui_{i,j}$, union-intersection of extended segments;
    \item[(iii)] $P$, phantom (dis)appearing (formally adding (deleting) an extended segment).
\end{enumerate}
He proved the following.

\begin{thm}[{\cite[Theorem 1.4]{Ato22}}]\label{thm Atobe's main thm}
Let $\EE_1$ and $\EE_2$ be two extended multi-segments for $G_n$. 
Suppose that $\pi(\EE_1) \not= 0$. 
Then $\pi(\EE_1) \cong \pi(\EE_2)$ if and only if 
$\EE_2$ can be obtained from $\EE_1$
by  a finite chain of three operators $R_k$, $ui_{i,j}$, P, and their inverses. 
\end{thm}

Our operators form a proper subset of those used by Atobe, see Remarks \ref{rmk partial dual}, \ref{rmk dual ui dual}, and Example \ref{exmp phantom} for comparisons between the two sets of operators. 
On the other hand, we remark that Theorem \ref{main thm intro}(2) is logically equivalent to Theorem \ref{thm Atobe's main thm}.


For any extended multi-segments $\EE$ such that $\pi(\EE)\neq 0$, in the proof of Theorem \ref{main thm intro}, we construct a particular extended multi-segment $\EE_{can}$ such that $\pi(\EE_{can})=\pi(\EE)$, which we call the {\it canonical form} of $\EE$ (see Definition \ref{def can form}). The canonical form $\EE_{can}$ has its own interest of study and representation theoretic importance. First, we have 
$\pi(\EE) \cong \pi(\EE')$ if and only if $\EE_{can}=\EE_{can}'$. Second, based on the definition of $\EE_{can}$, we give an algorithm to systematically compute the set 
\[\Psi(\pi):=\{\psi \in \Psi^+(G_n) \ | \ \pi(\EE) \in \Pi_{\psi}\}\]
in Theorem \ref{thm exhaustion of symbol}. We explicate this algorithm in Examples \ref{ex exhaustion} and \ref{ex Atobe}, where Example \ref{ex Atobe} is a reconsideration of Atobe's example in \cite[Section 3.4]{Ato22}. Third, we show that $\EE_{can}$ is uniquely characterized by the derivative information of $\pi(\EE)$ (see Theorem \ref{thm canonical form and derivatives}), which leads to our Algorithm \ref{alg Arthur type} towards the Goal \ref{goal}(1) above.  

Here we remark the difference between our Algorithm \ref{alg Arthur type} and 
Atobe's algorithm (\cite[Algorithm 3.3]{Ato22}), which are both towards the Goal \ref{goal}(1) above. 
Atobe's algorithm constructs several representations $\pi_i$ of smaller rank, and uses the information from the derivatives of $\pi$ and $\pi_i$ together with the construction of $\Psi(\pi_i)$. On the other hand, our Algorithm \ref{alg Arthur type} only uses the derivatives information of $\pi$ to construct a local Arthur parameter $\psi$ such that $\pi$ is of Arthur type if and only if $\pi \in \Pi_{\psi}$. We remark that both algorithms above result in constructing a possibly non-tempered local Arthur packet $\Pi_{\psi}$, and checking whether $\pi$ or $\pi_i$ is in this packet or not. In \cite{HJLLZ22}, we give new algorithms towards the Goal \ref{goal}(1) above, which do not require computing derivatives of representations, but use the construction of $\Psi(\pi_i)$ only, where $\pi_i$'s are several representations of smaller rank constructed from $\pi$.
Moreover, there is no construction of a non-tempered local Arthur packet involved. In \cite{HJLLZ22}, these new algorithms have been used in an inductive approach to classify certain families of unitary representations which are in local Arthur packets of symplectic or split odd special orthogonal groups.


In the following subsections, we give several applications of our main results, which are expected to play important roles in many problems related to local Arthur packets, for example, the local  non-tempered Gan-Gross-Prasad problem as in \cite{GGP20}.

\subsection{Counting of tempered representations in local Arthur packets}

It is known that tempered representations may occur in non-tempered local Arthur packets. It is a very interesting question to determine whether a given non-tempered local Arthur packet contains a tempered representation.
Notice that a local Arthur packet $\Pi_{\psi}$ contains a tempered representation only when $\psi \in \Psi(G_n)$. 


Let $\psi$ be any local Arthur parameter in $\Psi(G_n)$. M{\oe}glin's work (\cite[Corollary 4.2]{Moe09a}) implies the following inclusion
\begin{equation}\label{set inclusion}
    \{ \pi\in \Pi_{\psi} \ | \ \pi \text{ is tempered}\} \subseteq \Pi_{\psi^{\Delta}}.
\end{equation}
Here, $\psi^{\Delta}$ is the diagonal restriction of $\psi$ as in \eqref{def diag rest}. Note that for $\psi \in \Psi(G_n)$, $\psi^{\Delta}$ is always a tempered $L$-parameter.

As an application of our main results, we give a precise counting of tempered representations in each local Arthur packet $\Pi_{\psi}$ and a description of their $L$-data (see Theorems \ref{thm cuspidal count}, \ref{thm temp non-negative}, \ref{thm tempered rep general}). 
The idea is to apply the precise formula in Theorem \ref{thm exhaustion of symbol} on an extended multi-segment whose associated local Arthur parameter is tempered.
Based on the explicit counting, we provide a family of examples for the local non-tempered Gan-Gross-Prasad conjecture in \cite{GGP20}, by considering all local Arthur packets containing tempered representations, 
see Example \ref{exmp GGP}. We expect these precise countings to play important roles towards the local non-tempered Gan-Gross-Prasad problem.

\subsection{\texorpdfstring{The local $L$-packets of Arthur type}{}}
Given any local Arthur parameter $\psi$ as in \eqref{lap}, one has a local $L$-parameter $\phi_{\psi}$ as in \eqref{apequ1}, and Arthur showed that $\Pi_{\phi_{\psi}} \subseteq \Pi_{\psi}$. Considering the importance of this subset $\Pi_{\phi_{\psi}}$ in the theories of local Arthur packets and Langlands program, we give the following second application. For a local Arthur parameter $\psi$ of good parity, we completely describe the extended multi-segments $\EE$ such that $\pi(\EE) \in \Pi_{\phi_{\psi}}$ as follows. 

\begin{thm}[Theorem \ref{thm in L-packet}]\label{thm in L-packet intro}
Let $\EE$ be an extended multi-segment with associated local Arthur parameter $\psi$ such that $\pi(\EE)\neq 0$. 
Apply a sequence of row exchanges to $\EE$ to obtain another extended multi-segment
$$\EE'= \cup_{\rho}\{ ([A_i,B_i]_{\rho},l_i,\eta_i) \}_{i \in (I_{\rho},>)}$$
such that for all $\rho$ and $i < j \in I_{\rho}$, $A_i +B_i \leq A_j +B_j.$
Then $\pi(\EE)$ is in the local $L$-packet $\Pi_{\phi_{\psi}}$ if and only if the following conditions hold for all $\rho$ and $i<j \in I_{\rho}$,
    \begin{enumerate}
        \item [$\oldbullet$] $(A_i-B_i+1)-2l_i \leq 1$,
        \item [$\oldbullet$] if $A_i+B_i= A_{j}+B_{j}$ and $(A_i-B_i+1)-2l_i=1$, then $\eta_i=\eta_{j}$.
    \end{enumerate}
\end{thm}

The result is also generalized to $\psi \in \Psi^+(G_n)$ (see \S \ref{section L-packet}).

\subsection{``The" local Arthur parameters of representations of Arthur type}

Since local Arthur packets can have nontrivial intersection, given an irreducible representation $\pi$ of Arthur type, it may lie in several local Arthur packets. Namely, the set
\[ \Psi(\pi):= \{ \psi \in \Psi^+(G_n) \ | \ \pi \in \Pi_{\psi}  \},\]
may not be a singleton. It is a very mysterious question to determine which of these local Arthur parameters in $\Psi(\pi)$ could be called ``the" local Arthur parameter for $\pi$. We define a candidate towards this question in the third application. 

\begin{thm}[Corollary \ref{cor max}, Theorem \ref{thm can max}]\label{thm can max intro}
Suppose $\pi$ is a representation of Arthur type and of good parity. There is a unique absolute maximal (see Definition \ref{def abs max} and Corollary \ref{cor max}) extended multi-segment  $\EE^{|max|}$ such that
\begin{enumerate}
    \item [(a)]$\pi(\EE^{|max|})= \pi$, and
    \item [(b)]  $\pi \in \Pi_{\phi_{\psi}}$ if and only if $\EE^{|max|}$ satisfies the conditions in Theorem \ref{thm in L-packet intro}. In this case, $\psi_{\EE^{|max|}}=\psi$.
\end{enumerate}
\end{thm}

We call $\EE^{|max|}$ the {\it max form} of $\EE$ and let
$$\psi^{max}(\pi):=\psi_{\EE^{|max|}}.$$
The definition of $\psi^{max}(\pi)$ is extended to general representations of Arthur type but not of good parity by Theorem \ref{thm red from nu to gp}.

The distinguished member $\psi^{max}(\pi)$ is of great 
significance in the theories of local Arthur packets and automorphic forms. 
Among all the parameters in the set $\Psi(\pi)$,
we show that $\psi^{max}(\pi)$ deserves the most to be called ``the" local Arthur parameter for $\pi$. This will be elaborated in \S \ref{sec: 5 orderings}. As an opposite of $\EE^{|max|}$, we also 
specify a unique absolute minimal extended multi-segment $\EE^{|min|}$ (see Definition \ref{def abs min} and Corollary \ref{cor abs min}), called the {\it minimal form} of $\EE$, such that $\pi(\EE^{|min|})= \pi$.
This allows us to define another local Arthur parameter for $\pi$
$$\psi^{min}(\pi):=\psi_{\EE^{|min|}}.$$

We remark that both the local Arthur parameters $\psi^{max}(\pi)$ and $\psi^{min}(\pi)$ have their own interests of study.  
Note that $\psi^{max}(\pi)$ provides a disjoint partition of local Arthur packets by grouping together the representations with the same $\psi^{max}(\pi)$ (Corollary \ref{cor psi max partition}). Hence, a feasible way to overcome the difficulty of a problem related to the nontrivial intersections of local Arthur packets is to reduce it to $\psi^{max}(\pi)$. 
In \cite{HLLZ22}, jointly with Zhang, via a key reduction to $\psi^{max}(\pi)$, 
we proved the closure ordering conjecture on local $L$-parameters for representations in local Arthur packets of $\mathrm{G}_n$, namely, given any representation $\pi$ in a local Arthur packet $\Pi_{\psi}$, the closure of the local $L$-parameter of $\pi$ in the Vogan variety must contain the local $L$-parameter corresponding to $\psi$. 
The ABV analogue of the closure ordering conjecture has been proved in \cite[Proposition 7.10]{CFMMX22}, using geometric methods.
Hence, the result in \cite{HLLZ22} provides evidence for the Vogan conjecture on the coincidence of the local Arthur packets with the ABV-packets as in \cite[Section 8.3, Conjecture 1]{CFMMX22}.

\subsection{\texorpdfstring{Characterizations of $\psi^{max}(\pi)$ and $\psi^{min}(\pi)$}{}}
\label{sec: 5 orderings}

It is expected that the reduction to the unique maximal element $\psi^{max}(\pi)$ plays a key role in many problems related to local Arthur packets. Hence, it is desirable to explore the representation theoretic characterizations of $\psi^{max}(\pi)$ and its counterpart $\psi^{min}(\pi)$. More generally, it is an interesting question to study $\psi^{max}(\pi)$ and $\psi^{min}(\pi)$ in terms of the structure of the set $\Psi(\pi)$.

Notice that the actions of the operators on extended multi-segments naturally induce actions on local Arthur parameters (see Definition \ref{def operators on parameters}). 
Among all the operators (including the inverses), $ ui_{i,j}^{-1}$, $dual \circ ui_{j,i} \circ dual $ and $dual_k^{-}$ raise the ``temperedness" of local Arthur parameters under a certain measurement of temperedness (see Theorem \ref{thm Jiang's partition}(1)). These are called  \emph{raising} operators (see Definition \ref{def raising operator}), which induces a partial ordering $\geq_{O}$ on $\Psi(\pi)$. We realize that this partial ordering dominates 4 other orderings arising from problems related to local Arthur packets (see Theorem \ref{thm 5 orderings}(1)). We recall the definitions of these orderings now.

\begin{enumerate}
    \item $O$ (stands for Operator) ordering (see Definition \ref{def operator ordering}): 
$\psi_1 \geq_{O} \psi_2$ if $\psi_1=\psi_2$ or there exists a sequence of raising operators $\{T_l\}_{l=1}^m$ such that
\[ \psi_1= T_1 \circ \cdots \circ T_m(\psi_2).\]

\item $D$ (stands for Deligne) ordering (see \cite[Definition 1.8]{HLLZ22}): $\psi_1 \geq_D \psi_2$ if $ \underline{p}^D(\psi_1) \geq \underline{p}^D(\psi_2)$ under the dominance order. Here $\underline{p}^D(\psi_i)$ is the partition obtained via restricting $\psi_i$ to the Deligne-$\SL_2(\mathbb{C})$. 

\item $A$ (stands for Arthur) ordering (see Definition \ref{A ordering}): $\psi_1 \geq_A \psi_2$ if $ \underline{p}^A(\psi_1) \leq \underline{p}^A(\psi_2)$ under the dominance order. Here $\underline{p}^A(\psi_i)$ is the partition obtained via restricting $\psi_i$ to the Arthur-$\SL_2(\mathbb{C})$. 

\item $N$ (stands for Normalized) ordering (see Definition \ref{M ordering}): $\psi_1 \geq_{N} \psi_2$ if for any $\sigma= \St(\rho',a_0)$ and $s_0 \in \R_{\geq_{\half{1}}}$, the following inequality holds
$$
 \ord_{s=s_0} N_{\psi_1}(s,\pi, \sigma)
 \leq \ord_{s=s_0} N_{\psi_2}(s,\pi, \sigma).
$$
Here $\sigma= \St(\rho',a_0)$ is a Steinberg representation,  $N_{\psi_i}(s,\pi,\sigma)$ is the Arthur normalized intertwining operator considered in \cite{Moe10} and \cite{Art13} (see \S \ref{sec: M ordering}), and $\ord_{s=s_0} N_{\psi_i}(s,\pi, \sigma)$ is the order of zeros of $N_{\psi_i}(s,\pi,\sigma)$ at $s=s_0$.

\item $C$ (stands for Closure) ordering (see \cite[Definition 1.11]{HLLZ22}): $\psi_1 \geq_C \psi_2$ if 
$\overline{C_{\phi_{\psi_1}}} \supseteq C_{\phi_{\psi_2}}$. Here $C_{\phi_{\psi_i}}$ is the orbit corresponding to $\psi_i$ in the Vogan variety (see \cite{Vog93}, \cite[Section 4]{CFMMX22}, or, \cite[Introduction]{HLLZ22}).
\end{enumerate}

Then, we have the following characterizations of $\psi^{max}(\pi)$ and $\psi^{min}(\pi)$ applying the 5 orderings above. 

\begin{thm}[Theorems \ref{thm structure of Psi(pi) intro}, \ref{thm Jiang's partition}, \ref{thm Moeglin}, and  {\cite[Theorems 1.9, 1.12]{HLLZ22}}]
\label{thm 5 orderings}
Let $\mathrm{G}_n$ be $\Sp_{2n}$ or split $\SO_{2n+1}$.
\begin{enumerate}
    \item Let $\geq$ be any of the 4 orderings $\geq_D, \geq_A, \geq_{N}, \geq_C$. 
    If $T$ is a raising operator, then 
    \[T(\psi) \geq \psi,\]
    for any $\psi \in \Psi^+(G_n).$
    In other words, if $\psi \geq_{O} \psi'$, then $\psi \geq \psi'$.
    \item  
    Let $\pi\in \Pi_A(G_n)$ and let $\geq$ be any of the 5 orderings $\geq_O, \geq_D, \geq_A, \geq_{N}, \geq_C$. The local Arthur parameters $\psi^{max}(\pi)$ and $ \psi^{min}(\pi)$ are the unique elements in $\Psi(\pi)$ satisfying the following inequality
    \[ \psi^{max}(\pi) \geq \psi \geq \psi^{min}(\pi),\]
    for any $\psi \in \Psi(\pi).$
\end{enumerate}
\end{thm}

 The $A$ ordering is related to Jiang's Conjecture on the upper bound of wavefront sets (\cite[Conjecture 4.2]{Jia14}). The $D$ ordering is a natural analogue of the $A$ ordering. The $C$ ordering is an evidence of Vogan Conjecture (\cite[Conjecture 8.3.1]{CFMMX22}) as remarked at the end of the previous subsection. We give further remarks on the $N$ ordering as follows.

Arthur's normalized intertwining operators 
$N_{\psi}(s,\pi,\sigma)$ play important roles in theories of endoscopy and trace formula (see \cite[\S 2.3 and \S 2.4]{Art13}). 
M{\oe}glin (\cite{Moe10}) showed that $N_\psi(s,\pi,\sigma)$ is holomorphic for any real $s\geq 0.$ Given $\pi \in \Pi_A(G_n)$, if 
$\pi\in\Pi_{\phi_\psi}$, then $N_\psi(s,\pi,\sigma)$ is the same as the Langlands-Shahidi normalized intertwining operator and
$$
 \ord_{s=s_0} N_\psi(s,\pi, \sigma)
 \leq \ord_{s=s_0} N_{\psi'}(s,\pi, \sigma)
$$
for any $s_0\in\mathbb{R}_{\geq\frac{1}{2}} $ and any $\psi' \in \Psi(\pi)$ (\cite[Proposition 4.1]{Moe12}). In this case, it is expected that $N_\psi(s,\pi,\sigma)$ is nonzero on a certain right half plane (e.g., \cite[Lemma 4.2]{Kim05}).
However, if $\psi'\neq \psi$, then $N_{\psi'}(s,\pi, \sigma)$ may vanish at some point $s_0 \geq \frac{1}{2}$, that is, the normalizing factor may overkill the poles of the intertwining operator (see Theorem \ref{thm M vanishing}, Corollary \ref{cor van N ordering}). Therefore, an interesting natural question is as follows.

\begin{ques}\label{ques zeros of intertwining operators}
Given $\pi \in \Pi_A(G_n)$, if 
$\pi\notin\Pi_{\phi_\psi}$, for any $\psi \in \Psi(\pi)$, then which local Arthur parameter in $\Psi(\pi)$ gives the least order of zeros for $N_\psi(s,\pi,\sigma)$?
\end{ques}

In Theorem \ref{thm 5 orderings} above (see Theorem \ref{thm Moeglin}), we show that the answer for Question \ref{ques zeros of intertwining operators} is exactly that $\psi=\psi^{max}(\pi)$. This result shows the significance of the unique maximal element $\psi^{max}(\pi)$ in the theories of local Arthur packets and automorphic forms.

From the discussions above, we can see that given a representation $\pi$ of Arthur type, $\psi^{max}(\pi)$ is entitled to be ``the" local Arthur parameter for $\pi$, no matter if $\pi$ belongs to any local $L$-packet of Arthur type or not.


\subsection{Further applications}\label{sec: enhanced shahidi's conjecture intro}

In this subsection, we list further applications of the results in this paper, which will appear in our subsequent work.\\

\noindent\textbf{Non-containment of local Arthur packets:}

As an application of the Theorem \ref{thm 5 orderings}(2) in the case of
ordering $\geq_O$, jointly with Zhang (\cite{HLLZ22}), we showed the non-containment of local Arthur packets for $\mathrm{G}_n$ being $\Sp_{2n}$ or split $\SO_{2n+1}$. That is, local Arthur packets of $G_n$ can not be fully contained in other ones, which is in contrast to the situation over Archimedean local fields. \\

\noindent\textbf{Enhanced Shahidi Conjecture:}

As an application of Theorem \ref{main thm intro}(3), in \cite{HLL24}, we proved the enhanced Shahidi conjecture for $G_n$, which strengthens the well-known Shahidi conjecture. More precisely, the enhanced Shahidi conjecture states that for any quasi-split reductive group $G$, assuming that there is a local Arthur packets theory for $G$ as conjectured in \cite[\S 6]{Art89}, then 
\begin{enumerate}
    \item For any local Arthur parameter $\psi \in \Psi(G)$, the local Arthur packet $\Pi_{\psi}$ is tempered if and only if it has a generic member.
    \item For any local Arthur parameter $\psi \in \Psi^+(G)$, the local Arthur packet $\Pi_{\psi}$ is generic if and only if it has a generic member.
\end{enumerate}
This result was then used to prove the following conjecture on local components of automorphic forms for $G_n$. 

\begin{conj}[{\cite[Conjecture 1.5]{HLL24}}]
\label{conj generic almost all intro}
    Let $\pi$ be an automorphic representation in the discrete spectrum of a reductive group $\mathrm{G}$. Suppose there exists a finite place $v_0$ of $k$ such that $\pi_{v_0}$ is generic, then $\pi_v$ is generic for almost all places.
\end{conj}

\noindent\textbf{Unramified representations of Arthur type:}

Considering the importance of unramified representations in the theory of automorphic forms and automorphic representations, it is desirable to study more closely the unramified representations of Arthur type. As further application of the results in this paper, in \cite{HLL24}, we gave a characterization of unramified representations of $G_n$ of Arthur type in terms of their $L$-data, making use of Algorithm \ref{alg Arthur type}. This result was then used to prove the following conjectures for $G_n$.

\begin{conj}\label{conj unram intro}
Let $\mathrm{G}$ be a connected reductive group defined over a non-Archimedean local field $F$. Assume that there is a local Arthur packets theory for $G$ as conjectured in \cite[\S 6]{Art89}.  
\begin{enumerate}
    \item $(${\cite[Conjecture 1.9]{HLL24}}$)$. Any unramified representation of $G(F)$ of Arthur type lies in exactly one local Arthur packet. Moreover, it lies in the $L$-packet associated to an anti-tempered local Arthur packet.
    \item $(${\cite[Conjecture 2A]{Clo07}}$)$. Suppose that $\psi$ is a local Arthur parameter of $G(F)$ such that $\phi_\psi$ is unramified. Then the local Arthur packet $\Pi_\psi$ contains a unique unramified representation. More specifically, the unramified representation is the one associated to $\phi_\psi$ via the Satake isomorphism $($\cite{Sat63}$)$.
    \item $($Arthur, Clozel, \cite[Conjecture 6.1]{Sha11}$)$.
Let $\mathrm{G}$ be a connected reductive group defined over $k$ and $\pi=\otimes_v \pi_v$ be a cuspidal automorphic representation of $\mathrm{G}(\mathbb{A}_k).$
Then, for almost all finite places $v$, we have $\pi_v\in\Pi_{\phi_{\psi_v}}.$ 
\end{enumerate}
\end{conj}

Based on the Sage code shared by Atobe on his reformulation of local Arthur packets, we provide the detailed code for all the algorithms in this paper, which has been posted publicly at Github:\\
\url{https://github.com/ChiHengLo/Intersection-of-local-Arthur-packets}

Following is the structure of this paper. In \S \ref{sec nota and prel}, we recall necessary notation and preliminaries. In \S \ref{sec atob refor}, we recall Atobe's reformulation on M{\oe}glin's construction of local Arthur packets, the non-vanishing criterion; define the operators of row exchange and union-intersection; and introduce the Aubert-Zelevinsky dual formula of Atobe on extended multi-segments. In \S \ref{sec shift and add}, first, we prove several technical results related to the operator of ``shift". Then we show the relation between 
the $L$-data of $\pi(\EE)$ and $\EE$, and determine certain invariants from the $L$-data of  $\pi(\EE)$ which narrow down the set of all possible extended multi-segments $\EE'$ such that  $\pi(\EE')\cong\pi(\EE).$ In \S \ref{sec union-intersection}, we define a generalized version of the union-intersection operator, and show that for each
extended multi-segment $\EE$ such that $\pi(\EE)\neq 0$, among the set of all extended multi-segments obtained from $\EE$ by successively applying union-intersections and row exchanges, there exists a unique (up to row exchanges) minimal element which carries rich derivative information of $\pi(\EE)$. In \S \ref{sec main thm}, we prove our main results, Theorem \ref{main thm intro}. In \S \ref{sec can
form}, we define and construct the canonical form $\EE_{can}$ of $\EE$, and study its properties. Then we give an algorithm to determine whether a representation $\pi$ of good parity is of Arthur type or not. In the affirmative case, we give a formula to exhaust the set $\{ \EE\ | \ \pi(\EE)=\pi\}$.
In \S \ref{section L-packet}, we give a criterion on $\EE$ such that $\pi(\EE)$ is in the local $L$-packet of Arthur type and prove Theorem \ref{thm in L-packet intro}. In \S \ref{section the Arthur parameter}, we construct the max form $\EE^{|max|}$ of $\EE$ and prove Theorem \ref{thm can max intro}. Then, we define $\psi^{max}(\pi)=\psi_{\EE^{|max|}}$ to be ``the" local Arthur parameter for any irreducible representation $\pi$ of Arthur type.   
In \S \ref{characterizations}, we show that $\psi^{max}(\pi)$ and $\psi^{min}(\pi)$ are exactly the unique maximal and minimal elements in $\Psi(\pi)$ under the two orderings $\geq_A$ and $\geq_N$. In particular, 
 $\psi^{max}(\pi) \in \Psi(\pi)$ gives the least order of zeros for $N_{\psi}(s,\pi,\sigma)$.

\subsection*{Acknowledgements} 
The authors would like to thank Dihua Jiang and Freydoon Shahidi for their interests, constant support, and helpful discussions. 
The authors also would like to thank Hiraku Atobe for helpful communications on certain statements in his paper \cite{Ato20b} (Theorems \ref{prop uniform} and \ref{thm shift left}), helpful comments and suggestions, and for generously sharing his sage code on the reformulation of M{\oe}glin construction of local Arthur packets. Finally, the authors would like to thank Wee Teck Gan, Tasho Kaletha, Anantharam Raghuram, David Renard, Bin Xu, and Lei Zhang for helpful comments and suggestions.

\section{Notation and preliminaries}\label{sec nota and prel}

Let $F$ be a local non-Archimedean field of characteristic $0$ with normalized absolute value $|\cdot|.$ We also regard $|\cdot|$ as a character of $\GL_n(F)$ via composition with the determinant. Set $G_n$ to be split $\SO_{2n+1}(F)$ or $\Sp_{2n}(F).$ We write $\Pi(G)$ for the set of equivalence classes of irreducible smooth representations of a group $G.$ We assume that every representation is smooth. 

Suppose $\Pi_1, \Pi_2$ are representations of finite length. We denote $[\Pi_1]$ the image of $\Pi_1$ in the Grothendieck group. We write $\Pi_1 \geq \Pi_2$ if $[\Pi_1]-[\Pi_2]$ is a non-negative linear combination of irreducible representations.

For a multi-set $X$ and $a\in X$, we let $m_X(a)$ denote the multiplicity of $a$ in $X.$ Let $Y$ be another multi-set. We define the following multi-sets $Z$ by specifying the multiplicity $m_Z(a)$ for each $a$ such that $m_X(a)+m_{Y}(a)>0$.
\begin{enumerate}
    \item [$\oldbullet$] The sum of multi-sets $X+Y$: $m_{X+Y}(a):=m_X(a)+m_Y(a) $.
    \item [$\oldbullet$] The union of multi-sets $X\cup Y$: $m_{X\cup Y}(a):=\max(m_X(a),m_Y(a)).$
    \item [$\oldbullet$] The difference of multi-sets $X\setminus Y$: $m_{X\setminus Y}(a):=\max(m_X(a)-m_Y(a),0).$
    \item [$\oldbullet$] The intersection of multi-sets $X \cap Y$: $m_{X\cap Y}(a):=\min(m_X(a),m_Y(a)).$
\end{enumerate}
Finally, we define the symmetric difference of multi-sets by $X\Delta Y:= (X\cup Y)\setminus(X\cap Y).$

\subsection{Langlands classification}
In this subsection, we recall the Langlands classification for $\GL_n(F)$ and $G_n$. See \cite{Kon03} for the general setup for general reductive $p$-adic groups. 

Let $n$ be a positive integer and fix a Borel subgroup of $\GL_n(F)$. Let $P$ be a standard parabolic subgroup of $\GL_n(F)$ with Levi subgroup $M\cong \GL_{n_1}(F)\times\cdots\times\GL_{n_r}(F).$ Let $\tau_i\in \Pi(\GL_{n_i}(F))$ for $i=1,2,\dots,r.$ We set
$$
\tau_1\times\cdots\times\tau_r := \mathrm{Ind}_{P}^{\GL_n(F)}(\tau_1\otimes\cdots\otimes\tau_r)
$$
to be the normalized parabolic induction. We define a \emph{segment}, denoted by $[x,y]_\rho$, to be a set of supercuspidal representations of the form
$$
[x,y]_\rho=\{\rho|\cdot|^x, \rho|\cdot|^{x-1},\dots,\rho|\cdot|^y\}
$$
where $\rho$ is an irreducible unitary supercuspidal representation of $\GL_n(F)$ and $x,y\in\mathbb{R}$ such that $x-y$ is a non-negative integer. We denote a \emph{Steinberg representation} attached to the segment $[x,y]_\rho$ by $\Delta_\rho[x,y].$ This is the unique irreducible subrepresentation of $\rho|\cdot|^x\times\cdots\times\rho|\cdot|^y.$ It is an essentially discrete series representation of $\GL_{n(x-y+1)}(F).$ When it is clear in context, we refer to both $[x,y]_\rho$ and $\Delta_\rho[x,y]$ as segments. We also set $Z_\rho[y,x]$ to be the unique irreducible quotient of $\rho|\cdot|^x\times\cdots\times\rho|\cdot|^y.$ In the case $y=x+1,$ we set $\Delta_\rho[x,x+1]=Z_\rho[x+1,x]$ to be the trivial representation of $\GL_0(F).$

We say two segments $[x_1,y_1]_{\rho_1}$ and $ [x_2,y_2]_{\rho_2}$ are \emph{linked} if $\rho_1\cong\rho_2,$ the set $[x_1,y_1]_{\rho_1} \cup [x_2,y_2]_{\rho_2}$ is again a segment, and neither $ [x_1,y_1]_{\rho_1} \supseteq [x_2,y_2]_{\rho_2}$ nor $[x_1,y_1]_{\rho_1} \subseteq [x_2,y_2]_{\rho_2}$. The following lemma describes when the product of two segments commute.
\begin{lemma}[{\cite[Theorem 1.1]{Tad14}, \cite[Corollary 6.10]{LM16}}]\label{lem commutativity}
Suppose the segments $[x_1,y_1]_{\rho_1}$ and $[x_2,y_2]_{\rho_2}$ are not linked. Then $\Delta_{\rho_1}[x_1,y_1] \times \Delta_{\rho_2}[x_2,y_2]$ and $    Z_{\rho_1}[y_1,x_1] \times Z_{\rho_1}[y_2,x_2] $ are both irreducible, and 
\begin{align*}
        \Delta_{\rho_1}[x_1,y_1] \times \Delta_{\rho_2}[x_2,y_2] &=        \Delta_{\rho_2}[x_2,y_2] \times \Delta_{\rho_1}[x_1,y_1],\\
                Z_{\rho_1}[y_1,x_1] \times Z_{\rho_2}[y_2,x_2] &=        Z_{\rho_2}[y_2,x_2] \times Z_{\rho_1}[y_1,x_1].
\end{align*}
\end{lemma}

The Langlands classification for $\GL_n(F)$ states that any irreducible representation $\tau$ of $\GL_n(F)$ can be realized as a unique irreducible subrepresentation of a parabolic induction of the form
\[\Delta_{\rho_1}[x_1,y_1]\times\cdots\times\Delta_{\rho_r}[x_r,y_r],\]
where $\rho_i$ is an irreducible unitary supercuspidal representation of $\GL_{n_i}(F),$ $[x_i,y_i]_{\rho_i}$ is a segment, and $x_1+y_1\leq\cdots\leq x_r+y_r.$ In this setting, we write
$$
\tau=L(\Delta_{\rho_1}[x_1,y_1],\dots,\Delta_{\rho_r}[x_r,y_r]).
$$

Let $(x_{i,j})_{1\leq i\leq s, 1\leq j \leq t}$ be real numbers such that $x_{i,j}=x_{1,1}-i+j.$ We define a \emph{(shifted) Speh representation} to be the irreducible representation given by
$$
\begin{pmatrix}
x_{1,1} & \cdots & x_{1,t} \\
\vdots & \ddots & \vdots \\
x_{s,1} & \cdots & x_{s,t}
\end{pmatrix}_{\rho}:=L(\Delta_{\rho}[x_{1,1},x_{s,1}],\dots,\Delta_{\rho}[x_{1,t}, x_{s,t}]).
$$

Fix a Borel subgroup of $G_n$ and let $P$ be a standard parabolic subgroup of $G_n$ with Levi subgroup $M\cong\GL_{n_1}(F)\times\cdots\times\GL_{n_r}(F)\times G_{m}.$ Let $\tau_i$ be a representation of $\GL_{n_i}(F)$ for $i=1,2,\dots,r$ and $\sigma$ be a representation of $G_{m}$. We set
$$
\tau_1\times\cdots\times\tau_r\rtimes\sigma := \mathrm{Ind}_{P}^{G_n}(\tau_1\otimes\cdots\otimes\tau_r\otimes\sigma)
$$
to be the normalized parabolic induction. 

The Langlands classification for $G_n$ states that every irreducible representation $\pi$ of $G_n$ is a unique irreducible subrepresentation of 
\[\Delta_{\rho_1}[x_1,y_1]\times\cdots\times\Delta_{\rho_r}[x_r,y_r]\rtimes\pi_{temp},\]
where $\rho_i$ is an irreducible unitary supercuspidal representation of $\GL_{n_i}(F),$ $x_1+y_1\leq\cdots\leq x_r+y_r<0,$ and $\pi_{temp}$ is an irreducible tempered representation of $G_{m}.$ In this case, we write
$$
\pi=L(\Delta_{\rho_1}[x_1,y_1],\dots,\Delta_{\rho_r}[x_r,y_r];\pi_{temp})
$$
and call $(\Delta_{\rho_1}[x_1,y_1],\dots,\Delta_{\rho_r}[x_r,y_r];\pi_{temp})$ the Langlands data, or $L$-data, of $\pi.$ In Section \ref{sec tempered parametrization}, We give more detailed parametrization of the tempered representation $\pi_{temp}$ using Arthur's theory 

\subsection{Derivative and socle}
\label{sec: der and socle}

Let $\pi$ be a smooth representation of $G_n$ of finite length. We let $Jac_{P}(\pi)$ be the Jacquet module of $\pi$ with respect to a parabolic subgroup $P$ of $G_n.$  We  denote the semisimplification of $Jac_{P}(\pi)$ by $[Jac_{P}(\pi)].$

\begin{defn}
Let $P_d$ be a standard parabolic subgroup of $G_n$ with Levi subgroup isomorphic to $\GL_{d}(F)\times G_{n-d},$ $x\in\mathbb{R},$ and $\rho$ be an irreducible unitary self-dual supercuspidal representation of $\GL_d(F).$ We define the $\rho|\cdot|^x$-derivative of $\pi$, denoted $D_{\rho|\cdot|^x}(\pi),$ to be a semisimple representation satisfying
$$
[Jac_{P_d}(\pi)]=\rho|\cdot|^x\otimes D_{\rho|\cdot|^x}(\pi) + \sum_i \tau_i\otimes\pi_i,
$$
where the sum is over all irreducible representations $\tau_i$ of $\GL_d(F)$ such that $\tau_i\not\cong\rho|\cdot|^x.$
\end{defn}

While we call these derivatives, one should be careful not to confuse these with the Bernstein-Zelevinsky derivatives defined in \cite{BZ76}. We set $D_{\rho|\cdot|^{x}}^{(0)}(\pi)=\pi$ and for $k$ a non-negative integer, we define recursively
$$
D_{\rho|\cdot|^{x}}^{(k)}(\pi)=\frac{1}{k}D_{\rho|\cdot|^{x}}\circ D_{\rho|\cdot|^{x}}^{(k-1)}(\pi).
$$
If $D_{\rho|\cdot|^{x}}^{(k)}(\pi)\neq 0$, but $D_{\rho|\cdot|^{x}}^{(k+1)}(\pi)=0,$ then we say that $D_{\rho|\cdot|^{x}}^{(k)}(\pi)$ is the \emph{highest $\rho|\cdot|^{x}$-derivative} of $\pi$. If $D_{\rho|\cdot|^{x}}(\pi)=0,$ then say that $\pi$ is \emph{$\rho|\cdot|^{x}$-reduced}.

We also need to define derivatives for $\GL_n(F).$ However, in this situation, we must distinguish between left and right derivatives. We follow \cite[\S5]{Xu17a}.

\begin{defn}
Let $P_d$ (resp. $Q_d$) be a standard parabolic subgroup of $\GL_n(F)$ with Levi subgroup isomorphic to $\GL_{d}(F)\times \GL_{n-d}(F)$ (resp. $\GL_{n-d}(F)\times \GL_{d}(F)$), $x\in\mathbb{R},$ $\sigma$ be a smooth representation of $\GL_n(F)$, and $\rho$ be an irreducible unitary self-dual supercuspidal representation of $\GL_d(F).$ We define the left (resp. right) $\rho|\cdot|^x$-derivative of $\sigma$, denoted $D_{\rho|\cdot|^x}(\sigma)$ (resp. $D_{\rho|\cdot|^x}^{op}(\sigma)$), to be a semisimple representation satisfying
$$
[Jac_{P_d}(\sigma)]=\rho|\cdot|^x\otimes D_{\rho|\cdot|^x}(\sigma) + \sum_i \tau_i\otimes\sigma_i,
$$

$$
\left( \text{resp.  } [Jac_{Q_d}(\sigma)]= D_{\rho|\cdot|^x}^{op}(\sigma)\otimes\rho|\cdot|^x + \sum_i \sigma_i\otimes\tau_i, \right)
$$
where the sum is over all irreducible representations $\tau_i$ of $\GL_d(F)$ such that $\tau_i\not\cong\rho|\cdot|^x.$
\end{defn}

Note that the right derivative defined in \cite[\S5]{Xu17a} uses the contragredient; however, for our purposes, we are only concerned with derivatives for self-dual $\rho.$ 

Atobe and M{\' i}nguez gave explicit formulas for computing the highest nonzero derivative in the good parity case which we will use frequently (\cite[Proposition 6.1, Theorem 7.1]{AM20}). We also note that these derivatives satisfy the Leibniz rules below.

\begin{lemma}[{\cite[\S5]{Xu17a}}] \label{lem Leibniz rule}
Let $\rho$ be an irreducible unitary self-dual supercuspidal representation of $\GL_d(F)$ and $x\in\mathbb{R}.$
\begin{enumerate}
    \item [1.] For  $\sigma,\sigma_1,\sigma_2 \in \Pi(\GL_n(F)), \ \tau \in \Pi(G_m)$, we have
    \[ D_{\rho|\cdot|^x}( \sigma \rtimes \tau)= D_{\rho|\cdot|^x}(\sigma) \times \tau + D_{\rho|\cdot|^{-x}}^{op}(\sigma) \times \tau + \sigma \rtimes D_{\rho|\cdot|^x}(\tau)\]
    \begin{align*}
    D_{\rho|\cdot|^x}(\sigma_1 \times \sigma_2)&= D_{\rho|\cdot|^x}(\sigma_1) \times \sigma_2+ \sigma_1 \times D_{\rho|\cdot|^x}(\sigma_2)\\
    D_{\rho|\cdot|^{x}}^{op}(\sigma_1 \times \sigma_2)&= D_{\rho|\cdot|^{x}}^{op}(\sigma_1) \times \sigma_2+ \sigma_1 \times D_{\rho|\cdot|^{x}}^{op}(\sigma_2)
\end{align*} 
\item [2.] For $a\geq b$,
\begin{align*}
      D_{\rho|\cdot|^x}(\Delta_{\rho}[a,b])&= \begin{cases} 0 & \text{ if } x \neq a,\\
\Delta_{\rho}[a-1,b] & \text{ if }x=a.\end{cases} \\
      D_{\rho|\cdot|^x}^{op}(\Delta_{\rho}[a,b])&= \begin{cases} 0 & \text{ if } x \neq b,\\
\Delta_{\rho}[a,b+1] & \text{ if }x=b.\end{cases} \\
D_{\rho|\cdot|^x}(Z_{\rho}[b,a])&= \begin{cases} 0 & \text{ if } x \neq b,\\
Z_{\rho}[b+1,a] & \text{ if }x=b.\end{cases} \\
      D_{\rho|\cdot|^x}^{op}(Z_{\rho}[b,a])&= \begin{cases} 0 & \text{ if } x \neq a,\\
Z_{\rho}[b,a-1] & \text{ if }x=a.\end{cases} 
\end{align*}
\item [3.] $D_{\rho|\cdot|^{x}}$ commutes with $D_{\rho|\cdot|^{y}}$ if $|x-y|>1$. 
\end{enumerate}
\end{lemma}
For a multi-set of real numbers $\{x_1,\dots,x_r\}$, we denote the composition of derivatives by
\[ D_{\rho|\cdot|^{x_1,\dots ,x_r}}(\pi):=D_{\rho|\cdot|^{x_r}} \circ \cdots \circ D_{\rho|\cdot|^{x_1}}(\pi).\]
For example, if $\{x,\dots,x\}$ contains $k$ copies of $x$, then
\[ D_{\rho|\cdot|^{x,\dots,x} }(\pi)= (k!)\cdot D_{\rho|\cdot|^{x}}^{(k)}(\pi). \]
If the $D_{\rho|\cdot|^{x}}^{(k)}(\pi)$ in the right hand side is the highest derivative, then we say the left hand side is a highest derivative up to a scalar.
\begin{lemma}\label{lem Frobenius}
Suppose $\rho$ is a self-dual supercuspidal representation of $\GL_d$, and $\pi$ is an irreducible representation of $G_{(n+dt)}$. Then there exists a representation $\sigma$ of $G_n$ such that
\[ \pi \hookrightarrow \rho|\cdot|^{x_1} \times \cdots \times \rho|\cdot|^{x_t} \rtimes \sigma \]
if and only if
\[ D_{\rho|\cdot|^{x_1,\dots,x_t}}( \pi) \neq 0. \]
In this case, if $\sigma$ has a unique irreducible subrepresentation $\sigma'$, then 
\[   D_{\rho|\cdot|^{x_1,\dots,x_t}}( \pi) \geq \sigma'.\]
\end{lemma}

\begin{proof}
Suppose $ D_{\rho|\cdot|^{x_1,\dots,x_t}}( \pi)\neq 0$. By \cite[Lemma 5.3]{Xu17a}, there exists an irreducible representation $\sigma$ of $G_n$ which satisfies the claim. Conversely, let $P$ be the standard parabolic subgroup of $G$ whose Levi subgroup is isomorphic to $\GL_d^t(F) \times G_n$. By  Frobenius reciprocity,
\begin{align*}
    0 &\neq \mathrm{Hom}( \pi, \rho|\cdot|^{x_1} \times \cdots \times \rho|\cdot|^{x_t} \rtimes \sigma )\\
    &= \mathrm{Hom}( Jac_P(\pi), \rho|\cdot|^{x_1} \otimes \cdots \otimes \rho|\cdot|^{x_t} \otimes \sigma).
\end{align*}
From left exactness of the Hom functor and the fact that $Jac_{P}(\pi)$ is of finite length, there exists an irreducible subquotient $\tau$ of $Jac_{P}(\pi)$ such that
\[ 0 \neq \mathrm{Hom}(\tau,\rho|\cdot|^{x_1} \otimes \cdots\otimes \rho|\cdot|^{x_t} \otimes \sigma). \]
Then $\tau$ is isomorphic to $\rho|\cdot|^{x_1} \otimes \cdots \otimes \rho|\cdot|^{x_t} \otimes \sigma'$, where $\sigma'$ is an irreducible subrepresentation of $\sigma$. This shows $ D_{\rho|\cdot|^{x_1,\dots ,x_t}}( \pi) \geq \sigma'$ is nonzero.
\end{proof}

Let $\pi$ be a representation of finite length. We define the \emph{socle} of $\pi$, denoted by $soc(\pi)$, to be the maximal semisimple subrepresentation of $\pi.$

\begin{defn}
Let $\pi$ be a representation of finite length, $x\in\mathbb{R},$ and $\rho$ be an irreducible unitary self-dual supercuspidal representation of $\GL_d(F).$ We define
$$
S_{\rho|\cdot|^{x}}^{(r)}(\pi):=soc((\rho|\cdot|^{x})^r\rtimes\pi).
$$
\end{defn}
Similar as derivatives, for a sequence of real number $\{x_1,\dots,x_r\}$, we define the composition of socles by
\[ S_{\rho|\cdot|^{x_1,\dots ,x_r}}(\pi):=S_{\rho|\cdot|^{x_r}} \circ \cdots \circ S_{\rho|\cdot|^{x_1}}(\pi).\]
\begin{thm}[{\cite[Lemma 3.1.3]{Jan14}, \cite[Propositions 3.3, 6.1, Theorem 7.1]{AM20}}]\label{thm derivative-socle}
Let $\rho$ be an irreducible unitary self-dual supercuspidal representation of $\GL_d(F),$ $\pi\in \Pi(G_n),$ and $x\in\mathbb{R}\setminus\{0\}.$ For any non-negative integers $k$ and $r$, we have the following. 
\begin{enumerate}
    \item The highest $\rho|\cdot|^{x}$-derivative of $\pi$, say $D_{\rho|\cdot|^{x}}^{(k)}(\pi),$ is irreducible.
    \item $S_{\rho|\cdot|^{x}}^{(r)}(\pi)$ is irreducible for any $r\geq 0.$
    \item 
    $$
    S_{\rho|\cdot|^{x}}^{(k)}(D_{\rho|\cdot|^{x}}^{(k)}(\pi))=\pi,
    $$
    and, 
    $$
    D_{\rho|\cdot|^{x}}^{(k+r)}(S_{\rho|\cdot|^{x}}^{(r)}(\pi))=D_{\rho|\cdot|^{x}}^{(k)}(\pi).
    $$
    \item The $L$-data of $D_{\rho|\cdot|^{x}}^{(k)}(\pi)$ and $S_{\rho|\cdot|^{x}}^{(k)}(\pi)$ can be explicitly described in terms of those of $\pi.$
\end{enumerate}
\end{thm}

When $x=0,$ computing the $\rho$-derivative is generally a problematic endeavor. As a remedy, Atobe and M{\' i}nguez considered the $\Delta_\rho[0,-1]$-derivative and $Z_\rho[0,1]$-derivative, denoted by $D_{\Delta_\rho[0,-1]}^{(k)}(\pi)$ and $D_{Z_\rho[0,1]}^{(k)}(\pi)$, respectively. These are semisimple representations of $G_{n-2dk}$  defined by
$$
[Jac_{P_{2dk}}(\pi)]=\Delta_\rho[0,-1]^k\otimes D_{\Delta_\rho[0,-1]}^{(k)}(\pi)+Z_\rho[0,1]^k\otimes D_{Z_\rho[0,1]}^{(k)}(\pi) + \sum_i \tau_i\otimes\pi_i,
$$
where the sum is over all irreducible representations $\tau_i$ of $\GL_{2dk}(F)$ such that $\tau_i$ is neither isomorphic to $\Delta_\rho[0,-1]^k$ nor $Z_\rho[0,1]^k.$ We also define
$$
S_{\Delta_\rho[0,-1]}^{(r)}(\pi):=soc(\Delta_\rho[0,-1]^r\rtimes\pi), \, \, S_{Z_\rho[0,1]}^{(r)}(\pi):=soc(Z_\rho[0,1]^r\rtimes\pi).
$$
These derivatives and socles satisfy similar results as in Theorem \ref{thm derivative-socle}.

\begin{thm}[{\cite[Proposition 3.7]{AM20}}]\label{thm self-dual derivative}
Let $\rho$ be an irreducible unitary self-dual supercuspidal representation of $\GL_d(F)$ and $\pi\in \Pi(G_n).$ Assume that $\pi$ is $\rho|\cdot|\inv$-reduced (respectively $\rho|\cdot|$-reduced). Then the results of Theorem \ref{thm derivative-socle}(1), (2), and (3) hold with $\rho|\cdot|^x$ replaced by $\Delta_\rho[0,-1]$ (respectively $Z_\rho[0,1]$).
\end{thm}

\subsection{Local Arthur packet}
Recall that a local Arthur parameter is
a direct sum of irreducible representations
$$\psi: W_F \times \SL_2(\mathbb{C}) \times \SL_2(\mathbb{C}) \rightarrow \widehat{G}_n(\BC)$$
\begin{equation}\label{eq decomp psi +}
  \psi = \bigoplus_{i=1}^r \phi_i|\cdot|^{x_i} \otimes S_{a_i} \otimes S_{b_i},  
\end{equation}
satisfying the following conditions:
\begin{enumerate}
    \item [(1)]$\phi_i(W_F)$ is bounded and consists of semi-simple elements, and $\dim(\phi_i)=d_i$;
    \item [(2)] $x_i \in \R$ and $|x_i|<\half{1}$;
    \item [(3)]the restrictions of $\psi$ to the two copies of $\SL_2(\mathbb{C})$ are analytic, $S_k$ is the $k$-dimensional irreducible representation of $\SL_2(\mathbb{C})$, and 
    $$\sum_{i=1}^r d_ia_ib_i = N:= 
\begin{cases}
2n+1 & \text{ when } G_n=\Sp_{2n}(F),\\
2n & \text{ when } G_n=\SO_{2n+1}(F).
\end{cases}
$$ 
\end{enumerate}
We remark that the bound $|x_i|<\half{1}$ follows from the trivial bound of the Ramanujan Conjecture.

Two local Arthur parameters are equivalent if they are conjugate under $\widehat{G}_n(\BC)$. We do not distinguish $\psi$ and its equivalence class in the rest of the paper. We denote $\Psi^{+}_{\frac{1}{2}}(G_n)$ the equivalence class of local Arthur parameter, and $\Psi(G_n)$ the subset of $\Psi^+_{\frac{1}{2}}(G_n)$ consisting of local Arthur parameters $\psi$ whose restriction to $W_F$ is bounded. In other words, $\psi$ is in $\Psi(G_n)$ if and only if $x_i=0$ for $i=1,\dots, r$ in the decomposition \eqref{eq decomp psi +}. For simplicity, we omit the subscript $\frac{1}{2}$ and write 
$\Psi^{+}(G_n):=\Psi^{+}_{\frac{1}{2}}(G_n)$.

By the Local Langlands Correspondence for $\GL_{d_i}(F)$, the bounded representation $\phi_i$ of $W_F$ can be identified with an irreducible unitary supercuspidal representation $\rho_i$ of $\GL_{d_i}(F)$ (\cite{Hen00, HT01, Sch13}). Consequently, we may write
\begin{equation}\label{A-param decomp}
  \psi = \bigoplus_{\rho}\left(\bigoplus_{i\in I_\rho} \rho|\cdot|^{x_i} \otimes S_{a_i} \otimes S_{b_i}\right),  
\end{equation}
where first sum runs over 
irreducible unitary supercuspidal representations $\rho$ of $\GL_n(F)$, $n \in \mathbb{Z}_{\geq 1}$. Occasionally, we write $\rho|\cdot|^{x}\otimes S_a=\rho|\cdot|^x\otimes S_a \otimes S_1.$ 

Let $\psi$ be a local Arthur parameter as in \eqref{A-param decomp}, we say that $\psi$ is of \emph{good parity} if $\psi \in \Psi(G_n)$ (i.e. $x_i=0$ for all $i$) and every summand $\rho \otimes S_{a_i} \otimes S_{b_i}$ is self-dual and of the same type as $\psi.$ That is, $\rho$ is self dual and
\begin{itemize}
    \item if $G_n=\Sp_{2n}(F)$ and $\rho$ is orthogonal (resp. symplectic), then $a_i+b_i$ is even (resp. odd);
    \item if $G_n=\SO_{2n+1}(F)$ and $\rho$ is orthogonal (resp. symplectic), then $a_i+b_i$ is odd, (resp. even).
\end{itemize}
We denote $\Psi_{gp}(G_n)$ the subset of $\Psi(G_n)$ consists of local Arthur parameters of good parity.

Let $\psi \in \Psi^{+}(G_n).$ From the decomposition \eqref{A-param decomp}, we define a subrepresentation $\psi_{nu,>0}$ of $\psi$ by
\[ \psi_{nu,>0}:= \bigoplus_{\rho}\left(\bigoplus_{\substack{i\in I_\rho,\\ x_i>0}} \rho|\cdot|^{x_i} \otimes S_{a_i} \otimes S_{b_i}\right). \]
Since the image of $\psi$ is contained in  $\widehat{G}_n(\BC)$, $\psi$ is self-dual, and hence $\psi$ also contains $(\psi_{nu,>0})^{\vee}$. We define $\psi_{u} \in \Psi(G_m)$ for some $m\leq n$ by 
\begin{align}\label{eq def of psi_u}
    \psi= \psi_{nu,>0} \oplus \psi_u \oplus (\psi_{nu,>0})^{\vee}.
\end{align}
Equivalently,
\[ \psi_{u}:= \bigoplus_{\rho}\left(\bigoplus_{\substack{i\in I_\rho,\\ x_i=0}} \rho \otimes S_{a_i} \otimes S_{b_i}\right). \]

In \cite{Art13}, for a local Arthur parameter $\psi \in \Psi(G_n)$, Arthur constructed a finite multi-set $\Pi_\psi$ consisting of irreducible unitary representations of $G_n.$ We call $\Pi_\psi$ the \emph{local Arthur packet} of $\psi.$ M{\oe}glin showed that $\Pi_\psi$ is multiplicity-free (\cite{Moe11}). For $\psi \in \Psi^+(G_n)$, Arthur defined (\cite[(1.5.1)]{Art13})
\begin{align}\label{eq def packet +}
    \Pi_{\psi}:= \{ \tau_{\psi_{nu,>0}} \rtimes \pi_u \ | \ \pi_{u} \in \Pi_{\psi_u}   \},
\end{align}
where $\tau_{\psi_{nu,>0}}$ is the following irreducible representation of $\GL$
$$
\tau_{\psi_{nu,>0}}=\bigtimes_\rho\bigtimes_{i\in I_\rho}\begin{pmatrix}
\frac{a_i-b_i}{2}+x_i & \cdots & \frac{a_i+b_i}{2}-1+x_i \\
\vdots & \ddots & \vdots \\
\frac{-a_i-b_i}{2}+1+x_i & \cdots & \frac{b_i-a_i}{2}+x_i
\end{pmatrix}_{\rho}.
$$
Since $|x_i|<\half{1}$ in the decomposition \eqref{A-param decomp}, the parabolic induction in \eqref{eq def packet +} is always irreducible (this follows from \cite[Theorem 9.3(6)]{Jan97}, \cite[Proposition 3.2(i)]{Tad09}, and is also proved in \cite[Proposition 5.1]{Moe11b}; see Theorem \ref{thm red from nu to gp} below). We say that an irreducible representation $\pi$ of $G_n$ is \emph{of Arthur type} if $\pi\in\Pi_\psi$ for some local Arthur parameter $\psi \in \Psi^+(G_n)$.

Next, we further decompose $\psi_{u}$. Suppose $\rho \otimes S_a \otimes S_b$ is an irreducible summand of $\psi_u$ that is either not self-dual, or self-dual but not of the same type as $\psi$. It follows that $\psi$ must contain the other summand $(\rho\otimes S_a\otimes S_b)^{\vee}=\rho^{\vee} \otimes S_a \otimes S_b$. Therefore, we may choose a subrepresentation $\psi_{np}$ of $\psi_u$ such that
\begin{align}\label{eq decomp of psi_u}
     \psi_{u}= \psi_{np} \oplus \psi_{gp} \oplus \psi_{np}^{\vee},
\end{align}
where $\psi_{gp} $ is of good parity, and any irreducible summand of $\psi_{np}$ is either not self-dual or self-dual but not of the same type as $\psi$. In \cite{Moe06a}, M{\oe}glin constructed the local Arthur packet $\Pi_{\psi_u}$ from $\Pi_{\psi_{gp}}$, which we record below.

\begin{thm}[{\cite[Theorem 6]{Moe06a}, \cite[Proposition 8.11]{Xu17b}}]\label{thm reduction to gp}
Let $\psi_u \in \Psi(G_n)$ with a choice of decomposition \eqref{eq decomp of psi_u}. Write
\[ \psi_{np}= \bigoplus_{\rho} \left( \bigoplus_{i \in I_{\rho}} \rho \otimes S_{a_i} \otimes S_{b_i}\right),  \]
and consider the following irreducible parabolic induction
$$
\tau_{\psi_{np}}=\bigtimes_\rho\bigtimes_{i\in I_\rho}\begin{pmatrix}
\frac{a_i-b_i}{2} & \cdots & \frac{a_i+b_i}{2}-1 \\
\vdots & \ddots & \vdots \\
\frac{-a_i-b_i}{2}+1 & \cdots & \frac{b_i-a_i}{2}
\end{pmatrix}_{\rho}.
$$
Then for any $\pi_{gp}\in\Pi_{\psi_{gp}}$ the induced representation $\tau_{\psi_{np}}\rtimes\pi_{gp}$ is irreducible, independent of choice of $\psi_{np}$. Moreover,
$$
\Pi_\psi=\{\tau_{\psi_{np}}\rtimes\pi_{gp} \, | \, \pi\in\Pi_{\psi_{gp}}\}.
$$
\end{thm}

Combined with \eqref{eq def packet +}, we obtain the following.

\begin{thm}[{\cite[Proposition 5.1]{Moe11b}}]\label{thm red from nu to gp}
Let $\psi\in\Psi^+(G_n)$ with decomposition $\psi=\psi_{nu,>0}+\psi_{np}+\psi_{gp}+\psi_{np}^\vee+\psi_{nu,>0}^\vee$ as above. Then, for any $\pi_{gp}\in\Pi_{\psi_{gp}},$ the induction $\tau_{\psi_{nu,>0}}\times\tau_{\psi_{np}}\rtimes\pi_{gp}$ is irreducible. As a consequence, \begin{equation}\label{non-unitary A-packet}
    \Pi_\psi=\{\tau_{\psi_{nu,>0}}\times\tau_{\psi_{np}}\rtimes\pi_{gp} , | , \pi_{gp}\in\Pi_{\psi_{gp}}\}.
\end{equation}
\end{thm}

Now we show that how to use Theorem \ref{thm red from nu to gp} to reduce the following main Problems we consider in this paper to the good parity case.

\begin{problems}\label{problem general case}\ 
\begin{enumerate}
    \item Given an irreducible representation $\pi$, determine whether it is of Arthur type, i.e. whether there exists a local Arthur parameter $\psi \in \Psi^+(G_n)$  such that $\pi \in \Pi_{\psi}$.
    \item Given an irreducible representation $\pi \in \Pi_{\psi}$, find all $\psi'$ such that $ \pi \in \Pi_{\psi'}$.
\end{enumerate}
\end{problems}

First, we give an analogous definition of good parity for representations.

\begin{defn}\label{def good parity reps}
We say an irreducible representation
\[ \pi= L(\Delta_{\rho_1}[x_1,y_1],\dots,\Delta_{\rho_r}[x_r,y_r]; \pi_{temp})\] 
of $G_n$ is of \emph{good parity} if the following hold:
\begin{enumerate}
    \item [$\oldbullet$] The tempered representation $\pi_{temp}$ lies in $\Pi_{\psi_{temp}}$ for some $\psi_{temp} \in \Psi_{gp}(G_m)$.
    \item [$\oldbullet$] For $1 \leq i \leq r$, $x_i,y_i \in \half{1} \Z$ and $\rho_i \otimes S_{x_i-y_i+1}\otimes S_1 $ is self-dual of the same type as $\widehat{G}_n$.
\end{enumerate}
\end{defn}

Let $\pi \in \Pi_{\psi}$. By the construction of local Arthur packets in the good parity case in the next section (see Definition \ref{def rep of segment} and Theorem \ref{thm Atobe's reformulation}) and Theorem \ref{thm reduction to gp}, one can show that $\pi$ is of good parity if and only if $\psi$ is of good parity.

\begin{cor}\label{cor reduction from nu to gp} \ 
\begin{enumerate}
    \item For $\pi \in \Pi(G_n)$, there exists a $\psi \in \Psi^+(G_n)$ such that $\pi \in \Pi_{\psi}$ if and only if $\pi$ is of the form
    \begin{align}\label{eq decomp pi red from nu to gp}
        \pi= \tau_{\psi_{nu,>0}} \times \tau_{\psi_{np}} \rtimes \pi_{gp}
    \end{align}
    and $\pi_{gp} \in \Pi_{\psi_{gp}}$ for some $\psi_{gp} \in \Psi_{gp}(G_n)$. Moreover, we have
    \begin{align*}
        \Psi(\pi):=& \{ \psi \in \Psi^+(G_n)\ | \ \pi \in \Pi_{\psi}\}\\
        =& \{ \psi_{nu,>0} + \psi_{np} + \psi_{gp} + \psi_{np}^{\vee}+ \psi_{nu,>0} \ | \ \psi_{gp} \in \Psi(\pi_{gp})\}.
    \end{align*}
    \item For $\psi, \psi' \in \Psi^+(G_n)$, decompose 
\begin{align*}
\psi&=\psi_{nu,>0}+\psi_{np}+\psi_{gp}+\psi_{np}^\vee+\psi_{nu,>0}^\vee,\\ \psi'&=\psi_{nu,>0}'+\psi_{np}'+\psi_{gp}'+(\psi_{np}')^\vee+(\psi_{nu,>0}')^\vee
\end{align*}
as \eqref{eq decomp psi +}, \eqref{eq decomp of psi_u}. Then $\Pi_{\psi}\cap \Pi_{\psi'}\neq \emptyset$ if and only if $\psi_{nu,>0}+\psi_{nu,>0}^\vee=\psi_{nu,>0}'+(\psi_{nu,>0}')^\vee$, $\psi_{np}+ \psi_{np}^{\vee}= \psi_{np}'+ (\psi_{np}')^{\vee}$, and $\Pi_{\psi_{gp}}\cap \Pi_{\psi_{gp}'}\neq \emptyset$.
\end{enumerate}  
\end{cor}

\begin{proof}
Part (1) follows directly from Theorem \ref{thm red from nu to gp}. Part (2) follows from the fact that if $\pi$ is of the form of \eqref{eq decomp pi red from nu to gp}, then $\psi_{nu,>0}+ \psi_{nu,>0}^{\vee}$ and $\psi_{np}+ \psi_{np}^{\vee}$ can be recovered from the $L$-data of $\pi$. This completes the proof of the corollary.
\end{proof}

Therefore, to answer Problems \ref{problem general case}, it suffices to answer the following good parity version. 

\begin{problems}\label{problem good parity}\ 
\begin{enumerate}
    \item Given an irreducible representation $\pi$ of good parity, determine whether it is of Arthur type, i.e. whether there exists a local Arthur parameter $\psi$ of good parity such that $\pi \in \Pi_{\psi}$.
    \item Given an irreducible representation $\pi \in \Pi_{\psi}$, where $\psi$ is of good parity, find all good parity $\psi'$ such that $ \pi \in \Pi_{\psi'}$.
\end{enumerate}
\end{problems}

\subsection{Parametrization of tempered spectrum}\label{sec tempered parametrization}

Let $\psi \in \Psi(G_n)$. An important ingredient in the construction of local Arthur packets is a map
\begin{align}\label{eq character of Arthur packet}
    \Pi_{\psi} &\longrightarrow \widehat{\mathcal{S}}_{\psi},\\
   \nonumber \pi & \longmapsto \langle \cdot, \pi \rangle,
\end{align}
where $\widehat{\mathcal{S}}_{\psi}$ is the Pontryagin dual of the \emph{component group}
\[  \mathcal{S}_{\psi}:= \pi_0( \mathrm{Cent}_{\widehat{G}_n(\BC)}(\mathrm{Im}(\psi))/Z(\widehat{G}_n(\BC))). \]
The map \eqref{eq character of Arthur packet} is not injective nor surjective in general. However, when $\psi$ is tempered, it is a bijection (see Theorem \ref{thm Arthur tempered} below). In this subsection, we recall a combinatorial description of $\widehat{\mathcal{S}}_{\psi}$ in \cite[Section 2]{Xu17b}, and use it to give a parametrization of the tempered spectrum of $G_n$.

Write $\psi=\psi_{np} + \psi_{gp}+ \psi_{np}^{\vee}$. There is a bijection between $ \mathcal{S}_{\psi}$ and $\mathcal{S}_{\psi_{gp}}$. Therefore, to describe $\widehat{\mathcal{S}}_{\psi}$, we may assume $\psi \in \Psi_{gp}(G_n)$. Write
\begin{align}\label{eq decomp psi parametrization of tempered}
    \psi= \bigoplus_{\rho} \bigoplus_{i \in I_{\rho}} \rho \otimes S_{a_i} \otimes S_{b_i}.
\end{align}
First, we consider the \emph{enhanced component group} of $\psi$ defined by
$$
\mathcal{A}_\psi=\bigoplus_{\rho}\bigoplus_{i\in I_\rho}(\mathbb{Z}/2\mathbb{Z}) \alpha_{\rho,i}.
$$
That is, $\mathcal{A}_\psi$ is the finite vector space over $\mathbb{Z}/2\mathbb{Z}$ with basis $\alpha_{\rho,i}$ corresponding to the summands $\rho \otimes S_{a_i} \otimes S_{b_i}$ of Equation (\ref{eq decomp psi parametrization of tempered}). While it is possible that for some $i,j\in I_\rho$, we have $\rho \otimes S_{a_i} \otimes S_{b_i}=\rho \otimes S_{a_j} \otimes S_{b_j}$, we distinguish these summands in $\mathcal{A}_\psi.$ That is,  $\alpha_{\rho,i}\neq\alpha_{\rho,j}$ in $\mathcal{A}_\psi.$ The \emph{central element} of $\mathcal{A}_\psi$ is $z_\psi:=\sum_\rho\sum_{i\in I_\rho}\alpha_{\rho,i}.$

The component group $\mathcal{S}_\psi$ of $\psi$ can be identified with the quotient of $\mathcal{A}_\psi$ by the subgroup generated by the central element and the elements $\alpha_{\rho,i}+\alpha_{\rho,j}$ such that $i,j\in I_\rho$ with $\rho \otimes S_{a_i} \otimes S_{b_i}=\rho \otimes S_{a_j} \otimes S_{b_j}.$ As a consequence, we may identify $\widehat{\mathcal{S}}_{\psi}$ as the set of functions $\varepsilon$ from the summands of $\psi$ to $\{\pm 1\}$ that satisfy
\begin{enumerate}
    \item [$\oldbullet$] $\varepsilon(\rho \otimes S_{a_i} \otimes S_{b_i})=\varepsilon(\rho \otimes S_{a_j} \otimes S_{b_j})$ if $\rho \otimes S_{a_i} \otimes S_{b_i}=\rho \otimes S_{a_i} \otimes S_{b_i}$, and
    \item [$\oldbullet$] $\prod_{\rho} \prod_{i \in I_{\rho}} \varepsilon(\rho \otimes S_{a_i}\otimes S_{b_i})=1$.
\end{enumerate}

Recall that a local Arthur parameter $\psi,$ decomposed as in Equation (\ref{eq decomp psi parametrization of tempered}), is \emph{tempered} if $b_i=1$ for any $\rho$ and $i\in I_\rho.$ That is, $\psi$ is trivial on the second $\SL_2(\mathbb{C}).$ We say that a local Arthur packet $\Pi_\psi$ is $\emph{tempered}$ if $\psi$ is tempered. A Whittaker datum for $G_n$ is a $G_n$-conjugacy class of a tuple $(B,\chi)$, where $B$ is an $F$-rational Borel subgroup of $\mathrm{G}_n$ and $\chi$ is a generic character of the $F$-points of the unipotent radical of $B$. The following theorem is Arthur's classification of the tempered representations.

\begin{thm}[{\cite[Theorem 1.5.1]{Art13}}]\label{thm Arthur tempered}
Any irreducible tempered representation of $G_n$ lies in $\Pi_\psi$ for some tempered local Arthur parameter $\psi.$ Moreover, if $\psi_1 $ and $\psi_2$ are two non-isomorphic tempered local Arthur parameters, then 
$$
\Pi_{\psi_1}\cap\Pi_{\psi_2}=\emptyset.
$$
Finally, if one fixes a choice of Whittaker datum for $G_n$ and $\psi$ is tempered, then there is a bijective map between the tempered local Arthur packet $\Pi_{\psi}$ and $\widehat{\mathcal{S}}_\psi.$
\end{thm}

Hereinafter, we implicitly fix a choice of Whittaker datum for $G_n.$ When $\psi$ is tempered and of good parity, we write $\pi(\psi,\varepsilon)$ for the element of $\Pi_\psi$ corresponding to $\varepsilon\in\widehat{\mathcal{S}}_\psi$ via the bijection in Theorem \ref{thm Arthur tempered}.

Following the notation in \cite{AM20}, when 
\[\psi=\bigoplus_{i=1}^m \rho_i \otimes S_{a_i} \otimes S_1\]
is a tempered local Arthur parameter, we denote $x_{i}=\frac{a_i-1}{2}$ for $i=1,\dots, m$. We also write $\pi(\psi,\varepsilon)\in\Pi_\psi$, the representation corresponding to $\varepsilon$ via Theorem \ref{thm Arthur tempered}, as
\begin{equation*}
\pi((x_{1},\rho_1)^{\varepsilon(\rho_1 \otimes S_{a_1})},\dots, (x_m,\rho_m)^{\varepsilon(\rho \otimes S_{a_m})}).
\end{equation*}
When $\psi$ has only one $\rho$ in its decomposition, we often ignore it in the notation and simply write
\begin{equation*}
\pi(x_{1}^{\varepsilon(\rho \otimes S_{a_1})},\dots,x_{m}^{\varepsilon(\rho \otimes S_{a_m})})=\pi(\psi,\varepsilon).
\end{equation*}

We say that a tempered local Arthur parameter $\psi,$ is \emph{discrete} if every summand of Equation (\ref{eq decomp psi parametrization of tempered}) is self-dual and the decomposition is multiplicity free. Arthur showed that the discrete series is parametrized by local Arthur packets of discrete local Arthur parameters.

\begin{thm}[{\cite[Theorem 1.5.1]{Art13}}]\label{thm Arthur discrete}
Any irreducible discrete series representation of $G_n$ lies in $\Pi_\psi$ for some discrete local Arthur parameter $\psi.$
\end{thm}

\section{Atobe's reformulation}\label{sec atob refor}
In this section, we recall the main definitions and results of \cite{Ato20b} for the construction of local Arthur packets of good parity.

We fix the following notation throughout the section. Let $\psi$ be any local Arthur parameter of good parity with decomposition 
\[ \psi= \bigoplus_{\rho} \bigoplus_{i \in I_{\rho}} \rho \otimes S_{a_i} \otimes S_{b_i}. \]
 We set $A_i=\frac{a_i+b_i}{2}-1$ and $B_i=\frac{a_i-b_i}{2}$ for $i\in I_\rho.$ 
 
 We say that a total order $>_\psi$ on $I_\rho$ is \emph{admissible} if satisfies:
\[
\tag{$P$}
\text{
For $i,j \in I_\rho$, 
if $A_i > A_j$ and $B_i > B_j$, 
then $i >_\psi j$.
}
\]
Sometimes we consider an order $>_\psi$ on $I_\rho$ satisfying:
\[
\tag{$P'$}
\text{
For $i,j \in I_\rho$, 
if $B_i > B_j$, 
then $i >_\psi j$.
}
\]

Note that ($P'$) implies ($P$). Often, we write $>$ instead of $>_\psi$ when it is clear that we are working with a fixed admissible order. 

Suppose now that we have fixed an admissible order for $\psi.$ Then we define the collection of ordered multi-sets 
$$\supp(\psi) := \cup_{\rho}\{ [A_i,B_i]_{\rho} \}_{i \in (I_\rho,>)}.
$$
We call this the support of $\psi$. Note that $\supp(\psi)$ depends implicitly on the fixed admissible order.

\subsection{Extended multi-segments and associated representations}

In this subsection, we recall the definition of extended multi-segments, their associated representations, and the explicit construction of $\Pi_{\psi}$ when $\psi$ is of good parity. We also give some notation on extended multi-segments that we will use throughout our arguments.

\begin{defn} [{\cite[Definition 3.1]{Ato20a}}]
(Extended multi-segments)\label{def multi-segment}
\begin{enumerate}
\item
An \emph{extended segment} is a triple $([A,B]_\rho, l, \eta)$,
where
\begin{itemize}
\item
$[A,B]_\rho = \{\rho|\cdot|^A, \rho|\cdot|^{A-1}, \dots, \rho|\cdot|^B \}$ is a segment 
for an irreducible unitary supercuspidal representation $\rho$ of some $\GL_d(F)$; 
\item
$l \in \Z$ with $0 \leq l \leq \frac{b}{2}$, where $b = \#[A,B]_\rho = A-B+1$; 
\item
$\eta \in \{\pm1\}$. 
\end{itemize}

\item
An \emph{extended multi-segment} for $G_n$ is 
an equivalence class (via the equivalence defined below) of multi-sets of extended segments 
\[
\EE = \cup_{\rho}\{ ([A_i,B_i]_{\rho}, l_i, \eta_i) \}_{i \in (I_\rho,>)}
\]
such that 
\begin{itemize}
\item
$I_\rho$ is a totally ordered finite set with a fixed admissible total order $>$;

\item
$A_i + B_i \geq 0$ for all $\rho$ and $i \in I_\rho$; 

\item
as a representation of $W_F \times \SL_2(\BC) \times \SL_2(\BC)$, 
\[
\psi_{\EE} = \bigoplus_\rho \bigoplus_{i \in I_\rho} \rho \otimes S_{a_i} \otimes S_{b_i} 
\]
where $(a_i, b_i) = (A_i+B_i+1, A_i-B_i+1)$,
is a local Arthur parameter for $G_n$ of good parity. We shall denote $\psi_{\EE}$ the local Arthur parameter associated with $\EE$. 
\item The sign condition
\begin{align}\label{eq sign condition}
\prod_{\rho} \prod_{i \in I_\rho} (-1)^{[\frac{b_i}{2}]+l_i} \eta_i^{b_i} = 1
\end{align}
holds.
\end{itemize}

\item
Two extended segments $([A,B]_\rho, l, \eta)$ and $([A',B']_{\rho'}, l', \eta')$ are \emph{weakly equivalent} 
if 
\begin{itemize}
\item
$[A,B]_\rho = [A',B']_{\rho'}$; 
\item
$l = l'$; and 
\item
$\eta = \eta'$ whenever $l = l' < \frac{b}{2}$. 
\end{itemize}
Two extended multi-segments 
$\EE = \cup_{\rho}\{ ([A_i,B_i]_{\rho}, l_i, \eta_i) \}_{i \in (I_\rho,>)}$ 
and 
$\EE' = \cup_{\rho}\{ ([A'_i,B'_i]_{\rho}, l'_i, \eta'_i) \}_{i \in (I_\rho,>)}$ 
are \emph{weakly equivalent}
if for any $\rho$ and $i \in I_\rho$, the extended segments $([A_i,B_i]_\rho, l_i, \eta_i)$ and $([A'_i,B'_i]_{\rho}, l'_i, \eta'_i)$ are weakly equivalent.

\item
We define the \emph{support} of $\EE$ to be the collection of ordered multi-sets 
\[
\supp(\EE) = \cup_{\rho}\{ [A_i,B_i]_{\rho} \}_{i \in (I_\rho,>)}.
\]
\end{enumerate}
\end{defn}

If the admissible order $>$ is clear in the context, for $k \in I_{\rho}$, we often denote $k+1 \in I_{\rho}$ to be the unique element adjacent with $k$ and $k+1>k$.

We attach a pictograph to each extended multi-segment by the same way in \cite[Section 3]{Ato20b}. We give an example to explain this.
\begin{exmp}
 Let $\rho$ be the trivial representation. The pictograph
\[\EE=\bordermatrix{
& -1 & 0 &1 & 2 &3 & 4\cr
& \lhd & \lhd & \oplus & \ominus & \rhd & \rhd \cr
&  &  &  & \lhd &\rhd  &  \cr
&  &  &  &  &  & \ominus \cr
}_{\rho}\]
corresponds to the extended multi-segment $\EE= \{ ([A_i,B_i]_{\rho},l_i,\eta_i)\}_{i \in (1<2<3)}$ of $\Sp_{44}(F)$ where
\begin{enumerate}
    \item [$\oldbullet$] $([A_1,B_1]_{\rho},[A_2,B_2]_{\rho},[A_3,B_3]_{\rho})=([4,-1]_{\rho},[3,2]_{\rho},[4,4]_{\rho})$ specify the ``support" of each row.
    \item [$\oldbullet$] $(l_1,l_2,l_3)=(2,1,0)$ counts the number of pairs of triangles in each row. 
    \item [$\oldbullet$] $(\eta_1,\eta_2,\eta_3)=(1, 1, -1)$ records the sign of the first circle in each row. Note that $\eta_2=1,-1$ are weakly equivalent.
\end{enumerate}
The associated local Arthur parameter is
    \[ \psi_{\EE}= \rho \otimes S_{4}\otimes S_{6} + \rho \otimes S_{6}\otimes S_{2} + \rho \otimes S_9 \otimes S_1.  \]

\end{exmp}

Next, we introduce two operators that we use frequently. They are considered in \cite{Ato20b}, but not written in the form of operators. 

\begin{defn}(shift, add)\\
Let $\EE = \cup_{\rho}\{ ([A_i,B_i]_{\rho}, l_i, \eta_i) \}_{i \in (I_\rho,>)}$ be an extended multi-segment. For $j \in I_{\rho'}$ and $d \in \Z$, we define the following operators. It is immediate that the operators commute with each other and so we denote the composition by summation. 

\begin{enumerate}
    \item [1.] $sh_j^{d}(\EE)= \cup_{\rho}\{ ([A_i',B_i']_{\rho}, l_i, \eta_i) \}_{i \in (I_\rho,>)}$ with 
    \[ [A_i',B_i']_{\rho}= \begin{cases}
    [A_i+d,B_i+d]_{\rho} & \text{ if }\rho=\rho' \text{ and } i = j,\\
     [A_i,B_i]_{\rho} & \text{ otherwise, }\end{cases} \]
    and $sh^d_{\rho'}=\sum_{j\in I_{\rho'}} sh_j^{d}$. Also, we define $sh^d:= \sum_{\rho} sh_{\rho}^d.$
     \item [2.] $add_j^{d}(\EE)= \cup_{\rho}\{ ([A_i',B_i']_{\rho}, l_i', \eta_i) \}_{i \in (I_\rho,>)}$ with 
    \[ ([A_i',B_i']_{\rho},l_i')= \begin{cases}
    ([A_i+d,B_i-d]_{\rho},l_i+d) & \text{ if }\rho=\rho' \text{ and } i = j,\\
     ([A_i,B_i]_{\rho},l_i) & \text{ otherwise, }\end{cases} 
    \]
    and $add^d_{\rho'}=\sum_{j\in I_{\rho'}} add_j^{d}$. Also, we define $add^d:= \sum_{\rho} add_{\rho}^d.$
\end{enumerate}
We will use these notation in the case that the resulting object is still an extended multi-segment.
\end{defn}

We remark that the local Arthur parameter $\psi_{sh_i^1(\EE)}$ (resp. $\psi_{add_i^1(\EE)}$) can be obtained from $\psi_{\EE}$ by replacing the summand $\rho \otimes S_{a_i} \otimes S_{b_i}$ with $\rho \otimes S_{a_i+2} \otimes S_{b_i}$ (resp. $\rho \otimes S_{a_i} \otimes S_{b_i+2}$). Therefore, the parity of its dimension is the same as that of $\psi$, so it is a local Arthur parameter of the same type of group as $\psi$ with larger rank.

For each extended multi-segment $\EE$, Atobe defined a representation $\pi(\EE)$ (possibly zero) as follows.

\begin{defn}[\S3.2, \cite{Ato20b}]\label{def rep of segment}
Suppose $\EE$ is an extended multi-segment such that for any $\rho$, if there exists $i \in I_{\rho}$ with $B_{i}<0$, then the admissible order on $I_{\rho}$ satisfies ($P'$). We first suppose that $\EE$ satisfies
\begin{enumerate}
    \item [$\oldbullet$]  $B_i \geq 0$ for any $i\in I_\rho,$
    \item [$\oldbullet$] for $i>j\in I_\rho$, $B_i >A_{j}$.
\end{enumerate}
Then we define
\[ \pi(\EE)= soc\left( \bigtimes_\rho\bigtimes_{i\in I_\rho} \begin{pmatrix} B_i &\cdots& B_i+l_i-1\\\vdots & \ddots & \vdots \\ -A_i&\cdots &-(A_i-l_i+1) \end{pmatrix}_\rho \rtimes \pi(\phi,\varepsilon)  \right), \]
where 
\[ \phi= \bigoplus_\rho \left( \rho \otimes \left(\bigoplus_{i\in I_\rho} S_{2(B_i+l_i)+1}\oplus\cdots\oplus S_{2(A_{i}-l_i)+1}  \right) \right),\]
and 
$ \varepsilon(\rho \otimes S_{2(B_i+l_i+k)+1})= \eta_i (-1)^{k} $ for $0 \leq  k \leq b_i-2l_i-1$.

In general, let $t_i \in \Z_{\geq 0}$ such that $\EE'=(\sum_\rho\sum_{i\in I_\rho} sh_i^{t_i})(\EE)$ satisfies above conditions. Then we define 
\[ \pi(\EE)= \circ_\rho\circ_{i\in I_\rho} \left(D_{\rho|\cdot|^{B_i+1,\dots,A_i+1}} \circ\cdots\circ D_{\rho|\cdot|^{B_i+t_i,\dots,A_i+t_i}}  \right)(\pi(\EE') ), \]
where if $I_\rho=\{1,\dots,n\}$ with $1<\cdots<n,$ we write $\circ_{i\in I_\rho}D_i= D_n\circ\cdots\circ D_1.$
\end{defn}

Atobe showed that the composition of derivatives in the definition above can be computed by the formulae in \cite{AM20} explicitly based on the proposition below. We give more details of this argument in Lemma \ref{lem derivative algorithm}.

\begin{prop}[{\cite[Proposition 8.3]{Xu17b}}]\label{prop derivative support}

Suppose $\pi$ is an irreducible representation in $\Pi_{\psi}$, where $\supp(\psi) = \cup_{\rho}\{ [A_i,B_i]_{\rho} \}_{i \in (I_\rho,>)}.$
\begin{enumerate}
    \item [(i)] If $D_{\rho|\cdot|^{x}}^{(k)}(\pi)\neq 0$, then 
    \[ k \leq \#\{ i \in I_{\rho} \ | \ B_i= x\}.\]
    \item [(ii)]There exists $0 \leq x \leq y$ such that
    \[ D_{\rho|\cdot|^{x,\dots,y}}(\pi(\EE)) \neq 0 \]
    only if there exists $i <j$ in $I_{\rho}$ such that
    \begin{enumerate}
        \item [$\oldbullet$] $B_i=x$, $A_j \geq y$,
        \item [$\oldbullet$] for $i \leq k < j$,  $B_k \leq B_{k+1}\leq A_k+1.$
    \end{enumerate}
\end{enumerate}
\end{prop}
Following is Atobe's parametrization of the local Arthur packet associated with local Arthur parameter of good parity.
\begin{thm}[{\cite[Theorem 3.3]{Ato20b}}]\label{thm Atobe's reformulation}
Suppose $\psi= \bigoplus_{\rho} \bigoplus_{i \in I_{\rho}} \rho \otimes S_{a_i} \otimes S_{b_i}$ is a local Arthur parameter of good parity of $G_n$. Choose an admissible order $>_{\psi}$ on $I_{\rho}$ for each $\rho$ that satisfies ($P'$) if $\half{a_i-b_i}<0$ for some $i \in I_{\rho}$. Then
\[ \bigoplus_{\pi \in \Pi_{\psi}} \pi= \bigoplus_{\EE} \pi(\EE),\]
where $\EE$ runs over all extended multi-segments with $\supp(\EE)= \supp(\psi)$ and $\pi(\EE) \neq 0$. 
\end{thm}

Note that by $\supp(\EE)= \supp(\psi),$ we mean that they are equal as multi-sets and also that the admissible orders on $I_\rho$ agree.

\begin{lemma}\label{lem Moeglin}
Suppose $\supp(\EE_1)=\supp(\EE_2)$ and $\pi(\EE_1)\cong \pi(\EE_2) \neq 0$. Then $\EE_1=\EE_2$.
\end{lemma}
\begin{proof}
This is a direct consequence of Theorem \ref{thm Atobe's reformulation} and Mœglin's work that $\Pi_{\psi}$ is multiplicity free (\cite{Moe11}).
\end{proof}

The following definition provides some of the notation and terminology used throughout this paper.

\begin{defn}
Suppose $\EE= \cup_{\rho} \{([A_i,B_i]_{\rho}, l_i, \eta_i)\}_{i \in (I_{\rho},>)}$ satisfies $\pi(\EE)\neq 0$. 
\begin{enumerate}
    \item In many applications, we want to study $\EE$ piece by piece, so we set  
\[ \EE_{\rho}=\{([A_i,B_i]_{\rho}, l_i, \eta_i)\}_{i \in (I_{\rho},>)},\  \EE^{\rho}=\cup_{\rho' \not\cong \rho}\{([A_i,B_i]_{\rho'}, l_i, \eta_i)\}_{i \in (I_{\rho'},>)}. \]
\item For $i \in I_{\rho}$, we usually call the extended segment $ ([A_i,B_i]_{\rho}, l_i, \eta_i)$ the $i$-th row of $ \EE_{\rho}$.
\item We define $sh^{d}_i(\EE_{\rho})=sh^d_i(\EE)_{\rho}$, and $sh^{d}(\EE_{\rho})=((\sum_{i \in I_{\rho}} sh_i^{d})(\EE))_{\rho} $ and similarly for $add^d_i$ and $add^d$.
\item If it is clear from the context we often use $\FF$ to denote $\EE_{\rho}$ for brevity. Suppose $I_{\rho}$ is the disjoint union $I_{\rho,1} \sqcup I_{\rho,2}$, and for any $i \in I_{\rho,1}$ and $j \in I_{\rho,2}$, $i<j$. Then we denote $\FF=\FF_1 + \FF_2$ where
\[ \FF_1=\{([A_i,B_i]_{\rho}, l_i, \eta_i)\}_{i \in (I_{\rho,1},>)} ,\ \FF_2=\{([A_i,B_i]_{\rho}, l_i, \eta_i)\}_{i \in (I_{\rho,2},>)}.\]
Similarly, if $I_{\rho}= I_{\rho,1} \sqcup\cdots\sqcup I_{\rho,k}$, we denote $\FF= \FF_1 +\cdots+\FF_k$. In this case, we define 
\[ \FF_1 +\cdots+\FF_{r-1}+ Sh^d(\FF_r) + \FF_{r+1} +\cdots+\FF_k= 
\left(\sum_{i \in I_{\rho,r}} Sh_i^d\right)(\FF).\]
and similarly for $add$. We denote the support of $\FF_j$ as the ordered multi-set 
\[ \supp(\FF_j)= \{ [A_i,B_i]_{\rho}\}_{i \in (I_{\rho,j},>)}. \]

\item Finally, we say $\FF= \{ ([A_i,B_i]_{\rho}, l_i,\eta_i)\}_{i \in I_{\rho}}$ is \emph{positive} (resp. \emph{non-negative}) if $B_i$ is positive (resp. non-negative) for all $i \in I_{\rho}$. We say $\EE= \cup_{\rho} \EE_{\rho}$ is \emph{positive} (resp. \emph{non-negative}) if $\EE_{\rho}$ is positive (resp. non-negative) for all $\rho$.

\end{enumerate}

\end{defn}

We state the following definition in order to give a technical lemma which we use multiple times in later sections.

\begin{defn}\label{def modification}
Suppose $\EE= \EE^{\rho} \cup \EE_{\rho}$ and $\EE_{\rho} = \FF_1+ \FF_2$ with 
\[ \FF_1=\{([A_i,B_i]_{\rho}, l_i, \eta_i)\}_{i \in (I_{\rho,1},>)} ,\ \FF_2=\{([A_i,B_i]_{\rho}, l_i, \eta_i)\}_{i \in (I_{\rho,2},>)}.\]

Then for any $\rho^{\ast}$ such that
\begin{enumerate}
    \item [$\oldbullet$] $\EE_{\rho^{\ast}}=\emptyset$,
    \item [$\oldbullet$] $\rho^{\ast}$ is of the same type as $\rho$,
    \item [$\oldbullet$]$\dim(\rho^{\ast})=\dim(\rho)$,
\end{enumerate}
we define the extended multi-segment 
\[\EE-(\FF_2)_{\rho} + (\FF_2)_{\rho^{\ast}}:= \EE^{\rho} \cup \FF_1 \cup \{([A_i,B_i]_{\rho^{\ast}}, l_i, \eta_i)\}_{i \in (I_{\rho,2},>)},\]
and often write it as $\EE^{\rho} \cup \FF_1 \cup (\FF_2)_{\rho^{\ast}}$ for short.
\end{defn}
The following lemma is used frequently in later sections to perform induction arguments. Atobe has also used this argument (see the proof of \cite[Theorem 3.6]{Ato20b}).
\begin{lemma}\label{lemma far away} \ 
\begin{enumerate}
    \item [(i)] Suppose $\pi(\EE)$ and $ \pi(\EE')$ are both nonzero. Then $\pi(\EE) \cong \pi(\EE')$ if and only if $\pi(\EE)\cong \pi(\EE^{\rho} \cup (\EE')_{\rho})$ for all $\rho$.
    \item [(ii)] Suppose $\EE$ and $\EE'$ are extended multi-segments such that $\pi(\EE)$ and $\pi(\EE')$ are both nonzero. We decompose $\EE_{\rho}= \FF_1 + \FF_2$, $\EE_{\rho}'= \FF_1'+ \FF_2'$ where 
\begin{align*}
    \FF_1=\{([A_i,B_i]_{\rho}, l_i, \eta_i)\}_{i \in (I_{\rho,1},>)} ,\ \FF_2=\{([A_i,B_i]_{\rho}, l_i, \eta_i)\}_{i \in (I_{\rho,2},>)},\\
    \FF_1'=\{([A_i',B_i']_{\rho}, l_i', \eta_i')\}_{i \in (I_{\rho,1}',>)} ,\ \FF_2'=\{([A_i',B_i']_{\rho}, l_i', \eta_i')\}_{i \in (I_{\rho,2}',>)}.
\end{align*}
Suppose that the following holds.
\begin{enumerate}
    \item [$\oldbullet$] We have an equality between the ordered multi-sets  
    \[ \supp(\FF_2)= \supp(\FF_2').\]
    \item [$\oldbullet$] There exists a sequence of positive integers $\{t_i\}_{i \in I_{\rho,1}}$ (resp. $\{t_i'\}_{i \in I_{\rho,1}'}$) such that
    \[\begin{cases}
    B_i+t_i \geq 0 &\text{ for }i \in I_{\rho,1},\\
    B_{i}+t_i >A_j+t_j &\text{ for }i>j \in I_{\rho,1},
    \end{cases} \]
    \[\begin{cases}
    B_i'+t_i' \geq 0 &\text{ for }i \in I_{\rho,1}',\\
    B_{i}'+t_i' >A_j'+t_j' &\text{ for }i>j \in I_{\rho,1}',
    \end{cases}\]
    and for any $i \in I_{\rho,1}$ (resp. $I_{\rho,1}'$) and $j \in I_{\rho,2}$ (resp. $I_{\rho,2}' $), we have $ A_i+t_i <B_j$ (resp. $A_i'+t_i' < B_j'$). 
\end{enumerate}
 Then, taking $\rho^{\ast}$ as in Definition \ref{def modification}, we have 
    \[ \pi(\EE) \cong \pi(\EE') \Longleftrightarrow \pi(\EE- (\FF_2)_{\rho}+ (\FF_2)_{\rho^{\ast}})\cong \pi(\EE'- (\FF_2')_{\rho}+ (\FF_2')_{\rho^{\ast}}).\]
    In this case, we have $\FF_2=\FF_2'$.
\end{enumerate}
\end{lemma}
\begin{proof}
Part (i) follows from the fact that the algorithm for computing derivatives in \cite{AM20} for distinct $\rho$ are independent.

The equivalence in Part (ii) follows similarly. The conditions imply that the computation of the derivative for $\FF_1$ and $\FF_2$ are independent and similar for $\FF_1'$ and $\FF_2'$. The equality $\FF_2=\FF_2'$ follows from Part (i) and Lemma \ref{lem Moeglin}. This completes the proof of the lemma.
\end{proof}

Now we define the sets that we are interested in the most.
\begin{defn}\label{def PRep Rep}
We define $\Rep$ to be the set of extended multi-segments $\EE$ satisfying the condition in Definition \ref{def rep of segment} and $\pi(\EE) \neq 0$. Furthermore, we let $\Rep^{(P')}$ be the subset of $\Rep$ consisting of extended multi-segments whose admissible orders already satisfy ($P'$).
\end{defn}

We end this subsection by rephrasing Problem \ref{problem good parity} in terms of extended multi-segments below.
\begin{problems}\label{problem ext seg}\ 
\begin{enumerate}
    \item Given an irreducible representation $\pi$ of good parity, determine whether the following set below is empty.
    $$\{\EE \in \Rep \ | \ \pi(\EE)= \pi\}$$ 
    \item Given $\EE \in \Rep$, construct the set 
    \[\{ \EE' \in \Rep \ | \ \pi(\EE')= \pi(\EE)\}.\]
\end{enumerate}
\end{problems}

Problems \ref{problem ext seg}(1) will be answered by Algorithm \ref{alg Arthur type}, and (2) will be answered by the formula in Theorem \ref{thm exhaustion of symbol}.

\subsection{Change of admissible order and non-vanishing conditions}
After giving the construction of $\pi(\EE)$, it is an interesting question to give a purely combinatorial criterion on $\EE$ such that $\pi(\EE)$ is nonzero. Xu gave an algorithm in \cite{Xu21b} to answer this question in M{\oe}glin's parametrization. Atobe reformulated this algorithm on non-negative extended multi-segments, and then extended the algorithm by relating $\pi(\EE)$ to $\pi(sh^{t}(\EE))$ for $t \in \Z_{\geq 0}$ such that $sh^{t}(\EE)$ is non-negative. We collect these results in this subsection.

First, we give a necessary condition.
\begin{prop}\label{prop positive non-vanishing} 
Let $\EE=\cup_{\rho}\{([A_i,B_i]_{\rho},l_i,\eta_i)\}_{i \in (I_{\rho},>)}$ be an extended multi-segment whose admissible order satisfies ($P'$) if $B_i<0$ for some $i \in I_{\rho}$.
\begin{enumerate}
    \item [(i)] \textup{(}\cite[Lemma 5.5, 5.6, 5.7]{Xu21b}\textup{)} $\pi(\EE)\neq 0$ only if the following conditions hold for all $k<k+1 \in I_{\rho}$.
 Denote $\epsilon= (-1)^{A_k-B_k}\eta_k \eta_{k+1}.$
    \begin{enumerate}
        \item [(1)] If $ A_k\leq  A_{k+1}$, $B_k \leq B_{k+1}$, then
        \[ \begin{cases}
        \epsilon=1 &\Rightarrow B_k+l_k \leq B_{k+1}+l_{k+1}, A_k-l_k \leq A_{k+1}-l_{k+1},\\
        \epsilon=-1 &\Rightarrow  A_k-l_k < B_{k+1}+l_{k+1}.
        \end{cases} \]
        \item [(2)] If $[A_k,B_k]_{\rho}\subset [A_{k+1},B_{k+1}]_{\rho}$, then
        \[ \begin{cases}
        \epsilon=1 &\Rightarrow 0 \leq l_{k+1} -l_{k} \leq b_{k+1}-b_k,\\
        \epsilon=-1 &\Rightarrow  l_k+l_{k+1} \geq b_k.    \end{cases} \]
        \item [(3)] If $[A_k,B_k]_{\rho}\supset [A_{k+1},B_{k+1}]_{\rho}$, then
        \[ \begin{cases}
        \epsilon=1 &\Rightarrow 0 \leq l_{k} -l_{k+1} \leq b_{k}-b_{k+1},\\
        \epsilon=-1 &\Rightarrow  l_k+l_{k+1} \geq b_{k+1}.    \end{cases} \]
       
    \end{enumerate} 
    \item [(ii)]\textup{(}\cite[Theorem A.3]{Xu21a}\textup{)} In a special case that
    \begin{enumerate}
        \item [$\oldbullet$]$\EE$ is non-negative, and
        \item [$\oldbullet$]$ A_k\leq  A_{k+1}$, $B_k \leq B_{k+1}$ for all $k<k+1 \in I_{\rho}$,
    \end{enumerate}
    $\pi(\EE) \neq 0$ if and only if condition in (i) holds for all $k<k+1 \in I_{\rho}$. Note that under these assumptions, we are always in case (1) in (i).
\end{enumerate}
\end{prop}

Xu also studied how the M{\oe}glin's parametrization for $\Pi_{\psi}$ changes for different admissible orders of $\psi$ in \cite[Section 6]{Xu21b}. Atobe translated this result into his notation in \cite[Section 4.2]{Ato20b}. To make the description notationally correct, we consider a wider class of objects as follows.
\begin{defn}\label{def symbol}
  A symbol is a multi-set of extended segments 
\[
\EE = \cup_{\rho}\{ ([A_i,B_i]_{\rho}, l_i, \eta_i) \}_{i \in (I_\rho>)},
\] which satisfies the same conditions in Definition \ref{def multi-segment}(2) except we drop the condition $0 \leq l_i \leq \half{b_i}$, for each $i \in I_{\rho}$.
\end{defn}

Any change of admissible orders can be derived from a composition of the operators $R_k$ defined below.

\begin{defn}[{\cite[Section 4.2]{Ato20b}}, Row exchange]\label{def row exchange} 
Suppose $\EE$ is a symbol where
$$\EE_{\rho}=\{([A_i,B_i]_{\rho},l_i,\eta_i)\}_{i \in (I_{\rho},>)}.$$
For $k<k+1 \in I_{\rho}$, let $\gg$ be the total order on $I_\rho$ defined by $k\gg k+1$ and if $(i,j)\neq (k,k+1)$, then $ i \gg j$ if and only if $
i >j .$ 

Suppose $\gg$ is not an admissible order on $I_{\rho}$, then we define $R_k(\EE)=\EE$. Otherwise, we define 
\[R_{k}(\EE_{\rho})=\{([A_i,B_i]_{\rho},l_i',\eta_i')\}_{i \in (I_{\rho},\gg)},\]
where $( l_i',\eta_i')=(l_i,\eta_i)$ for $i \neq k,k+1$, and $(l_k',\eta_k')$ and $(l_{k+1}', \eta_{k+1}')$ are given as follows: Denote $\epsilon=(-1)^{A_k-B_k}\eta_k\eta_{k+1}$.
\begin{enumerate}
    \item [Case 1.] $ [A_k,B_k]_{\rho} \supset [A_{k+1},B_{k+1}]_{\rho}$:
    
    In this case, we set $(l_{k+1}',\eta_{k+1}')=(l_{k+1}, (-1)^{A_k-B_k}\eta_{k+1})$, and
    \begin{enumerate}
    \item [(a)] If $\epsilon=1$ and $b_k- 2l_k < 2(b_{k+1}-2l_{k+1})$, then
    \[ (l_k', \eta_{k}')= (b_k-(l_k+ (b_{k+1}-2l_{k+1})), (-1)^{A_{k+1}-B_{k+1}} \eta_k).  \]
    \item [(b)] If $\epsilon=1$ and $b_k- 2l_k \geq  2(b_{k+1}-2l_{k+1})$, then
    \[ (l_{k}', \eta_{k}')= (l_k+ (b_{k+1}-2l_{k+1}), (-1)^{A_{k+1}-B_{k+1}+1} \eta_k).  \]
    \item [(c)] If $\epsilon=-1$, then
    \[ (l_{k}', \eta_{k}')= (l_k- (b_{k+1}-2l_{k+1}), (-1)^{A_{k+1}-B_{k+1}+1} \eta_k).  \]
\end{enumerate}
    \item [Case 2.] $ [A_k,B_k]_{\rho} \subseteq [A_{k+1},B_{k+1}]_{\rho}$:
    
    In this case, we set $(l_{k}',\eta_{k}')=(l_{k}, (-1)^{A_{k+1}-B_{k+1}}\eta_{k})$, and
    \begin{enumerate}
   \item [(a)] If $\epsilon=1$ and $b_{k+1}- 2l_{k+1} < 2(b_{k}-2l_{k})$, then
    \[ (l_{k+1}', \eta_{k+1}')= (b_{k+1}-(l_{k+1}+ (b_{k}-2l_{k})), (-1)^{A_{k}-B_{k}} \eta_{k+1}).  \]
    \item [(b)] If $\epsilon=1$ and $b_{k+1}- 2l_{k+1} \geq  2(b_{k}-2l_{k})$,
    then
    \[ (l_{k+1}', \eta_{k+1}')= (l_{k+1}+ (b_{k}-2l_{k}), (-1)^{A_{k}-B_{k}+1} \eta_{k+1}).  \]
    \item [(c)] If $\epsilon=-1$, then
    \[ (l_{k+1}', \eta_{k+1}')= (l_{k+1}- (b_{k}-2l_{k}), (-1)^{A_{k}-B_{k}+1} \eta_{k+1}).  \]
\end{enumerate}
\end{enumerate}
Finally, we define $R_{k}(\EE)= \EE^{\rho} \cup R_{k}(\EE_{\rho})$.
\end{defn}

\begin{remark}\label{rmk on row exchange}
Definition \ref{def row exchange} is slightly different from the one in \cite{Ato20b}. Let us explain the difference in Case 1. Case 2 follows similarly.

In \cite[Section 4.2]{Ato20b}, Atobe identified the set $\{l \in \Z \ | \ 0 \leq l \leq \half{b_k}\}$ with $(\Z/b_k\Z)/\{\pm 1\}$ and directly defined $l_{k}'$ by $l_k+\epsilon (b_{k+1}-2 l_{k+1})$ in $(\Z/b_{k}\Z)/\{\pm 1\}$. On the other hand, suppose $\EE$ satisfies Proposition \ref{prop positive non-vanishing}(i); then one can check that the $l_{k}'$ in our definition gives the same element as $l_k+\epsilon (b_{k+1}-2 l_{k+1})$ in $(\Z/b_{k}\Z)/\{\pm 1\}$ and $l_{k}' \in \{l \in \Z \ | \ 0 \leq l \leq \half{b_{k}}\}$, and hence these two definitions coincide on $\EE$. In particular, two definitions agree on $\Rep$, the set that we are interested in the most.

Our definition has a disadvantage that for an extended multi-segment $\EE$, $R_k(\EE)$ is only a symbol as the condition $0 \leq l_i \leq \half{b_i}$ may not be preserved. However, Definition \ref{def row exchange} makes some identities among operators more clear. See Lemma \ref{lem identities} for example. 
\end{remark}

\begin{exmp}
Let $\rho$ be the trivial representation and $$\psi=\rho\otimes S_5 \otimes S_2+ \rho\otimes S_5 \otimes S_4 + \rho\otimes S_3 \otimes S_2$$ be a local Arthur parameter of good parity for $G_n=\SO_{37}(F).$ Let $I_\rho=\{1,2,3\}$ and $A_1=\half{3}$, $B_1=\half{1},$ $A_2=\half{7},$ $B_2=\half{1},$ $A_3=\half{5},$ and $B_3=\half{3}.$ Since $A_3>A_2$ and $B_3>B_2,$ there are only 3 admissible orders on $I_\rho$ which we denote by $1<_1 2 <_1 3,$ $1<_2 3 <_2 2,$ and $2<_3 1 <_3 3.$ Let $\EE=\{([A_i,B_i]_\rho,l_i,\eta_i\}_{i\in (I_\rho, >_1)}$ with $l_1=l_3=0,$ $l_2=2,$ and $\eta_1=\eta_2=\eta_3=1.$ That is,
$$
\EE=
\bordermatrix{
 &\half{1}&\half{3}&\half{5}&\half{7} \cr
& \oplus& \ominus& & \cr
 & \lhd&  \lhd&\rhd&\rhd \cr
& &\oplus &   \ominus& \cr
}_{\rho}.
$$
We have $\pi(\EE)= L(\Delta_{\rho}[\half{1},\half{-5}];\pi(\half{1}^+,\half{3}^-,\half{3}^-,\half{3}^-,\half{5}^+,\half{7}^-))\in\Pi_\psi.$ Exchanging the first and second rows gives 
$$
R_1(\EE)=
\bordermatrix{
 &\half{1}&\half{3}&\half{5}&\half{7} \cr
& \oplus& \ominus& \oplus & \ominus \cr
 & \ominus&  \oplus& & \cr
& &\oplus &   \ominus& \cr
}_{\rho}.
$$
We have $\pi(R_1(\EE))\cong\pi(\EE).$ We can also exchange the second and third rows of $\EE$. In this case, we find
$$
R_2(\EE)=
\bordermatrix{
 &\half{1}&\half{3}&\half{5}&\half{7} \cr
& \oplus& \ominus& & \cr
 & &  \ominus&\oplus& \cr
& \oplus &\ominus &   \oplus& \ominus \cr
}_{\rho}.
$$
Again, we have $\pi(R_2(\EE))\cong\pi(\EE).$ 
\end{exmp}

In the above example, we saw that every row exchange gave the same representation. This was not a coincidence as the next theorem shows.

\begin{thm} [{\cite[Theorem 4.3]{Ato20b}}] \label{thm row exchange}
Let $\EE \in \Rep$. If $\EE$ is non-negative, then 
    \[ \pi(R_k(\EE))\cong \pi(\EE). \]
\end{thm}

In above theorem, we assume $\EE$ is non-negative since for $\EE$ that is not non-negative, $\pi(\EE)$ is only defined if $\EE \in \Rep^{(P')}$. However, it is possible that $\EE \in \Rep^{(P')}$ but $R_{k}(\EE) \not\in \Rep^{(P')}$.

We give the following notation for general change of admissible orders.

\begin{defn}
Suppose $\EE= \cup_{\rho}\EE_{\rho}$ is an extended multi-segment, and $\gg$ is an admissible order on $I_{\rho}$. Write $\EE_{\rho}=\{([A_i,B_i]_{\rho},l_i,\eta_i)\}_{i \in (I_{\rho, >})}$. Then we define 
\[ \EE_{\rho, \gg}= \{([A_i,B_i]_{\rho},l_i',\eta_i')\}_{i \in (I_{\rho, \gg)}}\]
by applying a sequence of $R_k$ on $\EE_{\rho}$ to change the admissible order from $>$ to $\gg$.

We say $(i,j,\gg)$ is an adjacent pair if $i\gg j$ are two adjacent elements in $I_{\rho}$ with respect to $\gg$. We say an adjacent pair $(i,j, \gg)$ satisfies Proposition \ref{prop positive non-vanishing}(i) if the conditions for $\EE_{\rho,\gg}$ are satisfied with respect to $i \gg j$.
\end{defn}

Now we recall Atobe's reformulation of the non-vanishing criterion for general extended multi-segments.

\begin{thm}[{\cite[Theorems 3.6, 4.4]{Ato20b}}]\label{thm non-vanishing}
Let $\EE$ be an extended multi-segment such that for any $\rho$, if there exists $i \in I_{\rho}$ with $B_{i}<0$, then the admissible order on $I_{\rho}$ satisfies ($P'$).
\begin{enumerate}
    \item [(i)] $\EE=\cup_{\rho} \{ ([A_i,B_i]_{\rho},l_i,\eta_i)\}_{i \in (I_{\rho},>)} \in \Rep$ if and only if $ sh^d (\EE) \in \Rep$ for any $d \gg 0$ such that $sh^d (\EE)$ is non-negative, and the following condition holds for all $\rho$ and $i\in I_{\rho}$
    \[ (\ast) \ \ \ \ \ \ \ \ \ B_i+l_i \geq \begin{cases}  0 & \text{ if }B_i \in \Z, \\
    \frac{1}{2} & \text{ if } B_i \not\in \Z \text{ and }\eta= (-1)^{\alpha_i+1},\\
        -\frac{1}{2} & \text{ if } B_i \not\in \Z \text{ and }\eta= (-1)^{\alpha_i},
        \end{cases} \]
        where 
        \[ \alpha_i:= \sum_{j < i }A_j+B_j+1. \]
    In this case, if  
    \[\pi\left( sh^d(\EE)\right) =  L\left( \Delta_{\rho_1}[x_1,-y_1],\dots,\Delta_{\rho_t}[x_t,-y_t]; \pi\left(\sum_{j \in J}\rho_j\otimes S_{2z_j+1},\varepsilon \right) \right),\]
    then 
    \begin{enumerate}
        \item [$\oldbullet$] $x_i+y_i+1 \geq 2d$,
        \item [$\oldbullet$] $2z_j+1 \geq 2d$,  
    \end{enumerate}
    and
    \[ \pi(\EE) \cong  L\left( \Delta_{\rho_1}[x_1-d,-(y_1-d)],\dots,\Delta_{\rho_t}[x_t-d,-(y_t-d)]; \pi\left(\sum_{j \in J} \rho\otimes S_{2(z_j-d)+1},\varepsilon_{-d} \right) \right), \]
    where $ \varepsilon_{-d}(\rho_j\otimes S_{2(z_i-d)+1})= \varepsilon( \rho_j \otimes S_{2z_i+1})$ and we omit any term of the form $\rho \otimes S_{0}$ and segment of the form $ \Delta_{\rho}[x,x+1]$.
    \item [(ii)]If $\EE$ is non-negative, then $\pi(\EE)\neq 0$ if and only if any adjacent pair $(i,j,\gg)$ satisfies Proposition \ref{prop positive non-vanishing}(i).
\end{enumerate}
\end{thm}

\begin{remark}\label{rmk change rho}
Suppose $\EE\in \Rep^{(P')}$ and $\EE_{\rho}= \FF_1 +\FF_2$. Then the theorem above shows 
\[ (\EE^{\rho} \cup (\FF_1 + sh^1(\FF_2)))- (sh^1(\FF_2))_{\rho}+ (sh^1(\FF_2))_{\rho^{\ast}} \in \Rep,\]
where the choice of $\rho^{\ast}$ is the same as Definition \ref{def modification}. We will denote this extended multi-segment by $\EE-(\FF_2)_{\rho}+ (sh^1(\FF_2))_{\rho^{\ast}}$ or $\EE^{\rho} \cup \FF_1 \cup (sh^1(\FF_2))_{\rho^{\ast}}$ for short. Note that it is not true in general that $ \EE- (\FF_2)_{\rho}+ (\FF_2)_{\rho^{\ast}} \in \Rep$ since the condition ($\ast$) in theorem above may fail. Here is an example. Let
\[ \FF_1= \bordermatrix{
& \half{-1} &  \half{1} \cr
& \oplus&\ominus \cr
}_{\rho}, \ \FF_2= \bordermatrix{
& \half{-1} &  \half{1} \cr
& \ominus&\oplus \cr
}_{\rho}\]
Then
\[ \EE= \FF_1+\FF_2=  \bordermatrix{
& \half{-1} &  \half{1} \cr
& \oplus&\ominus \cr
& \ominus &\oplus\cr
}_{\rho} \in \Rep,\]
but 
\[ \EE- (\FF_2)_{\rho}+ (\FF_2)_{\rho^{\ast}}=  \bordermatrix{
& \half{-1} &  \half{1} \cr
& \oplus&\ominus \cr
}_{\rho} \cup \bordermatrix{
& \half{-1} &  \half{1} \cr
& \ominus&\oplus \cr
}_{\rho^{\ast}} \not\in \Rep.\]
\end{remark}

We need a version of Theorem \ref{thm non-vanishing}(i) with only $\EE_{\rho}$ shifted, which we state below.
\begin{cor}\label{cor shift in single rho}
Suppose $\EE \in \Rep$. Write 
\[ \pi(\EE^{\rho} \cup sh^d(\EE_{\rho}))= L\left( \Delta_{\rho_1}[x_1,-y_1],\dots,\Delta_{\rho_t}[x_t,-y_t]; \pi\left(\sum_{j \in J}\rho_j\otimes S_{2z_j+1},\varepsilon \right) \right) .\]
Then 
\begin{enumerate}
    \item [$\oldbullet$] $x_i+y_i+1 \geq 2d$ if $\rho_i \cong \rho$,
    \item [$\oldbullet$] $2z_j+1 \geq 2d$ if $\rho_j \cong \rho$.
\end{enumerate}
Moreover, we have
\[ \pi(\EE)=L\left( \Delta_{\rho_1}[x_1',-y_1'],\dots,\Delta_{\rho_t}[x_t',-y_t']; \pi\left(\sum_{j \in J}\rho_j\otimes S_{2z_j'+1},\varepsilon' \right) \right), \]
where
\begin{align*}
    \Delta_{\rho_i}[x_i',-y_i']&:= \begin{cases}
    \Delta_{\rho_i}[x_i,-y_i] & \text{ if } \rho_i \not\cong \rho,\\
\Delta_{\rho_i}[x_i-d,-(y_i-d)] & \text{ otherwise,}
\end{cases}\\
    \rho_j\otimes S_{2z_j'+1}&:= \begin{cases}
    \rho_j\otimes S_{2z_j+1} & \text{ if } \rho_i \not\cong \rho,\\
\rho_j\otimes S_{2(z_j-d)+1}& \text{ otherwise,}
    \end{cases}\\
    \varepsilon'(\rho_j\otimes S_{2z_j'+1})&:=\varepsilon(\rho_j\otimes S_{2z_j+1}).
\end{align*}
We omit any term of the form $\rho \otimes S_{0}$ and segment of the form $ \Delta_{\rho}[x,x+1]$.
\end{cor}
\begin{proof}
This follows from the algorithms for taking derivative in \cite{AM20} and Theorem \ref{thm non-vanishing}(i).
\end{proof}

\subsection{Deformation and Aubert-Zelevinsky dual formula}
In this subsection, we introduce another operator, called  union-intersection, which was defined by Atobe to give a quicker algorithm to compute $\pi(\EE)$. We also recall the Aubert-Zelevinsky involution and Atobe's formula for computing it on extended multi-segments.
\begin{defn}\label{ui def} (union-intersection) \cite[Section 5.2]{Ato20b}\\
 Let $\EE$ be an extended multi-segment. For $k< k+1 \in I_{\rho}$, we define an operator $ui_k$, called union-intersection, on $\EE$ as follows. Write 
 \[ \EE_{\rho}= \{([A_i,B_i]_\rho, l_i,\eta_i)\}_{i \in (I_{\rho},>)}.\]
  Denote $\epsilon=(-1)^{A_k-B_k}\eta_k \eta_{k+1}.$ If $A_{k+1}>A_k$, $B_{k+1}>B_k$ and any of the following cases holds:
\begin{enumerate}
    \item [{Case 1}.] $ \epsilon=1$ and $A_{k+1}-l_{k+1}=A_k-l_k,$
    \item [{Case 2}.] $ \epsilon=1$ and $B_{k+1}+l_{k+1}=B_k+l_k,$
    \item [{Case 3}.] $ \epsilon=-1$ and $B_{k+1}+l_{k+1}=A_k-l_k+1,$
\end{enumerate}
we define
\begin{align*}
     ui_{k}(\EE_{\rho})=\{ ([A_i',B_i']_{\rho},l_i',\eta_i')\}_{i \in (I_{\rho}, >)},
\end{align*} 
where $ ([A_i',B_i']_{\rho},l_i',\eta_i')=([A_i,B_i]_{\rho},l_i,\eta_i)$ for $i \neq k,k+1$, and $[A_k',B_k']_{\rho}=[A_{k+1},B_k]_{\rho}$, $[A_{k+1}',B_{k+1}']_{\rho}=[A_k,B_{k+1}]_{\rho}$, and $( l_k', \eta_k', l_{k+1}',\eta_{k+1}' )$ are given case by case as follows:
\begin{enumerate}
    \item[$(1)$] in Case 1, $( l_k', \eta_k', l_{k+1}',\eta_{k+1}' )= (l_k,\eta_k, l_{k+1}-(A_{k+1}-A_k), (-1)^{A_{k+1}-A_k}\eta_{k+1})$;
    \item [$(2)$] in Case 2, if $b_k-2l_k \geq A_{k+1}-A_k$, then
    \[( l_k', \eta_k', l_{k+1}',\eta_{k+1}' )= (l_k+(A_{k+1}-A_k),\eta_k, l_{k+1}, (-1)^{A_{k+1}-A_k}\eta_{k+1}),\]
    if $b_k-2l_k < A_{k+1}-A_k$, then
    \[( l_k', \eta_k', l_{k+1}',\eta_{k+1}' )= (b_k-l_k,-\eta_k, l_{k+1}, (-1)^{A_{k+1}-A_k}\eta_{k+1});\]
    \item [$(3)$] in Case 3, if $l_{k+1} \leq  l_k$, then
    \[( l_k', \eta_k', l_{k+1}',\eta_{k+1}' )= (l_k,\eta_k, l_{k+1}, (-1)^{A_{k+1}-A_k}\eta_{k+1}),\]
    if $l_{k+1}> l_{k}$, then
    \[( l_k', \eta_k', l_{k+1}',\eta_{k+1}' )= (l_k,\eta_k, l_{k}, (-1)^{A_{k+1}-A_k+1}\eta_{k+1});\]
    \item [$(3')$] if we are in Case 3 and $l_k=l_{k+1}=0$, then we delete $ ([A_{k+1}',B_{k+1}']_{\rho},l_{k+1}',\eta_{k+1}')$ from $ui_k(\EE_{\rho})$.
\end{enumerate}
Otherwise, we define $ui_k(\EE_{\rho})=\EE_{\rho}$. In any case, we define $ui_k(\EE)= \EE^{\rho} \cup ui_k(\EE_{\rho})$.

We say $ui_k$ is applicable on $\EE$ or $\EE_{\rho}$ if $ui_k(\EE)\neq \EE$. We say this $ui_k$ is of type 1 (resp. 2, 3, 3') if $\EE_{\rho}$ is in case 1 (resp. 2, 3, 3'). 
\end{defn}
We remark that the three cases in the above definition are exactly the extreme cases of the conditions in Proposition \ref{prop positive non-vanishing}(i)(1).

\begin{exmp}\label{ui case 2}
Let $\rho$ be the trivial representation and $$\psi=\rho\otimes S_3\otimes S_1+\rho\otimes S_4\otimes S_2+\rho\otimes S_6\otimes S_2$$ be a local Arthur parameter of good parity for $\Sp_{22}(F).$ Let 
$$\EE=\bordermatrix{
 &{1}&2&3 \cr
& \ominus& &  \cr
 & \lhd&  \rhd& \cr
& &\oplus &   \ominus \cr
}_{\rho}.$$
We have $\pi(\EE)= L(\Delta_{\rho}[1,-2];\pi(1^-,2^+,3^-))\in\Pi_\psi.$ Identify $(I_{\rho},>)$ as $\{1,2,3\}$ with $1<2<3$. We can apply $ui_2$ of type 2 (that is, we apply union-intersection to the second and the third rows of the pictograph of $\EE$) to obtain
$$
ui_2(\EE)=\bordermatrix{
 &{1}&2&3 \cr
& \ominus& &  \cr
 & \lhd&  \oplus & \rhd \cr
& &\ominus &   \cr
}_{\rho}.
$$
One can check that $\pi(ui_2(\EE))\cong\pi(\EE).$
\end{exmp}

\begin{exmp}
Let $\rho$ be the trivial representation and $$\psi=\rho\otimes S_4\otimes S_2+\rho\otimes S_7\otimes S_1$$ be a local Arthur parameter of good parity for $\Sp_{14}(F).$ Let 
$$\EE=\bordermatrix{
 &{1}&2&3 \cr
& \ominus& \oplus &  \cr
 & & & \ominus \cr
}_{\rho}.$$
We have $\pi(\EE)= L(\pi(1^-,2^+,3^-))\in\Pi_\psi.$
We can apply $ui_1$ of type 3' to obtain
$$
ui_1(\EE)=\bordermatrix{
 &{1}&2&3 \cr
& \ominus& \oplus & \ominus  \cr
}_{\rho}.
$$
Again, one can check by definition that $\pi(ui_1(\EE))\cong\pi(\EE).$
\end{exmp}
In the above examples, each union-intersection preserves the representation. Atobe showed that this is true in general.
\begin{thm}[{\cite[Corollary 5.3]{Ato20b}}]\label{thm ui}
Let $\EE \in \Rep$, we have 
     \[ \pi( ui_k (\EE)) \cong \pi(\EE). \]
\end{thm}

Along with row exchange and union-intersection, to describe another basic operator, we need to introduce the Aubert-Zelevinsky involution. In \cite{Aub95}, Aubert defined an involution on the irreducible representations of $G_n$ generalizing the one defined by Zelevinsky in \cite{Zel80}, which is called the Aubert-Zelevinsky dual or Aubert-Zelevinsky involution. We describe this construction below. 

Let $\pi$ be an irreducible representation of $G_n.$
In \cite{Aub95}, Aubert showed that there exists $\varepsilon\in\{\pm 1\}$ such that
$$
\widehat{\pi}:=\varepsilon\sum_P (-1)^{\mathrm{dim}(A_P)}[\mathrm{Ind}_{P}^{G_n}(Jac_P(\pi))]
$$
gives an irreducible representation. Here the sum is over all standard parabolic subgroups $P$ of $G_n$ and $A_P$ is the maximal split torus of the center of the Levi subgroup of $P.$ We say $\widehat{\pi}$ is the Aubert-Zelevinsky dual or Aubert-Zelevinsky involution of $\pi.$

Let $\psi=\bigoplus_\rho \bigoplus_{i \in I_\rho} \rho \otimes S_{a_i} \otimes S_{b_i}$ be a local Arthur parameter of good parity and set $\widehat{\psi}:=\bigoplus_\rho \bigoplus_{i \in I_\rho} \rho \otimes S_{b_i} \otimes S_{a_i}.$ From the compatibility of Aubert-Zelevinsky duality and twisted endoscopic character identities (see \cite[\S A]{Xu17b}), we have 
$$
\Pi_{\widehat{\psi}}=\{\widehat{\pi} \, | \, \pi\in\Pi_\psi\}.
$$

The next proposition shows that the Aubert-Zelevinsky dual is also compatible with taking derivatives.

\begin{prop}[{\cite[Proposition 3.9]{AM20}}]\label{compatability of Aubert-Zelevinsky dual and derivative} 
Let $\pi\in \Pi(G_n)$ and $\rho$ be an irreducible unitary self-dual supercuspidal representation of $\GL_d(F).$
\begin{enumerate}
    \item If  $D_{\rho|\cdot|^{x}}^{(k)}( \pi)$ is the highest $\rho|\cdot|^{x}$-derivative of $\pi$, then
    $$
    \widehat{D_{\rho|\cdot|^{x}}^{(k)}(\pi)}=D_{\rho|\cdot|^{-x}}^{(k)}(\widehat{\pi}).
    $$
    \item If $\pi$ is $\rho|\cdot|\inv$-reduced and $D_{\Delta_\rho[0,-1]}^{(k)}(\pi)$ is the highest $\Delta_\rho[0,-1]$-derivative of $\pi,$ then
    $$
    \widehat{D_{\Delta_\rho[0,-1]}^{(k)}(\pi)}=D_{Z_\rho[0,1]}^{(k)}(\widehat{\pi}).
    $$
\end{enumerate}
\end{prop}

Given an $\EE\in\Rep^{(P')},$ Atobe gave a formula \cite[\S6]{Ato20b} to compute a $dual(\EE)\in\Rep^{(P')}$ such that $\pi(dual(\EE))$ is the Aubert-Zelevinsky dual of $\pi(\EE).$ We record the definition of the operator \emph{dual}.
\begin{defn}[{\cite[Definition 6.1]{Ato20b}}]\label{dual segment}
Let $\EE= \cup_\rho \{([A_i,B_i]_{\rho},l_i,\eta_i)\}_{i\in (I_\rho, >)}$ be an extended multi-segment such that the admissible order $>$ on $I_{\rho}$ satisfies (P') for all $\rho$. We define 
$$dual(\EE)=\cup_{\rho}\{([A_i,-B_i]_{\rho},l_i',\eta_i')\}_{i\in (I_\rho, >')}$$ as follows:
\begin{enumerate}
    \item The order $>'$ is defined by $i>'j$ if and only if $j>i.$ 
    \item We set \begin{align*}
l_i'=\begin{cases}
l_i+B_i  & \mathrm{if} \, B_i\in\mathbb{Z},\\
 l_i+B_i+\frac{1}{2}(-1)^{\alpha_{i}}\eta_i  & \mathrm{if} \, B_i\not\in\mathbb{Z},
\end{cases}
\end{align*}
and
\begin{align*}
\eta_i'=\begin{cases}
(-1)^{\alpha_i+\beta_i}\eta_i  & \mathrm{if} \, B_i\in\mathbb{Z},\\
 (-1)^{\alpha_i+\beta_i+1}\eta_i  & \mathrm{if} \, B_i\not\in\mathbb{Z},
\end{cases}
\end{align*}
where $\alpha_{i}=\sum_{j\in I_\rho, j<i}a_j,$ and $\beta_{i}=\sum_{j\in I_\rho, j>i}b_j,$ $a_j=A_j+B_j+1$, $b_j=A_j-B_j+1$.
\item When $B_i\not\in\mathbb{Z}$ and $l_i=\frac{b_i}{2}$, we set $\eta_i=(-1)^{\alpha_i+1}.$
\end{enumerate}
If $\FF= \EE_{\rho}$, we define $dual(\FF):= (dual(\EE))_{\rho}$.
\end{defn} 

\begin{exmp}
Let $\rho$ be a symplectic representation of dimension $d$ and
$$\psi=\rho\otimes S_1\otimes S_5+\rho\otimes S_7\otimes S_1$$ be a local Arthur parameter of good parity for $\SO_{12d+1}(F).$ Let 
$$\EE=\bordermatrix{
 &-3 &-2&-1&0&1&2&3 \cr
& & \lhd& \lhd& \oplus & \rhd &\rhd &  \cr
 & & & & & & & \oplus \cr
}_{\rho}.$$
We have $\pi(\EE)= L(\Delta_\rho[-2,-2],\Delta_\rho[-1,-1];\pi(0^+,3^+))\in\Pi_\psi.$
Now,
$$
dual(\EE)=\bordermatrix{
 &-3 &-2&-1&0&1&2&3 \cr
&\lhd & \lhd& \lhd& \ominus & \rhd &\rhd & \rhd \cr
 & & & & & & \ominus &  \cr
}_{\rho}.$$
We compute $$\pi(dual(\EE))=L(\Delta_\rho[-3,-3],\Delta_\rho[-2,-2],\Delta_\rho[-1,-1];\pi(0^-,2^-))\in\Pi_{\widehat{\psi}}.$$
Using \cite[Algorithm 4.1]{AM20}, we see $\widehat{\pi(\EE)}\cong\pi(dual(\EE)).$
\end{exmp}

Atobe proved that the dual of an extended multi-segment gives the Aubert-Zelevinsky dual of the corresponding representation. 

\begin{thm}[{\cite[Theorem 6.2]{Ato20b}}]\label{thm Aubert-Zelevinsky dual formula}

Suppose $\EE \in \Rep^{(P')}$, then 
\[ \pi(dual(\EE)) \cong \widehat{\pi(\EE)}. \]
\end{thm}

Finally, we list several identities among the operators defined in this section, which follow from direct computations, details are omitted. 
\begin{lemma}\label{lem identities}
Suppose $\EE= \cup_{\rho} \{ ([A_i,B_i]_{\rho}, l_{i}, \eta_i) \}_{i \in (I_{\rho, >})}$. Let $i,k \in I_{\rho}$ and $t,s \in \Z$. The following hold.
\begin{enumerate}
    \item [(i)]Let $T \in \{sh_i^{t}, add_i^{s}\}$. Suppose $R_{k}(\EE) \neq \EE$ and $R_{k}(T(\EE))\neq T(\EE)$. Then
    \[ T \circ R_k (\EE)= R_k \circ T(\EE). \]
    In other words, $sh_i^{t}$, $add_i^{s}$ commute with $R_{k}$.
    \item [(ii)] Suppose the admissible orders of $\EE$ and $sh_k^t(\EE)$ both satisfy (P'). Then
        \[dual \circ sh_{k}^t(\EE)= add_{k}^t \circ dual(\EE).\]
\end{enumerate}
\end{lemma}

\section{Shift of extended multi-segments}\label{sec shift and add}

In this section, we compare the representations of $\EE$ and certain types of shift of $\EE$. We show that there is a close relation between the $L$-data of $\pi(\EE)$ and $\EE$ (see Lemma \ref{lem $L$-data}). From this relation, we determine invariants that can be read from $\pi(\EE)$ (see Theorems \ref{thm max A}, \ref{thm min B+l}). These invariants allow us to narrow down the set of all possible extended multi-segments $\EE'$ such that  $\pi(\EE')\cong\pi(\EE).$

We recall the importance of the segment $[A_i,B_i]_{\rho}$ associated with the summand $\rho \otimes S_{a_i} \otimes S_{b_i}$. Let $\Delta: \SL_2(\BC) \to \SL_2(\BC) \times \SL_2(\BC)$ be the diagonal embedding. The composition, which we denote by the diagonal restriction of $\rho \otimes S_{a_i} \otimes S_{b_i}$,
\[ W_F \times \SL_2(\BC) \xrightarrow[]{id\times \Delta} W_F \times \SL_2(\BC) \times \SL_2(\BC) \xrightarrow[]{\rho \otimes S_{a_i}\otimes S_{b_i}} GL_{d a_i b_i} \]
decomposes into a direct sum
\[ \rho \otimes S_{a_i+b_i-1} \oplus \rho \otimes S_{a+b-3} \oplus \cdots \oplus \rho \otimes S_{|a_i-b_i|+1}. \]
As $A_i= \half{a_i+b_i}-1$ and $B_i= \half{a_i-b_i}$, the segment $[A_i,B_i]_{\rho}$ records the information of this diagonal restriction and the order of the pair $(a_i,b_i)$. Under the identification
\[ [A_i,B_i]_{\rho} \longleftrightarrow (\rho, A, |B|, B/|B|),\]
(if $B=0$, then $B/|B|$ is replaced by an arbitrary choice in $\{+1,-1\}$) these segments are exactly the ``Jordan blocks" in the theory of M{\oe}glin and Xu (\cite{Moe06a, Xu17a}), which plays an important role in the construction.

Based on the discussion, when $\FF$ is a part of an extended multi-segment $\EE$, we give the following definition to collect the sum of segments in $\FF$.

\begin{defn}\label{def omega E} \
\begin{enumerate}
    \item Suppose 
\[ \FF=\{([A_i,B_i]_{\rho},l_i,\eta_i)\}_{i \in (I_{\rho, >})}. \]
(For example, $\FF= \EE_{\rho}$ for some extended multi-segment $\EE$.) We define ordered multi-sets  
\begin{align*}
    \Omega(\FF):=& \sum_{i\in I_{\rho}} [A_i,B_i]_{\rho}= \{ \rho|\cdot|^{\alpha_1} ,\dots,\rho|\cdot|^{\alpha_t} \},\\
    \overline{\Omega(\FF)}:=& \sum_{i\in I_{\rho}} [B_i,-A_i]_{\rho}= \{ \rho|\cdot|^{\beta_1} ,\dots,\rho|\cdot|^{\beta_r} \},
\end{align*}
where $\alpha_1 \leq \cdots\leq \alpha_t$ and $\beta_1 \geq \cdots\geq \beta_r$. 
Suppose $\EE=\cup_{\rho} \EE_{\rho}$ is an extended multi-segment. Fix an arbitrary order on the set 
\[ \{ \rho \ | \ \EE_{\rho} \neq \emptyset \}= \{ \rho_1,\dots,\rho_r\}.\]
We define multi-sets $\Omega(\EE)$ and $\overline{\Omega(\EE)}$ to be the sum of multi-sets 
\begin{align*}
    \Omega(\EE)&:= \Omega(\EE_{\rho_1}) +\cdots+ \Omega(\EE_{\rho_r}),\\
    \overline{\Omega(\EE)}&:= \overline{\Omega(\EE_{\rho_1})} +\cdots+ \overline{\Omega(\EE_{\rho_r})}.
\end{align*}

\item For each ordered multi-set $\Omega= \{ \rho_1|\cdot|^{\gamma_1} ,\dots,\rho_t|\cdot|^{\gamma_t} \}$, we define
\begin{align*}
     D_{\Omega}&:= D_{\rho_t|\cdot|^{\gamma_t}} \circ\cdots\circ D_{\rho_1|\cdot|^{\gamma_1}},\\
      S_{\Omega}&:= S_{\rho_1|\cdot|^{\gamma_1}} \circ\cdots\circ S_{\rho_t|\cdot|^{\gamma_t}}.
\end{align*}
For an extended multi-segment $\EE$, we define $D_{\Omega(\EE)}= \circ_{\rho} D_{\Omega(\EE_{\rho})}$. Note that the derivative is independent of the composition order of $D_{\Omega(\EE_\rho)}$.
\end{enumerate}
\end{defn}

We remark that $\overline{\Omega(\FF)}$ can be obtained by negating the exponent in $\Omega(dual(\FF))$. If $\FF$ is positive and $D_{\Omega(\FF)}(\pi)$ is a composition of highest derivatives up to a scalar, then we have 
\[ D_{\overline{\Omega(dual(\FF))} }( \widehat{\pi})= \reallywidehat{D_{\Omega(\FF)}(\pi)}  \]
 by Proposition \ref{compatability of Aubert-Zelevinsky dual and derivative}.

\subsection{Shift of a block}
In this subsection, we prove several technical lemmas which will be used frequently. 

The following lemma is a step in the proof of \cite[Proposition 8.5]{Xu17b}. We record it here for completeness. It shows that if $\EE \in \Rep^{(P')}$, then each derivative in the construction of $\pi(\EE)$ in Definition \ref{def rep of segment} is highest, and hence it can be computed by the formulae in \cite{AM20}.
\begin{lemma} \label{lem derivative algorithm}
Suppose $\EE \in \Rep^{(P')}$ where $\EE_{\rho}= \{([A_i,B_i]_{\rho},l_i, \eta_i)\}_{i \in (I_\rho, >)}$. Take a sequence of non-negative integers $\{t_i\}_{i\in I_{\rho}}$ such that 
\[\begin{cases}
 B_i+t_i \geq 0 & \text{ for } i \in I_{\rho},\\
 B_i + t_i > A_{j}+t_{j}& \text{ for all } i >j \in I_{\rho}. \, \, \\
\end{cases}\]
 Fix $k \in I_{\rho}$ and denote $\EE'=(\sum_{i>k} sh_{i}^{t_i})(\EE)$. Then the following holds.
 \begin{enumerate}
     \item [(i)] We have
     \[D_{\rho|\cdot|^{B_k+1,\dots,A_k+1}}(\pi(sh^{1}_k (\EE')))= \pi(\EE').\]
     \item [(ii)] There is an injection
     \[ \pi(sh^{1}_k (\EE')) \hookrightarrow Z_{\rho}[B_k+1, A_k+1] \rtimes \pi(\EE'). \]
     \item [(iii)] Each derivative in (i) is highest.
 \end{enumerate}
\end{lemma}

\begin{proof}
Part (i) follows directly from the definition. Indeed, denote $\EE''= (\sum_{i \leq k}sh_i^{t_i} )(\EE')$. Then we have 
\begin{align*}
    \pi(\EE')= &\left(D_{\rho|\cdot|^{B_k+1,\dots,A_k+1}} \circ\cdots\circ D_{\rho|\cdot|^{B_k+t_k,\dots,A_k+t_k}}  \right)\\
    &\circ_{i \leq k-1} \left(D_{\rho|\cdot|^{B_i+1,\dots,A_i+1}} \circ\cdots\circ D_{\rho|\cdot|^{B_i+t_i,\dots,A_i+t_i}}  \right)\pi(\EE''),\\
    \pi(sh_k^1(\EE'))= &\left(D_{\rho|\cdot|^{B_k+2,\dots,A_k+2}} \circ\cdots\circ D_{\rho|\cdot|^{B_k+t_k,\dots,A_k+t_k}}  \right)\\
    &\circ_{i\leq k-1} \left(D_{\rho|\cdot|^{B_i+1,\dots,A_i+1}} \circ\cdots\circ D_{\rho|\cdot|^{B_i+t_i,\dots,A_i+t_i}}  \right)\pi(\EE''),
\end{align*}
where the composition order is the same as the one in Definition \ref{def rep of segment}.

For Part (ii), by Lemma \ref{lem Frobenius},
\[ \pi(sh^{1}_k (\EE')) \hookrightarrow \rho|\cdot|^{B_k+1}\times\cdots\times \rho|\cdot|^{A_k+1} \rtimes \pi(\EE'), \]
and there exists an irreducible constituent $\tau$ of $\rho|\cdot|^{B_k+1}\times\cdots\times \rho|\cdot|^{A_k+1}$ such that $\Hom(\pi(sh_k^{1}( \EE')), \tau \rtimes \pi(\EE')) \neq 0$. Applying Frobenius reciprocity, for any $x \in \R$, if
$D_{\rho|\cdot|^{x}}( \tau) \neq 0$, then  $D_{\rho|\cdot|^{x}}( \pi(sh_k^{1}( \EE')))\neq 0.$
On the other hand, since the admissible order of $\EE$ satisfies ($P'$), $B_i<B_k+1$ for any $i<k$, and hence Proposition \ref{prop derivative support} implies  
\[ \{ x \in \{B_k+1,\dots,A_k+1\}\ | \ D_{\rho|\cdot|^{x}}(\pi(sh_k^{1}( \EE')))\neq 0\} \subseteq \{ B_k+1\}. \]
Therefore, $D_{\rho|\cdot|^{x}}(\tau)=0$ unless $x= B_{k}+1$. Consequently, we must have $\tau=Z_{\rho}[B_k+1,A_k+1]$ by \cite[Lemma 5.7]{Xu17a} and hence proves Part (ii).

Applying Proposition \ref{prop derivative support} again, we see that $\pi(\EE')$ is $\rho|\cdot|^{x}$-reduced for $x \in \{B_k+1,\dots,A_k+1\}.$ Thus Part (iii) follows directly from the injection in Part (ii) and Lemma \ref{lem Leibniz rule}. This completes the proof of the lemma. 
\end{proof}

Next, we generalize the above lemma in several aspects.

\begin{lemma} \label{lem shift}
Let $\EE\in \Rep^{(P')}$ and 
\[\EE_{\rho}= \{([A_i,B_i]_{\rho},l_i,\eta_i)\}_{i \in (I_{\rho}, >)}.\] The following holds. 
\begin{enumerate}
    \item [(i)] For any $j \in I_{\rho}$, $ \left(\sum_{i>j} sh_i^1\right)(\EE) \in \Rep^{(P')} $.
    \item [(ii)] Suppose $sh^1_k(\EE) \in \Rep^{(P')}$ for some $k \in I_{\rho}$ (in particular, $B_i \geq  B_k+1$ for $i>k$). Then  \[\pi(sh^{1}_k(\EE)) \hookrightarrow Z_{\rho}[ B_{k}+1, A_{k}+1] \rtimes \pi(\EE).\]
    In particular, $D_{\rho|\cdot|^{B_k+1,\dots,A_k+1}}(\pi(sh_{k}^{1}(\EE))) \geq \pi(\EE)\neq 0$.
    \item [(iii)] Suppose $k \in I_{\rho}$ satisfies
    \begin{enumerate}
        \item [$\oldbullet$] $B_k \geq 0$;
        \item [$\oldbullet$]$ B_k<B_i, A_k\geq A_i$ for all $i >k$.
    \end{enumerate}
    Then $sh_k^{1}(\EE) \in \Rep^{(P')}$. In this case, $\pi(sh^{1}_k(\EE))$ is the unique irreducible subrepresentation of $ Z_{\rho}[ B_{k}+1, A_{k}+1] \rtimes \pi(\EE)$, and 
    \[D_{\rho|\cdot|^{B_k+1,\dots,A_k+1}}(\pi(sh_{k}^{1}(\EE))) = \pi(\EE).\]
    \item [(iv)] Fix $j \in I_{\rho}$ and decompose $ \EE_{\rho}= \FF_1 + \FF_2$ where
    \begin{align*}
        \FF_1&=\{([A_i,B_i]_{\rho},l_i,\eta_i)\}_{i<j},\\
        \FF_2&=\{([A_i,B_i]_{\rho},l_i,\eta_i)\}_{i \geq j}.
    \end{align*}
    Then we have an injection
    \[ \pi\left( \EE^{\rho} \cup (\FF_1 + sh^1(\FF_2))\right) \hookrightarrow \times_{i\geq j} Z_\rho[B_i+1, A_i+1] \rtimes \pi(\EE). \]
    Moreover, $D_{\Omega(sh^1(\FF_2))} (\pi(\EE^{\rho} \cup (\FF_1 + sh^1(\FF_2)) )$ is a composition of highest derivative (modulo the factorial in the definition), and up to a multiplicity,
    \[ D_{\Omega(sh^1(\FF_2))}\left(\pi\left(\EE^{\rho} \cup (\FF_1 + sh^1(\FF_2)\right)\right)= \pi(\EE).\]
    If $B_i+1>0$ for all $i \geq j$, we also have 
    \[ \pi\left(\EE^{\rho} \cup (\FF_1 + sh^1(\FF_2)\right)= S_{\Omega(sh^1(\FF_2))}(\EE). \]
    \item [(v)] Suppose there exists a decomposition $ \EE_{\rho}= \FF_1 + \FF_2 +\FF_3$, where 
    \begin{align*}
    \FF_1&=\{([A_i,B_i]_{\rho},l_i,\eta_i)\}_{i<k},\\
        \FF_2&=\{([A_i,B_i]_{\rho},l_i,\eta_i)\}_{k \leq i<m},\\
        \FF_3&=\{([A_i,B_i]_{\rho},l_i,\eta_i)\}_{k \geq i\geq m},
    \end{align*}
    such that 
    \begin{enumerate}
        \item [$\oldbullet$] $ B_i=B$ for all $  k\leq i <m$,
        \item [$\oldbullet$] $B_i \leq B$ for all $ i <k$,
        \item [$\oldbullet$] $B_i >A_j+1$ for all $i\geq m$ and $ k \leq j <m$,
    \end{enumerate}
    for some $B \in \half{1} \Z$. Then $\pi( \EE^{\rho} \cup (\FF_1 + sh^1(\FF_2)+\FF_3))$ is the unique irreducible subrepresentation of 
    \[ \times_{i=k}^{m-1} Z_{\rho}[B+1, A_i+1] \rtimes \pi(\EE), \]
    and up to a multiplicity,
    \[ D_{\Omega(sh^1(\FF_2))}(\pi(\EE^{\rho} \cup (\FF_1 + sh^1(\FF_2)+\FF_3)))  \cong \pi(\EE).\]
\end{enumerate}
\end{lemma}

\begin{proof}
\textbf{Proof of (i).}
We check that $\left(\sum_{i\geq j} sh_i^1\right)(\EE_{\rho})$ satisfies Theorem \ref{thm non-vanishing} directly. The condition ($\ast$) in Theorem \ref{thm non-vanishing}(i) is clearly preserved under shift, so we focus on the examination of Part (ii) of Theorem \ref{thm non-vanishing}.

For brevity, denote $\FF= \EE_{\rho}$ and $\FF'=\left(\sum_{i \geq j} sh_i^1\right)(\FF)$, and denote their index sets by $I, I'$. We may identify $I$ and $I'$ in the obvious manner. It is not hard to see that any admissible order $\gg'$ on $I'$ is also an admissible order on $I$, which we denote by $\gg$. 

Now we fix an admissible order $\gg'$ of $I'$ and check the non-vanishing conditions in Proposition \ref{prop positive non-vanishing}(i) for all adjacent pair $(i_1, i_2,\gg')$ of $ \FF'$. We again identify $i_1,i_2$ as elements of $I$.

By Lemma \ref{lem identities}(i), we have 
\[ \FF_{\gg'}'=  \left(\sum_{i \geq j} sh_i^1\right)(\FF_{\gg}). \]
Now suppose both of $i_1\geq j $ and $i_2 \geq j$ (or $i_1 <j$ and $i_2<j$), then since the adjacent pair $(i_1,i_2,\gg)$ of $\FF$ satisfies Proposition \ref{prop positive non-vanishing}, so does $(i_1,i_2, \gg')$ of $\FF'$. Therefore, we only need to deal with the following cases
\begin{enumerate}
    \item [(a)] $i_1 \geq j$ and $i_2 <j$,
    \item [(b)] $i_1 <j$ and $i_2\geq j$.
\end{enumerate}
 A key observation is that since the admissible order $>$ of $I$ satisfies ($P'$), (a) implies $ B_{i_1} \geq B_{i_2}$, and (b) implies $ B_{i_1} \leq B_{i_2}$.

Suppose the adjacent pair $(i_1,i_2,\gg)$ is in the case of Proposition \ref{prop positive non-vanishing}(i)(1) but not in (2). Then $ B_{i_1}> B_{i_2}$, so only case (a) is possible. One can see that $(i_1,i_2, \gg')$ is still in the case of Proposition \ref{prop positive non-vanishing}(i)(1), and the condition is satisfied.  

Suppose the adjacent pair $(i_1,i_2,\gg)$ is in the case of Proposition \ref{prop positive non-vanishing}(i)(2). Then there are two possibilities.
\begin{enumerate}
    \item [$\oldbullet$] If $B_{i_1}<B_{i_2}$, then it is of case (b). Therefore, the adjacent pair $(i_1,i_2,\gg')$ is still in the case of Proposition \ref{prop positive non-vanishing}(i)(2), and the condition is satisfied.
    \item [$\oldbullet$] If $B_{i_1}= B_{i_2}$, then both cases (a) and (b) are possible. If it is of case (a), then the adjacent pair $(i_1,i_2,\gg')$ is in the setting of Proposition \ref{prop positive non-vanishing}(i)(1), and one can see that 
    \begin{align*}
    0 \leq l_{i_1}-l_{i_2} \leq b_{i_1}-b_{i_2} &\Longrightarrow \begin{cases} (A_{i_2}+1) -l_{i_2} \geq A_{i_1}-l_{i_1}\\
    (B_{i_2}+1)+l_{i_2} \geq B_{i_1}+l_{i_1}\end{cases}\\
    l_{i_2}+l_{i_1} \geq b_{i_2} &\Longrightarrow (B_{i_2}+1)+l_{i_2} \geq A_{i_1}-l_{i_1},
\end{align*}
and hence the condition is satisfied. If it is of case (b), then one can see that $(i_1,i_2, \gg')$ is still in the setting of Proposition \ref{prop positive non-vanishing}(i)(2), and the condition is satisfied.  
\end{enumerate}
The case that the adjacent pair $(i_1,i_2,\gg)$ is in the case of Proposition \ref{prop positive non-vanishing}(i)(3) is similar as the previous one and we omit the detail. This completes the proof of (i).

\textbf{Proof of (ii).}
Take a sequence of non-negative integers $\{t_i\}_{i \in I_{\rho}}$ such that \[\begin{cases}
 B_i+t_i \geq 0 & \text{ for } i \in I_{\rho},\\
 B_i + t_i > A_{j}+t_{j}& \text{ for all } i >j \in I_{\rho}. \, \, \\
\end{cases}\]
Denote $\EE'=(\sum_{i > k} sh^{t_i}_{i}) (\EE) $ and 
\[ D= \circ_{i>k} \left( D_{ \rho|\cdot|^{ B_{i}+1,\dots,A_i+1}} \circ\cdots \circ D_{\rho|\cdot|^{B_i+t_i,\dots,A_i+t_i}} \right), \]
where the order of the composition is the same as the one in Definition \ref{def rep of segment}. (If $k=n$, then $\EE'=\EE$ and we take $D$ to be identity.) By definition, we have \begin{enumerate}
    \item [$\oldbullet$] $D_{\rho|\cdot|^{B_{k}+1,\dots,A_k+1}}(\pi(sh_k^{1}(\EE'))) =\pi( \EE')$,
    \item [$\oldbullet$] $D(\pi(\EE'))=\pi(\EE)$,
    \item [$\oldbullet$] $D(\pi(sh^{1}_k(\EE')))= \pi(sh^{1}_k(\EE))$.
\end{enumerate}
Then, by Lemma \ref{lem derivative algorithm}(ii), 
\begin{align}\label{eq lem shift of a block (ii)}
    \pi(sh^{1}_k(\EE')) \hookrightarrow  Z_{\rho}[B_k+1,A_k+1] \rtimes \pi( \EE').
\end{align}

The idea of the proof is to pass the above injection from $\EE'$ to $\EE$ by taking the  derivative $D$. Let $\Omega=\{\rho|\cdot|^{x_1},\dots,\rho|\cdot|^{x_t}\}$ be the ordered multi-set such that $D= D_{\Omega}$. We define $\pi_r$ and $\pi_r'$ for $0\leq r \leq t$, inductively by 
\[ \begin{cases} \pi_0= \pi(\EE'),\\
\pi_r= D_{\rho|\cdot|^{x_r}}(\pi_{r-1}),  \end{cases}\ \begin{cases}
\pi_0'=\pi(sh_k^1(\EE')),\\
\pi_r'= D_{\rho|\cdot|^{x_r}}(\pi_{r-1}').  \end{cases}\]
Note that $\pi_t= \pi(\EE)$ and $\pi_t'= \pi(sh_k^1(\EE))$. By Lemma \ref{lem derivative algorithm}(iii), the derivative $D_{\rho|\cdot|^{x_r}}(\pi_{r-1})$ (resp. $D_{\rho|\cdot|^{x_r}}(\pi_{r-1}')$) is highest, hence, each $\pi_r$ (resp. $\pi_{r}'$) is irreducible and  $D_{\rho|\cdot|^{x_r}}$-reduced.

To finish the proof of (ii), it suffices to show that for any $ 0 \leq r \leq t$, we have
\[ \pi_r' \hookrightarrow  Z_{\rho}[B_k+1,A_k+1] \rtimes \pi_r. \]
We show this by applying induction on $r$. The case $r=0$ is already done in \eqref{eq lem shift of a block (ii)}. 

First, we claim that
\[ \pi_{r-1}' \hookrightarrow  \rho|\cdot|^{x_r} \times Z_{\rho}[B_k+1,A_k+1] \rtimes \pi_r. \]
Indeed, if $ x_r \neq A_k+2$, then
\begin{align*}
    \pi_{r-1}' &\hookrightarrow Z_{\rho} [B_k+1,A_k+1] \rtimes \pi_{r-1}\\
    & \hookrightarrow Z_{\rho} [B_k+1,A_k+1] \times \rho|\cdot|^{x_r} \rtimes \pi_r\\
    &= \rho|\cdot|^{x_r} \times Z_{\rho} [B_k+1,A_k+1]  \rtimes \pi_r,
\end{align*}
where the first injection is the induction hypothesis, the second injection follows from Lemma \ref{lem Frobenius}, and the last equation holds since the segments $[x_r,x_r]_{\rho}$ and $[A_k+1,B_k+1]_{\rho}$ are not linked. ($x_r \geq B_k+1$ by construction.)

If $x_r= A_k+2$, then the segments $[x_r,x_r]_{\rho}$ and $[A_k+1,B_k+1]_{\rho}$ are linked. However, we have a short exact sequence 
\begin{align*}
    0 \to Z_{\rho}[B_k+1, A_k+2]  &\to  Z_{\rho}[B_k+1,A_k+1] \times \rho|\cdot|^{A_k+2}  \\
    &\to  soc(\rho|\cdot|^{A_k+2}  \times Z_{\rho}[B_k+1,A_k+1]) \to 0.
\end{align*}
Hence, using the left exactness of Hom functor, we see $\pi_{r-1}'$ injects to one of 
\begin{align*}
    \tau_1&= soc(\rho|\cdot|^{x_r}  \times Z_{\rho}[B_k+1,A_k+1]) \rtimes \pi_r,\\
    \tau_2&=  Z_{\rho}[B_k+1,A_k+2] \rtimes \pi_r.
\end{align*}
If $\pi_{r-1}'$ injects into $\tau_2$, then
\begin{align*}
\pi_{r}'=&D_{\rho|\cdot|^{x_r}}(\pi_{r-1}')\\
\leq& D_{\rho|\cdot|^{x_r}}(\tau_2) \\
=& D_{\rho|\cdot|^{x_r}}( Z_{\rho}[B_k+1,A_k+2] )\rtimes \pi_r + D_{\rho|\cdot|^{-x_r}}^{op}( Z_{\rho}[B_k+1,A_k+2] )\rtimes \pi_r\\
&+ Z_{\rho}[B_k+1,A_k+2] \rtimes D_{\rho|\cdot|^{x_r}}(  \pi_r)\\
=&0,
\end{align*}
where the second and third equality follows from Lemma \ref{lem Leibniz rule} and the fact that $\pi_r$ is $D_{\rho|\cdot|^{x_r}}$-reduced. This contradicts to the fact that $\pi_r'\neq 0$. Therefore,
\begin{align*}
     \pi_{r-1}' &\hookrightarrow  soc(\rho|\cdot|^{x_r}  \times Z_{\rho}[B_k+1,A_k+1]) \rtimes \pi_r \\
     &\hookrightarrow \rho|\cdot|^{x_r}  \times Z_{\rho}[B_k+1,A_k+1] \rtimes \pi_r.
\end{align*}
This completes the proof of the claim.

Finally, we apply Frobenius reciprocity to show that 
\[0 \neq \Hom( \pi_{r}', Z_{\rho}[B_k+1,A_k+1] \rtimes \pi_r).\]
Indeed, say $\pi_{r-1}'$ (resp. $\pi_r$) is a representation on the group $G$ (resp. $G^{-}$). Let $ P=MN$ be the standard parabolic subgroup of $G$ such that $\rho|\cdot|^{x_r} \otimes \pi_{r}' $ is a representation on $M$. Then
\begin{align*}
    0 &\neq \Hom_G(\pi_{r-1}', \rho|\cdot|^{x_r}  \rtimes (Z_{\rho}[B_k+1,A_k+1] \rtimes \pi_r) )\\
    &=\Hom_M( Jac_P(\pi_{r-1}'), \rho|\cdot|^{x_r}  \otimes (Z_{\rho}[B_k+1,A_k+1] \rtimes \pi_r) )\\
    &=\Hom_{G^{-}}( D_{\rho|\cdot|^{x_r}}(\pi_{r-1}'),  Z_{\rho}[B_k+1,A_k+1] \rtimes \pi_r )\\
     &=\Hom_{G^{-}}( \pi_{r}',  Z_{\rho}[B_k+1,A_k+1] \rtimes \pi_r ).
\end{align*}
This completes the proof of Part (ii).

\textbf{Proof of (iii).}
We first show that $\pi(sh_k^1(\EE)) \neq 0$. Indeed, by Theorem \ref{thm non-vanishing}, we may assume $B_i >0$ for all $i$. In this case, the assumption shows $[A_k,B_k]_{\rho} \supsetneq [A_i,B_i]_{\rho}$ for all $i >k$, and hence we may consider a new total order $\gg$ defined by
\begin{enumerate}
    \item [$\oldbullet$] $ k \gg i$ for all $ i \in I_{\rho}-\{k\}$.
    \item [$\oldbullet$] For $i,j \in I_{\rho}- \{k\}$, $i \gg j$ if and only if $i >j$. 
\end{enumerate}
It is admissible by the assumption. Denote $\EE'= \EE^{\rho} \cup \EE_{\rho, \gg}$. Note that Theorem \ref{thm row exchange} indicates $\pi(\EE')\cong \pi(\EE)$. Then by the construction of  $\pi(\EE')$, we see that
\[ D_{\rho|\cdot|^{B_k+1,\dots,A_k+1}} (\pi(sh_k^1(\EE')))= \pi(\EE') \cong \pi(\EE)\neq 0. \]
Therefore, $\pi(sh_k^1(\EE'))\neq 0$. By Lemma \ref{lem identities}(i) and Theorem \ref{thm row exchange}, we have $\pi(sh_k^1(\EE)) =\pi(sh_k^1(\EE'))\neq 0$ and
\[ D_{\rho|\cdot|^{B_k+1,\dots,A_k+1}} (\pi(sh_k^1(\EE)))= \pi(\EE). \]

By applying Part (ii), we obtain an injection
\[ \pi(sh_k^1(\EE)) \hookrightarrow Z_{\rho}[B_k+1,A_k+1] \rtimes \pi(\EE).\]
Now we show that $\pi(sh_k^1(\EE))$ is the unique irreducible subrepresentation of the right hand side. 

Indeed, if $\pi$ is an irreducible subrepresentation of the right hand side, then Frobenius reciprocity shows 
\[ D_{\rho|\cdot|^{B_k+1,\dots,A_k+1}}(\pi) \neq 0.\]
Thus, if one can show
\begin{align}\label{eq 4.1}
    D_{\rho|\cdot|^{B_k+1,\dots,A_k+1}}  (Z_{\rho}[ B_k+1, A_k+1] \rtimes \pi( \EE))= \pi(\EE),
\end{align}
then by comparing the length,
\[ D_{\rho|\cdot|^{B_k+1,\dots,A_k+1}}  (Z_{\rho}[ B_k+1, A_k+1] \rtimes \pi( \EE)-\pi)=0\]
in the Grothendieck group. Therefore, $Z_{\rho}[ B_k+1, A_k+1] \rtimes \pi( \EE)$ has a unique irreducible subrepresentation.

Now we show \eqref{eq 4.1}. By Lemma \ref{lem Leibniz rule},
\begin{align*}
     &D_{\rho|\cdot|^{B_k+1,\dots,A_k+1}}  (Z_{\rho}[ B_k+1, A_k+1] \rtimes \pi( \EE))\\
     =&\pi(\EE) +\sum_{t=0}^{A_k-B_k} Z_{\rho}[B_k+1+t,A_k+1] \rtimes D_{ \rho|\cdot|^{B_{k}+t+1,\dots,A_{k}+1}}( \pi(\EE)) \\
     =& \pi(\EE),
\end{align*}
where the last equality follows from Proposition \ref{prop derivative support}(ii). This completes the proof of Part (iii).\\

\textbf{Proof of (iv).}
Combining Parts (i) and (ii), one can derive the injection in Part (iv). By Frobenius reciprocity, we have 
\[D_{\Omega(sh^1(\FF_2))}(\pi\left(\EE^{\rho} \cup (\FF_1 + sh^1(\FF_2)\right)) \geq \pi(\EE)\neq 0.\] 

Now we show that $D_{\Omega(sh^1(\FF_2))}$ is a composition of highest derivatives modulo a scalar, and then the assertion about $S_{\Omega(sh^1(\EE_2))}$ follows from Theorem \ref{thm derivative-socle}(3).

We apply induction on the cardinality of the set $\mathcal{B}=\{B_i \ | \ i \geq j\}$. If $\mathcal{B}$ is a singleton, then $\pi(\EE)$ is $\rho|\cdot|^{B_n+1}$-reduced by Proposition \ref{prop derivative support}(i), and hence for some positive integer $m$, 
\begin{align*}
    &D_{\Omega(sh^1(\FF_2))}(\times_{i\geq j} Z_{\rho}[B_i+1,A_i+1]\rtimes \pi(\EE))\\
    =&D_{\Omega(sh^1(\FF_2))}(\times_{i \geq j} Z_{\rho}[B_i+1,A_i+1]) \rtimes \pi(\EE)\\
    =&m \cdot \pi(\EE),
\end{align*}
and it is a composition of highest derivatives up to a scalar. Therefore, $D_{\Omega(sh^1(\FF_2))}(\pi)$ is also a composition of highest derivatives for any irreducible subrepresentation $\pi$ of $ \times_{i\geq j} Z_{\rho}[B_i+1,A_i+1]\rtimes \pi(\EE)$ as long as $D_{\Omega(sh^1(\FF_2))}(\pi) \neq 0$. As a consequence, $D_{\Omega(sh^1(\FF_2))}(\pi\left(\EE^{\rho} \cup (\FF_1 + sh^1(\FF_2)\right))$ is a composition of highest derivatives up to a scalar.

When $\mathcal{B}$ is not a singleton, we assume that $ B_i=B_j$ for $j \leq i<k$ and $B_{k}> B_j$. We let $B=B_j$ and let
\[ \FF_3= \{([A_i, B]_{\rho},l_i,\eta_i)\}_{j \leq i <k},\ \FF_4= \{([A_i, B_i]_{\rho},l_i,\eta_i)\}_{i \geq k}.\]
Note that $\FF_2=\FF_3+\FF_4.$ We also define the ordered multi-sets
\begin{align*}
    \Omega_1&=\{ \rho|\cdot|^{x} \in \Omega(sh^1(\FF_3)) \ | \ x < B_{k}+1\}, \\
    \Omega_2&=\{ \rho|\cdot|^{x} \in \Omega(sh^1(\FF_2)) \ | \ x \geq B_{k}+1\},
\end{align*}
with the orders induced from the bigger ordered multi-sets containing them. Note that we have $\Omega(sh^1(\FF_2))=\Omega_1 \cup \Omega_2$, and hence $D_{\Omega(sh^1(\FF_2))}= D_{\Omega_2} \circ D_{\Omega_1}$.

Since $\pi( \EE^{\rho} \cup(\FF_1+ \FF_3 + sh^1(\FF_4)))$ and $\pi( \EE^{\rho} \cup(\FF_1+ sh^1(\FF_3) + sh^1(\FF_4))) $ are both in $\Rep^{(P')}$, applying Part (ii) repeatedly, we have
\[ \pi( \EE^{\rho} \cup(\FF_1+ sh^1(\FF_3) + sh^1(\FF_4))) \hookrightarrow \times_{j \leq i <k}Z_{\rho}[B+1,A_i+1] \rtimes  \pi( \EE^{\rho} \cup(\FF_1+ \FF_3 + sh^1(\FF_4))). \]
Proposition \ref{prop derivative support}(i) implies $\pi( \EE^{\rho} \cup(\FF_1+ \FF_3 + sh^1(\FF_4)))$ is $D_{\rho|\cdot|^{x}}$-reduced for $B+1 \leq x <B_k+1$. Therefore, for some positive integer $m_1$,
\begin{align*}
     &D_{\Omega_1}(\ \times_{j \leq i <k}Z_{\rho}[B+1,A_i+1] \rtimes  \pi( \EE^{\rho} \cup(\FF_1+ \FF_3 + sh^1(\FF_4))))\\
     =& m_1 \cdot  \times_{j \leq i <k}Z_{\rho}[B_k+1,A_i+1] \rtimes  \pi( \EE^{\rho} \cup(\FF_1+ \FF_3 + sh^1(\FF_4))),
\end{align*}
where we omit $Z_{\rho}[B_k+1,A_i+1]$ in the product if $B_k+1>A_{i}+1$. Note that this is a composition of highest derivatives up to a scalar. The induction hypothesis implies the following derivative is a composition of highest derivatives up to a scalar
\[ D_{\Omega(sh^1(\FF_4))} ( \pi( \EE^{\rho} \cup(\FF_1+ \FF_3 + sh^1(\FF_4))))=m_2 \cdot \pi( \EE^{\rho} \cup(\FF_1+ \FF_3 + \FF_4)),\]
where $m_2$ is some positive integer. Therefore, there is a positive integer $m_3$ such that
\[ D_{\Omega_2} (\times_{j \leq i <k}Z_{\rho}[B+1,A_i+1] \rtimes  \pi( \EE^{\rho} \cup(\FF_1+ \FF_3 + sh^1(\FF_4))))= m_3 \cdot \pi(\EE^{\rho}  \cup (\FF_1+  \FF_3 + \FF_4)), \]
and this is also a composition of highest derivatives up to a scalar. In conclusion, 
\[D_{\Omega(sh^{1}( \FF_2))}(\ \times_{j \leq i <k}Z_{\rho}[B+1,A_i+1] \rtimes  \pi(\EE))\]
is a composition of highest derivative up to a scalar, and hence the same holds for $$D_{\Omega(sh^1(\FF_2))}\left(\pi\left(\EE^{\rho} \cup (\FF_1 + sh^1(\FF_2)\right)\right).$$ This completes the proof of Part (iv).\\

\textbf{Proof of (v).}
We remark that when $\FF_3$ is empty and $B>0$, Part (v) is the same as \cite[Theorem 5.1]{Ato20b}. In general, the assumptions imply $\pi( \EE^{\rho} \cup (\FF_1 + sh^1(\FF_2)+\FF_3)) \neq 0$, and hence Parts (i) and (ii) give an injection 
\[ \pi( \EE^{\rho} \cup (\FF_1 + sh^1(\FF_2)+\FF_3)) \hookrightarrow \times_{k\leq i <m} Z_{\rho}[B+1, A_i+1] \rtimes \pi(\EE). \]
Suppose $\pi$ is any irreducible subrepresentation of $\times_{k\leq i <m} Z_{\rho}[B+1, A_i+1] \rtimes \pi(\EE)$. Then Frobenius reciprocity implies that  $D_{\Omega(sh^1(\FF_2))}(\pi) \neq 0$. Applying Lemma \ref{lem Leibniz rule}, the derivative $D_{\Omega(sh^1(\FF_2))}(\times_{k\leq i <m} Z_{\rho}[B+1, A_i+1] \rtimes \pi(\EE))$ is exactly a direct sum of $M$ copies of $\pi(\EE)$, where
\[ M:=\Pi_{x \geq B} (\text{the multiplicity of } \rho|\cdot|^{x} \text{ in } \Omega(sh^1(\FF_2)))!. \]
In particular, it has length $M$.

On the other hand, for an arbitrary irreducible representation $\sigma$, the length of $D_{\Omega(sh^1(\FF_2))}(\sigma)$ is either $0$ or $M\cdot r$ for some positive integer $r$ coming from self-dual derivatives if there is any. As a consequence, one can see that 
\[D_{\Omega(sh^1(\FF_2))}(\times_{k\leq i <m} Z_{\rho}[B+1, A_i+1] \rtimes \pi(\EE)-\pi)=0.\]
This shows that $\times_{k\leq i <m} Z_{\rho}[B+1, A_i+1] \rtimes \pi(\EE)$  has a unique irreducible subrepresentation, which is $\pi( \EE^{\rho} \cup (\FF_1 + sh^1(\FF_2)+\FF_3))$. This completes the proof of Part (v) and the proof of Lemma \ref{lem shift}.
\end{proof}

The following corollary allows us to ``cancel" certain parts of extended multi-segments. 

\begin{cor} \label{cor shift add}\
Suppose $\EE, \EE' \in \Rep^{(P')}$ and $\pi(\EE)\cong \pi(\EE')$. We fix decompositions $\EE_{\rho}=\FF_{1}+\FF_2$, $\EE_{\rho}'= \FF_1'+ \FF_2'.$
\begin{enumerate}
    \item [(i)] Suppose $\supp(\FF_2)= \supp(\FF_2')$. Then for any $d \in \Z_{\geq 0},$ we have 
    \[ \pi(\EE^{\rho} \cup (\FF_1+ sh^d(\FF_2)))\cong \pi((\EE')^{\rho} \cup (\FF_1'+ sh^d(\FF_2'))).  \]
    As a consequence, $\FF_2=\FF_2'$, and for any choice of $\rho^{\ast}$ in Definition \ref{def modification}, we have
    \[ \pi( \EE^{\rho} \cup \FF_1 \cup (sh^1(\FF_2))_{\rho^{\ast}}) \cong \pi( (\EE')^{\rho} \cup \FF_1' \cup (sh^1(\FF_2'))_{\rho^{\ast}}).\]
    \item [(ii)] Under the same assumptions as (i), suppose that both $\EE^{\rho} \cup (\FF_1+ sh^{-1}(\FF_2))$ and $(\EE')^{\rho} \cup (\FF_1'+ sh^{-1}(\FF_2'))$ still satisfy (P'). Then
    
   \[ \pi(\EE^{\rho} \cup (\FF_1+ sh^{-1}(\FF_2)))  \neq 0 \Leftrightarrow \pi((\EE')^{\rho} \cup (\FF_1'+ sh^{-1}(\FF_2'))) \neq 0. \]
   If $\pi(\EE^{\rho} \cup (\FF_1+ sh^{-1}(\FF_2)))  \neq 0,$ then
    \[   \pi(\EE^{\rho} \cup (\FF_1+ sh^{-1}(\FF_2))) ) \cong  \pi((\EE')^{\rho} \cup (\FF_1'+ sh^{-1}(\FF_2'))). \]
    \item [(iii)] Suppose $\supp(\FF_1)=\supp(\FF_1')$. Then for any $d \in \Z_{\geq 0}$, we have 
    \[ \pi(\EE^{\rho} \cup (add^d(\FF_1)+ \FF_2))\cong \pi((\EE')^{\rho} \cup (add^d(\FF_1')+ \FF_2')).  \]
\end{enumerate}
\end{cor}
\begin{proof}
We may assume $\EE^{\rho}= (\EE')^{\rho}$ by Lemma \ref{lemma far away}(i).

For Part (i), write 
\[ \FF_2= \{ ([A_i,B_i]_{\rho}, l_i,\eta_i)\}_{i=1}^n, \ \FF_2'= \{ ([A_i,B_i]_{\rho} ,l_i',\eta_i')\}_{i =1}^n.\]
We take a sequence of positive integers $ \{t_i \}_{ i=1}^n$ such that 
\[\begin{cases}
t_i>d  &\text{ for }1 \leq i\leq n,\\
t_i >t_j &\text{ for }1 \leq j<i \leq n,\\
B_i+t_i > A_{j}+t_{j} &\text{ for }1 \leq j<i \leq n.
\end{cases}\]
Then applying Lemma  \ref{lem shift}(v) repeatedly, we get
\[ \pi\left(\EE^{\rho} \cup \left(\FF_1 + \left(\sum_{i=1}^n sh^{t_i}_i\right)(\FF_2)\right)\right) \cong \pi\left(\EE^{\rho} \cup \left(\FF_1' + \left(\sum_{i=1}^n sh^{t_i}_i\right)(\FF_2')\right)\right). \]
When we take $t_1$ large enough, we may apply Lemma \ref{lem derivative algorithm} repeatedly to construct an ordered multi-set $\Omega$ such that
\begin{align*}
    &\pi(\EE^{\rho} \cup (\FF_1+sh^d(\FF_2)))\\
    =& D_{\Omega}\left( \pi\left(\EE^{\rho} \cup \left(\FF_1 + \left(\sum_{i=1}^n sh^{t_i}_i\right)(\FF_2)\right)\right)\right)\\
    =&D_{\Omega}\left( \pi\left(\EE^{\rho} \cup \left(\FF_1' + \left(\sum_{i=1}^n sh^{t_i}_i\right)(\FF_2')\right)\right)\right)\\
    =&\pi(\EE^{\rho} \cup (\FF_1'+sh^d(\FF_2'))).
\end{align*}
When $d$ is large enough, the conditions in Lemma \ref{lemma far away}(ii) are satisfied, so we have $ sh^d(\FF_2)= sh^d(\FF_2')$, which implies $\FF_2=\FF_2'$, and 
\[ \pi(\EE^{\rho}\cup \FF_1\cup (sh^d(\FF_2))_{\rho^{\ast}} ) \cong  \pi((\EE')^{\rho}\cup \FF_1\cup (sh^d(\FF_2'))_{\rho^{\ast}} ). \]
On the other hand, Remark \ref{rmk change rho} ensures that both $\EE^{\rho}\cup \FF_1\cup (sh^1(\FF_2))_{\rho^{\ast}}$ and $ (\EE')^{\rho}\cup \FF_1'\cup (sh^1(\FF_2'))_{\rho^{\ast}}$ are in $\Rep^{(P')}$. Applying Lemma \ref{lem shift}(v) repeatedly, we get
\[ \pi(\EE^{\rho}\cup \FF_1\cup (sh^1(\FF_2))_{\rho^{\ast}} ) \cong  \pi((\EE')^{\rho}\cup \FF_1\cup (sh^1(\FF_2'))_{\rho^{\ast}} ). \]
This completes the proof of Part (i).
 
 For Part (ii), it suffices to show 
  \[ \pi(\EE^{\rho} \cup (\FF_1+ sh^{-1}(\FF_2)))  \neq 0 \Rightarrow \pi(\EE^{\rho} \cup (\FF_1+ sh^{-1}(\FF_2))) \cong \pi((\EE')^{\rho} \cup (\FF_1'+ sh^{-1}(\FF_2'))). \]
 We use the same notation as in the proof of Part (i) except that $t_i> d $ is replaced by $t_i \gg 0$. We apply induction on $n$. When $n=1$, it follows from definition. In general, we first shift the top row of $\FF_2'$ to left by one. We have (the isomorphism is by Part (i))
 \begin{align*}
    &\,\pi\left( \EE^{\rho} \cup\left(\FF_1'+ \left(sh_{1}^{-1} + \sum_{i=2}^n sh_i^{t_i} \right)(\FF_2')\right) \right)\\
     =&\, D_{\rho|\cdot|^{B_1 ,\dots,A_1}}\left( \pi\left( \EE^{\rho} \cup\left(\FF_1'+ \left( \sum_{i=2}^n sh_i^{t_i} \right)(\FF_2')\right)\right) \right)\ \ 
     \\
     \cong &\,  D_{\rho|\cdot|^{B_1 ,\dots,A_1}}\left( \pi\left( \EE^{\rho} \cup\left(\FF_1+ \left( \sum_{i=2}^n sh_i^{t_i} \right)(\FF_2)\right)\right) \right)\\
     =&\, \pi\left( \EE^{\rho} \cup\left(\FF_1+ \left(sh_{1}^{-1} + \sum_{i=2}^n sh_i^{t_i} \right)(\FF_2)\right) \right)\\
     =& \,\pi\left( \EE^{\rho} \cup\left(\FF_1+ \left(\sum_{i=2}^n sh_i^{t_i+1} \right)(sh^{-1}(\FF_2))\right) \right)\\
     \neq&\, 0  \text{ (by applying Lemma  \ref{lem shift}(i) repeatedly).}
\end{align*}
By definition, this implies 
\[ \pi(\EE^{\rho} \cup( \FF_1 + sh_1^{-1}(\FF_2))) \cong \pi(\EE^{\rho} \cup( \FF_1' + sh_1^{-1}(\FF_2'))). \]
It remains to apply $\sum_{i=2}^n sh_i^{-1}$ to $sh_1^{-1}(\FF_2)$, and hence we are done by induction hypothesis. This completes the proof of Part (ii).
 
For Part (iii), we apply Part (i) on $dual(\EE)$ and $dual(\EE')$, then the equation follows from Theorem \ref{thm Aubert-Zelevinsky dual formula} and Lemma \ref{lem identities}(ii). This completes the proof of the corollary. 
\end{proof}

\subsection{Uniform shift}

In this subsection, we recall a statement from the proof of Theorem \ref{thm Aubert-Zelevinsky dual formula} in \cite{Ato20b}, which will be used in later sections. To be complete, we give a proof here, which is based on suggestions communicated by Atobe. 
Then we give a corollary on the condition that $D_{\Omega(\EE_{\rho})}(\pi(\EE))$ is nonzero.  

\begin{prop}[{\cite[\S 3]{Ato20b}}]\label{prop uniform}
Suppose $\pi(\EE) \neq 0$. Up to a multiplicity, we have
$$D_{ \Omega(sh^1(\EE_{\rho}))} (\pi(\EE^{\rho} \cup sh^1(\EE_{\rho}))) = \pi(\EE).$$
\end{prop}

First, we associate a multi-set with an irreducible representation. 
\begin{defn}
For an irreducible representation
\[ \pi= L\left( \Delta_{\rho_1}[x_1,-y_1],\dots,\Delta_{\rho_t}[x_t,-y_t]; \pi(\sum_{j=t+1}^m \rho_j\otimes S_{2z_j+1},\varepsilon) \right), \]
We define
\[ \Omega(\pi):= \{ \rho_1|\cdot|^{x_1},\dots,\rho_t|\cdot|^{x_t} \} + \{ \rho_1|\cdot|^{y_1},\dots, \rho_t|\cdot|^{y_t} \} + \{ \rho_{t+1}|\cdot|^{z_{t+1}},\dots,\rho_m|\cdot|^{z_m} \}. \]
We denote $\Omega(\pi)_\rho$ to be the maximal sub-multi-set of $\Omega(\pi)$ whose elements are all of the form $\rho|\cdot|^x$ for some $x \in \R$.
\end{defn}

The following lemma shows that if $\pi=\pi(\EE),$ then the multi-set $\Omega(\EE)$ is closely related with $\Omega(\pi)$. Note that $ \Omega(\EE)$ only depends on the local Arthur parameter $\psi_{\EE}$.

\begin{lemma}\label{lem $L$-data}
For any $\EE \in \Rep$, we have 
\[ \Omega(\EE_{\rho}) \supseteq \Omega(\pi(\EE))_{\rho}\]
as multi-sets. Moreover, 
\begin{enumerate}
    \item [(i)] If $\pi(\EE^{\rho} \cup sh^{-1}(\EE_{\rho})) \neq 0$, then $\Omega(\EE_{\rho}) = \Omega(\pi(\EE))_{\rho}$.
    \item [(ii)] The difference multi-set $\Omega(\EE_{\rho})\setminus \Omega(\pi(\EE))_{\rho}$ is symmetric about $\rho|\cdot|^{-1/2}$ in the following sense: The multiplicity of $\rho|\cdot|^x$ in $ \Omega(\EE_{\rho})\setminus\Omega(\pi(\EE))_{\rho}$ is the same as that of  $\rho|\cdot|^{-x-1}$.
\end{enumerate}
\end{lemma}

\begin{proof}
  We first show that it is sufficient to prove the identity 
  \begin{align}
      \Omega(sh^t(\EE_{\rho}))=\Omega(\pi(\EE^{\rho} \cup (sh^t(\EE_{\rho}))))_{\rho}
  \end{align}
  for a large enough integer $t$. Indeed, Corollary \ref{cor shift in single rho} implies for any $t \in \N$,
 \[ \Omega(sh^t(\EE_{\rho}))=\Omega(\pi(\EE^{\rho} \cup (sh^t(\EE_{\rho}))))_{\rho} \Longrightarrow   \Omega(\EE_{\rho}) \supseteq \Omega(\pi(\EE))_{\rho},\]
and $ \Omega(\EE_{\rho}) \neq \Omega(\pi(\EE))_{\rho}$ if and only if there are segments $\Delta_{\rho}[x,x+1]$ or summands $\rho \otimes S_0$ being omitted, which happens only when $\pi(\EE^{\rho} \cup sh^{-1}(\EE_{\rho}))=0$. In this case, each segment $\Delta_{\rho}[x,x+1]$ (resp. each summand $\rho \otimes S_{0}$) contributes a pair $\{\rho|\cdot|^{x}, \rho|\cdot|^{-x-1}\}$ (resp. an element $\{\rho|\cdot|^{-1/2}\}$) to the difference multi-set $\Omega(\EE_{\rho})\setminus \Omega(\pi(\EE))_{\rho}$, so Part (ii) also follows.

Next, we show that $\Omega(sh^t(\EE_{\rho}))=\Omega(\pi(\EE^{\rho} \cup (sh^t(\EE_{\rho}))))_{\rho}$ for a large enough integer $t$. Write
\[\EE_{\rho}= \{ ([A_i,B_i]_{\rho},l_i,\eta_i)\}_{i=1}^n.\]
By replacing $\EE_{\rho}$ with $sh^t(\EE_{\rho})$ for $t$ large, we may assume
\begin{enumerate}
    \item [$\oldbullet$]$B_i >0$ for all $1 \leq i\leq n$;
    \item [$\oldbullet$]$\EE_{\rho}$ satisfies ($P'$);
    \item [$\oldbullet$] $B_i>A_j-B_j+1$ for any $1 \leq i,j\leq n$.
\end{enumerate}
Now we apply induction on $n$. If $n=1$, the claim follows from the definition. Assume that $n>1$.
Let $\EE^d= sh_n^{d}(\EE)$. When $d$ is large, we write $ sh_n^{d}(\EE)_{\rho}=\FF_1+\FF_2$ where
\[\FF_1=\{([A_i,B_i]_{\rho},l_i,\eta_i)\}_{i=1}^{n-1},\  \FF_2=\{([A_n+d,B_n+d]_{\rho},l_n,\eta_n)\}. \]
Then by Definition \ref{def rep of segment}, the multiplicities for $\rho|\cdot|^{x}$ for $A_n+d \leq x \leq B_n+d$ of the two multi-sets $\Omega(\EE^d_{\rho})$, $\Omega(\pi(\EE^{d}))_{\rho}$ agree. Therefore, $\Omega(\EE^d_{\rho})= \Omega(\pi(\EE^d))_{\rho}$ if and only if 
\[ \Omega(\FF_1)= \Omega(\pi(\EE^d- (\FF_2)_{\rho}+ (\FF_2)_{\rho^{\ast}}))_{\rho}, \]
where we take $\rho^{\ast}$ as in Definition \ref{def modification}. Then we reduce $n$ by 1, and the equality follows from induction hypothesis. 

It remains to show that $\Omega(\EE^d_{\rho})=\Omega(\pi(\EE^{d}))_{\rho}$ implies $\Omega(\EE^{d-1}_{\rho})=\Omega(\pi(\EE^{d-1}))_{\rho}$. By Lemma \ref{lem derivative algorithm}, we have
\[ D_{ \rho|\cdot|^{A_{n}+d}} \circ\cdots\circ D_{\rho|\cdot|^{B_{n}+d}}(\pi(\EE^{d}))= \pi(\EE^{d-1}),  \]
 and each derivative is the highest derivative. Then by the algorithm for taking positive derivatives in \cite[Theorem 7.1]{AM20}, each derivative $D_{\rho |\cdot|^{u}}(\pi)$ decreases the multiplicity of $\rho|\cdot|^u$ in $\Omega(\pi)$ by one, and increases the multiplicity of $\rho|\cdot|^{u-1}$ by one. Note that the assumption $B_i>A_j-B_j+1$ assures that there are no segments of the form $\Delta_{\rho}[u,u]$ in each stage, hence the derivative $D_{\rho |\cdot|^{u}}$ won't decrease the multiplicity of $ \rho|\cdot|^{-u}$. Therefore, $ \Omega(\EE^d_{\rho})= \Omega(\pi(\EE^d))_{\rho}$ implies that $ \Omega(\EE^{d-1}_{\rho})= \Omega(\pi(\EE^{d-1}))_{\rho}$. 
 This completes the proof of Lemma \ref{lem $L$-data}.
\end{proof}

Now we give a proof of Proposition \ref{prop uniform}. 

\begin{proof}[Proof of Proposition \ref{prop uniform}] 
We first describe the strategy of the proof, which is based on the algorithm for computing derivatives in \cite{AM20}. Suppose $D_{\rho|\cdot|^{x}}^{(k)}$ is the highest derivative in each stage. We use Lemma \ref{lem $L$-data} to show that $k$ is exactly the multiplicity of $\rho|\cdot|^{x}$ in $\Omega(sh^1(\EE_{\rho}))$, and hence $D_{\Omega(sh^1(\EE_{\rho}))}(\pi(\EE^{\rho} \cup sh^1(\EE_{\rho})))$ is a composition of highest derivatives. Finally, by explicit computation, we show that $D_{\Omega(sh^1(\EE_{\rho}))}(\pi(\EE^{\rho} \cup sh^1(\EE_{\rho})))$ is isotypic, and the $L$-data of any irreducible subrepresentation of it is the same as the $L$-data of $\pi(\EE)$ by Theorem \ref{thm non-vanishing}(i).

To be explicit, we write the multi-set $\Omega(sh^1(\EE_{\rho}))$ as 
\[ \Omega(sh^1(\EE_{\rho}))= \{ (\rho|\cdot|^{a_1})^{r_1},\dots,(\rho|\cdot|^{a_t})^{r_t}, \} \]
where $a_1 < \cdots < a_t$ and $r_i$ denotes the multiplicity of $\rho|\cdot|^{a_i}$ in $\Omega(sh^1(\EE_{\rho}))$. Write
\[ \pi(\EE^{\rho} \cup sh^1(\EE_{\rho}))= L\left( \Delta_{\rho_1}[x_1,-y_1],\dots,\Delta_{\rho_p}[x_p,-y_p]; \pi\left(\sum_{j=p+1}^{m}\rho_{j}\otimes S_{2z_j+1},\varepsilon \right) \right). \]
Lemma \ref{lem $L$-data} shows that $\Omega( sh^1(\EE_{\rho}))= \Omega(\pi(\EE^{\rho} \cup sh^1(\EE_{\rho}))_{\rho}$, and hence for any $a_i$, we have 
\begin{align}\label{eq proof of uniform shift left}
    r_i= \#\{ j \ | \ \rho_j \cong \rho,\  x_j=a_i\}+\#\{ j \ | \ \rho_j \cong \rho,\ y_j=a_i\}+\#\{ j \ | \ \rho_j \cong \rho,\ z_j=a_i\}.
\end{align}
To describe the $L$-data of $ D_{\rho|\cdot|^{a_i}}^{(r_i)}\circ\cdots\circ D_{\rho|\cdot|^{a_1}}^{(r_1)}(\pi(\EE^{\rho}\cup sh^{1}(\EE_{\rho})))$, we give the following notation. For any real number $a$, we denote 
 \[ a^{(i)}= \begin{cases}
    a-1 & \text{ if }a \leq a_i,\\
    a & \text{ otherwise,}
    \end{cases}\]
and 
\begin{align*}
    \Delta_{\rho_j}[x_j, -y_j]^{(i)}:=& \begin{cases}
    \Delta_{\rho_j}[x_j, -y_j] & \text{ if }\rho_j \not\cong \rho,\\
    \Delta_{\rho_j}[x_j^{(i)}, -y_j^{(i)}] & \text{ otherwise, }
    \end{cases}\\
    (\rho_j \otimes S_{2 z_j +1})^{(i)}:=&\begin{cases}
    \rho_j \otimes S_{2 z_j +1} & \text{ if }\rho_j \not\cong \rho,\\
    \rho_j \otimes S_{2 z_j^{(i)} +1} &\text{ otherwise.}
    \end{cases} 
\end{align*}
Now we show by induction on $i$ that up to a multiplicity,
\begin{align*}
  &   D_{\rho|\cdot|^{a_i}}^{(r_i)}\circ\cdots\circ D_{\rho|\cdot|^{a_1}}^{(r_1)}(\pi(\EE^{\rho} \cup sh^1(\EE_{\rho})))\\
=& L\left( \Delta_{\rho}[x_1,-y_1]^{(i)},\dots,\Delta_{\rho}[x_p,-y_p]^{(i)}; \pi\left(\sum_{j=p+1}^{m} (\rho_j\otimes S_{2z_j+1})^{(i)},\varepsilon^{(i)} \right) 
\right),
\end{align*}
where $\varepsilon^{(i)}((\rho_j\otimes S_{2z_j+1})^{(i)})= \varepsilon(\rho_j\otimes S_{2z_j+1})$. Denote the right hand side by $\pi_i$ for brevity, and set $\pi_0= \pi(sh^1(\EE))$. We separate into three cases:
(1) $a_i<0$, (2) $a_i=0$, and (3) $a_i>0$. 

Case (1) Suppose $a_i<0$. We follow the notation and definitions of \cite[Section 6.2]{AM20} and apply \cite[Proposition 6.1]{AM20} to compute the highest derivative $D_{\rho|\cdot|^{a_i}}^{(k)}(\pi_{i-1})$. 
    
    Since $ y_j$ and $z_{j}$ are all positive, \eqref{eq proof of uniform shift left} implies
    \[ r_i=\# \{j \ | \  \rho_j \cong \rho,\  x_j=a_i\}= \# \{j \ | \ \rho_j \cong \rho,\  x_j^{(i-1)}=a_i\}. \]
    On the other hand, by the induction hypothesis, 
    \[A_{\rho|\cdot|^{a_i-1}}:=\{j \ | \  \rho_j \cong \rho,\ x_{j}^{(i-1)}=a_i-1\}=\emptyset,\]
    hence in this special case 
    \[A_{\rho|\cdot|^{a_i}}^c= A_{\rho|\cdot|^{a_i}}:= \{j \ | \  \rho_j \cong \rho,\ x_j^{(i-1)}=a_i\}.\]
    For the definitions of $A$ and $A^c$, see \cite[Section 6.2]{AM20}. 
    Therefore, the formula gives $k=|A_{\rho|\cdot|^{a_i}}^{c}|=r_i$, and the $L$-data of $D_{\rho|\cdot|^{a_i}}^{(r_i)}(\pi_{i-1})$ is obtained from that of $\pi_{i-1}$ by replacing $\Delta_{\rho_j}[x_j, -y_j]^{(i-1)}$ with $\Delta_{\rho_j}[x_j, -y_j]^{(i)}$. Note that since $\pi(\EE) \neq 0$, by Theorem \ref{thm non-vanishing}(i) we have $x_j^{(i)} \geq -y_j^{(i)}$ whenever $\rho_j \cong \rho$. In other words, no segments are deleted in this case.
    
Case (2) Suppose $a_i=0$. According to the Langlands classification, we have an injection 
    \[ \pi_{i-1} \hookrightarrow  \times_{j=1}^p \Delta_{\rho_j}[x_j,-y_j]^{(i-1)} \rtimes \pi\left(\sum_{j=p+1}^{m} \rho_j\otimes S_{2z_j+1},\varepsilon\right),   \]
    where $\pi_{i-1}$ is the unique irreducible subrepresentation of the right hand side. (Note that $(\rho_j\otimes S_{2z_j+1})^{(i-1)}=\rho_j\otimes S_{2z_j+1}$.) 

    Since $\pi(\EE) \neq 0$, Corollary \ref{cor shift in single rho} implies $ 2z_j+1 \geq  2$ whenever $\rho_j \cong \rho$. Therefore, \eqref{eq proof of uniform shift left} implies
    \[ r_i=\# \{ j \ | \ \rho_j \cong \rho,\  x_j^{(i-1)}=0 \}. \]
    Also, we know $ x_j^{(i-1)} \neq -1$ for any $j$ with $\rho_j \cong \rho$ by the induction hypothesis, and hence the segments $ [0,0]_\rho $ and $[x_j^{(i-1)},-(y_j^{(i-1)})]_\rho$ are not linked for any $j$. Therefore, by Lemma \ref{lem commutativity}, the following product commutes for any $j$
    \[ \Delta_{\rho_j}[x_j,-y_j]^{(i-1)} \times \rho= \rho \times \Delta_{\rho_j}[x_j,-y_j]^{(i-1)}.  \]
    As a consequence, we have an injection $\pi_{i-1} \hookrightarrow \rho^{r_i} \times \sigma$, where
    \[\sigma= \times_{j=1}^p \Delta_{\rho_j}[x_j,-y_j]^{(i)} \rtimes \pi\left(\sum_{j=p+1}^{m} (\rho\otimes S_{2z_j+1})^{(i)},\varepsilon^{(i)}\right). \]
    Note that $\pi_{i}$ is the unique irreducible subrepresentation of $\sigma$ from the Langlands classification. We know $\sigma$ is $\rho$-reduced by \cite[Proposition 3.6]{Ato20a} (or \cite[Theorem 3.1]{Jan18}). Therefore, Lemma \ref{lem Frobenius} indicates
    \[ D_{\rho}^{(r_i)}(\pi_{i-1}) \geq \pi_i,\]
    and it is the highest derivative. From \cite[Proposition 2.7]{Ato20a}, we know the highest derivative of an irreducible representation is always isotypic, so $D_{\rho}^{(r_i)}(\pi_{i-1})$ is a multiple of $\pi_i$. 
    
Case (3) Suppose $a_i>0$. We use the notation of \cite[Section 7.1]{AM20} and compute the highest derivative of $D_{\rho|\cdot|^{a_i}}^{(k)}(\pi_{i-1})$ according to \cite[Theorem 7.1]{AM20}.
    
    The induction hypothesis implies that any segment $\Delta_{\rho}[x_j^{(i-1)},-(y_j^{(i-1)})]$ in the $L$-data of $\pi_{i-1}$ satisfies $x_j^{(i-1)} \neq a_i-1$ and $y_j^{(i-1)} \neq a_i-1$. Therefore, 
    \begin{align*}
        t&:=\#\{j\ | \  \rho_j \cong \rho,\  \Delta_{\rho}[x_j^{(i-1)},-(y_j^{(i-1)})]= \Delta_{\rho}[a_i-1,-a_i]\}=0,\\
        A_{\rho|\cdot|^{a_i-1}}&:= \{j \ | \ \rho_j \cong \rho,\  x_j^{(i-1)}=a_i-1\}=\emptyset,\\
        B_{\rho|\cdot|^{a_i-1}}&:= \{j \ | \  \rho_j \cong \rho,\  y_j^{(i-1)}=a_i-1\}=\emptyset,
    \end{align*}
    and hence in this special case, we have
    \begin{align*}
        A_{\rho|\cdot|^{a_i}}^c=A_{\rho|\cdot|^{a_i}}:=\{j \ | \  \rho_j \cong \rho,\  x_j^{(i-1)}=a_i\},&\ \ A_{\rho|\cdot|^{a_i-1}}^c=\emptyset,\\
        B_{\rho|\cdot|^{a_i}}^c=B_{\rho|\cdot|^{a_i}}:=\{j \ | \ \rho_j \cong \rho,\  y_j^{(i-1)}=a_i\},&\ \ B_{\rho|\cdot|^{a_i-1}}^c=\emptyset.
    \end{align*}
    Finally, again by induction hypothesis, we have
    \begin{align*}
        m':&= \# \{j \ | \  \rho_j \cong \rho,\  z_j^{(i-1)}= a_i-1 \}= 0,\\
        m:&= \# \{j \ | \  \rho_j \cong \rho,\  z_j^{(i-1)}= a_i \}.
    \end{align*}
    These are all ingredients needed in the formula. We have
    \begin{align*}
        k&= |A_{\rho|\cdot|^{a_i}}^c|+ \max \{m + \max\{ |B_{\rho|\cdot|^{a_i}}^c|-m',0\}-|A_{\rho|\cdot|^{a_i-1}}^{c}|,0\}\\
        &= |A_{\rho|\cdot|^{a_i}}|+m+|B_{\rho|\cdot|^{a_i}}|\\
        &= \#\{ j \ | \  \rho_j \cong \rho,\  x_j^{(i-1)}=a_i\}+\#\{ j \ | \ \rho_j \cong \rho,\  y_j^{(i-1)}=a_i\}\\
        & \quad +\#\{ j \ \rho_j \cong \rho,\  | \ z_j^{(i-1)}=a_i\}\\
        &=\#\{ j \ | \  \rho_j \cong \rho,\  x_j=a_i\}+\#\{ j \ | \ \rho_j \cong \rho,\  y_j=a_i\}+\#\{ j \ | \ \rho_j \cong \rho,\  z_j=a_i\}\\
        &=r_i,
    \end{align*}
    and the $L$-data of $D_{\rho|\cdot|^{a_i}}^{(r_i)}(\pi_{i-1})$ is obtained from $\pi_{i-1}$ by replacing $\Delta_{\rho_j}[x_j, -y_j]^{(i-1)}$ with $\Delta_{\rho_j}[x_j, -y_j]^{(i)}$ and replacing $(\rho_j \otimes S_{2 z_j+1})^{(i-1)}$ with $(\rho_j \otimes S_{2 z_j+1})^{(i)}$. Note that the combinatorics here use the fact that no segments are removed in Case (1). In conclusion, we have $\pi_{i}= D_{\rho|\cdot|^{x}}^{(r_i)}(\pi_{i-1}).$

In summary, we show that $D_{\Omega(sh^1(\EE_{\rho}))}(\pi(\EE^{\rho} \cup sh^1(\EE_{\rho})))$ is isotypic, and each irreducible subrepresentation is exactly $\pi(\EE)$ according to Corollary \ref{cor shift in single rho}. This completes the proof of the proposition. 
\end{proof}

We obtain the following corollary by the same computation in the proof above.
\begin{cor} \label{cor uniform derivative}
Suppose $\EE \in \Rep$. Write the multi-set $\Omega(\EE_{\rho})$ as 
\[ \Omega(\EE_{\rho})= \{ (\rho|\cdot|^{a_1})^{r_1},\dots,(\rho|\cdot|^{a_t})^{r_t} \}, \]
where $a_1 < \cdots< a_t$ and $r_i$ denotes the multiplicity of $\rho|\cdot|^{a_i}$ in $\Omega(\EE_{\rho})$. Define $\pi_i$ recursively by 
\[ \begin{cases}
\pi_0= \pi(\EE),\\
\pi_i= D_{\rho|\cdot|^{a_i}}^{(r_i)}(\pi_{i-1}).
\end{cases} \]
If $\pi_{i-1}\neq 0$ and $D_{\rho|\cdot|^{a_i}}^{(k)}(\pi_i)$ is the highest derivative, then $k \leq r_i$. Moreover, $D_{\Omega(\EE_{\rho})}(\pi(\EE))\neq 0$ if and only if the following conditions hold:
\begin{enumerate}
    \item [$\oldbullet$] $\Omega(\EE_{\rho})= \Omega(\pi(\EE))_{\rho}$;
    \item [$\oldbullet$] any segment of the form $\Delta_{\rho}[\alpha,\beta]$ in the $L$-data of $\pi(
    \EE)$ satisfies $\alpha-\beta\geq 1$;
    \item [$\oldbullet$] the tempered $L$-parameter in the $L$-data of $\pi(\EE)$ does not contain $\rho\otimes S_{1}$. If it contains $\rho\otimes S_{2}$, then the value of the character at this summand is $1$.
\end{enumerate}
In this case $D_{\Omega(\EE_{\rho})}(\pi(\EE))$ is isotypic, and the $L$-data of its irreducible subrepresentation is obtained from the $L$-data of $\pi$ by replacing each segment of the form $\Delta_{\rho}[\alpha,\beta]$ with $\Delta_{\rho} [\alpha-1, \beta+1]$ and each summand of the form $\rho\otimes S_{2z+1}$ by $\rho \otimes S_{2(z-1)+1}$, where we omit the segments (resp. summands) of the form $\Delta_{\rho}[x,x+1]$ (resp. $\rho\otimes S_{0}$).
\end{cor}
\begin{proof}
The computation of $D_{\rho|\cdot|^{a_i}}^{(k)}(\pi_{i-1})$ is identical to the proof of Proposition \ref{prop uniform} except for $ a_i=0,\half{1}$. We treat these two  cases separately. 

Suppose $\pi_{i-1} \neq 0$ where $a_i= 0 $ or $\half{1}$. Then the proof of Proposition \ref{prop uniform} shows that $\pi_{i-1}$ is irreducible and we write
\[ \pi_{i-1}= L(\Delta_{\rho_1}[x_1,-y_1],\dots,\Delta_{\rho_t}[x_t,-y_t]; \pi(\phi,\varepsilon)).\]
Let $ D_{\rho|\cdot|^{a_i}}^{(k)}(\pi_{i-1})$ be the highest derivative.

Case (1): Suppose $a_i=0$. Let $m_1$ be the multiplicity of $\rho\otimes S_{1}$ in $\phi$. The computation is the same if $m_1=0$. On the other hand, if $m_1 \neq 0$, we need to show $k<r_i$.

    By the proof of Proposition \ref{prop uniform} at the stage $a_{i-1}$, the multiplicities of $\rho|\cdot|^{-1}$ in $\Omega({\EE_{\rho}})$ and $\Omega(\pi(\EE))_{\rho}$ agree, so Lemma \ref{lem $L$-data}(ii) implies so do the multiplicities of $ \rho$ in the two multi-sets. Therefore, we have
    \[ r_i= \# \{j \ | \ \rho_j \cong \rho, x_j = 0\}+ m_1.\]
    Then \cite[Proposition 3.6]{Ato20a} (or \cite[Theorem 3.1]{Jan18}) implies
    \[ D_{\rho}^{(\lfloor m_1/2 \rfloor)} (\pi(\phi, \varepsilon))= c\cdot \pi(\phi- \rho^{2 \cdot \lfloor m_1/2 \rfloor} \otimes S_1, \varepsilon) \]
    is the highest derivative (for some positive integer $c$). Therefore, by the same argument in case (2) in the proof of Proposition \ref{prop uniform}, we have $\pi_{i-1} \hookrightarrow \rho^{r_i- \lfloor m_1/2 \rfloor } \rtimes \sigma$ where $\sigma$ is $\rho$-reduced. This concludes
    \[ k = \# \{j \ | \ \rho_j \cong \rho, x_j = 0\}+ \lfloor m_1/2 \rfloor= r_i - \lceil m_1/2 \rceil <r_i. \]

Case (2): Suppose $a_i=\half{1}$. Denote $m_2$ the multiplicity of $\rho \otimes S_{2}$ in $\phi$. If $m_2=0$, then there is no difference between this situation and case (3) in the proof of Proposition \ref{prop uniform}, and so $k=r_i$. On the other hand, when $m_2\neq 0$, $k=r_i$ if and only if $ \varepsilon(\rho\otimes S_{2})=1$ since the convention in \cite{AM20} sets the multiplicity of $\rho\otimes S_0$ in $\phi$ to be one, and $\varepsilon(\rho\otimes S_0)=1$.

This completes the proof of the corollary.
\end{proof}

We end this subsection by giving a comparison between the multi-set $\Omega(\pi)$ we defined and the extended cuspidal support $\exsupp(\pi)$ defined in the following remark.
\begin{remark}\label{rmk extended cuspidal support}
 We recall the definition of the extended cuspidal support of a representation defined in \cite{Moe09b}, which is crucial in \cite{Ato22}.

Suppose $\pi$ is an irreducible representation of $G_n$, then there exists an injection
\[ \pi \hookrightarrow \rho_1 \times \cdots \times \rho_{r} \rtimes \sigma\]
where $\rho_i$ are irreducible supercuspidal (not necessarily self-dual) representation of $\GL_{d_i}(F)$ and $\sigma= \pi(\phi, \varepsilon)$ is a supercuspidal representation of $G_{n'}$ for some $n' \leq n$. Write 
\[\phi= \bigoplus_{j=r+1}^m \rho_j \otimes S_{a_j}. \]
We define the extended cuspidal support to be the multi-set
\[ \exsupp(\pi):=\{\rho_1,\dots , \rho_r, \rho_1^{\vee}, \dots ,\rho_r^{\vee}\}+ \sum_{j= r+1}^m \{\rho_j|\cdot|^{\half{a_j-1}},\rho_j|\cdot|^{\half{a_j-3}}, \dots ,\rho_j|\cdot|^{-\half{a_j-1}}\}.\]
M{\oe}glin showed that all representations in $\Pi_{\psi}$ share the same extended cuspidal support and this can be computed from $ \psi$ (\cite[Proposition 4.1]{Moe09b}).

In \cite{Ato22}, Atobe considered this definition for $\pi$ satisfying the following condition: 
for any self-dual irreducible supercuspidal representation $\rho$ of $\GL_d(F)$,
\[D_{\rho|\cdot|^{x}}(\pi)\neq 0 \Longrightarrow  x \in \{0,1/2\}. \]
Atobe computed $\exsupp(\pi)$ from the $L$-data of $\pi$ explicitly in this case in \cite[Section 4.2]{Ato22}. If we denote the multiplicity of $\rho|\cdot|^{x}$ in $\Omega(\pi)$ (resp. $\exsupp(\pi)$) by $k_{\rho,x}$ (resp. $M_{\rho,x}$), then 
\begin{align*}
    M_{\rho,x}&= \sum_{z \in \half{1}\Z,\  z \geq x} k_{\rho,z},\\
    k_{\rho,x}&=M_{\rho, x}- M_{\rho,x+1}.
\end{align*}
In other words, we can recover any one of the multi-sets $\Omega(\pi)$ and $\exsupp(\pi)$ from the other. In particular, the associated local Arthur parameter (if any) is the same.

In general, it is not always possible to recover $\exsupp(\pi)$ from $\Omega(\pi)$, since $\Omega(\pi)$ only depends on $\phi_{\pi}$, the $L$-parameter of $\pi$, but the computation of $\exsupp(\pi)$ involves the character of the tempered representation in the $L$-data of $\pi$. For our purpose, it is sufficient to consider $\Omega(\pi)$.
\end{remark}

\subsection{Invariants of \texorpdfstring{$\pi(\EE)$}{}}
As corollaries of Lemma \ref{lem $L$-data} and the computation in the previous subsection, we show that if $\EE \in \Rep$ and $\EE_{\rho} = \{ ([A_i,B_i]_{\rho},l_i, \eta_i)\}_{i\in (I_{\rho},>)}$, then the following information
\begin{enumerate}
    \item [(a)] $A=\max\{ A_i \ | \ i \in I_{\rho}\}$ and the multiplicity $\# \{ i \in I_{\rho} \ | \ A_i= A \}$.
    \item [(b)] $\max\{ t \in \Z \ | \ \pi(\EE^{\rho} \cup sh^{-t}(\EE_{\rho})) \neq 0 \}$.
\end{enumerate}
can be recovered from the $L$-data of $\pi(\EE)$. To be precise, the $L$-data of $\pi(\EE)$ directly gives (a) (see Theorem \ref{thm max A}), but only an upper bound for (b) (see Theorem \ref{thm min B+l}). In particular, $A$ is independent of the choice of $\EE$ which gives the same representation.

\begin{thm} \label{thm max A}
 Suppose $\EE\in \Rep$, $\EE_{\rho}= \{([A_i,B_i]_{\rho},l_i,\eta_i)\}_{i\in (I_{\rho},>)}$ and denote  $A= \max\{ A_i \ | \ i \in I_{\rho} \}$, $m= \# \{ i \in I_{\rho} \ | \ A_i=A \}$.
Then we have
\begin{enumerate}
    \item [$\oldbullet$] $A= \max\{x \in \R \ | \ \rho|\cdot|^{x} \in \Omega(\pi(\EE))_{\rho} \}$,
    \item [$\oldbullet$] $m= \text{the multiplicity of }\rho|\cdot|^{A} \text{ in }\Omega(\pi(\EE))_{\rho}$.
\end{enumerate}
\end{thm}
\begin{proof}
Clearly, $ \rho|\cdot|^{A}$ appears in $\Omega(\EE_{\rho})$ with multiplicity $m$. On the other hand, there is no $\rho|\cdot|^{-A-1}$ in the multi-set $\Omega(\EE_{\rho})$ since 
\[ A_i+ B_i \geq 0 \Rightarrow -A-1\leq  -A_i-1 < -A_i \leq B_i.  \]
Thus, Lemma \ref{lem $L$-data}(ii) implies $\rho|\cdot|^{A}$ is not in the multi-set $\Omega(\EE_{\rho})\setminus \Omega(\pi(\EE))_{\rho}.$ It follows that the multiplicity of $\rho|\cdot|^{A}$ in $\Omega(\pi(\EE))$ is exactly $m$.
This completes the proof of the theorem. 
\end{proof}

Next, we give a characterization of $\max\{ t \in \Z \ | \ \pi(sh^{-t}(\EE)) \neq 0 \}$. Theorem \ref{thm shift left} below shows that if $(A_i-1)+ (B_i-1)\geq 0$ for all $i \in I_{\rho}$, then $\EE^{\rho} \cup sh^{-1}(\EE_{\rho}) \in \Rep$ if and only if $D_{\Omega(\EE_{\rho})} (\pi(\EE)) \neq 0$.
We thank Atobe for helpful communications on this theorem.
On the other hand, Corollary \ref{cor uniform derivative} translates the non-vanishing of $D_{\Omega(\EE_{\rho})}(\pi(\EE))$ into a simple combinatorial invariant of the $L$-data of $\pi(\EE)$. Combining them together, we have the following. 

\begin{thm}\label{thm min B+l}
Suppose $\EE\in \Rep$. Write $ \EE_{\rho}= \{([A_i,B_i]_{\rho},l_i,\eta_i)\}_{i \in (I_{\rho}, >)}$, and 
\[ \pi(\EE)= L\left( \Delta_{\rho_1}[x_1,-y_1],\dots,\Delta_{\rho_p}[x_p,-y_p]; \pi\left(\sum_{j=p+1}^{m}\rho_j\otimes S_{2z_j+1},\varepsilon \right) \right). \]
Then  
\[T:=\max\{ t \in \Z \ | \ \pi(\EE^{\rho} \cup sh^{-t}(\EE_{\rho})) \neq 0 \}\] 
is zero if $\Omega(\EE_{\rho}) \supsetneq \Omega(\pi(\EE))_{\rho}$. Otherwise, it is given by the maximum of $T_1 \cap (T_2 \cup T_3)$, where
\begin{align*}
   T_1= &\{t \in \Z \ | \ \forall i \in I_{\rho},\   (A_i-t)+(B_i-t) \geq 0 \}\\
    T_2=&\{ t \in \Z \ | \ \forall 1 \leq i\leq p \text{ and } \rho_i \cong \rho ,\  (x_i- (t-1))+ (y_i- (t-1)) \geq 1   \}\\
     T_3=&  \{ t \in \Z \ | \ \forall p+1 \leq j\leq m \text{ and } \rho_j \cong \rho,\  z_j-t\geq -1/2,\\
    & \text{ and }\varepsilon(\rho \otimes S_{2z_{j}+1})=1 \text{ if }  z_j-t=-1/2   \}.
\end{align*}
Note that $T_2,T_3$ only depend on $\pi(\EE)$, $T_1$ only depends on $ \psi_{\EE}$, and the condition $\Omega(\EE_{\rho}) = \Omega(\pi(\EE))_{\rho}$ only depends on $\pi(\EE)$ and $\psi_{\EE}$. In conclusion, $T$ is determined by $\pi(\EE)$ and $\psi_{\EE}$.
\end{thm}

We work towards proving Theorem \ref{thm shift left} below. First, we establish the following lemma.

\begin{lemma}\label{lem Lapid}
For $y \geq x+1$, the parabolic induction
\[ soc(Z_{\rho}[x+1,y]\times \rho|\cdot|^{x}) \times \rho|\cdot|^{x}\]
is irreducible, and hence the product commutes.
\end{lemma}
\begin{proof}
This is a consequence of \cite[Proposition 4.1(4)]{LM16} and we adopt their notation for our proof. We set $\Delta=[x,x]_{\rho}$, $\mathfrak{m}^t=[x+1,y]_{\rho}+ [x,x]_{\rho}$, and $\sigma=Z(\Delta)$. Then $soc(Z_{\rho}[x+1,y]\times \rho|\cdot|^{x}) = Z(\mathfrak{m}^t)=L(\mathfrak{m})$, where \[\mathfrak{m}= (\mathfrak{m}^t)^t=[y,y]_{\rho}+[y-1,y-1]_{\rho}+\cdots+[x+2,x+2]_{\rho}+ [ x,x+1]_{\rho}\]
via the Mœglin–Waldspurger algorithm (see \cite{MW86}). It follows that we have $\mathfrak{m}_1= \mathfrak{m}_{\leq_e e(\Delta) }= 0, \mathfrak{m}_2= \mathfrak{m}_{ \geq_{b} b(\Delta)}= \mathfrak{m}$. We verify directly that
\begin{align*}
    (\mathfrak{m}_1^t+\Delta)^t&=\mathfrak{m}_1+\Delta^t,\\
    (\mathfrak{m}_2^t+\Delta)^t&= \mathfrak{m}_2+\Delta^t.
\end{align*}
By \cite[Proposition 4.1(4)]{LM16}, $ L(\mathfrak{m}) \times Z(\Delta)$ is irreducible, and the product commutes. This completes the proof of the lemma.
\end{proof}

Now we prove the following theorem.

\begin{thm} \label{thm shift left}
Let $\EE\in \Rep^{(P')}$ and write $\EE_{\rho}=\{ ([A_i,B_i]_{\rho},l_i,\eta_i) \}_{i \in (I_{\rho},>)}$. Suppose $ (A_i-1)+(B_i-1) \geq 0$ for all $i \in I_{\rho}$ (i.e. $\EE^{\rho} \cup sh^{-1}(\EE_{\rho})$ is still an extended multi-segment). Then $ D_{\Omega(\EE_{\rho})}(\pi(\EE))=\pi( \EE^{\rho} \cup sh^{-1}(\EE_{\rho}))$ (possibly vanishing) up to a multiplicity.
\end{thm}

\begin{proof}
We begin by showing
\begin{align}\label{eq thm 4.13}
     \pi( \EE^{\rho} \cup sh^{-1}(\EE_{\rho})) \neq 0 \Longleftrightarrow D_{\Omega(\EE_{\rho})}(\pi(\EE)) \neq 0.
\end{align}
This is sufficient to prove the theorem. Indeed, if $\pi( \EE^{\rho} \cup sh^{-1}(\EE_{\rho})) \neq 0$, then Proposition \ref{prop uniform} gives the theorem. If $ \pi( \EE^{\rho} \cup sh^{-1}(\EE_{\rho}))=0$, then \eqref{eq thm 4.13} shows $ D_{\Omega(\EE_{\rho})}(\pi(\EE))=0$ also.

Assume that $\pi( \EE^{\rho} \cup sh^{-1}(\EE_{\rho})) \neq 0.$ From Proposition \ref{prop uniform} we have that
 \[ D_{\Omega(\EE_{\rho})}(\pi(\EE))= \pi(\EE^{\rho} \cup sh^{-1}(\EE_{\rho})) \]
 up to a multiplicity and hence $D_{\Omega(\EE_{\rho})}(\pi(\EE)) \neq 0$.
 
 Now we prove the backward direction of \eqref{eq thm 4.13}. To make the picture clear, we identify $I_{\rho}=\{1, \dots ,n\}$ with $1< \cdots < n$. Now we apply induction on $n$. The case $n=1$ follows from the definition.
 
 Denote $ \EE_{t}= \EE^{\rho} \cup sh_n^{t}(\EE_{\rho})$. When $t \gg 0$, the conclusion for $\EE_{t}$ follows from the induction hypothesis for $n-1$, so it suffices to show that for any $t \in \Z_{\geq 0}$, the conclusion for $\EE_{t+1}$ implies that of $\EE_{t}$. Replacing $\EE$ by $\EE_t$, we may assume $t=0$.
 
Write the multi-set $\Omega(\EE_{\rho})$ as 
\[ \Omega(\EE_{\rho})= \{ (\rho|\cdot|^{a_1})^{r_1},\dots,(\rho|\cdot|^{a_t})^{r_t} \} \]
where $a_1 < \cdots< a_t$ and $r_i$ denotes the multiplicity of $\rho|\cdot|^{a_i}$ in $\Omega(\EE)$. Assume $ a_p= B_n, a_q= A_n$. Then we know
\[ \Omega((\EE_1)_{\rho})= \{ (\rho|\cdot|^{a_1})^{r_1'},\dots,(\rho|\cdot|^{a_s})^{r_s'} \} \]
where 
\[ r_i'=\begin{cases}
r_i-1 &\text{ if } i=p,\\
r_i+1 & \text{ if }i=q+1,\\
r_i & \text{otherwise.}
\end{cases} \]
Define irreducible representation $\pi_i$ recursively by 
\[ \begin{cases}
\pi_0= \pi(\EE),\\
\pi_{i-1} \hookrightarrow (\rho|\cdot|^{a_i})^{r_i} \rtimes \pi_i . 
\end{cases} \]
This is well-defined (see the proofs of Proposition \ref{prop uniform} and Corollary \ref{cor uniform derivative}). Our strategy is to show 
\begin{align}\label{eq uniform shift}
    \pi(\EE_1) \hookrightarrow \times_{i=1}^s (\rho|\cdot|^{a_i})^{r_i'} \times Z_{\rho}[B_n,A_n] \rtimes \pi_s.
\end{align} 
If this is done, then Frobenius reciprocity says $D_{\Omega((\EE_1)_{\rho})}(\pi(\EE_1)) \neq 0$. The conclusion for $\EE_1$ implies $\pi(\EE^{\rho} \cup sh^{-1}((\EE_1)_{\rho})) \neq 0$. Moreover, Proposition \ref{prop uniform} shows $D_{\Omega((\EE_1)_{\rho})}(\pi(\EE_1))$ is a multiple of $\pi(\EE^{\rho} \cup sh^{-1}((\EE_1)_{\rho}))$. Now  \eqref{eq uniform shift} implies 
\[ D_{\rho|\cdot|^{ B_n,\dots,A_n}} \circ  D_{\Omega((\EE_1)_{\rho})}( \pi(\EE_1) )\neq 0,\]
and hence 
\[ \pi(\EE^{\rho} \cup sh^{-1}(\EE_{\rho}))=D_{\rho|\cdot|^{ B_n,\dots,A_n}} ( \pi(\EE^{\rho} \cup sh^{-1}((\EE_1)_{\rho}))) \neq 0, \]
and we are done.

To prove (\ref{eq uniform shift}), we proceed in four steps.
\begin{enumerate}
    \item [\textbf{Step 1:}] For $j \leq p-1$ ($a_j < B_n$), we show that
    \[ \pi(\EE_1) \hookrightarrow \times_{i=1}^j (\rho|\cdot|^{a_i})^{r_i} \times Z_{\rho}[B_n+1,A_n+1] \rtimes \pi_{j}.  \]
    
    We apply induction on $j$. When $j=0$, this is Lemma  \ref{lem shift}(ii). For $j >0$, it follows from the fact that the segments $[a_j,a_j]_{\rho}$ and $[A_n+1,B_n+1]_{\rho}$ are not linked, so the product of $\rho|\cdot|^{a_j}=Z_{\rho}[a_j,a_j]$ and $Z_{\rho}[B_n+1,A_n+1]$ commute. To be explicit,
    \begin{align*}
        \pi(\EE_1) & \hookrightarrow \times_{i=1}^{j-1} (\rho|\cdot|^{a_i})^{r_i} \times Z_{\rho}[B_n+1,A_n+1] \rtimes \pi_{j-1}\\
       &\hookrightarrow \times_{i=1}^{j-1} (\rho|\cdot|^{a_i})^{r_i} \times Z_{\rho}[B_n+1,A_n+1] \times (\rho|\cdot|^{a_j})^{r_j} \rtimes \pi_{j}\\
       &= \times_{i=1}^{j} (\rho|\cdot|^{a_i})^{r_i} \times Z_{\rho}[B_n+1,A_n+1] \rtimes \pi_{j}. 
    \end{align*}
  \item[\textbf{Step 2:}] For $p \leq j \leq q$ ($a_j= B_n,\dots,A_n$), we  show that
  \[ \pi(\EE_{1}) \hookrightarrow \times_{i=1}^j (\rho|\cdot|^{a_i})^{r_i'} \times Z_{\rho}[a_j+1, A_n+1] \times Z_{\rho}[B_n, a_j] \rtimes \pi_j. \]
  
  For $j=p$, we have 
  \[ \pi(\EE_1) \hookrightarrow \times_{i=1}^{j-1} (\rho|\cdot|^{a_i})^{r_i} \times Z_{\rho}[B_n+1,A_n+1] \times (\rho|\cdot|^{B_n})^{r_j} \rtimes \pi_{j}. \]
  We know the two irreducible constituents of $Z_{\rho}[B_n+1,A_n+1] \times \rho|\cdot|^{B_n}$ are $Z_{\rho}[B_n,A_n+1]$ and $soc(Z_{\rho}[B_n+1,A_n+1] \times \rho|\cdot|^{B_n})$, so $\pi(\EE_1)$ injects to one of 
  \begin{align*}
      \tau_1 := & \times_{i=1}^{j-1} (\rho|\cdot|^{a_i})^{r_i} 
       \times soc(Z_{\rho}[B_{n}+1, A_{n}+1] \times \rho|\cdot|^{B_{n}})\times (\rho|\cdot|^{B_{n}})^{r_j-1} \rtimes \pi_{j}, \\
       \tau_2:= & \times_{i=1}^{j-1} (\rho|\cdot|^{a_i})^{r_i} 
       \times Z_{\rho}[B_{n}, A_{n}+1] \times (\rho|\cdot|^{B_{n}})^{r_j-1} \rtimes \pi_{j}\\
       =&\times_{i=1}^{j-1} (\rho|\cdot|^{a_i})^{r_i} \times (\rho|\cdot|^{B_{n}})^{r_j-1}
       \times Z_{\rho}[B_{n}, A_{n}+1]  \rtimes \pi_{j}.
  \end{align*}
However, if $\pi(\EE_1)$ injected into $\tau_2$, then
\[ D_{\rho|\cdot|^{a_j}}^{(r_j'+1)} \circ D_{\rho|\cdot|^{a_{j-1}}}^{(r_{j-1}')}\circ\cdots\circ D_{\rho|\cdot|^{a_{1}}}^{(r_1')}(\pi(\EE_1)) \neq 0 . \]
This contradicts to Corollary \ref{cor uniform derivative}. As a consequence $\pi(\EE_1)$ must inject into $\tau_1$. Then applying Lemma \ref{lem Lapid}, we obtain
\begin{align*}
    \pi(\EE_1) &\hookrightarrow \times_{i=1}^{j} (\rho|\cdot|^{a_i})^{r_i'} \times soc(Z_{\rho}[B_n+1,A_n+1] \times \rho|\cdot|^{B_n})  \rtimes \pi_{j}\\
    & \hookrightarrow \times_{i=1}^{j} (\rho|\cdot|^{a_i})^{r_i'} \times Z_{\rho}[B_n+1,A_n+1] \times Z_{\rho}[B_n,B_n]  \rtimes \pi_{j}.
\end{align*}
This finishes the case $j=p$. 

For $p<j \leq q $, we have 
\begin{align*}
    \pi(\EE_1) \hookrightarrow \times_{i=1}^{j-1} (\rho|\cdot|^{a_i})^{r_i'} \times Z_{\rho}[a_j,A_n+1] \times Z_{\rho}[B_n,a_j-1]\times (\rho|\cdot|^{a_j})^{r_j}  \rtimes \pi_{j}.
\end{align*}
Therefore, $\pi(\EE_1)$ injects into one of 
\begin{align*}
      \tau_1 := & \times_{i=1}^{j-1} (\rho|\cdot|^{a_i})^{r_i} \times (\rho|\cdot|^{a_j})^{r_j}
       \times Z_{\rho}[a_j, A_{n}+1] \times Z_{\rho}[B_n,a_j-1] \rtimes \pi_{j}, \\
      \tau_2 := & \times_{i=1}^{j-1} (\rho|\cdot|^{a_i})^{r_i} \times (\rho|\cdot|^{a_j})^{r_j-1}
       \times Z_{\rho}[a_j, A_{n}+1] \times Z_{\rho}[B_n,a_j] \rtimes \pi_{j} .
  \end{align*}
  However, if it injected into $\tau_1$, we would have
  \[ D_{\rho|\cdot|^{a_j}}^{(r_j'+1)} \circ D_{\rho|\cdot|^{a_{j-1}}}^{(r_{j-1}')}\circ\cdots\circ D_{\rho|\cdot|^{a_{1}}}^{(r_1')}(\pi(\EE_1)) \neq 0,  \]
  which again contradicts to Corollary \ref{cor uniform derivative}. Therefore, $\pi(\EE_1)$ must inject into $\tau_2$, and we have 
\begin{align*}
    \pi(\EE_1)  \hookrightarrow \times_{i=1}^{j} (\rho|\cdot|^{a_i})^{r_i'} \times Z_{\rho}[a_j+1,A_n+1] \times Z_{\rho}[B_n,a_j] \rtimes \pi_{j}.
\end{align*}
\item [\textbf{Step 3:}] For $j = q+1$ ($a_j=A_n+1$), we  show that
 \[ \pi(\EE_{1}) \hookrightarrow \times_{i=1}^j (\rho|\cdot|^{a_i})^{r_i'} \times Z_{\rho}[B_n, A_n] \rtimes \pi_j. \]
 We have 
 \begin{align*}
      \pi(\EE_1)  \hookrightarrow \times_{i=1}^{j-1} (\rho|\cdot|^{a_i})^{r_i'} \times \rho|\cdot|^{A_n+1}\times Z_{\rho}[B_n,A_n] \times (\rho|\cdot|^{A_n+1})^{r_j} \rtimes \pi_{j},
 \end{align*}
 so it injects into one of 
 \begin{align*}
     \tau_1&:= \times_{i=1}^{j-1} (\rho|\cdot|^{a_i})^{r_i'} \times (\rho|\cdot|^{A_n+1})^{r_j+1} \times Z_{\rho}[B_n,A_n] \rtimes \pi_{j},\\
     \tau_2&:= \times_{i=1}^{j-1} (\rho|\cdot|^{a_i})^{r_i'} \times (\rho|\cdot|^{A_n+1})^{r_j} \times Z_{\rho}[B_n,A_n+1] \rtimes \pi_{j}.
 \end{align*}
 It remains to show that 
 \[ D_{\rho|\cdot|^{a_j}}^{(r_j')} \circ D_{\rho|\cdot|^{a_{j-1}}}^{(r_{j-1}')}\circ\cdots\circ D_{\rho|\cdot|^{a_{1}}}^{(r_1')}(\pi(\EE_1)) \neq 0,  \]
 and hence $\pi(\EE_1)$ must inject into $\tau_1$. Indeed, the computation so far shows (see the proof of Proposition \ref{prop uniform})
 \[ \{\rho|\cdot|^x \in \Omega((\EE_1)_{\rho})\setminus \Omega(\pi(\EE_1))_{\rho} \ | \ x \leq A_n\} =\emptyset .\]
 Since $A_n \geq 0$, Lemma \ref{lem $L$-data}(ii) implies $\Omega((\EE_1)_{\rho})=\Omega(\pi(\EE_1))_{\rho}$. Then it suffices to show the second condition in Corollary \ref{cor uniform derivative}. That is, in the $L$-data of $\pi(\EE_1)$, there is no segment of the form $\Delta_{\rho} [-A_n-1, -A_n-1]$.
 
By definition, we have 
\[ \pi(\EE)= D_{\rho|\cdot|^{A_n+1}} \circ\cdots \circ D_{\rho|\cdot|^{B_n+1}}(\pi(\EE_1)), \]
and each derivative is highest. Denote $\sigma_{B_n}= \pi(\EE_1)$ and
\[ \sigma_i= D_{\rho|\cdot|^{i}} \circ\cdots\circ D_{\rho|\cdot|^{B_n+2}} \circ D_{\rho|\cdot|^{B_n+1}}(\pi(\EE_1)). \]
We show that the multiplicities of the segment of the form $\Delta_{\rho}[ -A_n-1, -A_n-1]$ in the $L$-data of $\pi(\EE_1)$ and $\pi(\EE)$ agree by keeping track of $ \#\Omega( \sigma_i)_{\rho}$.

From the algorithm for taking a nonzero derivative (\cite[Proposition 6.1, Theorem 7.1]{AM20}), for $i \neq 0$, we have $\#\Omega( \sigma_i)_{\rho} \leq \#\Omega(\sigma_{i-1})_{\rho}$. The equality holds if and only if the derivative $D_{\rho|\cdot|^{i}}$ does not remove the segment of the form $\Delta_{\rho}[-|i|,-|i|]$ or summand $\rho \otimes S_{2}$.

For $i=0$, we write the injection given by Langlands classification as  
\[ \sigma_{i-1} \hookrightarrow  \times_{j=1}^t \Delta_{\rho_j}[\alpha_j,\beta_j] \rtimes \pi(\phi,\varepsilon). \]
 By the algorithm of taking negative derivative, $\pi(\phi,\varepsilon)$ is the same as the tempered part in the $L$-data of $\pi(\EE_1)$. Then since (say $a_j=0$)  
\[ D_{\rho}^{(r_j')} \circ\cdots\circ D_{\rho|\cdot|^{a_1}}^{(r_1')}(\pi(\EE_1)) \neq 0, \]
the multiplicity of $ \rho \otimes S_{1}$ in $\phi$ is zero. Therefore, $\pi(\phi,\varepsilon)$ is $D_{\rho}$-reduced (see the second step of the proof of Proposition \ref{prop uniform}), and we have 
\[ D_{\rho}^{(1)} ( soc( \times_{j=1}^t \Delta_{\rho_j}[\alpha_j,\beta_j])) \neq 0. \]
As a consequence, the $L$-data of $\sigma_i$ is obtained from that of $\sigma$ by changing one of the $\alpha_j=0$ by $-1$ ($\rho_j \cong \rho$), and hence $ \#\Omega( \sigma_i)=\#\Omega( \sigma_{i-1})$. In summary, we have 
\begin{align}\label{eq 3.7}
    \# \Omega(\sigma_{B_n}) \geq \cdots\geq \# \Omega(\sigma_{A_n}) \geq \# \Omega(\sigma_{A_n+1}).
\end{align}

Now we draw the conclusion from above discussion. By Corollary \ref{cor uniform derivative}, we have
\[ \# \Omega(\pi(\EE))= \# \Omega(\EE)= \# \Omega(\EE_1) = \#\Omega(\pi(\EE_1)) , \]
so the inequalities in (\ref{eq 3.7}) are all equalities. This shows the multiplicity of the segment of the form $\Delta_{\rho}[-A_n-1,-A_n-1] $ in the $L$-data of $ \pi(\EE_1)$ is the same as that of $\pi(\EE)$, which is zero by Corollary \ref{cor uniform derivative} and the assumption that $D_{\Omega(\EE_{\rho})}(\pi(\EE))\neq 0$.

\item [\textbf{Step 4:}] For $j > q+1$ ($a_j>A_n+1$), we  show that
 \[ \pi(\EE_{1}) \hookrightarrow \times_{i=1}^j (\rho|\cdot|^{a_i})^{r_i'} \times Z_{\rho}[B_n, A_n] \rtimes \pi_j \]
 
 Indeed, since the segments $[a_j,a_j]_{\rho}$ and $[A_n,B_n]_{\rho}$ are not linked it follows from the same reason in step 1.
\end{enumerate}

This completes the proof of the Theorem \ref{thm shift left}.
\end{proof}

\section{Union-intersection}\label{sec union-intersection}

In this section, we give a generalization of the operator $ui_k$, which we denote by $ui_{i,j}$. We use it to define a preorder on the collection of extended multi-segments, and show that for each $\EE \in \Rep$, there exists a unique (up to row exchanges) minimal element $\EE^{min} \leq \EE$ with $\pi(\EE^{min}) \cong \pi(\EE)$. This minimal element carries rich derivative information of $\pi(\EE)$.

\subsection{Definition and well-definedness}
In this subsection, we give the definition of $ui_{i,j}$, and show that it is well-defined.
\begin{defn} \label{def ui}
Suppose $\EE$ is an extended multi-segment, and write
\[\EE_{\rho}=\{ ([A_i,B_i]_{\rho},l_i,\eta_i)\}_{i\in (I_{\rho,>})}.\] 
Given $i,j \in I_{\rho}$, we define $ui_{i,j}(\EE_{\rho})=\EE_{\rho}$ unless
\begin{enumerate}
    \item [1.] We have $ A_i< A_j$, $B_i <B_j$ and $(j,i,>')$ is an adjacent pair for some admissible order $>'$ on $I_{\rho}$. 
    \item [2.] $ui_i$ is applicable on $\EE_{\rho,>'}$
\end{enumerate}
In this case, we define $ui_{i,j}(\EE_{\rho}):=(ui_{i}(\EE_{\rho,>'}))_{>}$, so that the admissible order of $ui_{i,j}(\EE_{\rho})$ and $\EE_{\rho}$ are the same. (If the $ui_i$ is of type 3', then we delete the $j$-th row.) Finally, we define $ui_{i,j}(\EE)= \EE^{\rho} \cup ui_{i,j}(\EE_{\rho})$.

We say $ui_{i,j}$ is applicable on $\EE$ if $ui_{i,j}(\EE) \neq \EE$. Furthermore, we say that $ui_{i,j}$ is of type 1, 2, 3, or 3' if the operation $ui_i$ is of type 1, 2, 3, or 3', respectively, in Definition \ref{ui def}.
\end{defn}

We remark that $ui_k$ in Definition \ref{ui def} is the same as $ui_{k,k+1}$.

\begin{exmp}
Let $\rho$ be the trivial representation and $$\psi=\rho\otimes S_5 \otimes S_2+ \rho\otimes S_5 \otimes S_4 + \rho\otimes S_3 \otimes S_2$$ be a local Arthur parameter of good parity for $G_n=\SO_{37}(F).$ Let $I_\rho=\{1,2,3\}$ and $A_1=\half{3}$, $B_1=\half{1},$ $A_2=\half{7},$ $B_2=\half{1},$ $A_3=\half{5},$ and $B_3=\half{3}.$ Since $A_3>A_2$ and $B_3>B_2,$ there are only 3 admissible orders on $I_\rho$ which we denote by $1<_1 2 <_1 3,$ $1<_2 3 <_2 2,$ and $2<_3 1 <_3 3.$
Let
$$
\EE=
\bordermatrix{
 &\half{1}&\half{3}&\half{5}&\half{7} \cr
& \ominus& \oplus& & \cr
 & \lhd&  \oplus&\ominus&\rhd \cr
& &\lhd &   \rhd& \cr
}_{\rho}.
$$
We have $\pi(\EE)= L(\Delta_{\rho}[\half{1},\half{-5}],\Delta_{\rho}[\half{3},\half{-7}];\pi(\half{1}^-,\half{3}^+,\half{3}^+,\half{5}^-))\in\Pi_\psi.$ $ui_k$ is not applicable on $\EE$ for $k=1,2$. The only possible union-intersection is for the first and third row.
Exchanging the second and the third rows gives 
$$
R_2(\EE)=
\bordermatrix{
 &\half{1}&\half{3}&\half{5}&\half{7} \cr
& \ominus& \oplus& & \cr
 & &  \lhd&\rhd& \cr
& \lhd &\oplus & \ominus  & \rhd \cr
}_{\rho}.
$$
Then we apply $ui_1$ of type 1 to obtain
$$
(ui_1\circ R_2)(\EE)=R_2(\EE)=
\bordermatrix{
 &\half{1}&\half{3}&\half{5}&\half{7} \cr
& \ominus& \oplus& \ominus & \cr
 & &\ominus& & \cr
& \lhd &\oplus & \ominus  & \rhd \cr
}_{\rho}.
$$ 
Finally we exchange the second and the third rows to return to the original order. We have
$$
ui_{1,3}(\EE)=(R_2\circ ui_1\circ R_2)(\EE)=R_2(\EE)=
\bordermatrix{
 &\half{1}&\half{3}&\half{5}&\half{7} \cr
& \ominus& \oplus& \ominus & \cr
& \ominus &\oplus & \ominus  & \oplus \cr
&  &\oplus &   &  \cr
}_{\rho}.
$$
At each stage the representation is preserved and hence $\pi(u_{1,3}(\EE))\cong\pi(\EE).$ 
\end{exmp}

Suppose $\EE_{\rho}\in \Rep$ is positive, then as $ui_{i,j}$ is a composition of row exchanges and $ui$, Theorem \ref{thm row exchange} and Theorem \ref{thm ui} shows that we have 
\[ \pi(\EE) \cong \pi(ui_{i,j}(\EE))\neq 0.\]

However, it is not immediate from the definition that it preserves the representation in general since if $B_i<0$ for some $i$, applying row exchange may not preserve the conditions in Definition \ref{def rep of segment}. To verify $ui_{i,j}$ preserves the representation, the main issue is to check whether the condition ($\ast$) in Theorem \ref{thm non-vanishing}(i) is preserved. We use Theorem \ref{thm shift left} to achieve this.

\begin{prop} \label{prop nonzero of ui}
Suppose $\EE \in \Rep$ and $ui_{i,j}$ is applicable on $\EE$. Then 
\[ \pi(\EE) \cong \pi(ui_{i,j} (\EE)) \neq 0. \]
In particular, $ui_{i,j}$ doesn't depend on the choice of admissible order $>'$ in the definition.
\end{prop}
\begin{proof}
 We take $t \gg 0$ such that $sh^{t}(\EE)$ is positive. Then since $ui_{i,j}$ is a composition of row exchange and $ui$, we have 
 \[ \pi(sh^t(\EE)) \cong \pi( sh^t(ui_{i,j}(\EE))).  \]
 By Theorem \ref{thm non-vanishing}(i), for $1 \leq d \leq t$, $\pi( sh^{t-d}(\EE)) \cong \pi( sh^{t-d}(ui_{i,j}(\EE)))$ provided $ \pi( sh^{t-d}(ui_{i,j}(\EE))) \neq 0$. Next, we show $\pi( sh^{t-d}(ui_{i,j}(\EE))) \neq 0,$ by induction on $d$.
 
 Note that we have $\Omega(ui_{i,j}(\EE))= \Omega(\EE)$. Since $\pi(\EE)$ is nonzero, by Theorem \ref{thm non-vanishing}, we have $\pi(sh^{t-d}(\EE))$ is nonzero for any $1\leq d \leq t.$
 Proposition \ref{prop uniform} then implies that  $D_{\Omega(sh^{t-d+1}(\EE))}(\pi(sh^{t-d+1}(\EE)))\neq 0$ and hence, from the induction hypothesis, we have $D_{\Omega(sh^{t-d+1}(ui_{i,j}(\EE)))}(\pi(sh^{t-d+1}(ui_{i,j}(\EE))))\neq 0.$ By Theorem \ref{thm shift left}, $\pi(sh^{t-d}(ui_{i,j}(\EE))) \neq 0$, which completes the proof of the proposition.
\end{proof}

We end this subsection with the following useful observation.

\begin{lemma}\label{lem ui shift}
If $\EE \in \Rep$ and $ui_{i,j}$ is applicable on $\EE$, then $\pi(sh_{j}^{-1}(\EE))=0$.   
\end{lemma}
\begin{proof}
    This follows directly from Definition \ref{ui def} and Proposition \ref{prop positive non-vanishing}.
\end{proof}

\subsection{Applicability of ui}\label{sec applicability of ui}
In this subsection, we prove that $ui_{i,j}$ is applicable on $\EE$ if and only if $\pi(\EE)$ is in the intersection of two local Arthur packets $\Pi_{\psi}$ and $\Pi_{ui_{i,j}(\psi)}$ defined below. As a consequence, we derive a formula for inverse of $ui_{i,j}$.

Given $\EE \in \Rep$, we denote $\psi= \bigoplus_{\rho} \bigoplus_{i \in I_{\rho}}  \rho \otimes S_{a_i} \otimes S_{b_i}$ the local Arthur parameter associated with $\EE$. That is, if $\EE= \cup_{\rho} \{ [A_i,B_i]_{\rho}, l_i,\eta_i\}_{i \in (I_{\rho},>)}$, then  
\[a_i= A_i+B_i+1,\ b_i= A_i-B_i+1.\]
Fix $i,j \in I_{\rho}$. Suppose $\EE$ satisfies the first condition in Definition \ref{def ui} for $ui_{i,j}$, then we denote $ui_{i,j}(\psi):= \bigoplus_{\rho} \bigoplus_{i \in I_{\rho}}  \rho \otimes S_{a_i'} \otimes S_{b_i'}$, where 
\[a_r'= A_r'+B_r'+1,\ b_r'= A_r'-B_r'+1,\]
and
\[ [A_r',B_r']= \begin{cases} [A_r,B_r] & \text{ if } r \neq i,j \text{ or } r \not\in I_{\rho},\\
[A_j,B_i] & \text{ if }r=i,\\
[A_i,B_j] & \text{ if }r=j.\end{cases} \]
We delete the summand of $ui_{i,j}(\psi)$ of the form $\rho \otimes S_{a_i}\otimes S_{0}$ if any. Then we can interpret Proposition \ref{prop nonzero of ui} as 
\[ ui_{i,j} \text{ is applicable on }\EE \Rightarrow \pi(\EE) \in \Pi_{\psi} \cap \Pi_{ui_{i,j}(\psi)}.  \]
We  show the backward direction in the next theorem. Before that, we fix a specific choice of admissible order in Definition \ref{def ui}.

Assume $(I_{\rho},>)$ satisfies ($P'$). We take $i,j \in I_{\rho}$ such that $\EE_{\rho}$ satisfies the first condition in Definition \ref{def ui}. It is necessary that $i<j$. For each $i < \alpha <j$, we have $ B_{i} \leq B_{\alpha} \leq B_{j}$ since $\mathcal{E}$ satisfies ($P'$). After row exchanges, we further assume that $ B_{i} < B_{\alpha} < B_{j}$ for every $i <\alpha <j$.
 
 If $(j,i,>')$ is an adjacent pair, then
\[ \begin{cases}
     \alpha <' i <' j &\Rightarrow A_{\alpha}\leq A_{i},\\
     i<'j<'\alpha &\Rightarrow A_{\alpha}\geq A_{j}.
\end{cases}\]
Therefore, we have a partition of the set $\{ \alpha \in I_{\rho} \ | \  i <\alpha< j \}= \Phi_1 \sqcup \Phi_2$, where
\begin{align*}
     \Phi_1=&\{ \alpha \ | \ A_{\alpha}  \leq A_{i}\}= \{ \alpha_1<\cdots<\alpha_r \}, \\
     \Phi_2=&\{ \alpha \ | \ A_{\alpha}  \geq A_{j}\}=\{ \beta_1<\cdots <\beta_s \}.
\end{align*} 
In other words, for $ \alpha \in \Phi_1, \beta \in \Phi_2$, we have 
\[ \begin{cases}
    [ A_\alpha, B_{\alpha} ]_{\rho} &\subseteq [A_{i},B_i]_{\rho},\\
    [ A_\beta, B_{\beta} ]_{\rho} &\supseteq [A_{j},B_j]_{\rho},\\
    [ A_\alpha, B_{\alpha} ]_{\rho} &\subseteq [A_{\beta},B_{\beta}]_{\rho}\ \ \text{ if } \alpha > \beta.
\end{cases}  \]

Then we fix a specific admissible order $\gg$ by 
\[ 1 \ll\cdots\ll i-1 \ll \alpha_1 \ll\cdots\ll \alpha_r \ll i \ll j \ll \beta_1 \ll\cdots\ll \beta_s \ll j+1 \ll\cdots\ll n, \]
where we identified $(I_{\rho},>)$ with $\{1,\dots,n\}$ and $1 < \cdots <n$.

\begin{thm} \label{thm ui packet}
Suppose $\EE \in \Rep$ and it satisfies condition 1. in Definition \ref{def ui} for $ui_{i,j}$. Then $ui_{i,j}$ is applicable on $\EE$ if and only if $\pi(\EE) \in \Pi_{\psi} \cap \Pi_{ui_{i,j}(\psi)}$.
\end{thm}

\begin{proof}
As stated above, one direction is already complete. We  show that if  $\pi(\EE) \in \Pi_{\psi} \cap \Pi_{ui_{i,j}(\psi)},$ then $ui_{i,j}$ is applicable on $\EE.$

Let $\EE'$ be the corresponding extended multi-segment in $\Pi_{ui_{i,j}(\psi)}$ such that $\pi(\EE) \cong \pi(\EE')$. For any $t >0$, since $\Omega(sh^t(\EE))= \Omega(sh^t(\EE'))$, Lemma \ref{lem $L$-data} and Theorem \ref{thm non-vanishing}(i) imply  $\pi(sh^t(\EE))\cong \pi(sh^t(\EE'))$. Therefore, we may assume that $sh^{-1}(\EE)$ and $sh^{-1}(\EE')$ are positive.

We identify $(I_{\rho},>)$ with $\{1, \dots, n\}$ and $1 < \cdots < n$, and write 
\begin{align*}
    \EE_{\rho}&= \{ ( [A_i,B_i]_{\rho},l_i,\eta_i)\}_{i=1}^n,\\
\EE_{\rho}'&=\{ ([ A_i',B_i']_{\rho},l_i',\eta_i')\}_{i=1}^n.
\end{align*}
We first assume $i=1, j=n$. We use the notation $\Phi_1,\Phi_2$ defined above.

Case (1): Suppose $\Phi_2$ is empty. We decompose $\EE_{\rho}=\FF_1 + \FF_2, \EE_{\rho}'= \FF_1' + \FF_2'$ where $\FF_1, \FF_1'$ are the first row of $\EE_{\rho},\EE_{\rho}'$ respectively. From the definition of union-intersection and Lemma \ref{lem identities}(i), $ui_{1,n}$ is not applicable on $\EE$ if and only if 
    \[\pi( \EE^{\rho} \cup( \FF_1 + sh^{-1}(\FF_2))) \neq 0.\]
    Hence if $ui_{1,n}$ is not applicable, then by Lemma \ref{lem shift}(iv), $D_{\Omega(\FF_2)} (\pi(\EE)) \neq 0.$ Now we show that this derivative must be zero for contradiction.
    
    Our assumption that $sh^{-1}(\EE')$ is positive implies $ \EE^{\rho} \cup (sh^{-1}(\EE_{\rho}'))$ is in $\Rep$ by Theorem \ref{thm non-vanishing}(i). Then Lemma  \ref{lem shift}(iv) shows
    \[ \pi(\EE') \hookrightarrow \times_{i=1}^n Z_{\rho}[B_i',A_i'] \rtimes \pi(\EE^{\rho} \cup (sh^{-1}(\EE_{\rho}'))). \]
    Also, Lemma \ref{lem Leibniz rule} implies 
    \begin{align*}
         D_{\Omega(\FF_2)}(\pi(\EE'))&\leq D_{\Omega(\FF_2)}(\times_{i=1}^n Z_{\rho}[B_i',A_i'] \rtimes \pi(\EE^{\rho} \cup (sh^{-1}(\EE_{\rho}'))))\\
         &=Z_{\rho}[B_1,A_n] \rtimes D_{\Omega(\FF_2)}(\times_{i=2}^n Z_{\rho}[B_i',A_i'] \rtimes \pi(\EE^{\rho} \cup (sh^{-1}(\EE_{\rho}')))).
    \end{align*}
     Finally, observe that $\rho|\cdot|^{A_n}$ is not in $\Omega(sh^{-1}(\EE_{\rho}'))$ and hence $\rho|\cdot|^{A_n}$ is also not in $\Omega(\pi(\EE^{\rho} \cup (sh^{-1}(\EE_{\rho}'))))$ by Lemma \ref{lem $L$-data}. Also, $A_i'<A_n$ for $2 \leq i \leq n$. We conclude that
    \[D_{\Omega(\FF_2)}(\times_{i=2}^n Z_{\rho}[B_i',A_i'] \rtimes \pi(\EE^{\rho} \cup (sh^{-1}(\EE_{\rho}'))))=0.\]
    As a consequence, $ui_{1,n}$ is applicable on $\EE$.

Case (2): In general, we apply induction on the size of $\Phi_2$. The case that $\Phi_2=\emptyset$ is done above. 
Write $\Phi_2= \{ \beta_1 < \cdots <\beta_s \}$. We define another admissible order $\gg$ by
\[ 1 \ll \cdots \ll \beta_s -1 \ll \beta_{s}+1 \ll \cdots \ll n \ll \beta_s. \]
 Lemma  \ref{lem shift}(iii) says that for any $d \in \N$, $\EE^{\rho} \cup sh_{\beta_s}^{d}(\EE_{\rho,\gg}), \EE^{\rho} \cup sh_{\beta_s}^{d}(\EE_{\rho,\gg}')$ are both in $\Rep$.

 We first claim that 
 \[ \pi(\EE) \cong \pi(\EE') \Longleftrightarrow \pi(\EE^{\rho} \cup sh_{\beta_s}^{d}(\EE_{\rho,\gg})) \cong \pi(\EE^{\rho} \cup sh_{\beta_s}^{d}(\EE_{\rho,\gg}')). \]
 Indeed, Lemma  \ref{lem shift}(iii) implies the forward direction. It follows from Definition \ref{def rep of segment} that
 \begin{align*}
      D_{\rho|\cdot|^{B_{\beta_s}+d,\dots,A_{\beta_s}+d}}(\pi(\EE^{\rho} \cup sh_{\beta_s}^{d}(\EE_{\rho,\gg})) &= \pi(\EE^{\rho} \cup sh_{\beta_s}^{d-1}(\EE_{\rho,\gg})),\\
       D_{\rho|\cdot|^{B_{\beta_s}+d,\dots,A_{\beta_s}+d}}(\pi(\EE^{\rho} \cup sh_{\beta_s}^{d}(\EE_{\rho,\gg}')) &= \pi(\EE^{\rho} \cup sh_{\beta_s}^{d-1}(\EE_{\rho,\gg}')),
 \end{align*}
 which shows the backward direction.
 
 As a consequence, let $\psi_d$ and $\psi_d'$ be the local Arthur parameters associated with $\EE^{\rho} \cup sh_{\beta_s}^{d}(\EE_{\rho,\gg})$ and $\EE^{\rho} \cup sh_{\beta_s}^{d}(\EE_{\rho,\gg}')$ respectively. We have 
 \[ \pi(\EE) \in \Pi_{\psi } \cap \Pi_{ui_{i,j}(\psi)} \Longleftrightarrow \pi( \EE^{\rho} \cup sh_{\beta_s}^{d}(\EE_{\rho,\gg})) \in \Pi_{\psi_d} \cap \Pi_{\psi_d'}. \]

When $d$ is sufficiently large, by Lemma \ref{lemma far away}(ii), the induction hypothesis gives that $\pi( \EE^{\rho} \cup sh_{\beta_s}^{d}(\EE_{\rho,\gg})) \in \Pi_{\psi_d} \cap \Pi_{\psi_d'}$ implies $ui_{1,n}$ is applicable on $\EE^{\rho} \cup sh_{\beta_s}^{d}(\EE_{\rho,\gg})$. This shows $ui_{1,n}$ is applicable on $\EE$.

Next, consider general case that $ 1\leq i<j \leq n$. We decompose $\EE_{\rho}=\FF_1 + \FF_2 + \FF_3$ where
\begin{align*}
    \FF_1&=\{([A_r,B_r]_{\rho},l_r,\eta_r)\}_{r=1}^{i-1},\\
    \FF_2&=\{([A_r,B_r]_{\rho},l_r,\eta_r)\}_{r=i}^{j},\\
    \FF_3&=\{([A_r,B_r]_{\rho},l_r,\eta_r)\}_{r=j+1}^{n},
\end{align*}
and similarly for $\EE_{\rho}'=\FF_1'+\FF_2' + \FF_3'$. Then Corollary \ref{cor shift add}(i), (iii) implies we may assume $\FF_3, \FF_3'$ are empty and replace $\FF_1$ (resp. $\FF_1'$) by $add^{t}(\FF_1)$ (resp. $add^{t}(\FF_1')$) for a large integer $t$ since
\[ \pi(\EE^{\rho} \cup (add^{t}(\FF_1)+ \FF_2) \cup (sh^1(\FF_3))_{\rho^{\ast}})\cong \pi(\EE^{\rho} \cup (add^{t}(\FF_1')+ \FF_2') \cup (sh^1(\FF_3)')_{\rho^{\ast}}). \]
Then we can apply the same argument above to show that $ui_{i,j}$ must be applicable on $add^{t}(\FF_1)+ \FF_2$, and hence applicable on $\EE$. This completes the proof of the theorem. 
\end{proof}

As a result, we give a formula of the inverse of $ui_{i,j}$ if it is not of type 3' in Definition \ref{def ui}. 

 \begin{cor} \label{cor ui inverse}
Suppose $ui_{i,j}$ is applicable on $\EE \in Rep^{(P')}$ and it is not of type 3' in Definition \ref{def ui}. Then $ui_{j,i}$ is applicable on $dual \circ ui_{i,j}(\EE)$, and we have 
\[ dual \circ ui_{j,i} \circ dual \circ ui_{i,j}(\EE)= \EE. \]
In other words, we may regard $(dual \circ ui_{j,i} \circ dual)$ as the inverse of $ui_{i,j}$.
\end{cor}
\begin{proof}
By assumption and Theorem \ref{thm ui packet}, $ \pi(\EE) \in \Pi_{\psi} \cap \Pi_{ui_{i,j}(\psi)}$. Then we have $\widehat{\pi(\EE)} \in \Pi_{\widehat{\psi}} \cap \Pi_{\widehat{ui_{i,j}(\psi)}}$. A simple computation shows $ \widehat{ui_{i,j}(\psi)}= ui_{j,i}(\widehat{\psi}),$
and the corollary follows.
\end{proof}

Alternatively, one can prove Corollary \ref{cor ui inverse} purely combinatorially from the definitions of these operators. 

\begin{exmp}
Recall the setup of Example \ref{ui case 2}. Let $\rho$ be the trivial representation and $$\psi=\rho\otimes S_3\otimes S_1+\rho\otimes S_4\otimes S_2+\rho\otimes S_6\otimes S_2$$ be a local Arthur parameter of good parity for $G_n=\Sp_{22}(F).$ Let 
$$\EE=\bordermatrix{
 &{1}&2&3 \cr
& \ominus& &  \cr
 & \lhd&  \rhd& \cr
& &\oplus &   \ominus \cr
}_{\rho}.$$
We have $\pi(\EE)= L(\Delta_{\rho}[1,-2];\pi(1^-,2^+,3^-))\in\Pi_\psi.$
Let
$$
\EE'=ui_{2,3}(\EE)=\bordermatrix{
 &{1}&2&3 \cr
& \ominus& &  \cr
 & \lhd&  \oplus & \rhd \cr
& &\ominus &   \cr
}_{\rho},
$$
where the union-intersection is of type 2.
Then $\EE$ is the inverse of $ui_{2,3}$ for $\EE'.$ On the other hand, Corollary \ref{cor ui inverse}, says $(dual\circ ui_{3,2} \circ dual)(\EE')=\EE.$ We check this explicitly.

We have 
$$
dual(\EE')=\bordermatrix{
 &-2 & -1 & 0 &{1}&2&3 \cr
&\lhd & \lhd & \ominus& \rhd & \rhd &  \cr
 & & \lhd & \lhd & \oplus &\rhd & \rhd  \cr
& &\lhd & \ominus & \rhd & &   \cr
}_{\rho}.
$$ Note that the ordering in the dual is $3<'2<'1.$ So $ui_{3,2}$ is the union-intersection of the first and second rows.
We have
$$
(ui_{3,2}\circ dual)(\EE')=\bordermatrix{
 &-2 & -1 & 0 &{1}&2&3 \cr
&\lhd & \lhd & \ominus& \oplus & \rhd & \rhd   \cr
 & & \lhd & \lhd  &\rhd & \rhd &  \cr
& &\lhd & \ominus & \rhd & &   \cr
}_{\rho}.
$$
Thus,
$$
(dual\circ ui_{3,2}\circ dual)(\EE')=\bordermatrix{
 &{1}&2&3 \cr
& \ominus& &  \cr
 & \lhd&  \rhd& \cr
& &\oplus &   \ominus \cr
}_{\rho}=\EE.
$$
\end{exmp}

Next, we describe the inverse of $ui_{i,j}$ of type 3' of Definition \ref{def ui}.

\begin{cor}\label{cor split}
Let $\EE \in \Rep$. 
\begin{enumerate}
    \item [(1)] Suppose $\EE_{\rho}$ is positive and there exists an admissible order $\gg$ such that $\EE_{\rho,\gg}= \{( [A_i,B_i]_{\rho},l_i',\eta_i')\}_{i \in (I_{\rho}, \gg)}$ with $l_j'=0$ for some $j$. Let 
    \[\mathcal{F}_1=\{ ([A_i,B_i]_{\rho},l_i',\eta_i')\}_{i \ll j },\  \mathcal{F}_2=\{ ([A_i,B_i]_{\rho},l_i',\eta_i')\}_{i \gg j}.\]
    For $0 \leq r \leq A_j-B_j -1$, we set
    \begin{align*}
         \EE_{\rho,r}:= \mathcal{F}_1 + \{([B_j+r, B_j]_{\rho},0,\eta_j')\} + \{([A_j,B_j+r+1]_{\rho},0, (-1)^{r+1}\eta_j')\} +\mathcal{F}_2.
    \end{align*}
    If the total order of $\EE_{\rho,r}$ satisfies ($P$), then for any admissible order $\gg'$ of $\EE_{\rho,r}$, we have 
    \[ 0 \neq \pi(\EE^{\rho} \cup (\EE_{\rho, r})_{\gg'}) \cong \pi(\EE) \]
    \item [(2)] The same statement in (1) holds for general $\EE_{\rho}$ if we further require $ 2B_j+r\geq 0$ and $\gg'$ satisfies ($P'$).
\end{enumerate}
\end{cor}

\begin{proof}
We identify $(I_{\rho},>)$ with $\{1, \dots , n\}$ with $1 < \cdots < n$ for simplicity.

For Part (1), we may assume $\EE_{\rho}= \EE_{\rho, \gg}$. It suffices to show $\pi(\EE) \cong \pi(\EE^{\rho} \cup \EE_{\rho, r})$. We first deal with the special case that $j=n$. We identify the index set of $\EE_{\rho,r}$ with $\{1, \dots, n, n+1\}$ and $1 < \cdots <n <n+1$.

By Theorem \ref{thm non-vanishing}(ii), we have 
\[ 0 \neq \pi( \EE^{\rho} \cup (sh^{t}_n+sh^{t}_{n+1})(\EE_{\rho, r})) \cong \pi(\EE^{\rho} \cup sh^{t}_n(\EE_{\rho})),  \]
when $t$ is large. By definition, we have 
\[0 \neq  \pi(\EE)= \circ_{i=t}^{1} D_{\rho|\cdot|^{B_n+t,\dots,A_n+t}}( \pi(\EE^{\rho} \cup sh^{t}_n(\EE_{\rho}))) ), \]
where $\circ_{i=t}^{1}$ means we take the derivative with $i=t, $ then $i=t-1,$ and so on until we arrive at $i=1.$
Therefore, it suffices to show that
\[\pi( \EE^{\rho} \cup (sh^{t-1}_n+sh^{t-1}_{n+1})(\EE_{\rho, r}))= D_{\rho|\cdot|^{B_n+t,\dots,A_n+t}}(\pi( \EE^{\rho} \cup (sh^{t}_n+sh^{t}_{n+1})(\EE_{\rho, r})). \]
We may assume $t=1$. By definition, we can take $s\gg 0$ such that
\[ \pi(\EE^{\rho} \cup \EE_{\rho,r})=  \left(\circ_{i=s}^{1}  D_{\rho |\cdot|^{B_n+r+ 1+i,\dots,A_n+i}} \right) \circ D_{\rho |\cdot|^{B_n+1,\dots,B_n+r+1}}(\pi(\EE^{\rho} \cup (sh^{1}_n+sh^{s}_{n+1})(\EE_{\rho, r})). \]
Then we apply the commutativity of $D_{\rho |\cdot|^{B_n+r+ 1+i,\dots,A_n+i}} $ and $D_{\rho |\cdot|^{B_n+1,\dots,B_n+r+1}}$ for $i \geq 2$ and obtain (note that $D_{\rho |\cdot|^{B_n+r+2,\dots,A_n+1}} \circ D_{\rho |\cdot|^{B_n+1,\dots,B_n+r+1}}= D_{\rho |\cdot|^{B_n+1,\dots,A_n+1}}$)
\begin{align*}
    \pi(\EE^{\rho} \cup \EE_{\rho,r})=&  D_{\rho |\cdot|^{B_n+1,\dots,A_n+1}}  \circ_{i=s}^{2}  D_{\rho |\cdot|^{B_n+r+ 1+i,\dots,A_n+i}} (\pi(\EE^{\rho} \cup (sh^{1}_n+sh^{s}_{n+1})(\EE_{\rho, r}))) \\
    &=D_{\rho |\cdot|^{B_n+1,\dots,A_n+1}} (\pi( (sh_n^{1}+sh_{n+1}^1)(\EE_r))).
\end{align*}
This completes the proof of the special case.

For the general case, we take a sequence of non-negative integers $\{t_i\}_{i=1}^n$ such that $B_{i}+t_i > A_{i-1}+t_{i-1} $ for $1 < i \leq n$. Then the same argument above shows 
\[ 0 \neq \pi\left(\EE^{\rho}\cup \left(\sum_{i=j+2}^{n+1} sh_i^{t_i}\right)(\EE_{\rho,r})\right)  \cong  \pi\left(\EE^{\rho} \cup \left(\sum_{i=j+1}^{n} sh_i^{t_i}\right)(\EE_{\rho})\right).  \]
Let 
\[ D= \circ_{i=j+1}^n \left( D_{ \rho|\cdot|^{ B_{i}+1,\dots,A_i+1}} \circ\cdots \circ D_{\rho|\cdot|^{B_i+t_i,\dots,A_i+t_i}} \right). \]
By definition,
\begin{align*}
    \pi(\EE^{\rho}\cup \EE_{\rho,r})&= D\left(\pi\left(\EE^{\rho}\cup \left(\sum_{i=j+2}^{n+1} sh_i^{t_i}\right)(\EE_{\rho,r})\right)\right)\\
    &=D\left(\pi\left(\EE^{\rho}\cup \left(\sum_{i=j+2}^{n+1} sh_i^{t_i}\right)(\EE_{\rho})\right)\right)\\
    &=\pi(\EE).
\end{align*}
This completes the proof of Part (1).

For Part (2), we take a large integer $t$ such that $sh^t(\EE)$ is positive. Then Part (1) implies
\[ 0 \neq \pi(sh^t(\EE) ) \cong \pi( sh^t(\EE^{\rho}) \cup sh^t(  (\EE_{\rho, r})_{\gg'})). \]
Now observe that $\Omega( sh^t(\EE))= \Omega(sh^t((\EE_r)_{\gg'}))$, so Theorem \ref{thm shift left} implies the result.
This completes the proof of the corollary.
\end{proof}

\begin{exmp}
Let $\rho$ be the trivial representation and 
\[\psi=\rho\otimes S_4\otimes S_6+ \rho \otimes S_2 \otimes S_2+ \rho \otimes S_7 \otimes S_1\] be a local Arthur parameter of good parity for $G_n=\Sp_{34}(F).$ We consider
\[\EE= \bordermatrix{
&-1&0&1&2&3&4\cr
&\lhd &\lhd & \ominus&\oplus & \rhd & \rhd \cr
&&\ominus & \oplus &&&\cr 
&&&&&\oplus&\cr
}_{\rho}, \]
where $\pi(\EE)= L( \Delta_{\rho}[ -1,-1], \Delta_{\rho}[ 3,-4]; \pi( 0^{+},0^{+},1^{-},2^{+},3^{-}))\in\Pi_\psi$. Then 
\[\EE'= R_1(\EE)= \bordermatrix{
&-1&0&1&2&3&4\cr
&&\oplus & \ominus &&&\cr 
&\ominus &\oplus & \ominus&\oplus & \ominus & \oplus \cr
&&&&&\oplus&\cr
}_{\rho}. \]
Corollary \ref{cor split} is applicable at the second row of $\EE'$. Then only $\EE_1',\EE_{2}', \EE_{3}'$ satisfy (P), and only $\EE_2',\EE_3'$ satisfies $ 2B_2+r \geq 0$. Thus, 
\[\EE_2'= \bordermatrix{
&-1&0&1&2&3&4\cr
&&\oplus & \ominus &&&\cr 
&\ominus &\oplus & \ominus&& &  \cr
& & & &\oplus & \ominus & \oplus \cr
&&&&&\oplus&\cr
}_{\rho},\ \EE_3'= \bordermatrix{
&-1&0&1&2&3&4\cr
&&\oplus & \ominus &&&\cr 
&\ominus &\oplus & \ominus& \oplus& &  \cr
& & & & & \ominus & \oplus \cr
&&&&&\oplus&\cr
}_{\rho}. \]
Since $B_1<0$, we need to do a row exchange to satisfy (P'). We obtain
\[R_1(\EE_2')= \bordermatrix{
&-1&0&1&2&3&4\cr
&\lhd &\oplus & \rhd&& &  \cr
&&\oplus & \ominus &&&\cr 
& & & &\oplus & \ominus & \oplus \cr
&&&&&\oplus&\cr
}_{\rho},\ R_1(\EE_3')= \bordermatrix{
&-1&0&1&2&3&4\cr
&\lhd &\lhd & \rhd& \rhd& &  \cr
&&\ominus & \oplus &&&\cr 
& & & & & \ominus & \oplus \cr
&&&&&\oplus&\cr
}_{\rho}. \]
By direct computation, we verify $\pi(\EE)\cong\pi(R_1(\EE'_2))\cong\pi(R_1(\EE'_3)).$ One can also check that we can apply $ui_{1,3}$ of type 3' to $R_1(\EE_2')$ and $R_1(\EE_3')$ to recover $\EE.$
\end{exmp}

\begin{defn}\label{def ui inverse}
Let $\EE \in \Rep$. We say that $ui_{i,j}\inv$ is applicable of type 1, 2, or 3 (excluding 3') if there exists $\EE'\in\Rep$ such that $\EE=ui_{i,j}(\EE')$ and this $ui_{i,j}$ is not of type 3' in Definition \ref{def ui}. Furthermore, we say $ui_{i}\inv$ is applicable of type 3' if there exists $\EE'\in\Rep$ such that $\EE=ui_{i,j}(\EE')$ and this $ui_{i,j}$ is of type 3' in Definition \ref{def ui}.
\end{defn}

Note that Corollary \ref{cor ui inverse} shows that for $ui_{i,j}\inv$ not of type 3', we have $dual \circ ui_{j,i} \circ dual=ui_{i,j}\inv.$ Also, Corollary \ref{cor split}, describes the image of $ui_i\inv.$

\subsection{Preorder on extended multi-segments}\label{sec preorder}
In this subsection, we define a preorder on the set of extended multi-segments. We  show that fixing $\EE \in \Rep$, the set $\{ \EE' \ | \ \EE' \leq \EE\}$ has a unique minimal element (up to row exchanges). Moreover, this minimal element carries certain invariants of $\pi(\EE)$.
\begin{defn}
For $\EE, \EE' \in \Rep$, we define $\EE \leq \EE'$ if $\EE$ can be obtained from $\EE'$ by successively applying $ui_{i,j}$ and row exchange $R_k$.

We say $\EE$ is minimal if for any $\EE'$ such that $\EE' \leq \EE$, the relation $\EE'\geq \EE$ also holds. In other words, $\EE'=\EE$ up to row exchanges. Equivalently, $\EE$ is minimal if and only if for any $\rho$ and $i,j \in I_{\rho}$, $ui_{i,j}$ is not applicable on $\EE$.

Denote $\EE_{\rho}= \FF$, and let $ \FF= \FF_1 + \FF_2 + \FF_3$ be an arbitrary decomposition (we allow $\FF_1$ and $\FF_3$ to be empty). We define $ \FF_2 \leq \FF_2'$ if $\EE^{\rho} \cup (\FF_1 + \FF_2+ \FF_3) \leq \EE^{\rho} \cup ( \FF_1 + \FF_2' + \FF_3)$. We say $\FF_2$ is minimal if it satisfies the same property in the previous paragraph.
\end{defn}

We remark that we consider this preorder on $\FF_2$ where $\EE_{\rho}= \FF_1+\FF_2+ \FF_3$ because sometimes we want to replace $\FF_2$ by a minimal element in the set $\{ \FF_2' \ | \ \FF_2' \leq \FF_2\}$ without changing the part $\FF_1$ and $ \FF_3$. This is crucial in the proof of the main theorems in the next section.

We fix the following notation. Suppose $\EE \in \Rep^{(P')}$ and  denote $$\FF=\EE_{\rho}= \{([A_i,B_i]_{\rho},l_i,\eta_i)\}_{ i \in (I_{\rho},>)}.$$ 
Let $ x \in \R$. We denote
\begin{align*}
\FF_{ >x}&=\{ ([A_i,B_i]_{\rho},l_i,\eta_i)\}_{i \in I_{\rho}, B_i>x},\\
    \FF_{=x}&=\{ ([A_i,B_i]_{\rho},l_i,\eta_i)\}_{i \in I_{\rho}, B_i=x},\\
    \FF_{ <x}&=\{ ([A_i,B_i]_{\rho},l_i,\eta_i)\}_{i \in I_{\rho}, B_i<x},
\end{align*}
with the admissible order inherited from $(I_{\rho}, >)$. Note that $\FF= \FF_{<x}+\FF_{=x}+ \FF_{>x}$. Also, we write $\FF_{\leq x}= \FF_{<x} + \FF_{=x}$ and $\FF_{\geq x}= \FF_{=x} + \FF_{>x}$.

We first show that minimal elements under this preorder enjoy the following property. 
\begin{lemma}\label{lem minimal shift}
Suppose $\EE \in \Rep^{(P')}$ is minimal. Denote $\FF= \EE_{\rho}$. Then for any $B>1/2$, we have $\EE^{\rho} \cup (\FF_{<B} + sh^{-1}(\FF_{\geq B})) \in \Rep$. In particular, we have 
\begin{align*}
    D_{\Omega(\FF_{\geq B})}(\pi(\EE))&= \pi(\EE^{\rho} \cup (\FF_{<B} + sh^{-1}(\FF_{\geq B})))
\end{align*}
is a composition of highest derivatives up to a multiplicity, and
\[S_{\Omega(\FF_{\geq B})} (\pi(\EE^{\rho} \cup(\FF_{<B} + sh^{-1}(\FF_{\geq B}))))= \pi(\EE).\]
\end{lemma}
\begin{proof}
The second assertion follows from the first one and Lemma  \ref{lem shift}(iv). Now we prove the first assertion. Fix $B>\frac{1}{2}$, and write $\FF= \{ ([A_i,B_i]_{\rho},l_i,\eta_i)\}_{i=1}^n$. We denote $\FF^{B}=\FF_{<B} + sh^{-1}(\FF_{\geq B}) $ and $ \EE^{B}= \EE^{\rho} \cup \FF^{B}$ for brevity.

Assume for the moment that $\EE^B$ is non-negative. 
Suppose for contradiction that $\EE^B$ is not in $\Rep$. 
Since $\EE^B\not\in \Rep$, by Theorem \ref{thm non-vanishing}(ii), there is an adjacent pair $(i,j, \gg)$ of $ \EE^{B}$ violates the non-vanishing conditions in Proposition \ref{prop positive non-vanishing}(i). Then $\gg$ must not be an admissible order of $\FF$ since $\EE \in \Rep$. After row exchanges, we may assume $\FF$ is of the form
\begin{align*}
    \FF_{<B}=& \{([A_1,B_1]_{\rho},l_i,\eta_i)\},\\
    \FF_{\geq B}=& \{([A_1+1,B_2]_{\rho},l_i,\eta_i)\}_{i=2}^{n-1} + \{([A_n,B_3]_{\rho},l_n,\eta_n)\},
\end{align*}
where $B_3 \geq B_2$ and $i=1, j=n$. Then $\FF_{<B} \cup sh^{-1}(\FF_{\geq B})$ satisfies the condition of Proposition \ref{prop positive non-vanishing}(ii), so we only need to check the non-vanishing conditions in Proposition \ref{prop positive non-vanishing}(i) in this specific order. It remains to check it for the first two rows. However, since $\EE$ is minimal, $ui_{1,2}$ is not applicable on $\FF$. Consequently, the non-vanishing conditions in Proposition \ref{prop positive non-vanishing}(i) hold in this specific order. This is a contradiction and hence verifies the lemma when $\EE^B$ is non-negative.

Suppose that $\EE^B$ is not necessarily non-negative. Let $t$ be any large enough integer such that $sh^t(\EE^B)=(sh^t(\EE))^B$ is non-negative. The previous case shows that $(sh^t(\EE))^B\in\Rep.$ Since $\EE,(sh^t(\EE))^B\in\Rep,$ it follows from Theorem \ref{thm non-vanishing}(i) that $\EE^B\in\Rep$. This completes the proof of the lemma.
\end{proof}

As a corollary, we show the uniqueness of the minimal element (up to row exchanges) in the set $\{\EE' \ | \ \EE' \leq \EE\}$, and denote the one with a specific admissible order by $\EE^{min}$.
\begin{cor}\label{cor minimal}\ 
\begin{enumerate}
    \item [(i)] Suppose $\EE ,\EE' \in \Rep^{(P')}$ are minimal elements under this preorder and $\pi(\EE) \cong \pi(\EE')$. Then for any $\rho$, we have up to row exchanges,
    \[ (\EE_{\rho})_{>1/2} =(\EE_{\rho}')_{>1/2}.\]
 \item[(ii)] Suppose $\EE \in \Rep$. There exists a unique minimal element (up to row exchanges) in the set $\{ \EE' \ | \ \EE' \leq \EE \}$. Let $\EE^{\min}=\cup_{\rho} \{([A_i,B_i]_{\rho},l_i,\eta_i)\}_{i \in (I_{\rho, >})}$ denote the unique minimal element in this set such that for any $i<j \in I_{\rho}$,
\begin{enumerate}
    \item [$\oldbullet$] $B_i \leq B_j$ , 
    \item [$\oldbullet$] if $B_i= B_{j}$, then $ A_i \leq A_{j}$.
\end{enumerate}
Suppose $\EE_{\rho}= \FF_1 +\FF_2 +\FF_3$. Then we denote $\FF_2^{min}$ the unique minimal element in the set $\{ \FF_2' \ | \ \FF_2' \leq \FF_2\}$ satisfying the same conditions above.
\end{enumerate}
\end{cor}
\begin{proof}
For Part (i), the previous lemma shows that for any $B >1/2$, $\Omega((\EE_{\rho})_{\geq B})$ can be recovered from the derivative information of $\pi(\EE)$, so
\[\Omega((\EE_{\rho})_{\geq B})=\Omega((\EE_{\rho}')_{\geq B}). \]
On the other hand, the multi-sets $\{\Omega((\EE_{\rho})_{\geq B})\}_{B >1/2}$ uniquely determine $\supp((\EE_{\rho})_{>1/2})$ as a multi-set. So Part (i) follows from Corollary \ref{cor shift add}(i).

For Part (ii), we observe that by definition, for any $t \in \N$,
\[ \EE \leq \EE' \Longleftrightarrow sh^t(\EE) \leq sh^t(\EE').   \]
 Therefore, we may assume for any $\rho$ and $\FF= \EE_{\rho}$, $\FF=\FF_{>1/2}$, and the conclusion follows from Part (i). This completes the proof of the corollary.
\end{proof}

\begin{remark}
\begin{itemize}
    \item [1.] Let $\EE\in \Rep$ with $\EE_{\rho}=\FF$. The comparison of $\FF_{\geq 1/2}$ between minimal elements in the set $\{\FF' \ | \ \pi(\EE^{\rho} \cup \FF')\cong \pi(\EE^{\rho} \cup \FF)\}$ is more subtle, which we defer to the next section (see Proposition \ref{prop shift left 1/2}).
    \item [2.] Though the minimal element in the set $\{ \EE' \ | \ \EE' \leq \EE \}$ is unique (up to row exchanges), there may be more than one minimal elements in the set $\{ \EE' \ | \ \pi(\EE') \cong \pi(\EE) \}$. Let us explain this phenomenon based on Example \ref{ex Atobe}, where we have 
    \[ \{ \EE' \ | \ \pi(\EE') \cong \pi(\EE) \}/\text{(row exchanges)}= \{ \EE_1 ,\dots , \EE_9\}, \]
    among which $\EE_1,\EE_5,\EE_6,\EE_7$ are all minimal. Indeed, one can check that
    \[ \{ \EE' \ | \ \pi(\EE') \cong \pi(\EE) \}= \bigcup_{i=4,7,8,9} \{ \EE' \ | \ \EE' \leq \EE_i\},  \]
    and $ \EE_1=\EE_4^{min}, \EE_5=\EE_8^{min}, \EE_6=\EE_9^{min}, \EE_7= \EE_7^{min}$. 
\end{itemize}
\end{remark}

\subsection{Algorithm for \texorpdfstring{$UI^{-1}(\EE)$}{}}
We end this section by showing that how to compute the set 
\[UI^{-1}(\EE'):=\{ \EE \ | \ \EE \geq \EE'\}/(\text{row exchanges}).\]
The quotient means any $\EE \geq \EE'$ is equal to exactly one element in $UI^{-1}(\EE')$ after row exchanges. Clearly, if we define 
\[ UI^{-1} (\EE_{\rho}'):= \{ \FF\ | \ \FF \geq \EE_{\rho}'\}/\text{(row exchanges)}, \]
then we have 
\[ UI^{-1}(\EE')= \{ \cup_\rho \FF_{\rho} \ | \ \FF_{\rho} \in UI^{-1}(\EE_{\rho})  \}. \]
We give an algorithm to compute the set $UI^{-1}(\FF')$. It is simple to check that the algorithm computes $UI\inv(\FF')$ using Corollary \ref{cor ui inverse} and \ref{cor split}.

\begin{algo} \label{algo ui inverse}
Given $\EE' \in \Rep$, and $\FF'= \EE_{\rho}'$ we proceed as follows:
\begin{enumerate}
    \item [\textbf{Step 1:}] Set $A=\{\FF'\}$, $B=\{\FF'\}$, and $C= \emptyset$.
    \item [\textbf{Step 2:}]  For each $\FF$ in $B$, we compute the set 
    \[S=\{dual(\FF)_{\gg} \ | \ \gg \text{is an admissible order on }dual(\FF)\}.\]
    For each element $\FF^{\ast}=\{([A_i,B_i]_{\rho},l_i, \eta_i)\}_{i=1}^n$ in $S$ and $1 \leq k <n$, if $ui_{k}$ is applicable and not of type 3' (see Definition \ref{ui def}) on $\FF^{\ast}$ and no element in $A$ is equal to $dual(ui_k(\FF^{\ast}))$ up to row exchanges, then we add $dual(ui_k(\FF^{\ast}))$ into both $A$ and $C$.
    \item [\textbf{Step 3:}]
        For each $\FF$ in $B$, we compute the set 
    \[S=\{\FF_{\gg} \ | \ \gg \text{is an admissible order on }\FF\}.\]
    For each $\FF_{\gg}= \{([A_i,B_i]_{\rho},l_i, \eta_i)\}_{i=1}^n$ and $1 \leq i \leq n$, if Corollary \ref{cor split} is applicable at the $i$-th row, and no element in $A$ is equal to the resulting extended multi-segment (up to row exchanges), then we add it into both $A$ and $C$.
\item [\textbf{Step 4:}] If $C$ is empty, then the procedure ends. Otherwise replace $B$ by $C$ and $C$ by $\emptyset$ and go back to step 2.
\end{enumerate}
When the procedure ends we have $A=UI^{-1}(\FF')$.
\end{algo}

We give a simple example for this algorithm.
\begin{exmp}\label{ex UI inverse}
We apply the algorithm to compute $UI^{-1}(\FF)$ where
\[\FF= \bordermatrix{
& 0& 1&2 \cr
& \ominus & \oplus & \ominus \cr
& &\ominus & \cr
}_{\rho}.\]
Initially, we set $A=\{\FF\}, B=\{\FF\}$ and $C=\emptyset$. In the second step, we see $ui_{1}$ is applicable on $dual(\FF)$ of type 2. We have
 \[dual(\FF)= \bordermatrix{
& 0& 1&2 \cr
& \ominus & \oplus &  \cr
& &\lhd & \rhd\cr
}_{\rho}, ui_{1}(dual(\FF)) = \bordermatrix{
& 0& 1&2 \cr
& \ominus & \oplus & \ominus \cr
& &\ominus & \cr
}_{\rho}.\]
So we add 
\[ \FF_2=dual \circ ui_{1,2} \circ dual (\FF)= \bordermatrix{
& 0& 1&2 \cr
& \ominus & \oplus &  \cr
& &\lhd & \rhd\cr
}_{\rho} \]
 in $A$ and $C$. Note that $\FF_2=ui_{1,2}^{-1}(\FF)$. Nothing happens in step 3. Since $C$ is non-empty, we let $B=C$, $C= \emptyset$ and go back to step 2.
 
 In the second loop, nothing happens in step 2, but in step 3, Corollary \ref{cor split} is applicable to the first row of $\FF_2$. So we add 
\[ \FF_3= \bordermatrix{
& 0& 1&2 \cr
& \ominus &  &  \cr
&&\oplus& \cr
& &\lhd & \rhd\cr
}_{\rho} \]
in $A$ and $C$. One can check that nothing happens in the third loop and the procedure terminates. Therefore, we get
\[UI^{-1}(\FF)=\left\{\bordermatrix{
& 0& 1&2 \cr
& \ominus & \oplus & \ominus \cr
& &\ominus & \cr
}_{\rho},\  \bordermatrix{
& 0& 1&2 \cr
& \ominus & \oplus &  \cr
& &\lhd & \rhd\cr
}_{\rho},\ \bordermatrix{
& 0& 1&2 \cr
& \ominus &  &  \cr
&&\oplus& \cr
& &\lhd & \rhd\cr
}_{\rho}  \right\}. \]
\end{exmp}

\section{Theorem \ref{main thm intro}(1) and (2)}\label{sec main thm}
In this section, we prove Theorem \ref{main thm intro}(1) (see Propositions \ref{prop basic operators} and \ref{prop partial dual} below) and (2) (see Theorems \ref{thm integer} and \ref{thm half integer} below). 
We separate into two cases: (i) integer case, i.e.,  $a_i+b_i$ being even; (ii) half integer case, i.e.,  $a_i+b_i$ being odd.  
To make the statements more concise, we give the following definition.

\begin{defn}\label{def basic operators}
We call the operators $R_k$, $ui_{i,j}$,  $dual \circ ui_{i,j} \circ dual$ and their inverses the basic operators. 
\end{defn}

Basic operators preserve representations as mentioned in Theorem \ref{main thm intro}(1).  

\begin{prop}\label{prop basic operators}
Let $\EE \in \Rep$ and $T$ be any one of the basic operators. Then we have 
$$\pi(\EE) \cong \pi(T(\EE)).$$ 
\end{prop}
\begin{proof}
It follows from Theorems \ref{thm row exchange}, \ref{thm Aubert-Zelevinsky dual formula} and Proposition \ref{prop nonzero of ui}.
\end{proof}

We will use the following immediate consequence of Lemma \ref{lem $L$-data}.
\begin{cor}\label{cor symmetric support}
Suppose $\pi(\EE)= \pi(\EE') \neq 0$. If the symmetric difference $\Omega(\EE_{\rho})\Delta \Omega(\EE_{\rho}')$ doesn't contain an element of the form $\rho|\cdot|^{x}$ with $x \leq -1/2$, then $\Omega(\EE_{\rho})=\Omega(\EE_{\rho}')$.
\end{cor}

\begin{proof}
Let $\pi= \pi(\EE) = \pi(\EE')$. This follows from
\[\Omega(\EE_{\rho}) \Delta \Omega(\EE_{\rho}')= (\Omega(\EE_{\rho})\setminus \Omega(\pi)_{\rho})\Delta (\Omega(\EE_{\rho}')\setminus \Omega(\pi)_{\rho})\]
and Lemma \ref{lem $L$-data}(ii).
\end{proof}

\subsection{Integer case}
In this subsection, we assume 
$$\FF=\{([A_i,B_i]_{\rho},l_i,\eta_i)\}_{i\in (I_{\rho},>)} \ \mathrm{and} \ \FF'=\{ ([A_j',B_j']_{\rho},l_j',\eta_j')\}_{j \in (I_{\rho}',>')}$$ are such that
\begin{enumerate}
    \item [$\oldbullet$] $\pi( \EE^{\rho} \cup \FF) \cong \pi(\EE^{\rho} \cup \FF') \neq 0,$
    \item [$\oldbullet$] $A_1 \in \Z$.
\end{enumerate}

We first prove Theorem \ref{main thm intro}(2) in the integer case.

\begin{thm}\label{thm integer}
In the setting above, $\FF'$ can be obtained from $\FF$ by applying a sequence of basic operators. 
\end{thm}
\begin{proof} We may assume both $\FF$ and $\FF'$ are minimal. We decompose them into
\[ \FF= \FF_{\leq 0} + \FF_{>0} ,\ \FF'= \FF_{\leq 0}' + \FF_{>0}'. \]
By Corollary \ref{cor minimal}(i), we have $\FF_{>0}= \FF_{>0}'$ after row exchanges.


Next, we compare $\FF_{\leq 0}$ and $\FF_{\leq 0}'$ further. To avoid complications from the positive parts, we consider $ \FF_t=\FF_{\leq 0} + sh^{t}(\FF_{>0}) , \FF_t'=\FF_{\leq 0}' + sh^{t}(\FF_{>0}') $ for $t \gg0$. Corollary \ref{cor shift add}(i) implies 
\[  \pi(\EE^{\rho} \cup \FF_t)  \cong \pi(\EE^{\rho} \cup\FF_t'). \]
Then we apply the Aubert-Zelevinsky involution to swap the positive and non-positive parts to negative and non-negative parts. Write
\begin{align*}
    \widetilde{\FF_t}:= dual(\FF_t)&= \widetilde{sh^t(\FF_{>0})}+ \widetilde{\FF_{\leq 0}},\\
    \widetilde{\FF_t'}:= dual(\FF_t')&= \widetilde{sh^t(\FF_{>0}')} + \widetilde{\FF_{\leq 0}'}.
\end{align*}
where $ \widetilde{sh^t(\FF_{>0})}= (dual(\FF_t))_{<0}$, $\widetilde{\FF_{\leq 0}}= (dual(\FF_t))_{\geq 0} $ and similar for $\FF'$. Theorem \ref{thm Aubert-Zelevinsky dual formula} shows 
\[ \pi(dual(\EE^{\rho})\cup( \widetilde{sh^t(\FF_{>0})} + \widetilde{\FF_{\leq 0}})) \cong \pi(dual(\EE^{\rho}) \cup (\widetilde{sh^t(\FF_{>0}')} + \widetilde{\FF_{\leq 0}'})). \]

We choose $t$ large enough such that if $[A_{i}+t, -B_{i}-t]_{\rho}$ (resp. $[A_{j},-B_{j}]_{\rho}$) is a support of a row in $\widetilde{sh^t(\FF_{>0})}$ (resp. $\widetilde{\FF_{\leq 0}}$), we have 
\[ [A_{i}+t, -B_{i}-t]_{\rho} \supset [A_j,-B_j]_{\rho},\]
so that $ ui_{i,j}$ is not applicable on $\widetilde{sh^t(\FF_{>0})} + \widetilde{\FF_{\leq 0}} $, and similarly for $\FF'$. Under this assumption, we have
\begin{align*}
    (\widetilde{\FF_t} )^{min}&=\left(\widetilde{sh^t(\FF_{>0})}\right)^{min} +  \left(\widetilde{\FF_{\leq 0}}  \right)^{min},\\
    (\widetilde{\FF_t'})^{min}&=\left(\widetilde{sh^t(\FF_{>0}')}\right)^{min} + \left(\widetilde{\FF_{\leq 0}'}  \right)^{min}.
\end{align*}  
We further decompose
\begin{align*}
     \left(\widetilde{\FF_{\leq 0}}  \right)^{min}&= \widetilde{\FF_{=0}} + \widetilde{\FF_{<0}},\\
      \left(\widetilde{\FF_{\leq 0}'}  \right)^{min}&= \widetilde{\FF_{=0}'} + \widetilde{\FF_{<0}'},
\end{align*}
where $ \widetilde{\FF_{=0}}= \left(\left(\widetilde{\FF_{\leq 0}}  \right)^{min}\right)_{=0}$, and similar for $\FF'$. Despite the notation, we note that $\widetilde{\FF_{<0}}$ is not necessarily $dual(\FF_t)_{>0}.$ By Corollary \ref{cor minimal}(i) again, we obtain $\widetilde{\FF_{<0}}= \widetilde{\FF_{<0}'}$. Taking the Aubert-Zelevinsky involution again, we write
\begin{align*}
   \widetilde{\widetilde{\FF_t}}:= dual (\widetilde{sh^t(\FF_{>0})}+ \widetilde{\FF_{=0}} + \widetilde{\FF_{<0}}  )= \widetilde{\widetilde{ \FF_{<0}}} + \widetilde{\widetilde{ \FF_{=0}}} + sh^{t}(\FF_{>0}),\\
    \widetilde{\widetilde{\FF_t'}}:=  dual (\widetilde{sh^t(\FF_{>0}')}+ \widetilde{\FF_{=0}'} + \widetilde{\FF_{<0}'}  )= \widetilde{\widetilde{ \FF_{<0}}'} + \widetilde{\widetilde{ \FF_{=0}}'} + sh^{t}(\FF_{>0}'),
\end{align*}
where
\[ \left( \widetilde{\widetilde{\FF_t}}\right)_{<0}=\widetilde{\widetilde{ \FF_{<0}}}, \ \left( \widetilde{\widetilde{\FF_t}}\right)_{=0}=\widetilde{\widetilde{ \FF_{=0}}}, \]
and similar for $\FF'$. Note that $\widetilde{\widetilde{\FF_t}}$ is not necessarily $dual(dual(\FF_t))=\FF_t$; however, they are related by a sequence of basic operators. Now we may apply Corollary \ref{cor shift add}(ii) to obtain
\[  \pi(\EE^{\rho} \cup (\widetilde{\widetilde{ \FF_{<0}}} + \widetilde{\widetilde{ \FF_{=0}}} +\FF_{>0})) \cong \pi(\EE^{\rho} \cup (\widetilde{\widetilde{ \FF_{<0}}'} + \widetilde{\widetilde{ \FF_{=0}}'} + \FF_{>0}')) \cong \pi(\EE^\rho \cup \FF). \]

We may replace $\FF$ with $\widetilde{\widetilde{ \FF_{<0}}} + \widetilde{\widetilde{ \FF_{=0}}} +\FF_{>0}$ (resp. $\FF'$ with $\widetilde{\widetilde{ \FF_{<0}}'} + \widetilde{\widetilde{ \FF_{=0}}'} + \FF_{>0}'$) since the latter is obtained from the former by a sequence of basic operators. Therefore, we reduce to the case that
\[ \FF= \FF_{<0}+ \FF_{=0} + \FF_{>0}, \  \FF'= \FF_{<0}' + \FF_{=0}' + \FF_{>0}' \]
such that
\begin{enumerate}
    \item [$\oldbullet$]$\FF_{>0}= \FF_{>0}'$,
    \item [$\oldbullet$]$]supp(\FF_{<0}) = ]supp(\FF_{<0}')$.
\end{enumerate}

We claim that $\supp(\FF_{=0})=\supp(\FF_{=0}')$ up to row exchanges. Indeed, it is equivalent to check
\[ \Omega(  \FF_{<0} + \FF_{=0} + \FF_{>0}) =  \Omega(\FF_{<0}'+ \FF_{=0}' + \FF_{>0}' ). \]
However,
\[ \Omega(  \FF_{<0} + \FF_{=0} + \FF_{>0}) \Delta \Omega(\FF_{<0}'+ \FF_{=0}' + \FF_{>0}' ) = \Omega(\FF_{=0}) \Delta \Omega(\FF_{=0}') \]
which must be empty by Corollary \ref{cor symmetric support}. This verifies the claim.

The claim shows $\supp(\FF)=\supp(\FF')$ as multi-sets. Then Lemma \ref{lem Moeglin} shows that $\FF=\FF'$ after row exchanges. This completes the proof of the theorem.
\end{proof}

\subsection{Half integer case}
In this subsection, we  consider the half integer case, that is $\EE_{\rho}=\FF=\{ ([A_i,B_i]_{\rho}, l_i, \eta_i)\}_{i \in (I_{\rho}, >)}$ where $A_i \in \Z +\half{1}$. The same conclusion in Theorem \ref{thm integer} does not hold in this case. A counter example is 
\begin{align*}
    \EE_1=\bordermatrix{
    & \half{-1} & \half{1}&\half{3} & \half{5} \cr
    & &\ominus &\oplus  &\ominus \cr
    }_{\rho}\ ,\ \EE_2=\bordermatrix{
    & \half{-1} & \half{1}&\half{3} & \half{5} \cr
    & \oplus &\ominus &\oplus  &\ominus \cr
    }_{\rho}
\end{align*}
$\EE_1$ and $\EE_2$ are not related by any combination of basic operators; however,
\[ \pi(\EE_1) \cong \pi(\EE_2)= \pi\left( \left(\half{1}\right)^{-},\left(\half{3}\right)^{+},\left(\half{5}\right)^{-}\right).\]

In the following, we define a new operator $dual_k$, which we call the partial dual, and show that if $\pi(\EE')\cong\pi(\EE),$ then $\EE'_\rho$ can be obtained from $\EE_\rho$ by a sequence of basic operators and at most one $dual_k$.

We continue to examine the example $\EE_1,\EE_2$ above. Note that $\EE_2=dual(\EE_1)$. Since the associated representation is supercuspidal (see \cite[Theorem 2.5.1]{Moe11}),  it is self-Aubert-Zelevinsky-dual, and hence $\pi(\EE_1)=\pi(dual(\EE_1))$ by Theorem \ref{thm Aubert-Zelevinsky dual formula}. We shall interpret the supercuspidal condition as the extreme case of Theorem \ref{thm non-vanishing}(i) along with $l_i=0$.

We observe a similar pattern in the following example. Let  
\begin{align*}
\EE_3= \begin{blockarray}{cccccc}
\frac{-5}{2}&\frac{-3}{2}&\frac{-1}{2}&\frac{1}{2}&\half{3}&\half{5}\\
\begin{block}{(cccccc)}
\lhd&\lhd&\lhd&\rhd&\rhd&\rhd\\
&&&\oplus&&\\
&&&&\oplus&\\\end{block}
\end{blockarray}\ , \ \EE_4= \begin{blockarray}{cccccc}
\frac{-5}{2}&\frac{-3}{2}&\frac{-1}{2}&\frac{1}{2}&\half{3}&\half{5}\\
\begin{block}{(cccccc)}
\lhd&\lhd&\oplus&\ominus&\rhd&\rhd\\
&&\ominus&\oplus&&\\
&&&&\oplus&\\\end{block}
\end{blockarray}.
\end{align*} 

We have $\pi(\EE_3)\cong \pi(\EE_4)$, and the second row of $\EE_3, \EE_4$ are both in the extreme case of Theorem \ref{thm non-vanishing}(i) and $l_2=0$. We see the third rows of $\EE_3$ and $\EE_4$ are identical, which is guaranteed by Corollary \ref{cor shift add}(i), but the first rows are different. On the other hand, we can take the dual for both of them.
\begin{align*}
dual(\EE_3)= \begin{blockarray}{ccccc}
\frac{-3}{2}&\frac{-1}{2}&\frac{1}{2}&\half{3}&\half{5}\\
\begin{block}{(ccccc)}
\lhd&\ominus&\oplus&\rhd&\\
&\ominus&\oplus&&\\
&&&&\oplus\\\end{block}
\end{blockarray}\ , \ dual(\EE_4)= \begin{blockarray}{ccccc}
\frac{-3}{2}&\frac{-1}{2}&\frac{1}{2}&\half{3}&\half{5}\\
\begin{block}{(ccccc)}
\lhd&\lhd&\rhd&\rhd&\\
&&\oplus&&\\
&&&&\oplus\\\end{block}
\end{blockarray}\ .
\end{align*} 
Then the first rows of $\EE_3$ and $\EE_4$ transfer to the third rows of $dual(\EE_3)$ and $dual(\EE_4)$ as dual reverses the order. Now the third rows of $dual(\EE_3)$ and $dual(\EE_4)$ are the same by Corollary \ref{cor shift add}(i). Based on this observation, we give the following definition. 

\begin{defn}\label{def partial dual}
Suppose $\EE\in \Rep^{(P')}$ and $\EE_{\rho}=\FF= \{([A_i,B_i]_{\rho},l_i,\eta_i)\}_{i \in (I_{\rho},>)}$. For $i \in I_{\rho}$, denote 
\[\alpha_i= \sum_{ j<i} (A_j+B_j+1) ,\ \beta_i= \sum_{j> i} (A_j+B_j+1).\]
Suppose there exists $k \in I_{\rho}$ such that
\begin{enumerate}
    \item [1.] $B_k=1/2,l_k=0,$
    \item [2.] $(-1)^{\alpha_k}\eta_k= -1$,
    \item [3.] for any $i < k$, $B_i < 1/2$.
\end{enumerate}
Then we define $dual_k^{+}(\FF)$ as follows. We write the decomposition
\[ \FF= \FF_1 + \{([A_k,1/2]_{\rho},0,\eta_k)\} + \FF_2,\]
where $\FF_1=\FF_{<1/2}$, and 
\[dual(\FF)= \widetilde{\FF_2} + \{([A_k,-1/2]_{\rho},0,(-1)^{\beta_k}) \}+ \widetilde{\FF_1},\]
where $\widetilde{\FF_1}=(dual(\FF))_{>-1/2}$. Finally, write
\[ dual(\widetilde{\FF_2} + \{([A_k,1/2]_{\rho},0,(-1)^{\beta_k+1})\} + \widetilde{\FF_1})= \widetilde{\widetilde{\FF_1}} + \{([A_k,-1/2]_{\rho},0,-\eta_k)\}+ \widetilde{\widetilde{\FF_2}},\]
where $\widetilde{\widetilde{\FF_2}}=(dual(\widetilde{\FF_2} + \{([A_k,1/2]_{\rho},0,(-1)^{\beta_k+1})\} + \widetilde{\FF_1}))_{>-1/2}$. Then we define
\[ dual_k^{+}( \FF)= \widetilde{\widetilde{\FF_1}} + \{([A_k,-1/2]_{\rho},0,-\eta_k)\} + \FF_2, \]
and say $dual_k^{+}$ is applicable on $\FF$.

Suppose $dual(\FF)$ satisfies above condition, then we define
\[ dual_k^{-}(\FF)= dual \circ dual_{k}^{+} \circ dual (\FF), \]
and say $dual_k^{-}$ is applicable on $\FF$.

We  call this operator partial dual, and use $dual_k$ to denote $dual_k^{+}$ or $dual_k^{-}$ if it is clear from the context.

Finally, we define $dual_k(\EE)= \EE^{\rho} \cup dual_k(\EE_{\rho}).$ 
\end{defn}

By the definition of $dual$, one can show that $dual_k^{-}$ is applicable on $\FF$ if and only if $B_k=-1/2$, $l_k=0$ and $B_{j} >-1/2$ for all $j >k$.
\begin{remark}
We remark that if $dual_k$ is applicable on $\EE$ and the local Arthur parameter $\psi$ associated with $\EE$ is elementary (see \cite[Section 5]{Xu17b}), then the local Arthur parameter associated with $dual_k(\EE)$ is $\psi^{\sharp}$ in \cite[Theorem 6.10]{Xu17b}, which corresponds to the change of $\psi$ under the generalized Aubert-Zelevinsky involution $inv_{<3}$ defined by Mœglin \cite{Moe06b}. 
\end{remark}

\begin{exmp}
Let $\rho$ be the trivial representation and
$$
\psi=\rho\otimes S_1\otimes S_6+\rho\otimes S_3\otimes S_2+ \rho\otimes S_7\otimes S_2
$$
be a local Arthur parameter of good parity for $G_n=\SO_{27}(F).$ Let 
$$
\EE=\bordermatrix{
 &\half{-5} &\half{-3}&\half{-1}&\half{1} &\half{3}&\half{5}&\half{7} \cr
 &\lhd&\lhd&\lhd&\rhd&\rhd&\rhd \cr
& & & & \oplus & \ominus & &  \cr
& & & & & & \oplus & \ominus  \cr
}_{\rho}.
$$
We have $\pi(\EE)=L(\Delta_\rho[\half{-5},\half{-5}],\Delta_\rho[\half{-3},\half{-3}],\Delta_\rho[\half{-1},\half{-1}];\pi(\half{1} ^+,\half{3}^-,\half{5}^+,\half{7}^-))\in\Pi_\psi.$ We see that this fits the situation of $dual_k^+.$ We compute it now. Let
$$
\FF_1=\bordermatrix{
 &\half{-5} &\half{-3}&\half{-1}&\half{1} &\half{3}&\half{5}&\half{7} \cr
 &\lhd&\lhd&\lhd&\rhd&\rhd&\rhd \cr
}_{\rho}
$$
and
$$
\FF_2=\bordermatrix{
 &\half{-5} &\half{-3}&\half{-1}&\half{1} &\half{3}&\half{5}&\half{7} \cr
& & & & & & \oplus & \ominus  \cr
}_{\rho}.
$$
We have
$$
dual(\EE)=\bordermatrix{
 &\half{-5} &\half{-3}&\half{-1}&\half{1} &\half{3}&\half{5}&\half{7} \cr
 &\lhd&\lhd&\lhd&\ominus &\rhd&\rhd&\rhd \cr
& & & \oplus & \ominus & \oplus & &  \cr
& & & & & & \oplus &   \cr
}_{\rho}.
$$
Then
$$
\widetilde{\FF}_1=\bordermatrix{
 &\half{-5} &\half{-3}&\half{-1}&\half{1} &\half{3}&\half{5}&\half{7} \cr
& & & & & & \oplus &   \cr
}_{\rho}
$$
and
$$
\widetilde{\FF}_2=\bordermatrix{
 &\half{-5} &\half{-3}&\half{-1}&\half{1} &\half{3}&\half{5}&\half{7} \cr
 &\lhd&\lhd&\lhd&\ominus &\rhd&\rhd&\rhd \cr
}_{\rho}.
$$ So
$$
\EE':=\widetilde{\FF}_2+\left\{\left(\left[\half{3},\half{1}\right]_{\rho},0,1\right)\right\}+\widetilde{\FF}_1=\bordermatrix{
 &\half{-5} &\half{-3}&\half{-1}&\half{1} &\half{3}&\half{5}&\half{7} \cr
 &\lhd&\lhd&\lhd&\ominus &\rhd&\rhd&\rhd \cr
& & & & \oplus & \ominus & &  \cr
& & & & & & \oplus &   \cr
}_{\rho}.
$$
We have 
$$
dual(\EE')=\bordermatrix{
 &\half{-5} &\half{-3}&\half{-1}&\half{1} &\half{3}&\half{5}&\half{7} \cr
 &\lhd&\lhd&\oplus&\ominus &\rhd&\rhd& \cr
& & & \lhd & \oplus & \rhd & &  \cr
& & & & & & \ominus & \oplus   \cr
}_{\rho}.
$$
Hence we see that 
$$
\widetilde{\widetilde{\FF_1}}=\bordermatrix{
 &\half{-5} &\half{-3}&\half{-1}&\half{1} &\half{3}&\half{5}&\half{7} \cr
 &\lhd&\lhd&\oplus&\ominus &\rhd&\rhd& \cr
}_{\rho}.
$$
Therefore,
$$
dual_2^+(\EE)=\widetilde{\widetilde{\FF_1}}+\left\{\left(\left[\half{3},\half{-1}\right]_{\rho},0,-1\right)\right\}+\FF_2=\bordermatrix{
 &\half{-5} &\half{-3}&\half{-1}&\half{1} &\half{3}&\half{5}&\half{7} \cr
 &\lhd&\lhd&\oplus&\ominus &\rhd&\rhd& \cr
& & & \ominus & \oplus & \ominus & &  \cr
& & & & & & \oplus & \ominus   \cr
}_{\rho}.
$$
One can check that $\pi(dual_2^+(\EE))\cong\pi(\EE).$
\end{exmp}

We show that the isomorphism of representations in the previous example is not a coincidence. That is, this operator also preserves representations as follows. This completes the proof of Theorem \ref{main thm intro}(1). 

\begin{prop} \label{prop partial dual}
Suppose $\EE \in \Rep$ and $\EE_{\rho}=\FF= \{([A_i,B_i]_{\rho},l_i,\eta_i)\}_{i \in (I_{\rho},>)}$. Suppose $dual_k$ is applicable on $\EE$ for some $k \in I_{\rho}$, then 
\[ \pi(dual_k(\EE))\cong \pi(\EE).\]
\end{prop}
\begin{proof}
By Theorem \ref{thm Aubert-Zelevinsky dual formula}, we only need to deal with the case that $dual_k^{+}$ is applicable on $\FF$. We keep the same notation as in the definition. When both $\FF_1$ and $\FF_2$ are empty, we have
\begin{align*}
    \FF=&\{([A_1, 1/2]_{\rho},0,-1)\},\\
    dual_k(\FF)=&\{([A_1, -1/2]_{\rho},0,1)\}.
\end{align*}
The conclusion follows from direct computation in this case. We reduce the general case to this case in two steps.

Step 1: We reduce $\FF_2$ to be  $\emptyset$. Note that by construction,
    \begin{align*}
        \FF&=\FF_1 + \{([A_k,1/2]_{\rho},0,\eta_k)\} + \FF_2\\
        dual_k(\FF)&= \widetilde{\widetilde{\FF_1}}+ \{([A_k,-1/2]_{\rho},0,-\eta_k)\} + \FF_2.
    \end{align*}
    As both of them satisfy ($P'$), Corollary \ref{cor shift add}(i) and (ii) implies the statement
 \begin{align*}
    0\neq &\pi(\EE^{\rho} \cup(\widetilde{\widetilde{\FF_1}}+ \{([A_k,-1/2]_{\rho},0,-\eta_k)\} + \FF_2))\\  \cong& \pi( \EE^{\rho} \cup(\FF_1 + \{([A_k,1/2]_{\rho},0,\eta_k)\} + \FF_2)
\end{align*}
is equivalent to the statement
\begin{align*} 
    0 \neq& \pi(\EE^{\rho} \cup (\widetilde{\widetilde{\FF_1}}+ \{([A_k,-1/2]_{\rho},0,-\eta_k)\} + sh^{d}(\FF_2))) \\
    \cong& \pi(\EE^{\rho} \cup(\FF_1 + \{([A_k,1/2]_{\rho},0,\eta_k)\} + sh^{d}(\FF_2))),
\end{align*}
 for any $d \in \N$. On the other hand, when $d$ is sufficiently large, Lemma \ref{lemma far away}(ii) shows that the previous statement is equivalent to the following statement  
\begin{align*}
    0& \neq \pi(\EE^{\rho }\cup (\widetilde{\widetilde{\FF_1}}+ \{([A_k,-1/2]_{\rho},0,-\eta_k)\}) \cup (sh^{d}(\FF_2))_{\rho^{\ast}} )\\
    &\cong \pi( \EE^{\rho} \cup (\FF_1 + \{([A_k,1/2]_{\rho},0,\eta_k)\})\cup (sh^{d}(\FF_2))_{\rho^{\ast}}).
\end{align*}
This completes the reduction in this step.

Step 2: We reduce $\FF_1$ to be $\emptyset$. Theorem \ref{thm Aubert-Zelevinsky dual formula} says 
\[ 0 \neq \pi(\EE^{\rho} \cup dual_k(\FF)) \cong \pi(\EE^{\rho} \cup \FF) \Longleftrightarrow 0 \neq \pi(dual \circ dual_k(\EE^{\rho} \cup\FF)) \cong \pi(dual(\EE^{\rho} \cup\FF)). \]
As we have assumed $\FF_2= \emptyset$, we have
\begin{align*}
    dual \circ dual_k(\FF)& =\{([A_k,1/2]_{\rho},0,-1)\} +  \widetilde{\FF_1},\\
    dual(\FF)&=\{([A_k,-1/2]_{\rho},0,1)\} + \widetilde{\FF_1}.
\end{align*}
Then we repeat step 1 with $\FF_2=\widetilde{\FF_1}$ now. Then, we have reduced back to the case $\FF=\{([A_1, 1/2]_{\rho},0,-1)\}$. This completes the proof of the proposition.
\end{proof}

 Next, we describe how $dual_k^+$ behaves on the local Arthur packets. The result is analogous to Theorem \ref{thm ui packet}. 

\begin{thm} \label{thm applicability of dual_k}
Suppose $\EE\in \Rep^{(P')}$, and write 
\[ \EE_{\rho}=\{ ([A_i,B_i]_{\rho},l_i,\eta_i)\}_{i \in (I_{\rho}, >)}.\]
Suppose there exists $k \in I_{\rho}$ such that $B_k= 1/2$ and for any $i < k$, $B_i<1/2$. Let $\psi= \bigoplus_{\rho} \bigoplus_{i \in I_{\rho}} \rho \otimes S_{a_i} \otimes  S_{b_i}$ be the local Arthur parameter associated with $\EE$. Denote $dual_k(\psi)$ the local Arthur parameter obtained from $\psi$ by replacing the summand $\rho \otimes S_{a_k} \otimes S_{b_k}$ by $ \rho \otimes S_{b_k} \otimes S_{a_k}$.

Then $dual_k^{+}$ is applicable on $\EE$ if and only if $\pi(\EE) \in \Pi_{\psi} \cap \Pi_{dual_k(\psi)}$. In this case, $dual_k^{+}(\EE)$ is the unique extended multi-segment with $\pi(\EE)\cong \pi(dual_k(\EE))$ and $\supp(dual_k^{+}(\EE))=\supp(dual_k(\psi))$. In particular, $dual_k \circ dual_k=id$.
\end{thm}
\begin{proof}
The forward direction follows from the previous proposition. We prove the backward direction now.

Let $\EE'$ be the multi-segment with associated local Arthur parameter $ dual_k(\psi)$ such that $\pi(\EE)\cong \pi(\EE')$. We write 
\[ \EE_{\rho}= \FF_1 + \{([A_k,1/2]_{\rho},l_k, \eta_k)\} + \FF_2,\  \EE_{\rho}'= \FF_1' + \{([A_k,-1/2]_{\rho},l_k', \eta_k') \}+ \FF_2'.\]
By Corollary \ref{cor shift add}(i) and (iii), we may fix $d \gg0$ and replace $\FF_1,\FF_1', \FF_2,\FF_2'$ by
\[add^d(\FF_1),add^d(\FF_1'),sh^{d}(\FF_2), sh^{d}(\FF_2')\]
respectively. When $d$ is large enough, the conditions 
\begin{enumerate}
    \item [1.] $l_k=0$,
    \item [2.] $ (-1)^{\alpha_k}\eta_k=-1$
\end{enumerate}
are equivalent to $ \pi( sh_k^{-1} (\EE)) = 0$. Therefore, supposing $dual_k$ is not applicable on $\EE$, Lemma  \ref{lem shift}(ii) shows
\[ 0 \neq D_{\rho|\cdot|^{1/2,\dots,A_k}}(\pi(\EE)) \cong D_{\rho|\cdot|^{1/2,\dots,A_k}}(\pi(\EE')). \]
In particular, $D_{\rho|\cdot|^{1/2}}(\pi(\EE')) \neq 0$. This contradicts to Proposition \ref{prop derivative support}(i). This completes the proof of the theorem.
\end{proof}

\begin{remark}\label{rmk partial dual}
When $dual_k^{+}$ is applicable on $\EE$, Atobe showed in step 4 of \cite[Algorithm 5.5]{Ato20b} that after applying his operator P to formally obtain 
\[\EE'=\{([-1/2,-1/2]_{\rho},0,1)\}\cup \EE,\]
$ui_{0,k+1}$ is applicable on $\EE'$ of type 3', where $0$ denotes the index of the phantom row we just added. He denoted the resulting extended multi-segment $\EE^{\ast}$, and claimed $\pi(\EE) \cong \pi(\EE^{\ast})$. 

Since the local Arthur parameter associated with $ \EE^{\ast}$ is exactly $dual_k(\psi)$. Our theorem above shows that $\EE^{\ast}=dual_k(\EE)$, which is not immediate from the two different definitions. In other words, our operator $dual_k$ can be realized as a composition of $R_i$, $ui_i$ and P and their inverses.
\end{remark}

We show that partial dual commutes with the basic operators in the following sense:
\begin{lemma}\label{lem commutativity of partial dual} \ 
\begin{enumerate}
    \item [(a)] Suppose $\EE\in \Rep^{(P')}$ and write
    \[ \EE_{\rho}=\{([A_r,B_r]_{\rho},l_r,\eta_r)\}_{r \in (I_{\rho,>})}. \]
    Suppose $ui_{i,j}$ is applicable on $\EE$ for some $i,j \in I_{\rho}$. Write 
    \[ ui_{i,j}(\EE_{\rho})= \{ ([A_r',B_r']_{\rho},l_r',\eta_r')\}_{r \in (I_{\rho},>)}.\]
    Under the assumption $B_k=-1/2$, we have $l_k=0$ if and only if $l_k'=0$. In other words, $dual_k^{-}$ is applicable on $\EE$ if and only if it is applicable on $ui_{i,j}(\EE)$.
    \item [(b)] Suppose $dual_k^{+}$ is applicable on $\EE$. Then
    \begin{enumerate}
    \item [(1)]$ui_{i,j}$ commutes with $dual_k^{+}$, i.e. 
\[ ui_{i,j} (dual_k^{+}(\EE))=  dual_k^{+}(ui_{i,j}(\EE)),  \]
unless $i=k$ and $B_j=1/2$.
\item [(2)] If $ui_{k,j}$ is applicable on $dual_k^{+}(\EE)$ with $B_j=1/2$, then we have 
\[ ui_{k,j} (dual_k^{+}(\EE_{\rho}))=  dual_k^{+}(\EE_{\rho,\gg})  \]
for some admissible order $\gg$ on $I_{\rho}$.
\end{enumerate}
\end{enumerate}
\end{lemma}
\begin{proof}
We first prove (a) combinatorially. We separate into three cases.

Case 1. Suppose $j=k$. Then we may reduce to the case that $I_{\rho}=\{1<2\}$ and $(i,j,k)=(1,2,2)$. If $l_2=0$, then the non-vanishing conditions force that $B_1+l_1 = -1/2$ and $ui_{1,2}$ is of type 2 in Definition \ref{def ui}. In this case we see $l_{2}'=l_2=0$ by definition. If $l_2'=0$, then the non-vanishing condition forces $(-1)^{A_1'-B_1'}\eta_1'\eta_{2}'=-1$. Then $l_2=0$ by the definition of $ui$.

Case 2. Suppose $i=k$ and $ B_{j}>1/2$. Then we may reduce to the case that $I_{\rho}=\{1<2\}$ and $(i,j,k)=(1,2,1)$. If $l_1=0$, then $ui_{1,2}$ is of type 3 or 3' in Definition \ref{def ui} and $l_1'=l_1=0$ by definition. If $l_1'=0$, then the non-vanishing condition forces $l_2'=0$ and $(-1)^{A_1'-B_1'}\eta_1'\eta_{2}'=1$. Then $l_1=0$ by the definition of $ui$.

Case 3. Suppose $k \neq i,j$. Then we may reduce to the case that $I_{\rho}=\{1<2<3\}$. If $(i,j,k)=(1,2,3)$ or $(2,3,1)$, then nothing to prove. Suppose $(i,j,k)=(1,3,2)$. In this case, in order that $ui_{1,3}$ is applicable, we have either $[A_1,B_1]_{\rho} \supseteq [A_2,B_2]_{\rho}$ or $[A_2,B_2]_{\rho} \supseteq [A_3,B_3]_{\rho}$. If $[A_1,B_1]_{\rho} \supseteq [A_2,B_2]_{\rho}$, then the assertion is clear by the definition of row exchange. If $[A_2,B_2]_{\rho} \supseteq [A_3,B_3]_{\rho}$, then $A_1 \leq A_2, B_1 \leq B_2$. The non-vanishing conditions in Proposition \ref{prop positive non-vanishing}(i) force $B_1+l_1=-1/2$ and $l_3=0$. Then $ui_{1,3}$ is not applicable.

Now we prove Part (b). For (1), if both of the $ui$ in the equation are not applicable, then there is nothing to prove. Suppose $ui_{i,j}$ is applicable on $dual_k(\EE)$, then Part (a) implies $dual_k^{-}$ is applicable on $ ui_{i,j}(dual_k(\EE))$. It is verified directly that
\[ \supp( dual_k \circ ui_{i,j} \circ dual_k(\psi))= \supp(ui_{i,j}(\psi)),\]
where $\psi$ is the local Arthur parameter associated with $\EE$. (Note that if $i=k$, the assumption $B_k>1/2$ is used in this computation.) Therefore, the conclusion follows from Theorem \ref{thm ui packet}.

On the other hand, suppose $ui_{i,j}$ is applicable on $\EE$. It suffices to show that $dual_k^{+}$ is applicable on $ui_{i,j}(\EE)$, or equivalently, $dual_{k}^{-}$ is applicable on $dual(ui_{i,j}(\EE))$. If $ui_{i,j}(\EE)$ is of type 3', then it is clear that $dual_k^{+}$ is still applicable on $ui_{i,j}(\EE)$. If $ui_{i,j}(\EE)$ is not of type 3', then $dual_k^{-}$ is applicable on 
\[ ui_{j,i} \circ dual\circ ui_{i,j}(\EE)= dual(\EE)\]
by Corollary \ref{cor ui inverse}. Then Part (a) implies $dual_{k}^{-}$ is applicable on $ dual(ui_{i,j})(\EE)$. This completes the proof of (1).

For (2), identify $ I_{\rho}= \{ 1 ,\dots ,n\}$ with $1 < \cdots <n$. Then we define a total order $\gg $ on $I_{\rho}$ by
\[1 \ll \cdots \ll k-1 \ll j \ll k+1 \ll \cdots \ll j-1 \ll k \ll j+1 \ll \cdots \ll n. \]
This is an admissible order for $\EE_{\rho}$ since $ui_{k,j}$ is applicable on $ dual_k^{+}( \EE_{\rho})$. Then by comparing the support, Theorem \ref{thm applicability of dual_k} implies $dual_k^{+}$ is applicable on $\EE_{\rho,\gg}$ and the equality holds. This completes the proof of the lemma.
\end{proof}
 As a consequence of the previous lemma, we show that if $dual_k^+$ is applicable on an extended multi-segment after a composition of basic operators, then there exists an admissible order of the original extended multi-segment for which we can apply $dual_k^+$ directly. This corollary is used as a reduction step in the arguments of later sections.
\begin{cor}\label{cor partial dual is always applicable}
Suppose $\EE\in\Rep$ and denote $\FF=\EE_{\rho}$. Suppose $dual_k^{+}$ is applicable on $\FF$, and $\FF'$ is obtained from $\FF$ by a sequence of basic operators. Then there exists an admissible $\gg$ order of $\FF'$ and $k'$ such that $dual_{k'}$ is applicable $(\FF')_{\gg}$.
\end{cor}
\begin{proof}
It suffices to show that if $\FF'=dual \circ ui_{i,j} \circ dual(\FF)$ or $\FF'=ui_{i,j}(\FF)$, then $dual_k^{+}$ is applicable on $\FF'$ if and only if $dual_k^{+}$ is applicable on $\FF$.

Suppose $\FF'=dual \circ ui_{i,j} \circ dual(\FF)$. Then \[dual(\FF')=ui_{i,j} \circ dual(\FF) .\] 
We know $dual_k^{+}$ is applicable on $\FF'$ (resp. $\FF$) if and only if $dual_{k}^{-}$ is applicable on $dual(\FF')$ (resp. $dual(\FF)$). On the other hand, Lemma \ref{lem commutativity of partial dual}(a) implies $dual_{k}^{-}$ is applicable on $dual(\FF')$ if and only if the same holds on $dual(\FF)$, and hence we are done.

Suppose $\FF' =ui_{i,j}(\FF)$. If this $ui$ is of type 3' in Definition \ref{def ui}, the assertion is clear. If it is not of 3', then we have 
\[ ui_{j,i} \circ dual(\FF')= dual(\FF).  \]
Hence, the conclusion holds by the same reasoning as the previous case. This completes the proof of the corollary.
\end{proof}

 Before we prove the main theorem of this subsection, we need the following proposition which extends Corollary \ref{cor minimal}(i).

\begin{prop}\label{prop shift left 1/2}
Suppose $\EE \in \Rep^{(P')}$ and write
\[\FF= \EE_{\rho} =\{ ([A_i,B_i]_{\rho}, l_i, \eta_i)\}_{i=1}^n,\]
where we identify $(I_{\rho},>)$ with $\{1, \dots , n\}$ where $1 <\cdots <n$. We assume that $A_i \in \Z+ \half{1}$.
\begin{enumerate}
    \item [(i)]  $\EE^{\rho} \cup (\FF_{<1/2} + sh^{-1}(\FF)_{\geq 1/2})$ is in $\Rep$ if $\FF$ satisfies all of the following conditions.  
\begin{enumerate}
    \item [(a)] $\FF$ is minimal.
    \item [(b)] If we write $\FF_{=1/2}= \{ ([A_i,1/2]_{\rho},l_i,\eta_i)\}_{i=k}^{m}$, then
\[ A_k \leq \cdots\leq A_{m}. \]
 \item [(c)] $dual_k^{+}$ is not applicable on $\FF$, where $\FF_{<1/2}$ has $k-1$ rows.
\end{enumerate}
In particular, suppose $\pi(\EE^{\rho} \cup \FF)\cong \pi(\EE^{\rho} \cup \FF')$ and both $\FF$ and $\FF'$ satisfy above conditions, then $\FF_{\geq 1/2}= \FF_{\geq 1/2}'$.
\item [(ii)] Suppose $\FF$ satisfies conditions (a) and (b) in (i) but (c) fails. Then either $dual_k(\FF)$ is already minimal or  $(dual_k(\FF))^{min}= ui_{k,j}(dual_k(\FF))$. Moreover, $(dual_k(\FF))^{min}$ satisfies all of (a), (b) and (c) and $dual_{k}^{-}$ is applicable on it with
\[dual_{k}((dual_k(\FF))^{min})= \FF\]
up to row exchanges.
\end{enumerate}

\end{prop}
\begin{proof}
For Part (i), if $\EE^{\rho} \cup (\FF_{<1/2} + sh^{-1}(\FF)_{\geq 1/2})\in \Rep$, then the second assertion follows from the same proof of Corollary \ref{cor minimal}(i). Therefore, it suffices to show $\EE^{\rho} \cup (\FF_{<1/2} + sh^{-1}(\FF)_{\geq 1/2})\in \Rep$, and it remains to show $\FF_{<1/2} \cup sh^{-1}(\FF_{\geq 1/2})$ satisfies the condition ($\ast$) in Theorem \ref{thm non-vanishing}(i). This also follows from \cite[Algorithm 5.5]{Ato20b}. We give the details below for completeness.

Write $\FF_{<1/2} + \FF_{=1/2}= \{ ([A_i,B_i]_{\rho},l_i, \eta_i)\}_{i=1}^{k-1} +\{ ([A_i, 1/2]_{\rho},l_i,\eta_i)\}_{i=k}^{m}$. By our assumption and non-vanishing conditions in Proposition \ref{prop positive non-vanishing}(i)(2), there exists a unique integer $r \geq k-1$ such that 
\[\begin{cases}
l_i=0 &\text{for } k \leq i \leq r,\\
(-1)^{A_i-B_i}\eta_i \eta_{i+1}=1 &\text{for } k \leq i <i+1 \leq r,\\
l_i>0 &\text{for }  i> r.
\end{cases}\]
To see whether $\FF_{<1/2} + sh^{-1}(\FF_{=1/2})$ satisfies the condition ($\ast$) in Theorem \ref{thm non-vanishing}, we only have to check it for the $i$-th rows for $k \leq i \leq r$. We apply induction on these $i$.

Since $dual_k^{+}$ is not applicable on $\FF$, we see either $l_k>0$ so nothing needs to be checked ($r=k-1$), or the induction hypothesis holds for $i=k$.

Suppose $k \leq i <i+1 \leq r $ and the induction hypothesis holds for $i$. Then 
\[ \begin{cases}
 (-1)^{\alpha_i} \eta_i= 1,\\
 \alpha_{i+1}= \alpha_i+ A_i+B_i+1,\\
 \eta_{i+1}= (-1)^{A_i-B_i} \eta_i.
\end{cases}\]
Since $A_i , B_i \in \Z+1/2$, we see $(-1)^{A_i-B_i}= (-1)^{A_i-B_i+ (2B_i+1)}=(-1)^{A_i+B_i+1}$. So we have
\[ (-1)^{\alpha_{i+1}}\eta_{i+1}=1.\]
Therefore, the $(i+1)$-th row of $\FF_{<1/2} + sh^{-1}(\FF_{=1/2})$ also satisfies $(\ast)$ in Theorem \ref{thm non-vanishing}(i). This proves Part (i).

Part (ii) follows directly from Lemma \ref{lem commutativity of partial dual}(b). Indeed, Lemma \ref{lem commutativity of partial dual}(b)(1) implies that if $\FF$ is minimal and $dual_k^{+}$ is applicable on $\FF$, then no $ui_{i,j}$ with $i \neq k$ is applicable on $dual_k(\FF)$. So $(dual_k(\FF))^{min}$ is obtained from $dual_k(\FF)$ by a sequence of $ui_{k,j}$ with $B_j=1/2$. However, Lemma \ref{lem commutativity of partial dual}(b)(2) helps us to reduce the composition of $ui_{k,j}$ into a single one, and shows that $dual_{k}^{-}$ is applicable on $(dual_k^{+}(\FF))^{min}$ with 
\[dual_{k}^{-}((dual_k^{+}(\FF))^{min})= \FF' =\FF\]
up to a row exchange. This completes the proof of the proposition.
\end{proof}

With all of the tools in hand, we now prove Theorem \ref{main thm intro}(2) in the half integer case.

\begin{thm}\label{thm half integer}
Suppose 
$$\FF= \{ ([A_i,B_i]_{\rho},l_i,\eta_i)\}_{i \in (I_{\rho},>)},\FF'=\{ ([A_j',B_j']_{\rho},l_j',\eta_j')\}_{ j \in (I_{\rho}', >')}$$ satisfy
\begin{enumerate}
    \item [$\oldbullet$] $\pi(\EE^{\rho} \cup \FF) \cong \pi(\EE^{\rho} \cup \FF') \neq 0$, for some $\EE^{\rho}$ such that both
    $\EE^{\rho} \cup \FF$ are $\EE^{\rho} \cup \FF'$ are extended multi-segments,
    \item [$\oldbullet$] $A_1 \in \Z +\half{1}$.
\end{enumerate}
Then $\FF'$ can be obtained from $\FF$ by a sequence of basic operators and at most one $dual_k$. 
\end{thm}

\begin{proof}
By the same argument as Theorem \ref{thm integer}, we assume both $\FF$ and $\FF'$ are minimal, and decompose
\[ \FF= \FF_{\leq 1/2} + \FF_{>1/2},\ \FF= \FF_{\leq 1/2}' + \FF_{>1/2}'. \]
Then $\FF_{>1/2} =\FF_{>1/2}'$ after row exchanges.

Next, denote $\FF_t= \FF_{\leq 1/2} + sh^t(\FF_{>1/2})$ for $t$ sufficiently large, similar for $\FF'$. Then take dual of them to get 
\begin{align*}
    \widetilde{\FF_t}:= dual(\FF_t)&= \widetilde{sh^t(\FF_{>1/2})} + \widetilde{\FF_{\leq 1/2}},\\
    \widetilde{\FF_t'}:= dual(\FF_t')&= \widetilde{sh^t(\FF_{>1/2}')} + \widetilde{\FF_{\leq 1/2}'},
\end{align*}
where 
\[\widetilde{sh^t(\FF_{>1/2})}= (\widetilde{\FF_t})_{< -1/2},\  \widetilde{\FF_{\leq 1/2}}= (\widetilde{\FF_t})_{\geq -1/2}, \]
and similar for $\widetilde{\FF_{\leq 1/2}'}$. For $t$ sufficiently large, we have 
\begin{align*}
    (\widetilde{\FF_t} )^{min}&=\left(\widetilde{sh^t(\FF_{>1/2})}\right)^{min} +  \left(\widetilde{\FF_{\leq 1/2}}  \right)^{min},\\
    (\widetilde{\FF_t'})^{min}&=\left(\widetilde{sh^t(\FF_{>1/2}')}\right)^{min} +  \left(\widetilde{\FF_{\leq 1/2}'}  \right)^{min}.
\end{align*}  
Suppose $(\widetilde{\FF_t} )^{min}$ satisfies condition (c) in Proposition \ref{prop shift left 1/2}(i), then we denote 
\[ \widetilde{\FF_{=1/2}}= \left(\left(\widetilde{\FF_{\leq 1/2}}  \right)^{min}\right)_{=1/2},\ \widetilde{\FF_{<1/2}}= \left(\left(\widetilde{\FF_{\leq 1/2}}  \right)^{min}\right)_{>1/2}.\]
Suppose $(\widetilde{\FF_t} )^{min}$ doesn't satisfy condition (c) in Proposition \ref{prop shift left 1/2}(i), then we denote
\[ \widetilde{\FF_{=1/2}}= \left(dual_k^{+}\left( (\widetilde{\FF_t} )^{min}\right)^{min}\right)_{=1/2},\ \widetilde{\FF_{<1/2}}= \left(dual_k^{+}\left( (\widetilde{\FF_t} )^{min}\right)^{min}\right)_{>1/2}.\]
We repeat this for $(\widetilde{\FF_{\leq 1/2}'})^{min}$. Then Proposition \ref{prop shift left 1/2} shows that $\widetilde{\FF_{< 1/2}}=\widetilde{\FF_{< 1/2}'} $ up to row exchanges. Finally, taking the dual of 
\begin{align*}
&\widetilde{sh^t(\FF_{>1/2})} + \widetilde{\FF_{= 1/2}} + \widetilde{\FF_{< 1/2}},\\
&\widetilde{sh^t(\FF_{>1/2}')} + \widetilde{\FF_{= 1/2}'} + \widetilde{\FF_{< 1/2}'},   
\end{align*}
and canceling the $sh^t$, we may assume $\FF$ and $\FF'$ have the following decomposition
\[ \FF= \FF_{<1/2} + \FF_{=1/2} + \FF_{>1/2},\ \FF'= \FF_{<1/2}' + \FF_{=1/2}' + \FF_{>1/2}',\]
where 
\begin{enumerate}
    \item [$\oldbullet$]$\FF_{>1/2}= \FF_{>1/2}'$,
    \item [$\oldbullet$] $\supp(\FF_{<1/2})= \supp(\FF_{<1/2}')$.
\end{enumerate}
 Using the same argument to compare $\FF_{=0}$ and $\FF_{=0}'$ as in integer case, we may conclude $\FF_{=1/2}=\FF_{=1/2}'$ and hence $\FF=\FF'$ up to row exchanges.

Under this procedure, to go from $\FF$ to $\FF'$, we may have to apply $dual_k$ twice. However, in that case, tracing back from the conclusion, we see that $\Omega(\widetilde{\FF_{\leq 1/2}})=\Omega(\widetilde{\FF_{\leq 1/2}'})$ before applying $dual_k$. Then indeed $(\widetilde{\FF_{\leq 1/2}})^{min}= (\widetilde{\FF_{\leq 1/2}})^{min}$ and we do not need to apply partial dual to go from $\FF$ to $\FF'$. 
\end{proof}


\begin{remark}
The assertion that at most one partial dual is needed to show that the set 
\[ \{ \FF \ | \ \pi(\EE^{\rho} \cup \FF) \cong \pi(\EE^{\rho} \cup \EE_{\rho})\}\]
splits into at most two orbits under the basic operators. 
\end{remark}

\section{Canonical form and Theorem \ref{main thm intro}(3)}\label{sec can form}

In this section, for each $\EE\in \Rep$, we construct an extended multi-segment $\EE_{can} \in \Rep$, which we call the \emph{canonical form} of $\EE$, such that $\pi(\EE) \cong \pi(\EE')$ if and only if $\EE_{can}=\EE_{can}'$. Based on this construction, we prove Theorem \ref{main thm intro}(3), that is, we give a formula  (Theorem \ref{thm exhaustion of symbol} below) to exhaust the set
\[ \Psi(\EE):=\{ \EE' \ | \ \pi(\EE') \cong \pi(\EE)\}/\text{(row exchanges)}.\]
Moreover, we show that $\EE_{can}$ is uniquely characterized by the derivative information of $\pi(\EE)$ (see Theorem \ref{thm canonical form and derivatives} below). From this property, we give an algorithm (Algorithm \ref{alg Arthur type} below) to determine whether a representation of good parity is of Arthur type or not.

Based on the proof of Theorems \ref{thm integer}, \ref{thm half integer}, we give the following definition.

\begin{defn}\label{def can form}
For $\EE \in \Rep$ and $\FF= \EE_{\rho}$, we define $\FF_{can}$ as follows:
\begin{itemize}
    \item [(i)] Suppose $dual_k^{+}$ is not applicable on 
    \[dual(\FF^{min})_{<-1/2} + (dual(\FF^{min})_{\geq -1/2})^{min}\]
    for any $k$. Then we define
    \[ \FF_{can}=dual(dual(\FF^{min})_{<-1/2} + (dual(\FF^{min})_{\geq -1/2})^{min}).\]
    \item [(ii)] Suppose $dual_k^{+}$ is applicable on $dual(\FF^{min})_{<-1/2} + (dual(\FF^{min})_{\geq -1/2})^{min}$ for some $k$. Apply $dual_k^{+}$ on it and denote the resulting extended multi-segment by  $\widetilde{\FF}$. Then we define
    \[ \FF_{can}= dual( \widetilde{\FF}_{<-1/2} + (\widetilde{\FF}_{\geq -1/2})^{min} ).\]
\end{itemize}
We define $\EE_{can}= \cup_{\rho} (\EE_{\rho})_{can}$.
\end{defn}

We give some examples explaining the above definition. 

\begin{exmp}
Let $\rho$ be the trivial representation. We compute the canonical forms of the following extended multi-segments.
\[ \EE_1= \bordermatrix{
&\half{-1}& \half{1} & \half{3} & \half{5} \cr 
&& \ominus &&\cr 
&&&\oplus & \cr 
&&&&\ominus \cr
}_{\rho}, \ \EE_2= \bordermatrix{
& \half{-1}& \half{1} & \half{3} & \half{5} \cr 
&\oplus& \ominus &&\cr 
&&&\oplus & \cr 
&&&&\ominus \cr
}_{\rho}, \]
where $\pi(\EE_1)= \pi(\EE_2) = \pi((1/2)^{-},(3/2)^{+}, (5/2)^{-})$ is a supercuspidal representation of $\SO_{13}(F)$.

For $\EE_1$, we have 
\[ \EE_1^{min}= \bordermatrix{
& \half{-1}& \half{1} & \half{3} & \half{5} \cr 
&& \ominus &\oplus&\ominus\cr 
}_{\rho}, \  dual(\EE_1^{min})=\bordermatrix{
&\half{-1}& \half{1} & \half{3} & \half{5} \cr 
& \oplus& \ominus &\oplus&\ominus\cr 
}_{\rho}.  \]
Thus $dual(\EE_1^{min})_{<-1/2}+ (dual(\EE_1^{min})_{\geq -1/2})^{min}= dual(\EE_1^{min}),$ and no $dual_k^{+}$ is applicable. Then $\EE_1$ is in case (i) in the definition above, and 
\[ (\EE_1)_{can}= dual(dual(\EE_1^{min}))= \bordermatrix{
& \half{-1}& \half{1} & \half{3} & \half{5} \cr 
&& \ominus &\oplus&\ominus\cr 
}_{\rho}.\]

For $\EE_2$, we have 
\[ \EE_2^{min}= \bordermatrix{
&\half{-1}& \half{1} & \half{3} & \half{5} \cr 
& \oplus & \ominus &\oplus&\ominus\cr 
}_{\rho},\  dual(\EE_2^{min})=\bordermatrix{
& \half{1} & \half{3} & \half{5} \cr 
& \ominus &\oplus&\ominus\cr 
}_{\rho}.  \]
Thus $dual(\EE_2^{min})_{<-1/2}+ (dual(\EE_2^{min})_{\geq -1/2})^{min}= dual(\EE_2^{min}),$ and $dual_1^{+}$ is applicable on it. Then $\EE_2$ is in case (ii) in the definition above, and 
\begin{align*}
    \widetilde{\EE_2}= dual_1^{+}(dual(\EE_2^{min}))&=\bordermatrix{
&\half{-1}& \half{1} & \half{3} & \half{5} \cr 
& \oplus & \ominus &\oplus&\ominus\cr 
}_{\rho},\\
(\EE_2)_{can}= dual(\widetilde{\EE_2})&= \bordermatrix{
& \half{-1}& \half{1} & \half{3} & \half{5} \cr 
& & \ominus &\oplus&\ominus\cr 
}_{\rho}.
\end{align*}
We see that $(\EE_1)_{can}=(\EE_2)_{can}$, which follows from the proof of Theorem \ref{thm half integer}.
\end{exmp}

The following is an explicit criterion for determining $\pi(\EE)\cong \pi(\EE')$ or not.
\begin{cor}\label{cor canonical form}
For $\EE ,\EE' \in \Rep$, we have 
\[\pi(\EE) \cong \pi(\EE')\Longleftrightarrow \EE_{can}=\EE_{can}'. \]
\end{cor}
\begin{proof}
It follows from the proofs of Theorems \ref{thm integer} and \ref{thm half integer}. 
\end{proof}

    



For $\EE =\cup_{\rho} \EE_{\rho} \in \Rep$, we describe how to compute the sets
\begin{align*}
\Psi (\EE_{\rho}):=&\{\FF' \ | \ \FF_{can}'= (\EE_{\rho})_{can} \}/(\text{row exchanges}),\\
    \Psi (\EE):=&\{\EE' \ | \ \EE_{can}= \EE_{can}' \}/(\text{row exchanges}),
\end{align*}
in the following theorem by reversing the construction of $\FF_{can}$ and $\EE_{can}$. Note that
\[ \{ \psi \ | \ \pi(\EE) \in \Pi_{\psi}\}=\{ \psi_{\EE'}\ | \ \EE' \in \Psi(\EE)\},\]
where $\psi_{\EE'}$ is the local Arthur parameter associated with $\EE'$.

To simplify the description, for any $\EE \in \Rep$ and $\FF= \EE_{\rho}$, we denote
\[ UI^{-1}_{\geq -1/2}(\FF):= \left\{ \FF_{<-1/2}+ \FF^{\ast} \ | \ \FF^{\ast} \in UI^{-1}( \FF_{\geq -1/2}) \right\}.\]

\begin{thm}\label{thm exhaustion of symbol} \
\begin{enumerate}
    \item [1.] Suppose $\EE \in \Rep$ and $\FF =\EE_{\rho}$. If $dual_k^{+}$ is not applicable on $\FF_{can}$ for any $k$, then we set $A=\{\FF_{can}\}$ to be a singleton. Otherwise, if $dual_k^{+}$ is applicable on $\FF_{can}$, we set $A= \{ \FF_{can} , dual_k(\FF_{can})\}$. Then we have   
     \[ \Psi(\FF)=\bigcup_{\FF^{\ast}\in A}\left( \bigcup_{\FF^{\ast \ast} \in UI^{-1}_{\geq-1/2}(dual(\FF^{\ast}))} UI^{-1}(dual(\FF^{\ast\ast})) \right). \]  
    \item [2.] Suppose $\EE= \cup_{\rho}\EE_{\rho} \in \Rep$, then
    \[ \Psi(\EE)= \{ \cup_{\rho}\FF_{\rho}\ | \ \FF_{\rho} \in \Psi(\EE_{\rho}) \}.\]
\end{enumerate}
\end{thm}
\begin{proof}
It follows from the construction of $\FF_{can}$ and $\EE_{can}$.
\end{proof}

We give two examples of this theorem.
\begin{exmp} \label{ex exhaustion}
Consider 
\[ \EE= \bordermatrix{
&-1&0&1&2 \cr
&\lhd&\ominus&\rhd&\cr
&&\ominus&\oplus&\ominus \cr
&&&&\ominus \cr
}_{\rho},\]
where $\rho$ is the trivial representation. We have $\EE \in \Rep$ and
\[\pi(\EE)= L(\Delta_{\rho}[-1,-1], \pi( 0^{-}, 0^{-}, 1^{+},2^{-},2^{-})) \]
is an irreducible representation of $\Sp_{16}(F)$.
Let us construct all local Arthur parameters $\psi$ such that $\pi(\EE) \in \Pi_{\psi}$.

We first compute the canonical form $\EE_{can}$. Denote $\EE_1=\EE$.  $\EE_1$ is not minimal since $ui_{1,2}$ is applicable. We have
\[ui_{1,2}(\EE_1)=\EE_2=\ \begin{blockarray}{cccc}-1&0&1&2\\
\begin{block}{(cccc)}
\lhd&\lhd&\rhd&\rhd\\
&\oplus&\ominus&\\
&&&\ominus\\
\end{block}
\end{blockarray}\ , \]
and it is minimal. After taking dual, $ui_{2,3}$ is applicable. 
\[dual(\EE_2)=\ \begin{blockarray}{ccccc}-2&-1&0&1&2\\
\begin{block}{(ccccc)}
\lhd&\lhd&\ominus&\rhd&\rhd\\
&&\ominus&\oplus&\\
&&&\lhd&\rhd\\
\end{block}
\end{blockarray}\ ,\  ui_{2,3}(dual(\EE_2))=\ \begin{blockarray}{ccccc}-2&-1&0&1&2\\
\begin{block}{(ccccc)}
\lhd&\lhd&\ominus&\rhd&\rhd\\
&&\ominus&\oplus&\ominus\\
&&&\ominus&\\
\end{block}
\end{blockarray}\ .\]
Then $ui_{2,3}(dual(\EE_2))_{\geq -1/2}$ is minimal. In fact, $ui_{2,3}(dual(\EE_2))=dual(\EE_1)$ by Corollary \ref{cor ui inverse} or direct computation. Therefore, $\EE_1=\EE_{can}$.

Next, we compute $UI^{-1}_{\geq -1/2}(dual(\EE_1))$. It is computed in Example \ref{ex UI inverse} that  
\[UI^{-1}( dual(\EE_1)_{\geq -1/2})=\left\{ \begin{blockarray}{ccc}0&1&2\\
\begin{block}{(ccc)}
\ominus&\oplus&\ominus\\
&\ominus&\\
\end{block}
\end{blockarray}\ ,\ \begin{blockarray}{ccc}0&1&2\\
\begin{block}{(ccc)}
\ominus&\oplus&\\
&\lhd&\rhd\\
\end{block}
\end{blockarray} \ ,\ \begin{blockarray}{ccc}0&1&2\\
\begin{block}{(ccc)}
\ominus&&\\
&\oplus & \\
&\lhd&\rhd\\
\end{block}
\end{blockarray}   \right\},  \]
and hence $UI^{-1}_{\geq-1/2}(dual(\EE_1))=\{dual(\EE_1), dual(\EE_2), dual(\EE_3)\}$, where 
\[\EE_3=\ \begin{blockarray}{cccc}-1&0&1&2\\
\begin{block}{(cccc)}
\lhd&\lhd&\rhd&\rhd\\
\lhd&\oplus&\rhd&\\
&\ominus&&\\
&&&\ominus\\
\end{block}
\end{blockarray}\ = dual\left( \begin{blockarray}{ccccc}-2&-1&0&1&2\\
\begin{block}{(ccccc)}
\lhd&\lhd&\ominus&\rhd&\rhd\\
&&\ominus&&\\
&&&\oplus&\\
&&&\lhd&\rhd\\
\end{block}
\end{blockarray} \right).\]
It remains to compute $UI^{-1}(\EE_1)$, $UI^{-1}(\EE_2)$ and $UI^{-1}(\EE_3)$. As $\EE_1= ui_{1,2}^{-1}(\EE_2)$, we have $ UI^{-1}(\EE_1) \subseteq UI^{-1}(\EE_2)$.

so we only need to apply Algorithm \ref{algo ui inverse} to compute $UI^{-1}(\EE_2)$ and $ UI^{-1}(\EE_3)$. 

We omit the details of the computation in Algorithm \ref{algo ui inverse}. When computing $UI^{-1}(\EE_2)$, the first time we enter step 4, $C=\{\EE_1,\EE_4,\EE_5\}$, where 
\[ \EE_4= \begin{blockarray}{cccc}-1&0&1&2\\
\begin{block}{(cccc)}
\lhd&\lhd&\rhd&\rhd\\
&\oplus&&\\
&&\ominus &\\
&&&\ominus\\
\end{block}
\end{blockarray}\ , \ \EE_5= \begin{blockarray}{cccc}-1&0&1&2\\
\begin{block}{(cccc)}
\lhd&\ominus&\rhd&\\
&\ominus&\oplus&\\
&&&\ominus \\
&&&\ominus\\
\end{block}
\end{blockarray}\ . \]
The second times we enter step 4 with $C=\{\EE_6, \EE_7\}$ where 
\[ \EE_6= \begin{blockarray}{cccc}-1&0&1&2\\
\begin{block}{(cccc)}
\lhd&\ominus&\rhd&\\
&\ominus&&\\
&&\oplus&\ominus \\
&&&\ominus\\
\end{block}
\end{blockarray}\ , \ \EE_7= \begin{blockarray}{cccc}-1&0&1&2\\
\begin{block}{(cccc)}
\lhd&\ominus&\rhd&\\
&\ominus&&\\
&&\oplus&\\
&&&\ominus \\
&&&\ominus\\
\end{block}
\end{blockarray}\ . \]
The third time we encounter step 4, we have $C=\emptyset$, and the procedure ends. For the computation of $UI^{-1}(\EE_3)$, the first time we apply step 4, $C=\{ \EE_8\}$ where
\[\EE_8= \begin{blockarray}{cccc}-1&0&1&2\\
\begin{block}{(cccc)}
\lhd&\ominus&\rhd&\\
\lhd&\ominus&\rhd&\\
&\oplus&&\\
&&&\ominus \\
&&&\ominus\\
\end{block}
\end{blockarray}\ . \]
The second time we enter step 4, $C$ is empty and the procedure ends. In conclusion, 
\[ \{ \EE' \ | \ \pi(\EE')\cong \pi(\EE)\}/\text{(row exchange)}=\{\EE_1,\ldots,\EE_8\}.\]
In the following diagram, we visualize the relations among these extended multi-segments that we have seen during the process. We write $ \EE \to \EE'$ (resp. $ \EE \dashrightarrow \EE'$) if $ \EE'= ui_{i,j}(\EE)$ (resp. $\EE'= dual\circ ui_{i,j} \circ dual (\EE)$) for some $i,j$.
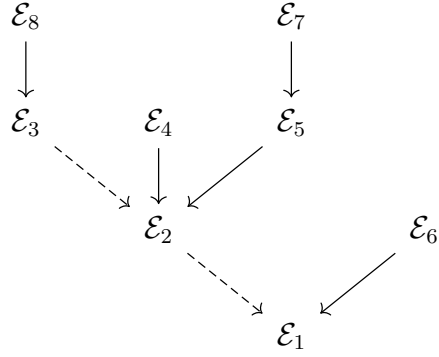
\[ \begin{tikzcd} 
\EE_8\arrow[d]& & \EE_7\arrow[d] & \\
\EE_3\arrow[rd, dashrightarrow]& \EE_4 \arrow[d]& \EE_5 \arrow[ld]&\\
& \EE_2\arrow[rd, dashrightarrow] & & \EE_6 \arrow[ld]\\
& & \EE_1 &
 \end{tikzcd}\]
Notice that to make the picture clean, we did not draw all relations among these extended multi-segments. For example, in the computation of $\EE_{can}$, we have seen $\EE_1 \to \EE_2$, but we decide to draw $\EE_2 \dashrightarrow \EE_1$ according to the formula in Theorem \ref{thm exhaustion of symbol}. Also, $ \EE_4=ui_{1,3}(\EE_6)$ and $ \EE_6= ui_{3,4}(\EE_{7})$, so $ \EE_7 \to \EE_6 \to \EE_4$, which is not drawn in the picture above. 

We list the local Arthur parameters $\psi_i$ associated with $\EE_i$.
\begin{align*}
    \psi_1&= \rho\otimes S_{1}\otimes S_{3}+\rho\otimes S_{3}\otimes S_{3}+\rho\otimes S_{5}\otimes S_{1},\\
     \psi_2&= \rho\otimes S_{2}\otimes S_{4}+\rho\otimes S_{2}\otimes S_{2}+\rho\otimes S_{5}\otimes S_{1},\\
      \psi_3&= \rho\otimes S_{1}\otimes S_{3}+\rho\otimes S_{2}\otimes S_{4}+\rho\otimes S_{1}\otimes S_{1}+\rho\otimes S_{5}\otimes S_{1},\\
      \psi_4&= \rho\otimes S_{2}\otimes S_{4}+\rho\otimes S_{1}\otimes S_{1}+\rho\otimes S_{3}\otimes S_{1}+\rho\otimes S_{5}\otimes S_{1},\\
      \psi_5&= \rho\otimes S_{1}\otimes S_{3}+\rho\otimes S_{2}\otimes S_{2}+\rho\otimes S_{5}\otimes S_{1}+\rho\otimes S_{5}\otimes S_{1},\\
      \psi_6&= \rho\otimes S_{1}\otimes S_{3}+\rho\otimes S_{1}\otimes S_{1}+\rho\otimes S_{4}\otimes S_{2}+\rho\otimes S_{5}\otimes S_{1},\\
      \psi_7&= \rho\otimes S_{1}\otimes S_{3}+\rho\otimes S_{1}\otimes S_{1}+\rho\otimes S_{3}\otimes S_{1}+\rho\otimes S_{5}\otimes S_{1}+S_{5}\otimes S_{1},\\
       \psi_8&= \rho\otimes S_{1}\otimes S_{3}+\rho\otimes S_{1}\otimes S_{3}+\rho\otimes S_{1}\otimes S_{1}+\rho\otimes S_{5}\otimes S_{1}+S_{5}\otimes S_{1}.
\end{align*}
\end{exmp}
\begin{exmp}\label{ex Atobe}
In this example, we apply Theorem \ref{thm exhaustion of symbol} on the same extended multi-segment considered in \cite[Section 3.4]{Ato22}. Let $\rho$ be the trivial representation and 
\[ \EE= \bordermatrix{
&0&1&2\cr
& \ominus & \oplus & \ominus
}_{\rho}, \]
where $\pi(\EE)= \pi(0^{-},1^{+},2^{-})$ is an irreducible representation of $\Sp_{8}(F)$.

First, we compute $\EE_{can}$. Let $\EE_1=\EE$. Since it is minimal and $dual(\EE_1)=\EE_1$, we have $\EE_1=\EE_{can}$ by definition.

Next, we compute $UI_{\geq -1/2}^{-1}(dual(\EE_{can}))$. It is not hard to see that Algorithm \ref{algo ui inverse} gives that
\[UI_{\geq -1/2}^{-1}(dual(\EE_{can})) = UI^{-1}(\EE_1)= \left\{ \EE_1, \EE_2,\EE_3, \EE_4 \right\}, \]
where 
\begin{align*}
    \EE_2= \bordermatrix{&0&1&2 \cr 
&\ominus&& \cr
&&\oplus&\ominus \cr}_{\rho} ,\ \EE_3= \bordermatrix{&0&1&2 \cr 
&\ominus& \oplus& \cr
&&&\ominus \cr}_{\rho} ,\ \EE_4= \bordermatrix{&0&1&2 \cr 
&\ominus&& \cr
&&\oplus& \cr
&&&\ominus \cr
}_{\rho}.
\end{align*}
Thus, we let 
\begin{align*}
    \EE_5&= dual(\EE_2)= \bordermatrix{
    &-1&0&1&2 \cr
    &\lhd & \ominus&\oplus&\rhd \cr
    & &\ominus&&\cr    
    }_{\rho},\\
    \EE_6&= dual(\EE_3)= \bordermatrix{
    &-2&-1&0&1&2 \cr
    &\lhd&\lhd & \ominus&\rhd&\rhd \cr
    && &\oplus&\ominus&\cr    
    }_{\rho},\\
    \EE_7&= dual(\EE_4)= \bordermatrix{
    &-2&-1&0&1&2 \cr
    &\lhd&\lhd & \ominus&\rhd&\rhd \cr
    &&\lhd &\oplus&\rhd&\cr
    &&&\ominus&&\cr
    }_{\rho}.
\end{align*}
From the formula in Theorem \ref{thm exhaustion of symbol}, we have 
$$\Psi(\EE)=  UI^{-1}(\EE_1) \cup UI^{-1}(\EE_5) \cup UI^{-1}(\EE_6) \cup UI^{-1}(\EE_7).$$ 
We have already computed $UI^{-1}(\EE_1)=\{ \EE_1 , \dots, \EE_4\}$. It is not hard to check that $ UI^{-1}(\EE_5)= \{\EE_5, \EE_8\}$, $UI^{-1}(\EE_6)=\{\EE_6,\EE_9\}$ and $UI^{-1}(\EE_7)= \{\EE_7\}$, where 
\[ \EE_8=  \bordermatrix{
    &-1&0&1&2 \cr
    &\lhd & \ominus&\rhd& \cr
    & &\oplus&&\cr  
    &&&&\ominus \cr
    }_{\rho},\ \EE_9= \bordermatrix{
    &-2&-1&0&1&2 \cr
    &\lhd&\lhd & \ominus&\rhd&\rhd \cr
    && &\oplus&&\cr
    &&&&\ominus&\cr
    }_{\rho}.\]
    In conclusion, $\Psi(\EE)=\{\EE_1,\dots, \EE_9\}$. In the following diagram, we visualize the relations among these extended multi-segments that we have seen during the process. We write $ \EE \to \EE'$ (resp. $ \EE \dashrightarrow \EE'$) if for some $i,j$, $ \EE'= ui_{i,j}(\EE)$ (resp. $\EE'= dual\circ ui_{i,j} \circ dual (\EE)$).
    \[
    \begin{tikzcd}
     \EE_9\ar[d]& \EE_8\ar[d]& \EE_7\ar[ld,dashrightarrow]&&\EE_4\ar[d]\\
    \EE_6\ar[drr,dashrightarrow] & \EE_5\ar[dr,dashrightarrow]&&\EE_2\ar[ld]&\EE_3\ar[lld]\\
    &&\EE_1&&
    \end{tikzcd}
    \]
The representation $\pi(\EE)$ is supercuspidal by \cite[Theorem 2.5.1]{Moe11}. 
As in Example \ref{ex exhaustion}, 
we did not draw all relations among these extended multi-segments. For example $ \EE_2=ui_{2,3}(\EE_4)$, so $ \EE_4 \to \EE_2$, which is not drawn in the picture above. 
\end{exmp}

Next, we show that $\EE_{can}$ can be uniquely characterized by the following derivative properties of $\pi(\EE)$. In some sense, among all extended multi-segments $\EE'$ with $\pi(\EE')\cong\pi(\EE)$, $\EE_{can}$ is the one that carries the most derivative information of $\pi(\EE)$.
\begin{thm}\label{thm canonical form and derivatives}
Suppose $\EE \in \Rep^{(P')}$ and denote
\[\FF= \EE_{\rho}=\{ ([A_i,B_i]_{\rho},l_i,\eta_i)\}_{i \in (I_{\rho},>)}.\]
Let $A= \max\{A_i\ | \ i \in I_{\rho}\}$. Then $\FF=\FF_{can}$ up to row exchanges if and only if the following properties hold:
\begin{enumerate}
    \item [(i)] For any $B > 1/2$, $D_{\Omega(\FF_{\geq B})} (\pi(\EE))\neq 0$ and this is a composition of highest derivatives (modulo a scalar). Also, we have 
       \[\pi(\EE^{\rho} \cup (\FF_{\leq 1/2} + sh^{\lceil A \rceil}(\FF_{>1/2})))= S_{\Omega(sh^{\lceil A \rceil}(\FF_{>1/2}))} \circ \cdots \circ S_{\Omega(sh^{1}(\FF_{>1/2}))} (\pi(\EE)). \]
       Denote this representation by $\pi_{A}$. 
    \item [(ii)] For any $B <0$, $ D_{\overline{\Omega(\FF_{\leq B})}}(\pi_{A}) \neq 0$ and this is a composition of highest derivatives (modulo a scalar).
\end{enumerate}
\end{thm}
\begin{proof}
First, suppose $\FF=\FF_{can}$. We are going to verify Conditions (i) and (ii). By construction, we have $(\FF_{can})_{>1/2}= (\FF^{min})_{>1/2}$, and Condition (i) follows from Lemma \ref{lem minimal shift}. Note that when $\FF$ is minimal, we have that $\FF_{\leq 1/2}+ sh^t(\FF_{> 1/2})$ is also minimal for any $t \in \Z_{>0}$.

For Condition (ii), denote $\widetilde{\FF}=  dual((\FF_{can})_{\leq 1/2} \cup sh^{\lceil A \rceil} (\FF_{can})_{>1/2})$. Then $\pi(dual (\EE^{\rho}) \cup \widetilde{\FF})$ is the Aubert-Zelevinsky dual of  $\pi(\EE^{\rho} \cup (\FF_{\leq 1/2} + sh^{\lceil A \rceil}(\FF_{>1/2})))$. Also, for any $B >0$, Proposition \ref{compatability of Aubert-Zelevinsky dual and derivative} implies
\[ D_{\overline{\Omega(\FF_{\leq -B})}}(\pi(\EE^{\rho} \cup (\FF_{\leq 1/2} + sh^{\lceil A \rceil}(\FF_{>1/2}))))= \reallywidehat{D_{ \Omega((\widetilde{\FF})_{\geq B})}(\pi(dual(\EE^{\rho}) \cup \widetilde{\FF}))}.\]
So it suffices to show $ D_{\Omega((\widetilde{\FF})_{\geq B})}(\pi(dual(\EE^{\rho}) \cup \widetilde{\FF}))$ is a composition of highest derivatives for any $B>0$. This follows from the construction that we have $(\widetilde{\FF})_{\geq 1/2}= ((\widetilde{\FF})^{min})_{\geq 1/2}$, and $\pi(\EE^{\rho}\cup (((\widetilde{\FF})^{min})_{< 1/2} + sh^{-1}(((\widetilde{\FF})^{min})_{\geq 1/2})))$ is always nonzero by Lemma \ref{lem minimal shift} and Proposition \ref{prop shift left 1/2}(i). 


Next, suppose $\FF$ satisfies both Conditions (i) and (ii). Observe that the multi-sets $\{\Omega(\FF_{\geq B})\}_{B >1/2}$ uniquely determine $\supp(\FF_{>1/2})$ as a multi-set, and $\{\overline{\Omega(\FF_{\leq  B})}\}_{B < 0}$ uniquely determine $\supp(\FF_{< 0})$ as a multi-set. As a consequence, Condition (i) implies that $ \supp(\FF_{>1/2})= \supp((\FF_{can})_{>1/2})$, and Condition (ii) implies that $ \supp(\FF_{<0})= \supp((\FF_{can})_{<0})$. In particular, we have 
\[ \Omega(\FF_{>1/2})= \Omega( (\FF_{can})_{>1/2}),\ \Omega(\FF_{<0})= \Omega( (\FF_{can})_{<0}).\]
Therefore, we have 
\[ \Omega(\FF)\Delta \Omega(\FF_{can})=\Omega(\FF_{=\epsilon})\Delta \Omega((\FF_{can})_{=\epsilon}) \]
where $\epsilon\in \{0,1/2\}$ such that $\epsilon +A \in \Z$. By Corollary \ref{cor symmetric support}, one can see that $\Omega(\FF_{=\epsilon})\Delta \Omega((\FF_{can})_{=\epsilon})$ is empty, and hence $\supp(\FF_{=\epsilon})= \supp((\FF_{can})_{=\epsilon})$ as multi-sets. The discussion so far indicates $ \supp(\FF)= \supp(\FF_{can})$ as multi-sets, so $\FF=\FF_{can}$ up to row exchanges by Lemma \ref{lem Moeglin}. 
This completes the proof of the theorem.
\end{proof}

Suppose $\EE \in \Rep^{(P')}$ and denote $\FF= \EE_{\rho}$. The next lemma shows that $\Omega(\FF_{=0})$ or $\Omega(\FF_{=1/2})$ can be computed from $ \Omega(\FF_{>1/2})$, $\Omega(\FF_{<0})$ and $\Omega(\pi(\EE))_{\rho}$. 

For $x\geq y$ with $x-y \in \Z$, we denote $[x,y]:= \{x,x-1,\ldots ,y\}$. 

\begin{lemma}
Suppose $\EE \in \Rep^{(P')}$ and 
\[ \FF= \EE_{\rho}= \{ ([A_i,B_i]_{\rho},l_i,\eta_i)\}_{i \in (I_{\rho}, >)}.\]
Let $A=\max\{A_i\ | \ i \in I_{\rho}\}$ and $\epsilon\in \{0,1/2\}$ such that $A_i+\epsilon \in \Z$. The multi-set $\Omega(\EE_{=\epsilon})$ can be computed from $\Omega(\pi(\EE))_{\rho}$ and $\Omega:=\Omega(\FF_{>1/2})\cup \Omega(\FF_{<0})$ as follows: For $t \in [A,-A]$, we denote $m_{1,t}$ (resp. $m_{2,t}$) the multiplicity of $\rho|\cdot|^{t}$ in the multi-set $\Omega\setminus\Omega(\pi(\EE))_{\rho}$ (resp. $\Omega(\pi(\EE))_{\rho}\setminus\Omega$). Then for $t \in [A,\epsilon]$, the multiplicity of $\rho|\cdot|^{t}$ in $\Omega(\FF_{=\epsilon})$ is given by $(m_{1,-t-1}-m_{1,t})+m_{2,t}$. 
\end{lemma}
\begin{proof}
We prove it by a purely combinatorial argument based on the following two observations:
\begin{enumerate}
    \item [$\oldbullet$] (Lemma \ref{lem $L$-data}(ii)) $\Omega(\FF)\setminus \Omega(\pi(\FF))$ is symmetric about $\rho|\cdot|^{-1/2}$ in the sense that the multiplicity of $\rho|\cdot|^{t}$ is the same as the multiplicity of $\rho|\cdot|^{-t-1}$ in this multi-set. 
    \item [$\oldbullet$] $\Omega(\FF_{=\epsilon})$ lies completely on one side of $\rho|\cdot|^{-1/2}$ in the sense that if we have $\rho|\cdot|^{x} \in \Omega(\FF_{=\epsilon})$, then $ \rho|\cdot|^{-x-1}$ is not in $\Omega(\FF_{=\epsilon})$.
\end{enumerate}
Let $M_t$ be the multiplicity of $\rho|\cdot|^{t}$ in $\Omega(\FF_{=\epsilon})$, and
\begin{align*}
\widetilde{\Omega}:=& \Omega \setminus \Omega(\pi(\EE))_{\rho},\\
\widetilde{\Omega(\pi)}:=&\Omega(\pi(\EE))_{\rho}\setminus \Omega,\\
\Omega^{sym}:=& \Omega(\FF)\setminus \Omega(\pi(\EE))_{\rho}.
\end{align*}
We have 
\[ \widetilde{\Omega} + \Omega(\FF_{=\epsilon})=\widetilde{\Omega(\pi)} +\Omega^{sym}. \]
There is a unique sub-multi-set $\widetilde{\Omega}^{sym}$ of $\widetilde{\Omega}$ which is symmetric about $\rho|\cdot|^{-1/2}$ and maximal with respect to this property. As $\widetilde{\Omega}$ and $\widetilde{\Omega(\pi)}$ are disjoint, we have $\widetilde{\Omega}^{sym} \subseteq \Omega^{sym}$. We rewrite our equation as 
\[ (\widetilde{\Omega}\setminus\widetilde{\Omega}^{sym}) + \Omega(\FF_{=\epsilon})= \widetilde{\Omega(\pi)} + (\Omega^{sym}\setminus \widetilde{\Omega}^{sym}). \]
Note that the multiplicity of $\rho|\cdot|^{t}$ in $\widetilde{\Omega}\setminus \widetilde{\Omega}^{sym}$ is given by $m_{1,t}-\min\{m_{1,t},m_{1,-t-1}\}$.

Now we compute $M_t$ for $t \in [A,\epsilon]$. If $M_t\geq  m_{2,t}$, then  
\[M_t- (m_{1,-t-1}-\min\{m_{1,t},m_{1,-t-1}\})=m_{2,t}\]
and hence $M_t= (m_{1,-t-1}- m_{1,t})+m_{2,t}$. On the other hand, if $M_t < m_{2,t}$, then
\[M_t+ (m_{1,t}-\min\{m_{1,t},m_{1,-t-1}\})=m_{2,t}\]
and hence $M_t= (m_{1,-t-1}- m_{1,t})+m_{2,t}$ still holds. This completes the proof of the lemma.
\end{proof}

Combining Theorem \ref{thm canonical form and derivatives} and the lemma above, if $\FF= \FF_{can}$, we can recover $\supp(\FF)$ completely from the derivative information of $\pi(\EE)$. As a consequence, we give the following algorithm to determine whether a representation $\pi$ of good parity is of Arthur type or not. If $\pi$ is of Arthur type, the algorithm outputs an extended multi-segment $\EE$ such that $\pi=\pi(\EE)$ and $\EE= \EE_{can}$.

\begin{algo} \label{alg Arthur type}
Given a representation $\pi$ of good parity, proceed as follows:
\begin{enumerate}
    \item [\textbf{Step 0}:] 
    Set $\psi=0$. Repeat steps 1 to 4 for each $\rho \in \{ \rho \ | \ \Omega(\pi)_{\rho} \neq \emptyset\}$.
    \item [\textbf{Step 1:}] Let $A= \max\{ x \ | \ \rho|\cdot|^{x} \in \Omega(\pi)_{\rho}\}$ and $\epsilon\in\{0,1/2\}$ such that $A+\epsilon \in \Z$. Set $\Omega^{+}=\emptyset=\Omega^{-}$.
    \item [\textbf{Step 2:}] Compute the following set
    \begin{align*}
        \mathcal{B}=& \{ B >1/2 \ | \ D_{\rho|\cdot|^{B}}(\pi) \neq 0 \}= \{ B_1,\dots,B_r\},
    \end{align*}
    where $B_i$ is decreasing. For each $1 \leq i \leq r$, we compute (by \cite[Theorem 7.1]{AM20}) recursively the integer $k_{i,t}$ and representation $\pi_{i,t}$ for $t$ in the segment $[A+1,B_i-1]$ via 
    \[\begin{cases}
     \pi_{i,B_{i}-1}= \pi,\\
     \pi_{i,t}=D_{\rho|\cdot|^{t}}^{(k_{i,t})}(\pi_{i,t-1}) \text{ is the highest derivative.}
    \end{cases}\]
If $k_{i,A+1} \neq 0$, then $\pi$ is not of Arthur type and the procedure ends. Set $k_{i,t}:=0$ if $t \not\in [A+1,B_i-1]$.
    
    Denote $K_{i,t}:=k_{i,t}-  k_{i-1,t}$. For $t \in [A+1,B_i+1]$, if $ K_{i,t}> K_{i,t-1}$, then $\pi$ is not of Arthur type and the procedure ends. If $K_{i,t}<K_{i,t-1}$, then add $K_{i,t-1}-K_{i,t}$ copies of $\rho \otimes S_{(t-1)+B_i+1}\otimes S_{(t-1)-B_i+1}$ to $\psi$ and add the same copies of elements in the segment $[t-1,B_i]_{\rho}$ in $\Omega^{+}$.
    \item [\textbf{Step 3:}]Reorder
    \[\Omega^{+}=\{\rho|\cdot|^{x_1},\dots,\rho|\cdot|^{x_r}\}\]
    such that $x_1 \leq \cdots\leq x_r$. For $t \in [A+ \epsilon, 1]$, denote
    \[ sh^t(\Omega^{+})=\{\rho|\cdot|^{x_1+t},\dots,\rho|\cdot|^{x_r+t}\}. \]
    Compute the representation (by \cite[Theorem 7.1]{AM20})
    \[ \pi_{A}= S_{ sh^{A+\epsilon}(\Omega^{+})} \circ\cdots \circ S_{sh^{1}(\Omega^{+})}(\pi),\]
    and the set 
    \[\overline{\mathcal{B}}= \{ B <0 \ | \ D_{\rho|\cdot|^{B}}(\pi_A) \neq 0 \}= \{ \overline{B_1},\dots,\overline{B_{\overline{r}}}\},\]
     where $\overline{B_j}$ is increasing. For each $1 \leq i \leq \overline{r}$, compute (by \cite[Proposition 6.1]{AM20}) the integer $\overline{k_{i,t}}$ and representation $\overline{\pi_{i,t}}$ for $t$ in the segment $[\overline{B_i}+1,-A-1]$ recursively by 
    \[\begin{cases}
     \overline{\pi_{i,\overline{B_{i}}+1}}= \pi_A, \\
    \overline{\pi_{i,t}}= D_{\rho|\cdot|^{t}}^{(\overline{k_{i,t}})}(\overline{\pi_{i,t+1}}) \text{ is the highest derivative.}
    \end{cases}\]
    If $\overline{k_{i,-A-1}} \neq 0$, then $\pi$ is not of Arthur type and the procedure ends. Set $\overline{k_{i,t}}:=0$ if $t \not\in [\overline{B_{i}}+1,-A-1]$.
    
    Denote $\overline{K_{i,t}}:=\overline{k_{i,t}}-  \overline{k_{i-1,t}}$. For $t \in [\overline{B_i}-1,-A-1]$, if $ \overline{K_{i,t}}> \overline{K_{i,t+1}}$, then $\pi$ is not of Arthur type and the procedure ends. If $\overline{K_{i,t}}<\overline{K_{i,t+1}}$, then add $\overline{K_{i,t+1}}-\overline{K_{i,t}}$ copies of $\rho \otimes S_{-(t+1)+\overline{B_i}+1}\otimes S_{-(t+1)-\overline{B_i}+1}$ to $\psi$, and add the same copies of elements in the segment $[ -t -1,\overline{B_i}]_{\rho}$ in $\Omega^{-}$.
    \item [\textbf{Step 4:}]
    Let $\Omega=\Omega^{+} + \Omega^{-}$. Denote the multiplicity of $\rho|\cdot|^{t}$ in $\Omega\setminus \Omega(\pi)$ (resp. $\Omega(\pi)\setminus \Omega$) by $ m_{1,t}$ (resp. $m_{2,t}$), and let $ M_t =(m_{1,-t-1}-m_{1,t})+m_{2,t}$. 
    
    For any $t \in [A+1, \epsilon+1]$, if $M_{t}>M_{t-1}$, then $\pi$ is not of Arthur type and the procedure ends. If $M_{t}<M_{t-1}$, add $M_{t}-M_{t-1}$ copies of $\rho \otimes S_{(t-1)+\epsilon+1}\otimes S_{(t-1)-\epsilon+1}$ to $\psi$. 
    \item [\textbf{Step 5:}] Construct the local Arthur packet $\Pi_{\psi}$. If there exists $\EE$ in this packet such that $\pi(\EE)=\pi$, then $\pi$ is of Arthur type and $\EE=\EE_{can}$ up to row exchanges. Otherwise, $\pi$ is not of Arthur type. 
    \end{enumerate}
\end{algo}

In \cite{Ato22}, Atobe also gave an algorithm to determine whether a representation $\pi$ is of Arthur type or not. We remark on the difference between Atobe's algorithm and ours here. The two algorithms are identical if $\pi$ is $\rho|\cdot|^{x}$-reduced for all $\rho$ and $x \not\in \{0,1/2\}$. So we focus on the case that $D_{\rho|\cdot|^{x}}(\pi)\neq 0$ for some $x >1/2$. 

Atobe's algorithm uses recursion. He computes $\pi^{+}=D_{\Omega}(\pi)$, a derivative of $\pi$, then applies the algorithm again on $\pi^{+}$. If $\pi^{+}$ is of Arthur type, then he exhausts the set 
\[ \{ \EE^{+} \ | \ \pi(\EE^{+})\cong \pi^{+}\} \]
to see whether there exists an $\EE^{+}$ with $\EE_{\rho}^{+}= \FF_1 +\FF_2$ such that $\Omega(sh^1(\FF_2))=\Omega$. If so, then $\pi$ is of Arthur type and $\pi= \pi((\EE^{+})^{\rho} \cup (\FF_1+ sh^1(\FF_2)))$.

On the other hand, our algorithm does not use recursion. We first collect the positive derivative information of $\pi$ in $\Omega^{+}$. Then we compute $\pi_A$ by taking successive socles with respect to $\Omega^{+}$, and collect the negative derivative information of $\pi_A$ in $\Omega^{-}$. Finally, from the $L$-data of $\pi$ and $\Omega:=\Omega^{+} + \Omega^{-}$, we nail down one local Arthur parameter $\psi$ such that $\pi$ is of Arthur type if and only if $\pi$ is in the local Arthur packet $\Pi_{\psi}$.

We give examples for the algorithm.
\begin{exmp}\label{ex algo 7.9}
\begin{enumerate}
    \item [1.] We apply our algorithm on \[\pi=L(\Delta_{\rho}[-1,-1], \pi( 0^{-}, 0^{-}, 1^{+},2^{-},2^{-})),\] 
    the representation in Example \ref{ex exhaustion}. We initially set $\psi=0.$
    \begin{enumerate}
        \item [\textbf{Step 1:}] $A=2,\epsilon=0$, and set $\Omega^{+}=\Omega^{-}=\emptyset$.
        \item [\textbf{Step 2:}] We have 
    \[ \mathcal{B}=\{B>1/2 \ | \ D_{\rho|\cdot|^{B}}(\pi)\neq 0\}= \{2\}. \]
    Since $D_{\rho|\cdot|^{2}}(\pi)=L(\Delta_{\rho}[-1,-1],\Delta_{\rho}[1,-2], \pi( 0^{-}, 0^{-}, 1^{+}))$ is $\rho|\cdot|^{3}$-reduced, we add $\rho\otimes S_{5} \otimes S_{1} $ in $\psi$ and add $\rho|\cdot|^{2}$ in $\Omega^{+}$.
    \item [\textbf{Step 3:}] We compute that
    \[ \pi_{A}= S_{\rho|\cdot|^{4}}\circ S_{\rho|\cdot|^{3}} (\pi)= L(\Delta_{\rho}[-1,-1], \pi( 0^{-}, 0^{-}, 1^{+},2^{-},4^{-})). \]
    Then
    \[\overline{\mathcal{B}}= \{ B <0 \ | \ D_{\rho|\cdot|^{B}}(\pi_A) \neq 0 \}= \{-1\}.\]
    Since $D_{\rho|\cdot|^{-1}}(\pi_A)$ is $\rho|\cdot|^{-2}$-reduced, we add $ \rho \otimes S_{1} \otimes S_{3}$ in $\psi$ and $ \rho|\cdot|^{-1}, \rho, \rho|\cdot|^{1}$ in $\Omega^{-}$.
    \item [\textbf{Step 4:}] We have $\Omega(\pi) \supseteq \Omega$ and the difference multi-set is $ \{ \rho, \rho|\cdot|^{1}, \rho|\cdot|^{2}\}$. So $(M_{0},M_{1},M_{2},M_3)=(1,1,1,0)$, and we add $\rho\otimes S_{3}\otimes S_{3}$ in $\psi$.
    \item [\textbf{Step 5:}] Now we exhaust the local Arthur packet associated with 
    \[\psi= \rho\otimes S_{1}\otimes S_{3}+\rho\otimes S_{3}\otimes S_{3}+\rho\otimes S_{5}\otimes S_{1}. \]
     Then we find $\pi=\pi(\EE)$ where 
  \begin{align}\label{eq exmp A-type}
       \EE= \bordermatrix{
&-1&0&1&2 \cr
&\lhd&\ominus&\rhd&\cr
&&\ominus&\oplus&\ominus \cr
&&&&\ominus \cr
}_{\rho},
  \end{align}
and $\EE=\EE_{can}$. 
    \end{enumerate}
    Indeed, the representation $\pi_{A}$ we computed in step 3 is the representation associated with the following extended multi-segment
\[ \EE_{\leq 1/2}+ sh^{2}(\EE_{>1/2})= \bordermatrix{
&-1&0&1&2&3&4 \cr
&\lhd&\ominus&\rhd&&&\cr
&&\ominus&\oplus&\ominus&& \cr
&&&&&&\ominus \cr
}_{\rho}. \]

\item [2.]  We apply our algorithm on 
\[ \pi= L(\Delta_{\rho}[-3,-3],\Delta_{\rho}[-1,-1], \pi( 0^{-}, 0^{-}, 1^{+},2^{-},2^{-})),\]
which is $S_{\rho|\cdot|^{-3}}(\pi(\EE))$ where $\EE$ is the extended multi-segment given in \eqref{eq exmp A-type}. We initially set $\psi=0$.
\begin{enumerate}
    \item [\textbf{Step 1:}]$A=3, \epsilon=0$. Set  $\Omega^{+}=\Omega^{-}=\emptyset$.
    \item [\textbf{Step 2:}]  We have 
    \[ \mathcal{B}=\{B>1/2 \ | \ D_{\rho|\cdot|^{B}}(\pi)\neq 0\}= \{2\}. \]
    Since $D_{\rho|\cdot|^{2}}(\pi)=L(\Delta_{\rho}[-3,-3],\Delta_{\rho}[-1,-1],\Delta_{\rho}[1,-2], \pi( 0^{-}, 0^{-}, 1^{+}))$ is $\rho|\cdot|^{3}$-reduced, we add $\rho\otimes S_{5} \otimes S_{1} $ in $\psi$ and add $\rho|\cdot|^{2}$ in $\Omega^{+}$.
    \item [\textbf{Step 3:}] We have 
    \begin{align*}
        \pi_A=& S_{\rho|\cdot|^{5}}\circ S_{\rho|\cdot|^{4}} \circ S_{\rho|\cdot|^{3}}(\pi)\\
        =&L(\Delta_{\rho}[-3,-3],\Delta_{\rho}[-1,-1], \pi( 0^{-}, 0^{-}, 1^{+},2^{-},5^{-})),
    \end{align*}
    and 
     \[\overline{\mathcal{B}}= \{ B <0 \ | \ D_{\rho|\cdot|^{B}}(\pi_A) \neq 0 \}= \{-3,-1\}.\]
     Since $D_{\rho|\cdot|^{-3}}(\pi_A)$ is $\rho|\cdot|^{-4}$-reduced, we add $ \rho\otimes S_{1}\otimes S_{7}$ to $\psi$ and add $\{ \rho|\cdot|^{-3},\rho|\cdot|^{-2},\dots,\rho|\cdot|^{3}\}$ in $\Omega^{-}$. 
     
     Since $D_{\rho|\cdot|^{-3}}^{(1)} \circ D_{\rho|\cdot|^{-2}}^{(0)}  \circ D_{\rho|\cdot|^{-1}}^{(1)}(\pi_A)$ is a composition of highest derivatives and the resulting representation is $\rho|\cdot|^{-4}$-reduced,  we add $ \rho\otimes S_{1}\otimes S_{3}$ to $\psi$ and add $\{ \rho|\cdot|^{-1},\rho,\rho|\cdot|^{1}\}$ in $\Omega^{-}$.
     \item [\textbf{Step 4:}] We have $\Omega \supseteq \Omega(\pi)$, and the difference multi-set is $\{\rho|\cdot|^{-2},\rho|\cdot|^{-1}\}$, so $(M_0,M_1,M_2,M_3,M_4)=(1,1,0,0,0)$, and we add $\rho\otimes S_{2} \otimes S_{2}$ to $\psi$.
     \item [\textbf{Step 5:}] Finally, $\pi$ is of Arthur type if and only if it is in the local Arthur packet associated with the parameter
    \[ \psi= \rho\otimes S_{5}\otimes S_{1}+\rho\otimes S_{1}\otimes S_{7}+\rho\otimes S_{1}\otimes S_{3}+\rho\otimes S_{2}\otimes S_{2}. \]
    By exhausting this packet, we conclude $\pi$ is not of Arthur type.
\end{enumerate}
\end{enumerate}
\end{exmp}

\begin{exmp}\label{ex algo 7.9 Atobe}
In this example, we apply Algorithm \ref{alg Arthur type} on two representations considered in \cite[Section 3.5]{Ato22}. 

\begin{enumerate}
    \item [1.] Consider 
    \[\pi= L\left(\Delta_{\rho}\left[ \half{-1},\half{-5}\right], \Delta_{\rho}\left[ \half{-1},\half{-1}\right],\Delta_{\rho}\left[ \half{3},\half{-5}\right]; \pi\left(\half{1^+},\half{3^+}, \half{5^+}\right)\right).\] Initially, we set $\psi=0$.
    \begin{enumerate}
        \item [\textbf{Step 1:}] $A=5/2, \epsilon=1/2$, and set $\Omega^{+}=\Omega^{-}=\emptyset$.
        \item [\textbf{Step 2:}]We have 
        \[ \mathcal{B}= \{B>1/2\ | \ D_{\rho|\cdot|^{B}}(\pi) \neq 0\}= \{ 5/2, 3/2\}.\]
         Since $ D_{\rho|\cdot|^{5/2}}(\pi)$ is a highest derivative, and the resulting representation is $\rho|\cdot|^{7/2}$-reduced, we add $ \rho\otimes S_6 \otimes S_1 $ to $\psi$ and $\rho|\cdot|^{5/2}$ to $\Omega^{+}$.
         
        Since $ D_{\rho|\cdot|^{5/2}}^{(3)} \circ D_{\rho|\cdot|^{3/2}}^{(2)}(\pi)$ is a composition of highest derivative, and the resulting representation is $\rho|\cdot|^{7/2}$-reduced, we add $2$ copies of $ \rho \otimes S_5 \otimes S_2$ to $\psi$ and add $ \{\rho|\cdot|^{3/2},\rho|\cdot|^{3/2},\rho|\cdot|^{5/2},\rho|\cdot|^{5/2}\}$ to $ \Omega^{+}$.
    \item[\textbf{Step 3:}] We have 
    \begin{align*}
        \pi_A&= (S_{\rho|\cdot|^{9/2}}^{(2)} \circ S_{\rho|\cdot|^{11/2}}^{(3)} )\circ (S_{\rho|\cdot|^{7/2}}^{(2)} \circ S_{\rho|\cdot|^{9/2}}^{(3)}) \circ (S_{\rho|\cdot|^{5/2}}^{(2)} \circ S_{\rho|\cdot|^{7/2}}^{(3)} )(\pi)\\
        &= L\left(\Delta_{\rho}\left[ \half{-1},\half{-11}\right], \Delta_{\rho}\left[ \half{-1},\half{-1}\right],\Delta_{\rho}\left[ \half{9},\half{-11}\right]; \pi\left(\half{1^+},\half{9^+}, \half{11^+}\right)\right),
    \end{align*}
         and 
     \[\overline{\mathcal{B}}= \{ B <0 \ | \ D_{\rho|\cdot|^{B}}(\pi_A) \neq 0 \}= \{-1/2\}.\]
     We have 
     \[D_{\rho|\cdot|^{-7/2}} \circ D_{\rho|\cdot|^{-5/2}}\circ D_{\rho|\cdot|^{-3/2}}\circ
     D_{\rho|\cdot|^{-1/2}}^{(2)}(\pi_A) \]
     is a composition of highest derivatives. Since $ -7/2= -A-1$, we conclude that $\pi$ is not of Arthur type.
    \end{enumerate}
    \item [2.] Consider 
    \[\pi= L\left(\Delta_{\rho}\left[ \half{-1},\half{-5}\right], \Delta_{\rho}\left[ \half{-1},\half{-1}\right],\Delta_{\rho}\left[ \half{3},\half{-5}\right]; \pi\left(\half{1^-},\half{3^+}, \half{5^-}\right)\right).\] Initially, we set $\psi=0$.
    \begin{enumerate}
        \item [\textbf{Step 1:}] $A=5/2,\epsilon=1/2$. Set $\Omega^{+}=\Omega^{-}= \emptyset$.
         \item [\textbf{Step 2:}]We have 
        \[ \mathcal{B}= \{B>1/2\ | \ D_{\rho|\cdot|^{B}}(\pi) \neq 0\}= \{ 5/2, 3/2\}.\]
         Since $ D_{\rho|\cdot|^{5/2}}(\pi)$ is a highest derivative, and the resulting representation is $\rho|\cdot|^{7/2}$-reduced, we add $ \rho\otimes S_6 \otimes S_1 $ in $\psi$ and $\rho|\cdot|^{5/2}$ in $\Omega^{+}$.
         
        Since $ D_{\rho|\cdot|^{5/2}}^{(2)} \circ D_{\rho|\cdot|^{3/2}}^{(1)}(\pi)$ is a composition of highest derivative, and the resulting representation is $\rho|\cdot|^{7/2}$-reduced, we add $ \rho \otimes S_5 \otimes S_2$ to $\psi$ and add $ \{\rho|\cdot|^{3/2},\rho|\cdot|^{5/2}\}$ to $ \Omega^{+}$.
    \item [\textbf{Step 3:}] We have 
    \begin{align*}
        \pi_A&= (S_{\rho|\cdot|^{9/2}}^{(1)} \circ S_{\rho|\cdot|^{11/2}}^{(2)} )\circ (S_{\rho|\cdot|^{7/2}}^{(1)} \circ S_{\rho|\cdot|^{9/2}}^{(2)}) \circ (S_{\rho|\cdot|^{5/2}}^{(1)} \circ S_{\rho|\cdot|^{7/2}}^{(2)} )(\pi)\\
        &= L\left(\Delta_{\rho}\left[ \half{-1},\half{-5}\right], \Delta_{\rho}\left[ \half{-1},\half{-1}\right],\Delta_{\rho}\left[ \half{9},\half{-11}\right]; \pi\left(\half{1^-},\half{2^+}, \half{11^-}\right)\right)
    \end{align*}
          and 
     \[\overline{\mathcal{B}}= \{ B <0 \ | \ D_{\rho|\cdot|^{B}}(\pi_A) \neq 0 \}= \{-1/2\}.\]
     We have 
     \[ D_{\rho|\cdot|^{-5/2}} \circ D_{\rho|\cdot|^{-3/2}} \circ D_{\rho|\cdot|^{-1/2}}(\pi_A)\]
     is a composition of highest derivative, and the resulting representation is $\rho|\cdot|^{-7/2}$-reduced. Therefore, we add $\rho\otimes S_1 \otimes S_2+ \rho \otimes S_3 \otimes S_3$ to $\psi$ and add $\{ \rho|\cdot|^{-1/2},\rho|\cdot|^{1/2}, \rho|\cdot|^{-1/2},\rho|\cdot|^{1/2}, \cdots ,\rho|\cdot|^{5/2}\}$ to $\Omega^{-}$.
     \item[\textbf{Step 4:}] We have
     \begin{align*}
         \Omega=\{ (\rho|\cdot|^{-1/2})^2, (\rho|\cdot|^{1/2})^2, (\rho|\cdot|^{3/2})^2, (\rho|\cdot|^{5/2})^3 \}=\Omega(\pi).
     \end{align*}
     Thus we don't need to add anything to $\psi$ in this step.
     \item[\textbf{Step 5:}] Now we exhaust the local Arthur packet associated with
     \[ \psi= \rho \otimes S_6 \otimes S_1+\rho \otimes S_5 \otimes S_2+\rho \otimes S_1 \otimes S_2+\rho \otimes S_3 \otimes S_4. \]
     Then we find $\pi=\pi(\EE)$ where
     \[ \EE= \bordermatrix{
     & \half{-1} & \half{1} & \half{3} & \half{5} \cr 
     & \lhd & \rhd && \cr 
     & \lhd & \ominus & \oplus & \rhd \cr
     && &\lhd & \rhd \cr
     && & & \ominus \cr
     }_{\rho}.\]
     One can check that $\EE=\EE_{can}$, and $\pi_A$ is the representation associated with
     \[  \EE_{\leq 1/2}+ sh^{3}(\EE_{>1/2})= \bordermatrix{
     & \half{-1} & \half{1} & \half{3} & \half{5} & \half{7}& \half{9} & \half{11} \cr 
     & \lhd & \rhd && &&&\cr 
     & \lhd & \ominus & \oplus & \rhd &&&\cr
     && &&&&\lhd & \rhd \cr
     && & & &&&\ominus \cr
     }_{\rho}.\]
    \end{enumerate}
\end{enumerate}
\end{exmp}

\section{Local Arthur packets containing tempered representations}\label{sec tempered rep}
In this section, we classify the set of local Arthur packets which contain tempered representations. Moreover, for each local Arthur parameter $\psi \in \Psi^{+}(G_n)$, we give explicit formula for the number of tempered representations inside $\Pi_{\psi}$ (Theorem \ref{thm tempered general}).

Given a general local Arthur parameter $\psi \in \Psi^+(G_n)$, we consider the decomposition $\psi= \psi_{nu,>0}+\psi_{np} + \psi_{gp} +\psi_{np}^{\vee}+ \psi_{nu,>0}^{\vee}$ as in Theorem \ref{thm red from nu to gp}. Then $\Pi_{\psi}$ contains tempered representations only if $\psi_{np}$ is tempered and $\psi_{nu,>0}=0$ and the set of tempered representations in $\Pi_{\psi}$ in this case can be described as
\[ \{ \tau_{\psi_{np}} \rtimes \pi_{gp} \ | \ \pi_{gp} \in \Pi_{\psi_{gp}}, \pi_{gp} \text{ is tempered}\}. \]
Therefore, it suffices to consider local Arthur parameters of good parity. Hence, we assume all local Arthur parameters are of good parity in this section. 

For convenience, we give the following definition.

\begin{defn}
We say that
\[\FF= \{([A_i,B_i]_{\rho}, l_i,\eta_i)\}_{i \in (I_{\rho, >})}\]
is tempered if $A_i=B_i$ for all $i \in I_{\rho}$ and an extended multi-segment $\EE= \cup_{\rho} \EE_{\rho}$ is tempered if $\EE_{\rho}$ is tempered for any $\rho$. Equivalently, $\EE$ is tempered if $\psi_{\EE}$, the local Arthur parameter associated with $\EE$, is tempered.

\end{defn}

\subsection{Local Arthur packets containing supercuspidal representations}\label{sec supercuspidal}
In this subsection, we give a necessary and sufficient condition on $\EE \in \Rep$ that $\pi(\EE)$ is supercuspidal (Corollary \ref{cor cuspidal shape}). Also, fix a local Arthur parameter $\psi \in \Psi^{+}(G_n)$, we give a formula for the number of supercuspidal representations in $\Pi_{\psi}$ (Theorem \ref{thm cuspidal count}).

M{\oe}glin's work suggested how to characterize local Arthur packets containing supercuspidal representations.  We first recall several notations and state two of M{\oe}glin's results.

Let $\Delta: \SL_2(\BC) \to \SL_2(\BC)\times \SL_2(\BC)$ be the diagonal embedding. Recall the diagonal restriction of $\psi$ as in \eqref{def diag rest}
\[ \psi^{\Delta}: W_{F} \times \SL_2(\BC) \xrightarrow[]{id \times \Delta} W_F \times \SL_2(\BC) \times \SL_2(\BC) \xrightarrow[]{\psi} \widehat{G}.\]

For any tempered $L$-parameter $ \phi= \bigoplus_{i} \rho_i \otimes S_{a_i}$, we write $\rho \otimes S_{a} \subset \phi$ if $\rho \otimes S_a$ appears as a direct summand in  $\phi$. 
A discrete $L$-parameter $ \phi= \bigoplus_{i} \rho_i \otimes S_{a_i}$ is said to be \emph{without gaps} if for $a \geq 1$,
\[ \rho \otimes S_{a+2} \subset \phi \Rightarrow \rho \otimes S_{a} \subset \phi. \]

M{\oe}glin gave the following characterization of the $L$-data of supercuspidal representations.
\begin{thm}[{\cite[Theorem 2.5.1]{Moe11},\cite[Theorem 3.3]{Xu17a}}] \label{thm characterizatioin of supercuspidal}
A tempered representation $\pi(\phi,\varepsilon)$ is supercuspidal if and only if 
\begin{enumerate}
    \item [$\oldbullet$] As a tempered $L$-parameter, $\phi$ is discrete and without gaps.
    \item [$\oldbullet$] If $\rho \otimes S_{a} \subset \phi$ and $\rho \otimes S_{a+2} \subset \phi$, then $\varepsilon(\rho \otimes S_{a}) \varepsilon(\rho \otimes S_{a+2})=-1.$
    \item [$\oldbullet$] If $\rho \otimes S_{2} \subset \phi$, then $ \varepsilon(\rho \otimes S_2)=-1$.
\end{enumerate}
\end{thm}

The next result follows from \cite{Moe06b,Moe09a} and we state the version in \cite{GGP20}.
\begin{thm}[{\cite[Theorem 7.5]{GGP20}}]\label{thm GGP20}
Suppose $\psi, \psi'$ are local Arthur parameters of $G_n$. Then
\[ \Pi_{\psi} \cap \Pi_{\psi'} \neq \emptyset \Longrightarrow \psi^{\Delta} \cong \psi'^{\Delta}.\]
Furthermore, if $\Pi_{\psi^{\Delta}}^{sc}$ denotes the set of supercuspidal representations in $\Pi_{\psi^{\Delta}}$, then $\Pi_{\psi^{\Delta}}^{sc} \subset \Pi_{\psi}$.
\end{thm}

Let $\phi$ be a tempered local Arthur parameter. Then Theorem \ref{thm GGP20} implies that 
$\Pi_{\psi} \cap \Pi_{\phi}\neq \emptyset$ if and only if $\Pi_{\psi}$ contains a tempered representation  
 and $\psi^{\Delta}=\phi$. Note that $\phi^{\Delta}=\phi$ and $\psi^{\Delta}$ is always a tempered $L$-parameter. 

 Now we assign an extended multi-segment for each supercuspidal representation $\pi$. By Theorem \ref{thm characterizatioin of supercuspidal}, the $L$-parameter of $\pi$ is of the form
$$
\phi=\bigoplus_{i=1}^{n_1}\rho_i\otimes (S_1+S_3+\cdots+S_{2A_i+1})\bigoplus_{i=n_1+1}^{n_1+n_2}\rho_i\otimes (S_2+S_4+\cdots+S_{2A_i+1}),
$$
where each $\rho_i$ is distinct and $A_i\in\mathbb{Z}$ if $i=1,\dots,n_1$ and $A_i\in\half{1}+\mathbb{Z}$ otherwise. Moreover, $\pi\cong\pi(\EE_0')$ where 
\begin{align*}
    \EE_0'=& \cup_{i=1}^{n_1} \{([ 0, 0]_{\rho_i},0,\eta_i),([ 1, 1]_{\rho_i},0,-\eta_i), \dots,([ A_i, A_i]_{\rho_i},0,(-1)^{A_i}\eta_i)\} \\
    &\cup_{i=n_1+1}^{n_1+n_2} \{([ 1/2, 1/2]_{\rho_i},0,-1),([ 3/2,3/2]_{\rho_i},0,1), \dots,([ A_i, A_i]_{\rho_i},0,(-1)^{A_i+\half{1}})\}.
\end{align*}
We remark that the sign condition \eqref{eq sign condition} should hold. In other words, we assume
\begin{align}\label{eq sign eq sc}
    \prod_{i=1}^{n_1}  (-1)^{[\frac{A_i+1}{2}]} \eta_i^{A_i+1} \prod_{i=n_1+1}^{n_2}  (-1)^{[\frac{A_i+1/2}{2}]} (-1)^{A_i+1/2} = 1.
\end{align}
By applying a sequence of union-intersection operators, we obtain that $\pi\cong\pi(\EE_0)$ where \begin{align}\label{eq EE_0}
    \EE_0= \cup_{i=1}^{n_1} \{([ A_i, 0]_{\rho_i},0,\eta_i)\} \cup_{i=n_1+1}^{n_1+n_2} \{([A_i,1/2]_{\rho_i},0,-1)\}.
\end{align}

 Our first goal in this subsection is to describe the set 
 \[E_{\pi}= \{\EE \ | \ \pi(\EE)\cong \pi\}/\textrm{(row exchanges)}.\] 
 One can directly apply Theorem \ref{thm exhaustion of symbol} to construct this set from $\EE_0$; however, we provide another construction which gives a better explanation of the size of $E_{\pi}$.

The construction is based on an induction on $\alpha(\pi):=\sum_{i=1}^{n_1+n_2} \lfloor A_i\rfloor$. The base case is $A_i=0$ or $1/2$ for every $i$. In this case, the only operators applicable on $\EE_0$ are partial duals for $ n_1+1 \leq i \leq n_1+n_2$. Therefore, there are exactly $2^{n_2}$ local Arthur packets containing $\pi$, with corresponding multi-segments being
\[  \left\{\bigcup_{i=1}^{n_1} \{([ 0, 0]_{\rho_i},0,\eta_i)\} \bigcup_{i=n_1+1}^{n_1+n_2} \{([1/2, \pm 1/2]_{\rho_i},0,\mp 1) \}\right\}. \]

Suppose $\alpha(\pi)>0$. We are going to construct another supercuspidal representation $\pi^{-}$ with $\alpha(\pi^{-})= \alpha(\pi)-1$. Fix an $i$ such that $ \lfloor A_i \rfloor > 0$. From now on, we denote $\rho= \rho_i$, $A:=A_i$, $\epsilon:=B_i$ and $\eta:=\eta_i$. We define $\EE_0^{-}$ by formally replacing $\{([A,\epsilon]_{\rho},0,\eta)\}$ in $\EE_0$ with $\{([A-1,\epsilon]_{\rho},0,\eta)\}$. Then we fix a $\rho^{\ast}$ distinct from $\{\rho_1,\ldots,\rho_{n_1+n_2}\}$ such that
\[ \EE_{0}^{-} \cup \{ ([0,0]_{\rho^{\ast}}, 0,\eta^{\ast})\}\]
is an extended multi-segment of the same type of group, where $\eta^{\ast}$ is determined by \eqref{eq sign eq sc}. We set
\[ \pi^{-}:=\pi(\EE_{0}^{-} \cup \{ ([0,0]_{\rho^{\ast}}, 0,\eta^{\ast})\}),\]
which is supercuspidal by Theorem \ref{thm characterizatioin of supercuspidal}.

We know how to construct $E_{\pi^{-}}$ by induction hypothesis, and any extended multi-segment in $E_{\pi^{-}}$ can be uniquely written as 
\[ \EE^{-} \cup \{ ([0,0]_{\rho^{\ast}},0, \eta^{\ast})\}.\]
Denote the set of such $\EE^{-}$ (up to admissible order) by $E^{-}_{\pi}$. Towards constructing $E_{\pi}$, for each $\EE^{-} \in E^{-}_\pi$, we define three extended multi-segments $\EE^{-}_{1},\EE^{-}_{2},\EE^{-}_3$ whose representations are $\pi$ in the following way:
\begin{enumerate}
    \item [$\oldbullet$] $\EE_1^{-}: = (\EE^{-})^{\rho} \cup((\EE^{-})_{\rho}+ \{([A,A]_{\rho},0, (-1)^{A-\epsilon} \eta)\}) $,
    \item [$\oldbullet$] $\EE_2^{-}:= dual(\EE_1^{-})$,
    \item [$\oldbullet$] Write $(\EE_1^{-})_{\rho}= \{([A_r,B_r]_{\rho},l_r, \eta_r)\}_{r \in (I_{\rho},>)}$. Denote $n$ the maximal element in $(I_{\rho},>)$, which corresponds to the row $([A,A]_{\rho},0, (-1)^{A-\epsilon} \eta)$. By Theorem \ref{thm max A}, there is a unique $j$ such that $A_j= A-1$. Then we define $\EE_{3}^{-}:= (\EE^{-})^{\rho} \cup ui_{j,n}((\EE_1^{-})_{\rho})$.

\end{enumerate}
Clearly $\pi(\EE_3^{-}) \cong \pi(\EE_1^{-}) \cong \pi \cong \hat{\pi} \cong \pi(\EE_2^{-})$. The following lemma shows that $ui_{j,n}$ is applicable on $\EE_1^{-}$ of type 3', so $\EE_1^{-} $ is different from $\EE_3^{-}$.

\begin{lemma}\label{lem l_n=0}
Suppose $\EE \in \Rep$ and denote $\FF=\EE_{\rho}$. Write
\[ \FF= \{([A_i,B_i]_{\rho},l_i,\eta_i)\}_{i =1}^{n}, \ \FF'=\{([A_i',B_i']_{\rho},l_i',\eta_i')\}_{ i =1}^{n'}  \]
and suppose further that they satisfy the following conditions.
\begin{enumerate}
    \item [$\oldbullet$] $\pi(\EE^{\rho} \cup \FF_{\gg}) \cong \pi(\EE^{\rho} \cup \FF_{\gg'}') $ for some admissible orders $\gg$ and $\gg'$ satisfy $(P')$.  
    \item [$ \oldbullet$]
    $ A_n= \max\{ A_i \ | \ i \in I_{\rho} \}$, $B_n= \max\{ B_j \ | \ A_j=A_n \}$, and similarly $ A_{n'}'=\max\{ A_{i}' \ | \ i \in I_{\rho} \}$, $B_{n'}'= \max\{ B_j' \ | \ A_j'=A_n' \}$.
    \item [$\oldbullet$] $l_n=0$.
\end{enumerate}
Then $l_{n'}'=0$ and $(-1)^{A_n-B_n} \eta_n= (-1)^{A_{n'}'-B_{n'}'}\eta_{n'}'$.
\end{lemma}
\begin{proof}
First we give an outline of the idea of this proof. Let $\eta= (-1)^{A_n-B_n}\eta_n.$ Then $\FF+ \{([A_n,A_n]_{\rho},0, -\eta)\}$ does not satisfy Proposition \ref{prop positive non-vanishing}(i), and hence neither does $\FF'+ \{([A_n,A_n]_{\rho},0, -\eta)\}$. This gives the desired conclusion. 

 We proceed with the details. For a large integer $t$, we consider the following extended multi-segments
\begin{align*}
    \EE_t:=&\EE^{\rho} \cup ( \FF_{\gg}+ \{ ([A_n+t,A_n+t]_{\rho},0,-\eta), ([A_n+t+1,A_n+t+1]_{\rho},0,-\eta) \}),\\
   \EE_t':= &\EE^{\rho} \cup ( \FF_{\gg}'+ \{ ([A_n+t,A_n+t]_{\rho},0,-\eta), ([A_n+t+1,A_n+t+1]_{\rho},0,-\eta) \}).
\end{align*}
 These extended multi-segments give nonzero representations by Theorem \ref{thm non-vanishing} and are isomorphic by Lemma \ref{lemma far away}. We have
\begin{align*}
    &\pi(\EE^{\rho} \cup( \FF_{\gg}+ \{([A_n,A_n]_{\rho},0, -\eta)\}+ \{([A_n+t+1,A_n+t+1]_{\rho},0, -\eta)\}) )\\
    =& D_{\rho|\cdot|^{ A_n+t,\dots,A_n+1} } (\pi(\EE_t))\\
    =& D_{\rho|\cdot|^{ A_n+t,\dots,A_n+1} } (\pi(\EE_t'))\\
    =&\pi(\EE^{\rho} \cup( \FF_{\gg'}'+ \{([A_n,A_n]_{\rho},0, -\eta)\}+ \{([A_n+t+1,A_n+t+1]_{\rho},0, -\eta)\}) ).
\end{align*}
The first representation is zero since $\FF+ \{([A_n,A_n]_{\rho},0, (-1)^{A_n-B_n+1} \eta_n)\}$ contains two adjacent rows
\[ \{([A_n,B_n]_{\rho},0,\eta_n) , ([A_n,A_n]_{\rho},0, (-1)^{A_n-B_n+1} \eta_n)\} \]
which do not satisfy Proposition \ref{prop positive non-vanishing}(i)(3). It remains to check that the last representation vanishes if and only if $l_{n'}'=0$ and $ -\eta= (-1)^{A_{n'}'-B_{n'}'+1} \eta_{n'}'$. Indeed, since $\EE^{\rho} \cup \FF'$ satisfies Theorem \ref{thm non-vanishing}(ii), we only need to check the following adjacent rows in $\FF'+ \{([A_n,A_n]_{\rho},0, (-1)^{A_n-B_n+1} \eta_n)\}$
\[ \{([A_{n'}',B_{n'}']_{\rho},l_{n'}',\eta_{n'}') , ([A_n,A_n]_{\rho},0, (-1)^{A_n-B_n+1} \eta_n)\} \]
by Proposition \ref{prop positive non-vanishing}(ii). Then Proposition \ref{prop positive non-vanishing}(i)(1) fails if and only if $l_{n'}'=0$ and $ -\eta= (-1)^{A_{n'}'-B_{n'}'+1} \eta_{n'}'$.
This completes the proof of the lemma.
\end{proof}

We give an example to explain the construction of $\EE_1^{-},\EE_2^{-},\EE_3^{-}$.
\begin{exmp} \label{ex supercuspidal}
Let $\rho$ be the trivial representation. Let $\pi=\pi(0^{-},1^+,2^{-})$ be a supercuspidal representation of $\Sp_{8}(F)$. Set
\[ \EE_0= \bordermatrix{  &0&1&2 \cr 
&\ominus & \oplus& \ominus \cr}_{\rho}.\]

It is not hard to show that $E_{\pi}^{-}$ is 
\[ \left\{  \bordermatrix{
&0&1 \cr 
&\ominus & \cr
& &\oplus \cr
}_{\rho}, \bordermatrix{
&-1&0&1 \cr 
&\lhd&\ominus &\rhd \cr
& &\oplus& \cr
}_{\rho}, \bordermatrix{
&0&1 \cr 
&\ominus & \oplus \cr
}_{\rho}
\right\} .\]

Then for each $\EE^{-} \in E_{\pi}^{-}$, the corresponding $\EE^{-}_1$ are
\[  \bordermatrix{&0&1&2 \cr 
&\ominus&& \cr
&&\oplus& \cr
&&&\ominus \cr
}_{\rho}, \bordermatrix{
    &-1&0&1&2 \cr
    &\lhd & \ominus&\rhd& \cr
    & &\oplus&&\cr  
    &&&&\ominus \cr
    }_{\rho}, \bordermatrix{&0&1&2 \cr 
&\ominus& \oplus& \cr
&&&\ominus \cr}_{\rho} \ . \]
The corresponding $\EE_2^{-}$ are 
\[ \bordermatrix{
    &-2&-1&0&1&2 \cr
    &\lhd&\lhd & \ominus&\rhd&\rhd \cr
    &&\lhd &\oplus&\rhd&\cr
    &&&\ominus&&\cr
    }_{\rho}, \bordermatrix{
    &-2&-1&0&1&2 \cr
    &\lhd&\lhd & \ominus&\rhd&\rhd \cr
    && &\oplus&&\cr
    &&&&\ominus&\cr
    }_{\rho}, \]
    \[ \bordermatrix{
    &-2&-1&0&1&2 \cr
    &\lhd&\lhd & \ominus&\rhd&\rhd \cr
    && &\oplus&\ominus&\cr    
    }_{\rho}\ .   \]
Finally, the corresponding $\EE^{-}_3$ are \[ \bordermatrix{&0&1&2 \cr 
&\ominus&& \cr
&&\oplus&\ominus \cr}_{\rho},  \bordermatrix{
    &-1&0&1&2 \cr
    &\lhd & \ominus&\oplus&\rhd \cr
    & &\ominus&&\cr    
    }_{\rho}, \bordermatrix{  &0&1&2 \cr 
&\ominus & \oplus& \ominus \cr}_{\rho}\ . \]
As we have shown in Example \ref{ex Atobe}, this exhausts the set $\Psi(\EE_0)=E_{\pi}$.
\end{exmp}

Now we show that the construction exhausts the set $E_{\pi}$.
\begin{thm}\label{thm surjection cuspidal}
Continue the notation from above. The map
\begin{align*}
   E_{\pi}^{-} \times \{1,2,3\} &\to E_{\pi}\\
    (\EE^{-},i)&\mapsto \EE_{i}^{-}
\end{align*}
is bijective. In particular, we have 
\[ |E_{\pi}|=2^{n_2} 3^{ \sum_{i=1}^{n_1+n_2} \lfloor A_i \rfloor}. \]
\end{thm}
\begin{proof}
We first show the map is surjective. Suppose $$\EE=\EE^{\rho} \cup \{([A_i,B_i]_{\rho},l_i,\eta_i)\}_{i \in (I_{\rho, >})} \in E_{\pi}.$$ Then by Theorem \ref{thm max A}, there exists one and only one $j \in I_{\rho}$ such that $A_j=A$. Since $A_j \geq B_j \geq  -A_j$, we separate into the following three cases.
\begin{enumerate}
    \item [\textbf{Case 1:}] $A_j=B_j$. Then $\EE= \EE_1^{-}$, where $\EE^{-}$ is defined by removing the $j$-th row of $\EE_{\rho}$. It is not hard to check $\EE^{-} \in E_{\pi}^{-}$.
    \item [\textbf{Case 2:}] $-A_j=B_j$. Then $dual(\EE)= \EE^{-}_1$ for some $\EE^{-} \in E_{\pi}^{-}$ by Case 1, and hence $ \EE= \EE^{-}_{2}$.
    \item [\textbf{Case 3:}]  $-A_j <B_j<A_j $. Identify $(I_{\rho},>)$ with the set $ \{1 ,\ldots ,n\}$ where $1 < \cdots <n$. Then we define a new admissible order $\gg$ on $I_{\rho}$ by
    \[ 1 \ll \cdots \ll j-1 \ll j+1 \ll \cdots \ll n \ll j. \]
    Write $\EE_{\rho,\gg}= \{ ([A_i,B_i]_{\rho}, l_i',\eta_i')\}_{i \in (I_{\rho},\gg)}$
     By Lemma \ref{lem l_n=0}, we have $l_r'=0$. Then applying Corollary \ref{cor split}, $ui_j^{-1}$ is applicable on $\EE_{\rho,\gg}$, and hence $\EE= \EE_{3}^{-}$ for some $\EE^{-} \in E_{\pi}^{-}$.
\end{enumerate}
This argument also gives the inverse of the map, which completes the proof of the theorem.
\end{proof}

We remark that the size $|E_{\pi}|$ also follows from M{\oe}glin's result (Theorem \ref{thm GGP20}).

The following definition rephrases the condition that $\psi^{\Delta}$ is discrete and without gaps. 
\begin{defn}
Let $S_{\rho}= \{ [A_i,B_i]_{\rho}\}_{i=1}^n$ be a finite multi-set of segments with the same $\rho$. We say that $S_{\rho}$ is of cuspidal shape if 
\begin{enumerate}
    \item [$\oldbullet$] $\{[ A_i, |B_i| ]_{\rho}\}_{i=1}^n$ is disjoint, and 
    \item [$\oldbullet$] if $A= \max\{ A_1,\ldots,A_n \}$, we have 
\[ \cup_{i=1}^{n} [A_i, |B_i|]_{\rho}= [A, B]_{\rho}  \]
for $B=0 $ or $1/2$.
\end{enumerate}
For a general finite multi-set of segments $S$, we decompose it according to supercuspidal representation $\rho$ as $S= \cup_{\rho} S_{\rho}$. Then we say $S$ is of cuspidal shape if every $S_{\rho}$ is of cuspidal shape.

We say a local Arthur parameter $\psi$ (resp, an extended multi-segment $\EE$) is of cuspidal shape if $\supp(\psi)$ (resp. $\supp(\EE)$) is of cuspidal shape.
\end{defn}

Let $\pi$ be a supercuspidal representation. By Theorem \ref{thm surjection cuspidal}, every extended multi-segment in $E_{\pi}$
is of cuspidal shape. Our next goal is to determine which extended multi-segment of cuspdial shape gives a supercuspidal representation. We define a new condition on an admissible order. It is useful to detect whether $\pi(\EE)$ is tempered or supercuspidal.
\begin{defn}
We say an admissible order $\gg$ of $I_{\rho}$ satisfies ($P''$) if for any $i,j \in I_{\rho}$,
\[ \begin{cases} A_i> A_j, \text{ or}  \\ A_i=A_j, B_i>B_j  \end{cases}\Longrightarrow i\gg j.  \]
We denote $\EE' \in \Rep^{(P'')}$ if there exists an $\EE \in \Rep$ such that $(\EE')_{\rho}= (\EE_\rho)_{\gg_{\rho}}$ for a collection of admissible order $\gg_{\rho}$ on $I_{\rho}$ satisfying $(P'')$. Finally, for $\EE \in \Rep$, we denote $\EE^{(P'')}$ the unique extended multi-segment in $\Rep^{(P'')}$ obtained from $\EE$ by row exchange. We define $ \EE_{\rho}^{(P'')}= (\EE^{(P'')})_{\rho}$. 
\end{defn}

The following corollary gives a necessary and sufficient condition on $\EE \in \Rep$ of cuspidal shape such that $\pi(\EE)$ is supercuspidal.
\begin{cor}\label{cor cuspidal shape}
Suppose $\EE \in \Rep$ is of cuspidal shape. As above, we write 
\[\EE^{(P'')}= \cup_{\rho}\{ ([A_i,B_i]_{\rho},l_i,\eta_i)\}_{i \in (I_{\rho},>)}.\]
Denote $n_{\rho}$ the maximal element in $(I_{\rho},>)$, $A_{\rho}:= \max\{ A_i \ | \ i \in I_{\rho}\}$ and take $\epsilon_{\rho} \in \{0,1/2\}$ such that $ A_{\rho}+\epsilon_{\rho} \in \Z$. Then 
\begin{enumerate}
    \item [(i)] The followings are equivalent.
    \begin{enumerate}
        \item [(a)] $\pi(\EE^{\rho} \cup \EE_{\rho})\cong \pi(\EE^{\rho} \cup \{ ([A_{\rho},\epsilon_{\rho}]_{\rho},0,\eta) \}) $ for some $\eta \in \{\pm 1\}$. We require $\eta =-1$ if $\epsilon_{\rho}=1/2$.
        \item [(b)] $\EE_{\rho}^{(P'')}=\{ ([A_i,B_i]_{\rho},l_i,\eta_i)\}_{i \in (I_{\rho},>)}$ satisfies the following conditions: For all $i,j \in I_{\rho}$,
        \begin{enumerate}
    \item [(1)] $l_i=0$, 
    \item [(2)] if $B_{j}=A_{i}+1$, then $(-1)^{A_{i}-B_i}\eta_{i}\eta_{j}=-1$,
    \item [(3)] if $B_i=1/2$, then $\eta_i=-1$.
\end{enumerate}
    \end{enumerate}
In this case, $\eta= (-1)^{B_{n_{\rho}}-\epsilon} \eta_{n_{\rho}}$. We say $\EE_{\rho}$ satisfies the \emph{cuspidal condition} if Condition (b) holds.
\item [(ii)] $\pi(\EE)$ is supercuspidal if and only if $\EE_{\rho}$ satisfies the cuspidal condition for all $\rho$. In this case, we have $\pi(\EE)= \pi(\phi,\varepsilon)$ where 
\begin{enumerate}
    \item [$\oldbullet$] $\phi= \sum_{\rho} \rho \otimes S_{ 2 A_{\rho} +1}$,
    \item [$\oldbullet$] if $ \rho \otimes S_{2} \subset \phi$, then $\varepsilon(\rho \otimes S_{2})=-1$,
    \item [$\oldbullet$] if $\rho \otimes S_{1} \subset \phi$, then $\varepsilon(\rho \otimes S_1)=(-1)^{B_{n_{\rho}}} \eta_{n_{\rho}}$.
\end{enumerate}
The rest of the values of $\varepsilon$ are determined by Theorem \ref{thm characterizatioin of supercuspidal}.
\end{enumerate}
\end{cor}
\begin{proof}
For Part (i), we first show that Condition (a) implies Condition (b). We apply induction on $A_{\rho}$. The base case that $A_{\rho}= \epsilon$ is obvious.

If $A_{\rho} > \epsilon$, Theorem \ref{thm surjection cuspidal} shows that there exists an $\FF$, which is obtained from $\{([A_{\rho}-1,\epsilon]_{\rho},0,\eta)\}$ by a sequence of basic operators and partial dual, such that $ \EE_{\rho}$ is in one of the following forms
\begin{enumerate}
    \item [$\oldbullet$] $\FF + \{([A_{\rho}, A_{\rho}]_{\rho},0, (-1)^{A_{\rho}-\epsilon}\eta)\}$,
    \item [$\oldbullet$] $dual(\FF + \{([A_{\rho}, A_{\rho}]_{\rho},0, (-1)^{A_{\rho}-\epsilon}\eta)\})$,
    \item [$\oldbullet$] $ ui_{j,n}(\FF + \{([A_{\rho}, A_{\rho}]_{\rho},0, (-1)^{A_{\rho}-\epsilon}\eta)\}) $ where $n$ is the index of the last row $\{([A_{\rho}, A_{\rho}]_{\rho},0, (-1)^{A_{\rho}-\epsilon}\eta)\}$ and $ui_{i,n}$ is applicable of type 3'.
\end{enumerate}
The induction hypothesis shows $\FF$ satisfies the cuspidal condition, and hence $\EE_{\rho}$ also satisfies the cuspidal condition. This completes the proof of this direction.

Next, we show Condition (b) implies Condition (a). We apply induction on $n:= |I_{\rho}|$. There is nothing to prove if $n=1$.

Assume $n>1$. If there exist $i,j \in I_{\rho}$ satisfying the Condition (2), then $(i,j,>)$ is an adjacent pair and $ui_{i,j}$ is applicable on $\EE_{\rho}$ of type 3' in Definition \ref{def ui}, and we reduce $n$ by 1. If Condition (2) is empty on $\EE_{\rho}$, then $\EE_{\rho}$ being of cuspidal shape implies $B_{i+1}=-A_i-1$ for any adjacent $i<i+1$. Now from Condition (1), for all $i$, $add_{i}^{-1}(\EE_{\rho})$ doesn't satisfy the non-vanishing conditions in Proposition \ref{prop positive non-vanishing}(i). Then Theorem \ref{thm Aubert-Zelevinsky dual formula} implies for all $i$, $ dual( add_{i}^{-1}(\EE_{\rho}))= sh^{-1}_{i}(dual(\EE_{\rho}))$ doesn't satisfy the non-vanishing conditions, either. Therefore, $ui_{i+1,i}$ is applicable on $dual(\EE_{\rho})$ of type 3' for all $i \in I_{\rho}$. It follows that $(dual(\EE_{\rho}))^{min}=([A, \epsilon]_{\rho},0,\eta)$ with $ \epsilon \in \{0, \pm 1/2\}$ so that $ A+\epsilon \in \Z$. Finally, Condition (3) shows that $\eta=-1$ if $\epsilon=1/2$, and hence $dual((dual(\EE_{\rho}))^{min})$ or its partial dual is of the form we want. This completes the proof of this direction.

Furthermore, Lemma \ref{lem l_n=0} implies 
\[ (-1)^{A_{\rho}- \epsilon_{\rho}}\eta=(-1)^{A_{n_{\rho}}-B_{n_{\rho}}}\eta_{n_{\rho}},\]
and hence the last assertion follows. This completes the proof of Part (i).

For Part (ii), $\pi(\EE)$ is supercuspidal if and only if $\pi(\EE) \cong \pi(\EE_0)$ for an $\EE_0$ of the form \eqref{eq EE_0}. Applying Part (i) for each $\rho$ and Lemma \ref{lemma far away}(i), we obtain the equivalence. The description of $\pi(\EE)$ is given by the last assertion of Part (i).
This completes the proof of the corollary.
\end{proof}

Suppose $\EE_{\rho}$ is of cuspidal shape and satisfies the cuspidal condition. Denote $\EE_{\rho}^{(P'')}=\{([A_i,B_i]_{\rho},l_i,\eta_i)\}_{i \in (I_{\rho,>})}$ and let $1$ denote the minimal element of $(I_{\rho},>)$. If we fix $\eta_1$, then all the other $\eta_i$ are determined by Corollary \ref{cor cuspidal shape}(i)(b)(1) and Proposition \ref{prop positive non-vanishing}(i) or Corollary \ref{cor cuspidal shape}(i)(b)(2). Moreover, Corollary \ref{cor cuspidal shape}(i)(b)(3) implies that we have no choice for $\eta_1$ in half integer cases. As a consequence, we can count the number of supercuspidal representations inside a local Arthur packet by taking the sign condition \eqref{eq sign condition} into account.

\begin{thm}\label{thm cuspidal count}
Let $\psi \in \Psi^+(G_n)$.
\begin{enumerate}
    \item [1.] $\Pi_{\psi}$ contains a supercuspidal representation only if $\psi$ is of good parity and of cuspidal shape.
    \item [2.] Suppose $\psi$ is of good parity and of cuspidal shape. Decompose $$\psi= \sum_{i=1}^{n_1} \psi_i+ \sum_{i=n_1+1}^{n_1+n_2} \psi_i,$$ where
    \[ \psi_i= \rho_i \otimes \left( \sum_{j\in I_{\rho}} S_{a_{i,j}} \otimes S_{b_{i,j}} \right), \]
    such that
    \begin{enumerate}
        \item [$\oldbullet$] each $\rho_i$ is distinct, and
        \item [ $\oldbullet$] $a_{i,j} + b_{i,j} \equiv \begin{cases}0 \mod 2 & \text{ if } 1 \leq i\leq n_1,\\ 1 \mod 2 & \text{ if } n_1+1 \leq i\leq n_1+n_2. \end{cases}  $
    \end{enumerate}
    Denote $d_i= \left\lfloor\max_{j}\left\{\frac{a_{i,j}+b_{i,j}}{2} \right\}\right\rfloor$ and
    \[ \epsilon_i= \begin{cases}
    1 &\text{ if } \begin{cases}
    d_i \equiv 2 \mod 4 &\text{ if } 1 \leq i \leq n_1,\\
    d_i \equiv 1 \text{ or }2 \mod 4 &\text{ if } n_1+1 \leq i \leq n_1+n_2,
    \end{cases}  \\
    0& \text{ otherwise.}
    \end{cases}\]
     Then the number of supercuspidal representations in $\Pi_{\psi}$ is given by
    \[ \begin{cases}
    2^{n_1-1} \ &\text{ if } d_i \equiv 1 \mod 2 \text{ for some } 1 \leq i \leq n_1, \\
    2^{n_1}\ &\text{ if }d_i\equiv 0 \mod 2 \text{ for any }1 \leq i \leq n_1, \text{ and } \sum_{i=1}^{n_1+n_2} \epsilon_i \equiv 0 \mod 2, \\
    0\ &\text{ otherwise.}
    \end{cases}\]
    Moreover, the $L$-data of these supercuspidal representations can be listed as in Corollary \ref{cor cuspidal shape}.
\end{enumerate}
\end{thm}

\subsection{Non-negative local Arthur packets containing tempered representations}

In this subsection, we classify all non-negative local Arthur packets (see Definition \ref{def non-neg arthur para} below) containing tempered representations, more precisely, we classify all non-negative extended multi-segments $\EE$ such that $\pi(\EE)$ is tempered. This is a precursor to the general case considered in the next section, with simpler statement and argument. 


We begin by defining non-negative local Arthur parameters and packets.


\begin{defn}\label{def non-neg arthur para}
Suppose $\psi$ is a local Arthur parameter of good parity.

\begin{enumerate}
    \item [(1)] If we write $\psi= \bigoplus_{i=1}^n \rho_i \otimes S_{a_i} \otimes S_{b_i}$, then $\psi$
    (resp. $\Pi_{\psi}$) 
    is said to be non-negative if $a_i \geq b_i$ for all $1\leq i \leq n$. Equivalently, $\psi$ is non-negative if any extended multi-segment $\EE$ such that $\supp(\EE)= \supp(\psi)$ is non-negative.
    \item [(2)] We define $\Omega(\psi):= \Omega(\EE)$ for any extended multi-segment $\EE$ such that $\supp(\EE)= \supp(\psi)$.
\end{enumerate}
\end{defn}
We remark that if $\psi_1, \psi_2$ are both non-negative, then 
\[ \Omega(\psi_1)= \Omega(\psi_2) \Longleftrightarrow \psi_1^{\Delta}= \psi_2^{\Delta}.\]

Suppose $\EE$ is non-negative, now we give a necessary condition on $\supp(\EE)$ such that $\pi(\EE)$ is tempered. 
\begin{defn}\label{def pos chain shape}
We define a multi-set of segments $S=\{ [A_i,B_i]_{\rho_i} \}_{i=1}^n$ with $B_i \geq 0$ is of chain shape if the following hold: 
\begin{enumerate}
    \item [(i)] For any $i,j$, $|[A_i,B_i]_{\rho_i} \cap [A_j,B_j]_{\rho_j}| \leq 1$.
    \item [(ii)] Suppose $A_i>B>B_i$ for some $1 \leq i\leq n$, then the multiplicity of $[B,B]_{\rho_i}$ in $S$ is even.   
\end{enumerate}
Suppose $\psi$ (resp. $\EE$) is a non-negative local Arthur parameter (resp. extended multi-segment). We say $\psi$ (resp. $\EE$) is of chain shape if $ \supp(\psi)$ (resp. $\supp(\EE)$) is of chain shape.
\end{defn}

\begin{exmp}\label{exmp chain shape}
Here are two examples of $\EE \in \Rep^{(P'')}$ of chain shape.
\begin{align*}
 \EE_1=& \bordermatrix{
 &1&2&3&4&5&6&7 \cr
&\oplus&&&&&&&\cr 
&\oplus&\ominus&\oplus&\ominus&&&\cr
&&&&\ominus&&&\cr
&&&&\ominus&&&\cr
&&&&\ominus &\oplus&\ominus&\oplus\cr
 }_{\rho},\\
 \EE_2=& \bordermatrix{
 &1&2&3&4&5&6&7 \cr
&&\ominus&&&&&\cr
&&\ominus&&&&&\cr
&\ominus&\oplus&\ominus&\oplus&&&&\cr 
&&&& &&\oplus&\cr
&&&& &&\oplus&\cr
&&&&\oplus&\ominus&\oplus&\ominus\cr
 }_{\rho}.   
\end{align*}

We have that
\begin{align*}
    \pi(\EE_1)&= \pi(1^{+},1^{+},2^{-},3^{+},4^{-},4^{-},4^{-},4^{-},5^{+},6^{-},7^{+} ),\\
    \pi(\EE_2)&= \pi(1^{-},2^{+},2^{+},2^{+},3^{-},4^{+},4^{+},5^{-},6^{+},6^{+},6^{+},7^{-} )
\end{align*}
are both tempered (see Corollary \ref{cor pos tempered} below).
\end{exmp}

Now we show that a non-negative extended multi-segment whose representation is tempered must be of chain shape.
\begin{prop}\label{prop pos chain shape} 
Suppose $\EE \in \Rep$ is non-negative and $\pi(\EE)$ is tempered. Then $\EE$ is of chain shape.
\end{prop}
\begin{proof}
We denote 
\[\FF=\EE_{\rho}=\{([A_i,B_i]_{\rho},l_i,\eta_i)\}_{i=1}^n\]
and check that it is of chain shape. We apply induction on $n$. There is nothing to prove when $n=1$. After applying several row exchanges, we assume $\FF= \FF^{(P'')}$.  Denote $A= \max\{A_i\ | \ i \in I_{\rho}\}$ and $k= \#\{i \in I_{\rho}\ | \ A_i =A\}.$ From Theorem \ref{thm Arthur tempered}, we can find a tempered extended multi-segment $\EE_{temp}$ such that $\pi(\EE_{temp}) \cong \pi(\EE)$. We denote $\FF_{temp}= (\EE_{temp})_{\rho}$.

Here is the key observation that will be used repeatedly in this proof. Suppose $ \FF= \FF_1 + \FF_2$, $\FF_{temp}= \FF_{1,temp} + \FF_{2,temp}$ with $\FF_{2}= \FF_{2,temp}$. Then we have $\FF_1+ \FF_2$ is of chain shape if and only if $\FF_1$ is of chain shape. On the other hand, by Corollary \ref{cor shift add}(i) we have
\[\pi(\EE^{\rho} \cup \FF_{1} \cup (sh^1(\FF_2))_{\rho^{\ast}}) \cong \pi(\EE^{\rho} \cup \FF_{1,temp} \cup (sh^1(\FF_{2,temp}))_{\rho^{\ast}}),\]
which are also tempered. By abuse of notation, we will write
\[ \pi(\EE^{\rho} \cup \FF_1) \cong \pi(\EE^{\rho} \cup \FF_{1,temp})\]
in this case, which should be understood that we include $ (sh^1(\FF_2))_{\rho^{\ast}}$ in $\EE^{\rho}$. In summary, we can cancel $\FF_2=\FF_{2,temp}$ in the argument.

 We separate the argument into following five  steps.
 
\textbf{Step 1.} Suppose $A_n=B_n$. Then $\FF$ is of chain shape by the key observation and induction hypothesis. 

\textbf{Step 2.} We check that $A_n>B_n$ implies  $A_{n-1}<A_n$ in this step.
    
Suppose the contrary, that is $A_{n-1}=A_n$. Lemma \ref{lem l_n=0} implies $l_n=0$ by comparing $\FF$ and $\FF_{temp}$, and hence we can use Corollary \ref{cor split} to split the last row of $\FF$. It becomes
\[ \{ ([A_i,B_i]_{\rho}, l_i, \eta_i)\}_{i=1}^{n-1} + \{([A_n-1,B_n]_{\rho}, 0, \eta_n)\}+\{([A_n,A_n]_{\rho}, 0, (-1)^{ A_n-B_n}\eta_n). \]
Then we exchange the $n$-th and the $(n-1)$-th row. It becomes $\FF_1+ \FF_2+\FF_3+\FF_4$ where
\begin{align*}
\FF_1&=\{([A_i,B_i]_{\rho},l_i, \eta_i)\}_{i=1}^{n-2},\\
    \FF_2&=\{([A_n-1,B_n]_{\rho},0, (-1)^{A_{n}-B_{n-1}}\eta_n)\},\\
    \FF_3&=\{([A_{n},B_{n-1}]_{\rho}, l_{n-1}', \eta_{n-1}')\},\\
    \FF_4&=\{([A_n,A_n]_{\rho}, 0, (-1)^{ A_n-B_n}\eta_n)\}.
\end{align*}

We first show that $l_{n-1}'=0$. Since $\FF_4$ is the same as the last row of $\FF_{temp}$, by the key observation, we may cancel then and hence Lemma \ref{lem l_n=0} implies $l_{n-1}'=0$.

Next, we check Proposition \ref{prop positive non-vanishing}(i) on $\FF_2+\FF_3+\FF_4$.  Proposition \ref{prop positive non-vanishing}(i)(2) for $\FF_2+\FF_3$ implies
\[(-1)^{A_n-1- B_n}(-1)^{A_{n}-B_{n-1}} \eta_n \eta_{n-1}'=1 .\]
However, then
\[ (-1)^{A_n-B_{n-1}} \eta_{n-1}' (-1)^{A_n-B_n}\eta_n=-1,\]
which contradicts to condition Proposition \ref{prop positive non-vanishing}(i)(3) for the $\FF_3+\FF_4$.

\textbf{Step 3.} In this step, we show that if $ A_{n-1} \leq B_n$, then $\FF$ is of chain shape.

In this case, $\FF$ is of chain shape if and only if $\FF_1= \{ ([A_i,B_i]_{\rho},l_i,\eta_i)\}_{i=1}^{n-1}$ is of chain shape. Lemma \ref{lem l_n=0} implies $l_n=0$, and hence we may apply Corollary \ref{cor split} successively to the last row of $\FF$. We obtain 
\[ \FF'= \FF_1 + \{([B_n+r,B_n+r]_{\rho},0, (-1)^r \eta_n)\}_{r=0}^{A_n-B_n}  \]
such that $\pi(\EE^{\rho} \cup \FF')\cong \pi(\EE^{\rho} \cup \FF_{temp})$. Then the key observation shows that $\pi(\EE^{\rho} \cup \FF_1)$ is also tempered, and hence induction hypothesis implies $\FF_1$ is of chain shape.


\textbf{Step 4.} Suppose $A_{n-1} >B_n$. We show that any row in $\FF$ of the form $([A_{n-1},B_i]_{\rho},l_i,\eta_i)$ satisfies $ B_i=A_{n-1}$ in this step.

Using Corollary \ref{cor split} again, we may replace the last row of $\FF$ by
\[ \{( [  A_{n-1}, B_n]_{\rho},0,\eta_n)\} +\{ ([A_{n-1}+r,A_{n-1}+r]_{\rho},0,\ast)\}_{r=1}^{ A_n-A_{n-1}}. \]
Then the key observation shows if we denote
\[ \FF_1 :=\{([A_i,B_i]_{\rho},l_i,\eta_i) \}_{i=1}^{n-1}+ \{([A_{n-1},B_n]_{\rho},0,\eta_n)\},\]
 then $\pi(\EE^{\rho} \cup \FF_1)$ is also tempered. Since $A_{n-1}>B_n$, after removing all rows of the form $ ([A_{n-1},A_{n-1}],0,\ast)$ in $\FF_1$, the rest of the rows should satisfy $A_i<A_{n-1}$ by step 2.

\textbf{Step 5.} Again suppose $A_{n-1}>B_n$. Let $s$ be the multiplicity of $\rho|\cdot|^{A_{n-1}}$ inside $\Omega(\FF)$. We  show that $s-1$ is even, and conclude that $\FF$ is of chain shape.

Looking at the $\FF_1$ defined in previous step, after row exchange, we obtain
\[ \FF_1':= \{ ([A_i,B_i]_{\rho},l_i,\eta_i)\}_{i=1}^{n-s} + \{([A_{n-1},B_n]_{\rho},l',\ast)\} + \{([A_{n-1},A_{n-1}]_{\rho},0,\ast)^{s-1}\} \]
where $A_i<A_{n-1}$ for $i \leq n-s$, and
\[ l'=\begin{cases}
1 &\text{ if }s-1 \text{ is odd,} \\
0 &\text{ if }s-1 \text{ is even.} 
\end{cases} \]
Finally, denote 
\[ \FF_2:= \{ ([A_i,B_i]_{\rho},l_i,\eta_i)\}_{i=1}^{n-s} + \{([A_{n-1},B_n]_{\rho},l',\ast)\}.\]
Applying the key observation, $\pi(\EE^{\rho} \cup \FF_2)$ is tempered, and hence Lemma \ref{lem l_n=0} shows $l'$ should be zero, so $s-1$ must be even. 

Recall
\[ \FF= \{ ([A_i,B_i]_{\rho},l_i,\eta_i)\}_{i=1}^{n-s} + \{([A_{n-1},A_{n-1}]_{\rho},0,\ast)^{s-1}\} + \{([A_n,B_n]_{\rho},0, \eta_n)\}.\]
As $s-1$ is even, $\FF$ is of chain shape if and only if $\FF_2$ is of chain shape. Since $\pi(\EE^{\rho}\cup \FF_2)$ is tempered and the number of rows of $\FF_2$ is less than $n$ ($s \geq 2$), $\FF_2$ is of chain shape by induction hypothesis.



The completes the proof of Proposition \ref{prop pos chain shape}. 
\end{proof}

Based on the proof of the previous proposition, given any non-negative extended multi-segment $\EE$, we can also give a sufficient condition for $\pi(\EE)$ being tempered.

\begin{cor} \label{cor pos tempered}
Suppose $\EE \in \Rep$ is a non-negative extended multi-segment of chain shape. Write $\EE^{(P'')}=\cup_{\rho} \{([A_i,B_i]_{\rho},l_i, \eta_i)\}_{i \in (I_{\rho},>)}$.
\begin{enumerate}
    \item [(i)] The followings are equivalent:
    \begin{enumerate}
        \item [(a)] $\pi(\EE^{\rho}\cup \EE_{\rho}) \cong \pi(\EE^{\rho} \cup \FF_{temp})$ for some tempered $\FF_{temp}$.
        \item [(b)] $\EE_{\rho}^{(P'')}= \{([A_i,B_i]_{\rho},l_i, \eta_i)\}_{i \in (I_{\rho},>)}$ satisfies $l_i=0$ for all $i \in I_{\rho}$.
    \end{enumerate}
    We say $\EE_{\rho}$ satisfies the \emph{non-negative tempered condition} if (b) holds.
    \item [(ii)] $\pi(\EE)$ is tempered if and only if $\EE_{\rho}$ satisfies the non-negative tempered condition for all $\rho$. In this case, let $\phi$ be the unique tempered local Arthur parameter with $\Omega(\phi)=\Omega(\EE)$. We have $\pi(\EE)= \pi(\phi,\varepsilon)$, where the value of $\varepsilon$ on $\rho\otimes S_{2C+1} \subset \phi$ is given by follows:
\begin{enumerate}
    \item [$\oldbullet$] If $ \rho|\cdot|^{C} \in [A_i,B_i]_{\rho}$ for some $i \in I_{\rho}$ with $A_i>B_i$, then
    \[\varepsilon(\rho\otimes S_{2C+1})= (-1)^{C-B_i}\eta_i.\]
    \item [$\oldbullet$] Otherwise, for any $i \in I_{\rho}$ such that $B_i=C=A_i$,
    \[\varepsilon(\rho\otimes S_{2C+1})=\eta_i.\]
\end{enumerate}
\end{enumerate}
\end{cor}
\begin{proof}
For (i), the proof of Proposition \ref{prop pos chain shape} already implies (a) implies (b). One can show that (b) implies (a) by applying Corollary \ref{cor split} repeatedly as in step 3 and step 4 in the proof of Proposition \ref{prop pos chain shape}. This procedure also shows (ii), and hence completes the proof of the corollary.
\end{proof}

Next, we explain that the non-negative tempered condition also imposes restrictions on $\eta_{i}$ implicitly by the non-vanishing condition in Theorem \ref{thm non-vanishing}(ii). Indeed, if we have two adjacent rows
\[\{([A_1,B_1]_{\rho}, 0, \eta_1),([A_2,B_2]_{\rho}, 0, \eta_2)\} \]
such that $[A_1,B_1]_{\rho} \cap [A_2,B_2]_{\rho} \neq \emptyset$, then they satisfy Proposition \ref{prop positive non-vanishing} if and only if
\[ (-1)^{A_1-B_1} \eta_1 \eta_{2}=1, \]
and hence $\eta_1$ is determined by $\eta_2$ and vice versa. 

Now suppose $\EE \in \Rep^{(P'')}$ is a non-negative extended multi-segment of chain shape and write $\EE_{\rho}= \{([A_i,B_i]_{\rho}, 0 ,\eta_i)\}_{ i \in (I_{\rho,>})}$. Let $i\neq j \in I_{\rho}$. By applying row exchanges, one can extend above argument to show that if $ [A_i, B_i]_{\rho} \cap [A_j,B_j]_{\rho} \neq \emptyset$, then $\eta_i$ is determined by $\eta_j$ and vice versa. Consequently, if there exists a sequence of indices $\{k_0,\cdots, k_r\}$ with $i=k_0$ and $j=k_r$ such that for $1 \leq s \leq r$,
\[ [A_{k_s}, B_{k_s}]_{\rho} \cap [A_{k_{s+1}}, B_{k_{s+1}}]_{\rho}  \neq \emptyset,\]
then $\eta_i$ is determined by $\eta_j$, and vice versa. 

Therefore, we define the connected components of a local Arthur parameter in the following sense.
\begin{defn} 
\begin{enumerate}
    \item  Suppose $S= \{ [A_i,B_i]_{\rho_i}\}_{i=1}^n$ is a multi-set of segments. We say $S$ is connected if, for any pair $1 \leq i \neq j \leq n$, there exists a sequence of indices $\{k_0,\cdots,k_r\}$ with $i=k_0$ and $j=k_r$ such that for $1 \leq s \leq r$,
    \[  \left[A_{k_s},B_{k_s} \right]_{\rho_{k_{s}}}\cap \left[A_{k_{s-1}},B_{k_{s-1}} \right]_{\rho_{k_{s-1}}} \neq \emptyset .\]
     Let $\psi=  \bigoplus_{i=1}^n \rho_i \otimes S_{a_i} \otimes S_{b_i}$ be a representation of $W_F \times \SL_2(\BC) \times \SL_2(\BC)$. We define that $\psi$ is connected if the multi-set of segments 
     \[ \supp(\psi):= \left\{ \left[ \half{a_i+b_i}-1, \half{a_i-b_i}  \right]_{\rho_i} \right\}_{i=1}^n \]
     is connected.
    \item Suppose $\psi$ is a local Arthur parameter. We define the connected components of $\psi$ to be the maximal elements of 
    \[ \{ \psi' \leq \psi\ | \ \psi' \text{ is connected} \},\]
    where we define $\psi_1 \leq \psi_2$ if and only if $\psi_1$ is a subrepresentation of $\psi_2$.
\end{enumerate}
\end{defn}

Here are some immediate consequences of the definition.
\begin{lemma}
\begin{enumerate}
    \item Any local Arthur parameter decomposes uniquely into sum of its connected components.
    \item Suppose $\psi$ is a non-negative local Arthur parameter. Then $\psi$  is of chain shape if and only if all of its connected components are of chain shape.
\end{enumerate}
\end{lemma}

Now suppose $\psi$ is a connected component of some non-negative local Arthur parameter of chain shape. Then there exist exactly two $\FF=\{ ([A_i,B_i]_{\rho},l_i,\eta_i)\}_{i\in (I_{\rho},>)}$ satisfying the following:
\begin{enumerate}
    \item [$\oldbullet$] $ \supp(\FF)= \supp(\psi)$ as multi-sets,
    \item [$\oldbullet$] $\FF$ satisfies the non-negative tempered condition,
    \item [$\oldbullet$]  $\FF$ satisfies Theorem \ref{thm non-vanishing}(ii).
\end{enumerate}
 We denote them by $ \FF_{\psi,temp,+}, \FF_{\psi,temp,-}$ according to the sign of $(-1)^{A_n-B_n}\eta_n$, where $n$ denotes the maximal element in $(I_{\rho},>)$. Note that it is the sign of the last circle of the last row of the pictograph associated with $\FF$.
 Pictorially, $\FF_{\psi,temp,-}$ can be obtained from $\FF_{\psi,temp,+}$ by replacing all $ \oplus$ with $\ominus$, and $\ominus$ with $\oplus$. 
 
Suppose $\psi$ is a non-negative local Arthur parameter of chain shape, and $\psi= \bigoplus_{i=1}^n \psi_i$ is a decomposition into sum of connected components. Then Corollary \ref{cor pos tempered} shows the set of tempered representations in $\Pi_{\psi}$ is given by
\[ \{ \cup_{i} \FF_{\psi_i,temp, \eta_i}\ | \ \text{ for }1 \leq i \leq n,\ \eta_i\in \{\pm \},\ \text{and }\cup_{i} \FF_{\psi_i,temp, \eta_i} \text{ satisfies }\eqref{eq sign condition}  \}.\]
To count the cardinality of above sets, we classify connected subrepresentations of non-negative local Arthur parameters of chain shape. 

\begin{defn}\label{def type of A-par}
Let $\psi$ be a connected subrepresentation of a non-negative local Arthur parameter of chain shape. We define
\begin{enumerate}
    \item [$\oldbullet$]$\psi$ is of type (I) if $\#\Omega(\FF_{\psi,temp,+})$ is odd. 
    \item [$\oldbullet$] $\psi$ is of type (II) if the pictograph $\FF_{\psi,temp,+}$ contains an even number of $\oplus$ and even number of $\ominus$.
    \item [$\oldbullet$] $\psi$ is of type (III) if the pictograph of $\FF_{\psi,temp,+}$ contains an odd number of $\oplus$ and odd number of $\ominus$.
\end{enumerate}
\end{defn}
We demonstrate the definition in the following example.
\begin{exmp}
Let $\psi_1, \psi_2$ be the local Arthur parameters associated with $\EE_1, \EE_2$ in Example \ref{exmp chain shape}. Then they are connected, and $\EE_1= \FF_{\psi_1,temp,+}$ and $\EE_{2}= \FF_{\psi_2,temp,-} $. We have $\psi_1$ is of type I and $\psi_2$ is of type II.
\end{exmp}

The following is the main theorem for non-negative case.
\begin{thm}\label{thm temp non-negative}Let $\psi$ be a non-negative local Arthur parameter.
\begin{enumerate}
    \item [1.] $\Pi_{\psi}$ contains a tempered representation only if $\psi$ is of chain shape. 
    \item [2.] Suppose $\psi$ is of chain shape. Denote 
\[ \psi= \sum_{i=1}^{r} \psi_i \]
the decomposition into connected components. Denote $r_I$ (resp. $r_{
II}, r_{III}$) the number of $\psi_i$ of type ($I$) (resp. type ($II$), type ($III$)).
Then the number of tempered representations inside $\Pi_{\psi}$ is given by
\[ \begin{cases}
2^{r-1} & \text{ if }r_{I} >0\\
2^{r} &\text{ if } r_{I}=0 \text{ and } r_{III} \text{ is even}\\
0 &\text{ if } r_{I}=0 \text{ and } r_{III} \text{ is odd}\\
\end{cases},\]
 and these representations are given by
 \[ \pi(\cup_{i=1}^{r} \FF_{\psi_i,temp,\eta_i}  ) \]
 where $\eta_i \in \{ +,-\}$ are chosen such that $\cup_{i=1}^{r} \FF_{\psi_i,temp,\eta_i}$ satisfies the sign condition (\ref{eq sign condition}). Moreover, their $L$-data are given explicitly in Corollary \ref{cor pos tempered}.
\end{enumerate}
 \end{thm}
\begin{proof}
This follows from Corollary \ref{cor pos tempered} and the definitions.
\end{proof}

We check that Theorem $\ref{thm temp non-negative}$ agrees with Arthur's parameterization of tempered local Arthur parameters of good parity (see Theorem \ref{thm Arthur tempered}). Let
  $$\psi = \bigoplus_{\rho}\left(\bigoplus_{i\in I_\rho} \rho \otimes S_{a_i} \otimes S_{1}\right)$$  
 be a tempered local Arthur parameter of good parity.  We have 
 $$\supp(\rho \otimes S_{a_i} \otimes S_{1})=\left\{\left[\frac{a_i-1}{2},\frac{a_i-1}{2}\right]\right\},$$ 
 and hence the connected components of $\psi$ are the summands of the above decomposition with multiplicity. Let $\psi'=m_{\rho,i}(\rho \otimes S_{a_i} \otimes S_{1})$ where $m_{\rho,i}$ is the multiplicity of $\rho \otimes S_{a_i} \otimes S_{1}$ in $\psi.$ By direct computation, we see $$\FF_{\psi',temp,+}=\left\{\left(\left[\frac{a_i-1}{2},\frac{a_i-1}{2}\right],0,1\right)^{m_{\rho,i}}\right\}$$ and hence $\psi'$ is of type (I) if $m_{\rho,i}$ is odd and type (II) if $m_{\rho,i}$ is even. Suppose that $\psi$ has $r$ connected components. By Theorem \ref{thm temp non-negative}, we have $|\Pi_\psi|=2^{r-1}$ if $m_{\rho,i}$ is odd for some pair $(\rho,i).$ Otherwise, $|\Pi_\psi|=2^{r}.$
 
 On the other hand, by Theorem \ref{thm Arthur tempered}, we can compute $|\Pi_\psi|$ by computing the size of the $\widehat{\mathcal{S}}_\psi.$ If $m_{(\rho,i)}$ is always even for any pair $(\rho,i)$, then the central element $z_\psi$ of $\mathcal{A}_\psi$ is generated by $\alpha_{\rho,i}+\alpha_{\rho,j}$ such that $i,j\in I_\rho$ with $$\rho \otimes S_{a_i} \otimes S_{1}=\rho \otimes S_{a_j} \otimes S_{1}.$$ Hence $\mathcal{S}_\psi$ is isomorphic with the Abelian group $\oplus_{i=1}^r (\mathbb{Z}/2\mathbb{Z}) \alpha_{\rho,i}$. Thus we have $|\widehat{\mathcal{S}}_\psi|=2^r=|\Pi_\psi|.$ If some $m_{\rho,i}$ is odd, then the central element $z_\psi$ of $\mathcal{A}_\psi$ is not generated by $\alpha_{\rho,i}+\alpha_{\rho,j}$ such that $i,j\in I_\rho$ with $$\rho \otimes S_{a_i} \otimes S_{1}=\rho \otimes S_{a_j} \otimes S_{1}.$$ Thus, $|\widehat{\mathcal{S}}_\psi|=2^{r-1}=|\Pi_\psi|.$ Therefore, we see that the size of $\Pi_\psi$ given by Theorem \ref{thm temp non-negative} agrees with the size given by Theorem \ref{thm Arthur tempered}.

\subsection{General local Arthur packets containing tempered representations}\label{general case tempered}

In this section, we classify all local Arthur packets containing tempered representations (see Definition \ref{def chain shape} below). More precisely, 
we classify all extended multi-segments that give tempered representations (see Theorem \ref{thm tempered general}). The idea is to apply Theorem \ref{thm exhaustion of symbol} to a tempered extended multi-segment $\EE_{temp}.$ By construction, $(\EE_{temp})_{can}$ is non-negative.
However, since we have already treated the non-negative case, we know that $(\EE_{temp})_{can}$ is of chain shape. This allows us to compute $\Psi(\EE_{temp}).$

We first separate into two cases.

\begin{lemma}\label{lem classification of negative tempered}  Suppose $\EE \in \Rep^{(P')}$ and there exists a tempered  $\FF_{temp}$ such that $\pi(\EE^{\rho}\cup \EE_{\rho}) \cong \pi(\EE^{\rho} \cup \FF_{temp})$. Denote 
\[\FF=\EE_{\rho}=\{([A_i,B_i],l_i,\eta_i)\}_{i \in (I_{\rho},>)}\]
and $1$ the minimal element in $I_{\rho}$.
Take $\epsilon \in \{0,1/2\}$ such that $A_i+\epsilon\in \Z$ for $i \in I_{\rho}$. Then
\begin{enumerate}
    \item [(i)] If $B_1<-\epsilon$, then 
    \[(\FF_{temp})_{=\epsilon}= \{([\epsilon,\epsilon]_{\rho},0,\eta)^{s}\}\]
    for some odd number $s$, and $\eta=-1$ if $\epsilon=1/2$.
    \item [(ii)]  If $B_1=-1/2$, then $dual_1^{-}$ is applicable on $\EE_{\rho}$, and hence $dual_1^-(\EE_{\rho})$ is non-negative and its support is of chain shape.
\end{enumerate}
\end{lemma}
\begin{proof}
Part (i) follows from applying Theorem \ref{thm exhaustion of symbol} to $\FF_{temp}$. It follows directly that $\FF_{can}=(\FF_{temp})^{min}$, which is non-negative and hence of chain shape by Proposition \ref{prop pos chain shape}. Then one can see that the set $UI^{-1}_{\geq-1/2}(dual(\FF_{can}))$ is not a singleton only if the conclusion in the statement holds.

For Part (ii), a similar argument shows none of the elements in $UI^{-1}_{\geq-1/2}(dual(\FF_{can}))$ satisfy $\max\{B_i\}=1/2$. Hence, any $\FF'=\{([A_i',B_i'],l_i',\eta_i')\}_{i \in (I_{\rho}',>)}$ obtained from $\FF_{can}$ by a sequence of basic operators should satisfy $\min\{B_i'\ | \ i \in I_{\rho}'\}\neq -1/2$.

By Theorem \ref{thm half integer} and the Corollary \ref{cor partial dual is always applicable}, $dual_1^{-}$ must be applicable on $\FF$, and $dual_1^{-}(\FF)$ is non-negative, and hence its support is of chain shape by Proposition \ref{prop pos chain shape}.
This proves the lemma.
\end{proof}

Now we apply Theorem \ref{thm exhaustion of symbol} to the tempered $\FF_{temp}$ such that $$(\FF_{temp})_{=\epsilon}= \{([\epsilon,\epsilon]_{\rho},0,\eta)^{2s+1}\},$$ where $\eta=-1$ if $\epsilon=1/2$. A simple computation shows $(dual((\FF_{temp})_{can}))_{\geq -1/2}$ is of the form
\[\{([\epsilon,-\epsilon]_{\rho},0,\ast)^{2s}\}+\{ ([A, -\epsilon]_{\rho},0,\ast)\}.\]
The only possible way to apply $ui^{-1}$ is to split the last row by Corollary \ref{cor split}. Therefore, one can see that $\pi(\EE^{\rho}\cup \EE_{\rho} ) \cong \pi(\EE^{\rho} \cup \FF_{temp})$ only if $\EE_{\rho}^{min}$ or $dual_k(\EE_{\rho}^{min})$ for some $k$ of the form
\[ \FF_{cusp} + \{([\epsilon,\epsilon]_{\rho},0,\ast)^{2s}\} +\FF_{chain},\]
where the support of $\FF_{cusp}$ is of cuspidal shape and the support of $\FF_{chain}$ is non-negative of chain shape. Due to the cuspidal condition in Corollary \ref{cor cuspidal shape}, it is possible that $ui\inv$ is applicable. Here is an example.
\begin{exmp}\label{ex chain shape gneral}
Consider the extended multi-segments
\begin{align*}
     \EE_1&=   \begin{blockarray}{cccccccc}
\frac{-7}{2}&\frac{-5}{2}&\frac{-3}{2}&\frac{-1}{2}&\frac{1}{2}&\frac{3}{2}&\frac{5}{2}&\frac{7}{2}\\
\begin{block}{(cccccccc)}
\lhd&\lhd &\lhd&\lhd&\rhd&\rhd&\rhd&\rhd\\
&&&\ominus&\oplus&\ominus&\oplus&\\
&&&&&\oplus&&\\
&&&&&\oplus&&\\
\end{block}
\end{blockarray} \ ,\\
     \EE_2&=   \begin{blockarray}{cccccccc}
\frac{-7}{2}&\frac{-5}{2}&\frac{-3}{2}&\frac{-1}{2}&\frac{1}{2}&\frac{3}{2}&\frac{5}{2}&\frac{7}{2}\\
\begin{block}{(cccccccc)}
\lhd&\lhd &\lhd&\lhd&\rhd&\rhd&\rhd&\rhd\\
&&&\ominus&\oplus&\ominus&&\\
&&&&&\ominus&&\\
&&&&&\ominus&\oplus&\\
\end{block}
\end{blockarray} \ .
\end{align*}
We have $\EE_1= ui_{2,4}(\EE_2)$, and
\[ \pi(\EE_1')=\pi(\EE_2')= \pi( (1/2)^{-},(3/2)^{+},(3/2)^{+},(3/2)^{+},(5/2)^{-},(7/2)^{+})\]
are tempered.
\end{exmp}

Based on these observations, we extend Definition \ref{def pos chain shape}.
\begin{defn}\label{def chain shape}
We extend Definition \ref{def pos chain shape} for general multi-sets of segments. 

First suppose $S$ is a multi-set of segments consisting of a single $\rho$. We say $S$ is of chain shape if the following hold:
\begin{enumerate}
    \item [(i)] There exists a multiplicity free ordered subset $S_{1}= \{ [A_i,B_i]_{\rho} \}_{i=1}^{n}$ of $S$ satisfying:
\begin{enumerate}
    \item [(a)] $B_1<0$,
    \item [(b)] for $1 \leq  i < n$, if $B_i<0$, then $A_{i+1}=-B_{i}-1$,
    \item [(c)] for $1 <  i < n$, if $B_i>0$, then either $ A_{i+1}=B_{i}-1$ or $i<n-1$, $[A_{i+1},B_{i+1}]_{\rho}=[B_{i},B_{i}]_{\rho}$ and $ A_{i+2}=B_i, B_{i+2}<B_i$.
    \item [(d)] $B_n \in \{0,1/2,-1/2\}$. $B_i=0$ or $-1/2$ only when $i=n$. 
\end{enumerate}
\item [(ii)] $S\setminus S_1$ consists of segments $[A,B]_{\rho}$ with $B\geq 0$. Moreover, we have $\{ [A_1,|B_n|]_{\rho} \} + (S\setminus S_1)$ is of chain shape in Definition \ref{def pos chain shape}.
\item[(iii)] If $S_1$ is non-empty and $ S_1 \neq \{ [1/2,-1/2]_{\rho}\}$, the multiplicity of $[1/2,1/2]_{\rho}$ in $S$ is even.
\end{enumerate}
In this case, we write
\[ \sum_{i=1}^n [A_i, |B_i|]_{\rho}= [A,\epsilon]_{\rho} +\sum_{j=1}^r ([C_j,C_j]_{\rho})^2.\]
We denote $\abs(S_1):=\{[A,\epsilon]_{\rho}, ([C_1,C_1]_{\rho})^{2},\dots,([C_r,C_r]_{\rho})^2\}$ and define $\abs(S):=(S\setminus S_1) + \abs(S_1)$. 

For general $S$, we decompose $S= \cup_{\rho} S_{\rho}$ and say that $S$ is of chain shape if every $S_{\rho}$ is of chain shape. In this case, we define $\abs(S)=\cup_{\rho} \abs(S_{\rho})$. 

We say a local Arthur parameter $\psi$ is of chain shape if $\supp(\psi)$ is of chain shape. In this case, we denote $\abs(\psi)$ to be the unique (non-negative) local Arthur parameter with $\supp(\abs(\psi))= \abs(\supp(\psi))$.

We say an extended multi-segment $\EE$ is of chain shape if $\supp(\EE)$ is of chain shape.
\end{defn}

We remark that if $\psi$ is of chain shape, then $\psi^{\Delta}= (\abs(\psi))^{\Delta}$.
With the definition of chain shape in general case, we generalize Proposition \ref{prop pos chain shape} and Corollary \ref{cor pos tempered} as follows.
\begin{thm}\label{thm tempered general}
Let $\EE \in \Rep^{(P')}$.
\begin{enumerate}
    \item [(i)] There exists a tempered  $\FF_{temp}$ such that $\pi(\EE^{\rho}\cup \EE_{\rho}) \cong \pi(\EE^{\rho} \cup \FF_{temp})$ only if $ \supp(\EE_{\rho})$ is of chain shape. In particular, $\pi(\EE)$ is tempered only if $\EE$ is of chain shape. 
    \item [(ii)] Suppose $\EE$ is of chain shape. Write $\EE^{(P'')}=\cup_{\rho}\{ ([A_i,B_i]_{\rho_i},l_i,\eta_i)\}_{i \in (I_{\rho,>})}$. Then the followings are equivalent:
    \begin{enumerate}
        \item [(a)] $\pi(\EE^{\rho} \cup \EE_{\rho}) \cong \pi( \EE^{\rho} \cup \FF_{chain})$ for some $\FF_{chain}$ non-negative of chain shape satisfying the non-negative tempered condition and also that $\supp(\FF_{chain})= \abs(\supp(\EE_{\rho}))$. 
        \item [(b)] $\EE_{\rho}^{(P'')}= \{([A_i,B_i]_{\rho},l_i, \eta_i)\}_{i \in (I_{\rho},>)}$ satisfies $l_i=0$ for all $i \in I_{\rho}$.
    \end{enumerate}
    We say $\EE_{\rho}$ satisfies the \emph{tempered condition} if (b) holds. 
\end{enumerate}
 
\end{thm}
\begin{proof}
For Part (i), we check that each support of $\EE_{\rho}$ is of chain shape. Denote 
\[ \FF= \EE_{\rho}= \{ ([A_i,B_i]_{\rho},l_i,\eta_i)\}_{i \in (I_{\rho},>)},\]
and let $1$ denote the minimal element in $(I_{\rho},>)$. Since the non-negative case is done in Proposition \ref{prop pos chain shape}, we assume $B_1 <0$.

We first deal with the case that $B_1=-1/2$. Lemma \ref{lem classification of negative tempered}(ii) implies that $dual_1^{-}$ is applicable on $\FF$ and $dual_1^{-}(\FF)$ is non-negative of chain shape. Then in Definition \ref{def chain shape}, we have $S_1=\{[A_1, -1/2]_{\rho}\}$ is a singleton. It remains to check Condition (iii) in the same definition in the case $A_1>1/2$. This can be done by the same argument in Step 5 of the proof of Proposition \ref{prop pos chain shape} and we omit the details. 

Now we assume $B_1 <-1/2$. Following the notation Lemma \ref{lem classification of negative tempered} and discussion after it, we may write 
\[\FF_{can}= \{([A,\epsilon]_{\rho},0,\eta)\} + \{([\epsilon,\epsilon]_{\rho},0,\ast)^{2s}\} + \FF_{chain},\]
up to row exchanges, and $\FF\in UI^{-1}(\FF')$ where 
\[ \FF'=\{ ([A_i,B_i]_{\rho},l_i, \eta_i)\}_{i=1}^n + \{([\epsilon,\epsilon]_{\rho},0,\ast)^{2s}\} + \FF_{chain} \]
with 
\begin{enumerate}
\item [$\oldbullet$] $A_1=A$,
    \item [$\oldbullet$] $B_i<0$ for $1 \leq i\leq n$,
    \item [$\oldbullet$] for $1 \leq i<n$, $ A_{i+1}=-B_{i}-1$.
\end{enumerate}
Note that $\{ ([A_i,B_i],l_i, \eta_i)\}_{i=1}^n=\FF_{<0}'$.

Since $\FF_{can}$ is of chain shape, we may decompose $\FF_{chain}$ into
\[\FF_{chain}= \{ ([C_j,C_j]_{\rho},0,\ast)^2\}_{j=1}^m + (\FF_{chain})_{\geq A} \]
for some $C_j<A$, and hence 
\[\FF_{<A}'= \{ ([A_i,B_i]_{\rho},l_i, \eta_i)\}_{i=1}^n + \{([\epsilon,\epsilon]_{\rho},0,\ast)^{2s}\} +\{ ([C_j,C_j]_{\rho},0,\ast)^2\}_{j=1}^m.  \]
By comparing the support, one can conclude
\[ UI^{-1}(\FF')= \{ \FF_1+ \FF_2 \ | \ \FF_1 \in UI^{-1}(\FF_{<A}'),\ \FF_2 \in UI^{-1}(\FF_{\geq A}') \}.\]

Finally, we know any element in $ UI^{-1}(\FF_{\geq A}')$ is of chain shape by Proposition \ref{prop pos chain shape}, and any element in $ UI^{-1}(\FF_{<A}')$ is of chain shape by definition. Therefore, $\FF$ is of chain shape. This completes the proof of (i).

For Part (ii), we first show that Part (a) implies Part (b). Corollary \ref{cor pos tempered} shows that there exists a tempered $\FF_{temp}$ such that $\pi(\EE^{\rho} \cup \EE_{\rho}) \cong \pi(\EE^{\rho} \cup \FF_{temp})$. Then the argument in Part (i) is valid for $\EE_{\rho}$.

If $B_1=-1/2$, then $\EE_{\rho}= dual_1(\FF_{chain})$, and the assertion is clear. Otherwise $B_1< -1/2$, we use the notation in Part (i). Write $\EE_{\rho}=\FF= \FF_{<A}+ \FF_{\geq A}$ where $\FF_{<A}\in UI^{-1}(\FF_{<A}')$ and $\FF_{\geq A} \in UI^{-1}(\FF_{<A}')$. We have 
\[ \FF^{(P'')}= \FF_{<A}^{(P'')}+ \FF_{\geq A}^{(P'')}. \]
We know $ \FF_{\geq A}$ satisfies the non-negative tempered condition in Corollary \ref{cor pos tempered}, and $\FF_{<0}'$ satisfies the cuspidal condition in Corollary \ref{cor cuspidal shape}. The conclusion then follows from the following simple computation: If $A>B>C$ are all in $\Z$ or $\Z+ \half{1}$, denote
\[\mathcal{F}=\{[A_i,B_i]_{\rho}, l_i, \eta_i\}_{i \in (1<2<3))}= \{ ([A,B]_{\rho},0,\eta)\}+ \{([C,C]_{\rho},0, (-1)^{A-B} \eta)^2\}.\]
Then we have 
\[ ui_{1,3}^{-1}(\FF)= \{([C,B]_{\rho},0,\eta)\}+ \{([C,C]_{\rho},0,(-1)^{C-B}\eta)\}+ \{([A,C]_{\rho},0, (-1)^{C-B}\eta) \}.\]
This completes the proof of this direction.

Finally, we show Part (b) implies Part (a). Again denote $\FF= \EE_{\rho}$. Since $\supp(\FF)$ is of chain shape, we let $S_1=\{[A_i,B_i]_{\rho}\}_{i=1}^{n}$ be as in Definition \ref{def chain shape}, and we apply induction on $n$.

The case $n=0$ is done in Corollary \ref{cor pos tempered}. When $n=1$, we must have $B_1=-1/2$, and Theorem \ref{thm non-vanishing}(i) forces $dual_1^{-}$ is applicable on $\FF$. Then we take $\FF_{chain}=dual_1^{-}(\FF)$. Assume $n >1$. Write $\FF=\{ ([A_i,B_i]_{\rho},l_i',\eta_i')\}_{i \in (I_{\rho},\gg)}$. We separate into four cases.

Case 1. Assume $ n=2$. Let $i,j \in I_{\rho}$ be such that $[A_i,B_i]_{\rho}= [A_1,B_1]_\rho$ and $[A_j,B_j]_{\rho}= [A_2,B_2]_\rho$. In this case we have $ A_2=-B_{1}-1$. The condition on $\FF^{(P'')}$ implies that $\EE^{\rho} \cup (add^{-1}_i(\FF))$ is not in $\Rep$. Then Theorem \ref{thm Aubert-Zelevinsky dual formula} implies that $dual(\EE^{\rho})\cup sh^{-1}_{i}(dual(\FF))$ is not in $\Rep$, either. However, this implies $ui_{j,i}$ is applicable on $dual(\FF)$, and it is of type 3'. So we may replace $\FF$ by $dual \circ ui_{j,i} \circ dual(\FF)$, which still satisfies the conditions in the statement, and the size of $S_1$ decreases. The claim then follows by induction. Note that this procedure does not change $abs(\supp(\FF))$.
    
Case 2. Suppose that there exists $1<k<n$ such that $B_{k}<0$. Let $i,j \in I_{\rho}$ be such that $[A_i,B_i]_{\rho}= [A_{k},B_{k}]_\rho$ and $[A_j,B_j]_{\rho}= [A_{k+1},B_{k+1}]_\rho$. Then the rest of the argument is identical with Case 1.

Case 3.  Assume that there exists $ 1< k< n-1$ such that $B_k>0$, $[A_{k+1},B_{k+1}]_{\rho}= [B_k,B_k]_{\rho}$ and $ A_{k+2}= B_k$, $B_{k+2}< B_k$. Let $i,j\in I_{\rho}$ be such that $[A_j,B_j]_{\rho}=[A_{k},B_{k}]_\rho$ and $[A_i,B_i]_{\rho}=[A_{k+2},B_{k+2}]_\rho$.
    In this case, $ui_{i,j}$ is applicable on $\FF$ (see the simple computation above), so we replace $\FF$ with $ui_{i,j}(\FF)$. This replacement changes the three consecutive segments $\{[A_k,B_k]_{\rho}, [A_{k+1},B_{k+1}]_{\rho}, [A_{k+2},B_{k+2}]_{\rho}\}$ in $S_1$ into a single $\{[A_k,B_{k+2}]_{\rho}\}$, and adds two copies of $[B_k,B_k]_{\rho}$ in $S \setminus S_1$. Then we decrease the size of $S_1$ and $abs(\supp(\FF))=abs(\supp(ui_{i,j}(\FF)))$. The claim then follows by induction.

Case 4. Suppose $B_2>0$ and $A_3= B_2-1$. Let $i,j \in I_{\rho}$ be such that $[A_i,B_i]_{\rho}= [A_{3},B_{3}]_\rho$ and $[A_j,B_j]_{\rho}= [A_{2},B_{2}]_\rho$. One can check that Theorem \ref{thm non-vanishing}(ii) forces $ui_{i,j}$ is applicable on $\FF$ of type 3'. This decreases the size of $S_1$ and $abs(\supp(\FF))=abs(\supp(ui_{i,j}(\FF)))$. Again, the claim then follows by induction.

 This completes the proof of this theorem.
\end{proof}

Similar as the non-negative case, suppose  $\psi$ is a connected component of some local Arthur parameter of chain shape. There
exist exactly two (up to row exchanges)  $\FF=\{([A_i,B_i]_{\rho},l_i,\eta_i)\}_{i\in (I_{\rho},>)} $ satisfying
\begin{enumerate}
    \item [$\oldbullet$] $\supp(\FF)=\supp(\psi)$ as multi-sets,
    \item [$\oldbullet$] $\FF$ satisfies the tempered condition,
    \item [$\oldbullet$] $\FF$ satisfies Theorem \ref{thm non-vanishing}(ii).
\end{enumerate}
We denote them by $\FF_{\psi,temp,+}$ and $\FF_{\psi,temp,-}$ according to the sign of $(-1)^{A_n-B_n}\eta_n$ where $n$ is the maximal element of $(I_{\rho},>)$. However, if $A_1\in\Z+1/2$, then only one of them satisfies Theorem \ref{thm non-vanishing}(i) after row exchange (see the case of $n=1$ in the induction step of the proof of Part (2) above). In this case, we denote this one by $\FF_{\psi,temp}$.

Let $\psi_{\EE}= \bigoplus_{j \in J} \psi_j$ be the decomposition into connected components. The proof of Theorem \ref{thm tempered general} shows that if $\pi(\EE)$ is tempered, then we may write $\EE= \cup_{j \in J} \FF_{\psi_j,temp,\eta_j} $, and we have  
\[ \pi\left(\cup_{j \in J} \FF_{\psi_j,temp,\eta_j} \right)\cong \pi\left(\cup_{j \in J} \FF_{\abs(\psi_j),temp,\eta_j'} \right),\]
for some $\eta_j' \in \{ +, -\}$. Indeed, one can show by Lemma \ref{lem l_n=0} that $\eta_j'=\eta_j$.



Now we extend Definition \ref{def type of A-par} to classify connected components of a general local Arthur parameter of chain shape (and of good parity).
\begin{defn}
Suppose $\psi=\bigoplus_{i=1}^n \rho\otimes S_{a_i}\otimes S_{b_i}$ is a connected component of a local Arthur parameter of chain shape such that $a_i < b_i$ for some $1 \leq i \leq n$.
\begin{enumerate}
    \item [(i)] If $a_1 \equiv b_1 \mod 2$, then we define $\psi$ is of type (I) (resp. (II), (III)) if it satisfies the same condition in Definition \ref{def type of A-par} of type (I) (resp. (II), (III)).
    \item [(ii)] If $a_1 \not\equiv b_1 \mod 2$, then we define 
    \begin{enumerate}
      \item [$\oldbullet$] $\psi$ is of type $(II,\half{1})$ if the pictograph of $\FF_{\psi,temp}$ contains an even number of $\ominus$.
    \item [$\oldbullet$] $\psi$ is of type $(III,\half{1})$ if the pictograph of $\FF_{\psi,temp}$ contains an odd number of $\ominus$.
    \end{enumerate}
\end{enumerate}
\end{defn}

We demonstrate the definition in the following example.
\begin{exmp}\label{ex general chain shape}
Let $\psi_1$ (resp. $\psi_2$) be the local Arthur parameter associated with extended multi-segments $\EE_1$ (resp. $\EE_2$) in Example \ref{ex chain shape gneral}. Then after row exchange, we get 
\begin{align*}
     \EE_1'&=   \begin{blockarray}{cccccccc}
\frac{-7}{2}&\frac{-5}{2}&\frac{-3}{2}&\frac{-1}{2}&\frac{1}{2}&\frac{3}{2}&\frac{5}{2}&\frac{7}{2}\\
\begin{block}{(cccccccc)}
&&&&&\oplus&&\\
&&&&&\oplus&&\\
&&&\oplus&\ominus&\oplus&\ominus&\\
\ominus&\oplus&\ominus&\oplus&\ominus&\oplus&\ominus&\oplus\\
\end{block}
\end{blockarray} \ , \\
\EE_2'&=   \begin{blockarray}{cccccccc}
\frac{-7}{2}&\frac{-5}{2}&\frac{-3}{2}&\frac{-1}{2}&\frac{1}{2}&\frac{3}{2}&\frac{5}{2}&\frac{7}{2}\\
\begin{block}{(cccccccc)}
&&&\oplus&\ominus&\oplus&&\\
&&&&&\oplus&&\\
&&&&&\oplus&\ominus&\\
\ominus&\oplus&\ominus&\oplus&\ominus&\oplus&\ominus&\oplus\\
\end{block}
\end{blockarray} \ .
\end{align*}
Then $\EE_i'= \FF_{\psi_i,temp,+}=\FF_{\psi_i,temp}$ for $i=1,2$. We have $\abs(\psi_1)=\abs(\psi_2)$ and 
\[\FF_{abs(\psi_1),temp,+}=\begin{blockarray}{cccc}
\frac{1}{2}&\frac{3}{2}&\frac{5}{2}&\frac{7}{2}\\
\begin{block}{(cccc)}
&\oplus&&\\
&\oplus&&\\
\ominus&\oplus&\ominus&\oplus\\
\end{block}
\end{blockarray}\]
and
\[ \pi(\EE_1')=\pi(\EE_2')=\pi(\FF_{\abs(\psi_1),temp,+})= \pi( (1/2)^{-},(3/2)^{+},(3/2)^{+},(3/2)^{+},(5/2)^{-},(7/2)^{+}).\]
Since there are 6 $\ominus$'s in the pictographs of both $\FF_{\psi_1,temp}$ and $\FF_{\psi_2,temp}$, $\psi_1,\psi_2$ are both of type $(II,\half{1})$.
\end{exmp}

Here is our main theorem in the general case.
\begin{thm}\label{thm tempered rep general}
Let $\psi$ be a local Arthur parameter.
\begin{enumerate}
    \item [1.] $ \Pi_{\psi}$ contains a tempered representation only if $\psi$ is of chain shape.
    \item [2.] Suppose $\psi$ is of chain shape. Let
\[\psi= \bigoplus_{j \in J} \psi_j\]
denote the decomposition of $\psi$ into connected components. Let $J_{I}$ (resp. $J_{II}$, $J_{III}$, $J_{II,\half{1}}$, $J_{III,\half{1}}$) denote the subset of $J$ consisting of those indices $j$ such that $\psi_j$ is of type (I) (resp. (II), (II,$\half{1}$), (III), (III,$\half{1}$)), and denote its cardinality by $r_{I}$ (resp. $r_{II}$, $r_{II,\half{1}}$, $r_{III}$, $r_{III,\half{1}}$). 
Let $r= r_I+r_{II}+r_{III}$. Then the number of tempered representations inside $\Pi_{\psi}$ is given by
\[ |\Pi_\psi^{temp}|=\begin{cases}
2^{r-1} & \text{ if }r_{I} >0,\\
2^{r} &\text{ if } r_{I}=0 \text{ and } r_{III}+r_{III,\half{1}} \text{ is even},\\
0 &\text{ if } r_{I}=0 \text{ and } r_{III}+r_{III,\half{1}} \text{ is odd},\\
\end{cases}\]
 and these representations are given by
 \[ \pi(\cup_{j \in J} \FF_{\psi_j, temp, \eta_j} ), \]
 where $\eta_j \in \{ +,-\}$ are chosen such that $\cup_{j \in J} \FF_{\psi_j, temp, \eta_j}$ satisfies the sign condition (\ref{eq sign condition}), and $\FF_{\psi_j,temp,\eta_j}=\FF_{\psi_j,temp}$ if $j \in J_{II,\half{1}} \cup J_{III,\half{1}}$. 
 
 Moreover, we have 
 \[ \pi(\cup_{j \in J} \FF_{\psi_j, temp, \eta_j} ) \cong \pi(\cup_{j \in J} \FF_{\abs(\psi_j), temp, \eta_j} ), \]
 and the $L$-data of right hand side are described in Corollary \ref{cor pos tempered}.
\end{enumerate}
\end{thm}

\begin{remark}
Let $\phi$ be a tempered local Arthur parameter. By Theorems \ref{thm GGP20} and \ref{thm Arthur tempered} we have that $\Pi_{\psi} \cap \Pi_{\phi} \neq \emptyset$ if and only if $\psi^{\Delta}= \phi$ and $\Pi_{\psi}$ contains a tempered representation. On the other hand, $\Pi_{\psi}$ contains a tempered representation if and only if $|\Pi_\psi^{temp}|$ is nonzero (see Theorem \ref{thm tempered rep general}(2)). Therefore, Theorem \ref{thm tempered rep general} gives a complete characterization of the set
\[ \{ \psi \ | \ \Pi_{\psi} \cap \Pi_{\phi} \neq \emptyset\}.\]
\end{remark}

Here is an example of the theorem.
\begin{exmp}
Let $\rho_1$ be the trivial representation and $\rho_2$ be symplectic of dimension $d$. Let 
\begin{align*}
    \psi_1&=\rho_1 \otimes S_{4} \otimes S_{7} + \rho_1\otimes S_{1} \otimes S_{2},\\
    \psi_2&= \rho_2 \otimes S_4 \otimes S_4,\\
    \psi_2'&= \rho_2 \otimes S_1 \otimes S_1,\\
    \psi_3'&= \rho_2 \otimes S_5 \otimes S_3.
\end{align*}
We consider $ \psi= \psi_1+ \psi_2$ and $\psi'= \psi_1+\psi_2'+\psi_3'$ which are both local Arthur parameters of $\SO_{30+16d+1}$. One can check that every $\psi_i, \psi_j'$ is connected. For $\psi_1$, we have 
\begin{align*}
    \FF_{\psi_1,temp,+}^{(P'')} &=\bordermatrix{
    &\half{-3}&\half{-1}&\half{1}&\half{3}&\half{5}&\half{7}&\half{9} \cr
    &&\ominus&\oplus&&&&\cr
    & \oplus &\ominus &\oplus &\ominus& \oplus & \ominus &\oplus \cr 
    }_{\rho_1}, \\ \FF_{\psi_1,temp,-}^{(P'')} &=\bordermatrix{
    &\half{-3}&\half{-1}&\half{1}&\half{3}&\half{5}&\half{7}&\half{9} \cr
    &&\oplus&\ominus&&&&\cr
    & \ominus &\oplus &\ominus &\oplus& \ominus & \oplus & \ominus \cr
    }_{\rho_1},\\
    \FF_{\psi_1,temp,+} &=\bordermatrix{
    &\half{-3}&\half{-1}&\half{1}&\half{3}&\half{5}&\half{7}&\half{9} \cr
    & \lhd &\lhd &\oplus &\ominus& \oplus & \rhd &\rhd \cr 
    &&\ominus&\oplus&&&&\cr
    }_{\rho_1}, \\ \FF_{\psi_1,temp,-} &=\bordermatrix{
    &\half{-3}&\half{-1}&\half{1}&\half{3}&\half{5}&\half{7}&\half{9} \cr
    & \lhd &\lhd &\ominus &\oplus& \ominus & \rhd &\rhd \cr 
    &&\oplus&\ominus&&&&\cr
    }_{\rho_1}.
\end{align*}
One can check that $ \FF_{\psi_1,temp, +}$ does not satisfy Theorem \ref{thm non-vanishing}(i), and hence we have $\FF_{\psi_1,temp}=\FF_{\psi_1,temp,-}$ and $\psi_1$ if of type $(III, \half{1})$. We have $\abs(\psi_1)= \rho_1 \otimes S_{6} \otimes S_5$ and 
\[ \FF_{\abs(\psi_1),temp,-}= \bordermatrix{
&\half{1}&\half{3}&\half{5}&\half{7}& \half{9} \cr
& \ominus&\oplus&\ominus&\oplus &\ominus \cr
}_{\rho_1}. \]
For $\psi_2$, we have 
\[ \FF_{\psi_2,temp,+}= \bordermatrix{
&0&1&2&3\cr
&\ominus&\oplus&\ominus &\oplus \cr
}_{\rho_2}, \ \FF_{\psi_2,temp,-}= \bordermatrix{
&0&1&2&3\cr
&\oplus&\ominus &\oplus&\ominus \cr
}_{\rho_2},  \]
and hence $\psi_2$ is of type (II). Then the formula in Theorem \ref{thm tempered rep general}(2) shows that there is no tempered representation in $\Pi_{\psi}$.

On the other hand, both $\psi_2'$ and $\psi_3'$ are of type (I) and we have
\begin{align*}
\FF_{\psi_2',temp,+}= \bordermatrix{
&0&1&2&3\cr
&\oplus& && \cr
}_{\rho_2}&,\  \FF_{\psi_2',temp,-}= \bordermatrix{
&0&1&2&3\cr
&\ominus& && \cr
}_{\rho_2},  \\
\FF_{\psi_3',temp,+}= \bordermatrix{
&0&1&2&3\cr
&&\oplus&\ominus&\oplus\cr
}_{\rho_2}&, \ \FF_{\psi_3',temp,-}= \bordermatrix{
&0&1&2&3\cr
&&\ominus &\oplus&\ominus \cr
}_{\rho_2}.
\end{align*}
So the formula in Theorem \ref{thm tempered rep general}(2) shows that there are $2$ tempered representations in $\Pi_{\psi'}$, and they are 
\begin{align*}
    \pi_1&=\pi( \FF_{\psi_1,temp,-} \cup \FF_{\psi_2',temp ,+} \cup \FF_{\psi_3',temp, +}),\\
    \pi_2&=\pi( \FF_{\psi_1,temp,-} \cup \FF_{\psi_2',temp ,-} \cup \FF_{\psi_3',temp, -}).
\end{align*}
Their $L$-data are the same if we replace $ \FF_{\psi_1,temp,-}$ by $\FF_{\abs(\psi_1),temp,-}$, and we have  
\begin{align*}
    \pi_1&= \pi\left( \left( \half{1}\right)_{\rho_1}^{-},  \left( \half{3}\right)_{\rho_1}^{+}, \left( \half{5}\right)_{\rho_1}^{-}, \left( \half{7}\right)_{\rho_1}^{+}, \left( \half{9}\right)_{\rho_1}^{-}, \left( 0\right)_{\rho_2}^{+},\left( 1\right)_{\rho_2}^{+},\left( 2\right)_{\rho_2}^{-},\left( 3\right)_{\rho_2}^{+}\right),\\
    \pi_2&= \pi\left( \left( \half{1}\right)_{\rho_1}^{-},  \left( \half{3}\right)_{\rho_1}^{+}, \left( \half{5}\right)_{\rho_1}^{-}, \left( \half{7}\right)_{\rho_1}^{+}, \left( \half{9}\right)_{\rho_1}^{-}, \left( 0\right)_{\rho_2}^{-},\left( 1\right)_{\rho_2}^{-},\left( 2\right)_{\rho_2}^{+},\left( 3\right)_{\rho_2}^{-}\right),
\end{align*}
by Corollary \ref{cor pos tempered}.
\end{exmp}

Given a tempered representation $\pi$, one can apply the converse of Theorem \ref{thm tempered rep general} to classify all local Arthur parameter $\psi$ such that $\Pi_{\psi}$ contains $\pi$, which provides a lot of examples for the non-tempered GGP conjecture considered in \cite{GGP20}. We explain a simple case in the following example, for notations, we refer to \cite{GGP20}. 

\begin{exmp}\label{exmp GGP}
Let $\phi_1\times \phi_2$ be a discrete tempered $L$-parameter of good parity of $\SO_{2n+1}\times \SO_{2m}$, and let $(\pi_1,\pi_2) \in \Pi_{\phi_1} \times \Pi_{\phi_2}$ (Vogan local $L$-packets) be the unique pair such that $d(\pi_1,\pi_2) =1$. We further assume that $\pi_1$ is a representation of the split $\SO_{2n+1}(F)$, and write $\pi_1= \pi(\phi_1, \varepsilon)$. Now we take any local Arthur parameter $\psi_1= \bigoplus_{i \in I} \rho_i \otimes S_{a_i} \otimes S_{b_i}$ such that $\pi_1 \in \Pi_{\psi_1}$ and $a_i \geq b_i$ for all $i$. By definition, we have
\[d(\psi_1, \phi_2)\geq d(\pi_1,\pi_2)= 1.\]
We claim that the pair $(\psi_1, \phi_2)$ is relevant if and only if $b_i \leq 2$ for all $i$. In this case, \cite[Conjecture 6.1]{GGP20} implies that
\[ d(\psi_1, \phi_2)\geq d(\phi_{\psi_1}, \phi_2)+ d(\pi_1,\pi_2) = 2. \]

Indeed, the necessity follows from the definition of relevance, so it suffices to show the sufficiency. Let $\EE, \EE'$ be the extended multi-segments such that $ \pi(\EE)=\pi(\EE')=\pi_1$ and $ \supp(\EE)= \supp(\phi_1)$, $\supp(\EE')= \supp(\psi_1)$. The assumption then implies $\EE'$ is obtained from $\EE$ by a sequence of $ui$ of type 3' (see also the proof of Proposition \ref{prop pos chain shape}). In each stage, the $ui$ changes the local Arthur parameter by
\[ \rho \otimes S_{a} \otimes S_1 + \rho \otimes S_{a+2} \otimes S_1 \mapsto \rho \otimes S_{a+1} \otimes S_2,\]
and the applicability of this $ui$ is equivalent to the equation
\begin{align*}
     \varepsilon(\rho \otimes S_{a} ) \varepsilon(\rho \otimes S_{a+2})=-1,
\end{align*}
which is also equivalent to the condition that $\phi_2$ contains $\rho \otimes S_{a+1}$ with odd multiplicity by \cite[Proposition 7.6]{GGP20}. Therefore, we can decompose
\begin{align*}
    \psi_1&=  \sum_{i \not\in I_1} \rho_i \otimes S_{a_i} \otimes S_1+\sum_{i \in I_1} \rho_{i} \otimes S_{a_i} \otimes S_2 ,\\
    \phi_2&=  \left(\phi_2- \sum_{i \in I_1} \rho_{i} \otimes S_{a_i} \otimes S_1\right)+\sum_{i \in I_1} \rho_{i} \otimes S_{a_i} \otimes S_1,
\end{align*}
where $I_1= \{i \in I \ | \ b_i=2\}$. This verifies the relevance of the pair $ (\psi_1,\phi_2)$.
\end{exmp}

\section{\texorpdfstring{The local $L$-packets of Arthur type}{}}\label{section L-packet}
Let $\psi$ be a local Arthur parameter of good parity, $\EE$ be an extended multi-segment in the packet of $\psi$, and $\phi_{\psi}$ be the $L$-parameter associated with $\psi$. To be more specific, write $$
\psi=\bigoplus_{i=1}^r (\rho_i\otimes S_{a_i} \otimes S_{b_i})^{\oplus n_i},
$$
where $n_i$ is the multiplicity. Then we have the following decomposition
\begin{align}\label{eq associated L-parameter}
\phi_{\psi}=\bigoplus_{i=1}^r  \left(\bigoplus_{j=0}^{b_i-1} \rho_i |\cdot|^{\frac{b_i-1}{2}-j}\otimes S_{a_i} \right)^{\oplus n_i}.    
\end{align}

The goal of this section is to give a necessary and sufficient condition on $\EE$ such that $\pi(\EE)$ is in the local $L$-packet $\Pi_{\phi_{\psi}}$. The results can be applied to general local Arthur parameters $\psi \in \Psi^{+}(G_n)$ not necessarily of good parity. To be explicit, decompose $\psi=\psi_{nu,>0}\oplus \psi_{np}\oplus \psi_{gp}\oplus \psi_{np}^{\vee} \oplus \psi_{nu,>0}^{\vee}$ as in Theorem \ref{thm red from nu to gp}. Then
\[ \Pi_{\psi}= \{ \tau_{\psi_{nu,>0}} \times \tau_{\psi_{np}} \rtimes \pi(\EE) \ | \ \pi(\EE) \in \Pi_{\psi_{gp}}\}, \]
and $\tau_{\psi_{nu,>0}} \times \tau_{\psi_{np}} \rtimes \pi(\EE) \in \Pi_{\phi_{\psi}}$ if and only if $ \pi(\EE) \in \Pi_{\phi_{\psi_{gp}}}$.

We give the following definition.

\begin{defn}\label{def (L)}
We say an extended multi-segment $\EE=\cup_{\rho} \{([A_i,B_i]_{\rho},l_i,\eta_i)\}_{i \in (I_{\rho},>)}$ satisfies (L) if after row exchanges, it satisfies the following conditions:
For all $\rho$ and $i<j \in I_{\rho}$,
\begin{enumerate}
    \item [(i)] $ A_i +B_i \leq A_{j}+B_{j}.$
    \item [(ii)] $(A_i-B_i+1)-2l_i \leq 1$.
    \item [(iii)] If $A_i+B_i=A_{j}+B_{j}$ and $A_i-B_i+1$ is odd, then $\eta_i=\eta_{j}$.
\end{enumerate}
\end{defn}

\begin{remark}\label{rmk (L)}\ 
\begin{enumerate}
    \item [(1)] The conditions on an extended multi-segment satisfying (L) is equivalent to the following conditions on the associated pictograph. Identify $(I_{\rho},>)$ with $\{1,\dots,n\}$ where $1<\cdots <n$.
\begin{enumerate}
    \item [(i)]The middle point of each row is non-decreasing.
    \item [(ii)] Every row contains at most one circle. Equivalently, every row has the maximal pairs of $\lhd,\rhd$.
    \item [(iii)] Any circles in the same column have the same sign.
\end{enumerate}
\item [(2)] Let $\EE= \cup_{\rho}\{([A_i,B_i]_{\rho},l_i,\eta_i)\}_{i \in (I_{\rho},>)} \in \Rep$ be an extended multi-segment with $A_i+B_i$ non-decreasing in $(I_{\rho},>)$ for any $\rho$. We construct another extended multi-segment
\[ \EE_0:=\left(\sum_{\rho} \sum_{i \in I_{\rho}} add_i^{-l_i}\right)(\EE). \]
Here if $A_i-B_i+1= 2l_i$, then the $add_{i}^{-l_i}$ is understood as removing the extended segment $([A_i,B_i]_{\rho},l_i,\eta_i)$ from $\EE$. Then $\EE$ satisfies (L) if and only if $\psi_{\EE_0}$ is tempered and $\pi(\EE_0)\neq 0$.
\end{enumerate}

\end{remark}

The following proposition describes a sufficient condition on $\EE$ such that $\pi(\EE) \in \Pi_{\phi_{\psi_{\EE}}}$ in the special case that $\EE$ is positive.
\begin{prop}\label{prop $L$-data L-packet}
If $\EE=\cup_{\rho} \{([A_i,B_i]_{\rho},l_i,\eta_i)\}_{i \in (I_{\rho},>)}$ is a positive extended multi-segment satisfying (L), then $\pi(\EE) \neq 0$. Moreover, its $L$-data can be described as follows. Write
\[ \pi(\EE)= L\left( \Delta_{\rho_1}[x_1,-y_1],\dots,\Delta_{\rho_t}[x_t,-y_t]; \pi\left(\sum_{i=t+1}^{m} \rho_i \otimes  S_{2z_i+1},\varepsilon \right) \right). \]
Then
\begin{enumerate}
    \item [(a)] We have an equality of multi-sets
    \begin{align*}
        \{ [x_i,-y_i]_{\rho_i} \}_{i=1}^t&= \bigcup_{\rho}\sum_{i\in I_{\rho}} \{ [B_i,-A_i]_{\rho},\dots,[B_i+l_i-1,-A_i+l_i-1]_{\rho} \}.
    \end{align*}
    \item [(b)] Let $I_{\rho}^{odd}= \{j \in I_{\rho}\ | \ A_j-B_j+1 \text{ is odd}\}$. Then we have equality of multi-sets
    \[\\
        \{(\rho_i, z_i)\}_{i=1}^m =\bigcup_{\rho} \left\{ \left(\rho,\frac{A_j+B_j}{2}\right)\ \middle| \ j \in I_{\rho}^{odd} \right\}. \]
    If $z_i= \half{A_j+B_j}$, the character $\varepsilon$ takes value $\eta_j$ on the summand $\rho\otimes S_{2z_i+1}$.
\end{enumerate}
In particular, $\pi(\EE) \in \Pi_{\phi_{\psi_{\EE}}}$.
\end{prop}

\begin{proof}
If $|I_{\rho}|=1$ for all $\rho$ such that $\EE_{\rho}$ is non-empty, the conclusions are clear by definition. We proceed by induction on $\sum_{\rho, \EE_{\rho}\neq \emptyset} (|I_{\rho}|-1)$.

Fix a $\rho$ such that $|I_{\rho}|>1$ and denote $\FF:=\EE_{\rho}= \{([A_i,B_i]_{\rho},l_i,\eta_i)\}_{i=1}^n$. Decompose $\FF=\FF_1+\FF_2$ where $\FF_2$ consists of the $n$-th row. We denote $\EE_{d}:=\EE^{\rho} \cup (sh_n^d(\FF))$. Clearly, $\EE_d$ also satisfies (L). When $d$ is large enough, the computation of derivative in \cite{AM20} shows that the conclusions hold for $\EE_{d} $ if and only if the same hold for $\EE^{\rho} \cup \FF_1 \cup (sh^d(\FF_2))_{\rho^{\ast}}$, where $\rho^{\ast}$ is chosen as in Definition \ref{def modification}. Therefore, the induction hypothesis implies $\EE_d $ satisfies all of the conclusions. Thus, it suffices to show the conclusions for $\EE_d$ implies that for  $\EE_{d-1}$. By replacing $\FF$ with $sh^{d-1}_n(\FF)$, we assume $d=1$ without loss of generality. 

Let $\pi$ be the representation whose $L$-data satisfies Conditions (a) and (b) in the statement (with respect to $\EE$). Our goal is to show
\begin{align} \label{eq L-data}
     D_{\rho|\cdot|^{A_n+1}}\circ \cdots\circ D_{\rho|\cdot|^{B_{n}+1}}(\pi(\EE_1) )\geq \pi.
\end{align}
By definition, the left hand side is exactly $\pi(\EE)$. Therefore, if \eqref{eq L-data} is proved, then $\pi(\EE)=\pi \neq 0$.

We remark that since the admissible order on $I_{\rho}$ does not necessarily satisfy (P'), some of the derivatives in \eqref{eq L-data} may not be highest, and hence those derivatives may not give irreducible representations. However, based on the form of $\pi(\EE_1)$, we can pinpoint the correct irreducible subquotient of these derivatives that contribute to $\pi$ eventually.

We define an intermediate irreducible representation $\pi'$ by replacing the segments $\{\Delta_{\rho}[B_n+1,-(A_n+1)],\dots,\Delta_{\rho}[(B_{n}+1)+l_n-1, -(A_n+1)+l_n-1] \}$ in the $L$-data of $\pi(\EE_1)$ with $\{\Delta_{\rho}[B_n,-(A_n+1)],\dots,\Delta_{\rho}[B_{n}+l_n-1, -(A_n+1)+l_n-1] \}$ (possibly up to a reordering). Then it suffices to show the following claims.
\begin{enumerate}
    \item [1.] $ D_{\rho|\cdot|^{B_n+l_n}}\circ \cdots \circ D_{\rho|\cdot|^{B_{n}+1}}(\pi(\EE_1)) \geq \pi'$.
    \item [2.] $ D_{\rho|\cdot|^{A_{n}+1}}\circ \cdots \circ D_{\rho|\cdot|^{B_n+l_n+1}}(\pi') =\pi$.
\end{enumerate}


Now we start the proof of Claim 1. Set $\pi_0':=\pi(\EE_1)$. For $i=1,\dots,l_n$, We define $\pi_i'$ recursively by replacing the segment $ \Delta_{\rho}[(B_n+1)+i-1,-(A_n+1)+i-1]$ in the $L$-data of $\pi_{i-1}'$ with $ \Delta_{\rho}[B_n+i-1,-(A_n+1)+i-1]$. Note that $\pi_{l_n}'=\pi'$. We claim that $\pi_i' \leq D_{\rho|\cdot|^{B_n+i}} (\pi_{i-1}')$. Indeed, write
\begin{align*}
    \pi_{i-1}' \hookrightarrow & \left(\bigtimes_{j=1}^{r} \Delta_{\rho_j}[\alpha_j,\beta_j]\right)\times \Delta_{\rho}[(B_n+1)+i-1,-(A_n+1)+i-1]\\
    &\times \left(\bigtimes_{j=r+2}^{t} \Delta_{\rho_j}[ \alpha_{j},\beta_{j}] \right)\rtimes \pi(\phi,\varepsilon)
\end{align*}
as the unique irreducible subrepresentation from Langlands classification. We may assume
\begin{align}\label{eq 10.2}
    \alpha_j+ \beta_j < ((B_n+1)+i-1) + (-(A_n+1)+i-1)=B_n-A_n+2i-2
\end{align}
for all $1 \leq j \leq r$. Moreover, using Condition (ii) of $\EE_1$ and Conclusion (a) for $\pi(\EE_1)$, one can see that for $1 \leq j \leq r$, if $\rho_j \cong \rho$, then 
\begin{align}\label{eq 10.3}
    \alpha_j- \beta_j  < (B_n+1)+ (A_n+1).
\end{align}
Therefore, from \eqref{eq 10.2} and \eqref{eq 10.3}, we have (note that $A_n \pm B_n \in \Z$)
\[ \alpha_j \leq  B_n+i -2. \]
As a consequence, the segments $[B_n+i,B_n+i]_{\rho}$ and $[\alpha_j,\beta_j]_{\rho_j}$ are not linked for all $1 \leq j\leq r$, and hence
\begin{align*}
    \pi_{i-1}' \hookrightarrow & \left(\bigtimes_{j=1}^{r} \Delta_{\rho_j}[\alpha_j,\beta_j]\right)\times \rho|\cdot|^{B_n+i} \times \Delta_{\rho}[B_n+i-1,-(A_n+1)+i-1]\\
    &\times \left(\bigtimes_{j=r+2}^{t} \Delta_{\rho_j}[ \alpha_{j},\beta_{j}]\right) \rtimes \pi(\phi,\varepsilon)\\
    = &\rho|\cdot|^{B_n+i} \times \left(\bigtimes_{j=1}^{r} \Delta_{\rho_j}[\alpha_j,\beta_j] \right) \\ &\times  \Delta_{\rho}[B_n+i-1,-(A_n+1)+i-1] \times \left(\bigtimes_{j=r+2}^{t} \Delta_{\rho_j}[ \alpha_{j},\beta_{j}] \right)\rtimes \pi(\phi,\varepsilon)\\
    =& \rho|\cdot|^{B_n+i} \rtimes \sigma,
\end{align*}
where
\[ \sigma= \left(\bigtimes_{j=1}^{r} \Delta_{\rho_j}[\alpha_j,\beta_j]\right)\times  \Delta_{\rho}[B_n+i-1,-(A_n+1)+i-1]\times \left(\bigtimes_{j=r+2}^{t} \Delta_{\rho_j}[ \alpha_{j},\beta_{j}] \right)\rtimes \pi(\phi,\varepsilon).  \]
Note that $\pi_i' $ is the unique irreducible subrepresentation of $\sigma$ from Langlands classification. Then Lemma \ref{lem Frobenius} implies $D_{\rho|\cdot|^{B_n+i}}(\pi_{i-1}') \geq \pi_i'$. This completes the proof of the first claim.

Claim 2 follows directly from the formula in \cite[Theorem 7.1]{AM20} and we omit the detail. We remark that each derivative in Claim 2 is highest due to Conditions (ii) and (iii) for $\EE_1$ in the statement. 
This completes the proof of the proposition.
\end{proof}

Fix a general extended multi-segment $\EE$. For $d \in \Z_{\geq 0}$, let $sh^{d}(\psi)$ be the local Arthur parameter associated to $sh^d(\EE)$, and let $sh^d(\phi_{\psi}):= \phi_{sh^d(\psi)}$. Note that if $\phi_{\psi}$ is given in \eqref{eq associated L-parameter}, then 
\[sh^d(\phi_{\psi})=\bigoplus_{i=1}^r  \left(\bigoplus_{j=0}^{b_i-1} \rho_i |\cdot|^{\frac{b_i-1}{2}-j}\otimes S_{a_i+2d} \right)^{\oplus n_i}.\]
The next proposition helps us to generalize Proposition \ref{prop $L$-data L-packet}.
\begin{prop} \label{prop L-packet shift}

Suppose $\pi(\EE)$ is in the L-packet $\Pi_{\phi_{\psi_{\EE}}}$. Then
\begin{enumerate}
    \item [(a)]$  \Omega(\EE)= \Omega(\pi(\EE)),$
    \item [(b)] $\pi(sh^d(\EE)) \in \Pi_{sh^d(\phi_{\psi_{\EE}})}$ for any $d \in \N$.
\end{enumerate}
\end{prop}

\begin{proof}
Write \[ \pi(\EE)= L\left( \Delta_{\rho_1}[x_1,-y_1],\dots,\Delta_{\rho_t}[x_t,-y_t]; \pi\left(\sum_{j \in J} \rho_j\otimes  S_{2z_j+1},\varepsilon \right) \right). \]

For Part (a), write  $\psi= \bigoplus_{i\in I} \rho_i \otimes S_{a_i}\otimes S_{b_i}$. The $L$-data of $\pi(\EE)$ shows that the $L$-parameter of $\pi(\EE)$ is
\[ \phi'=\bigoplus_{s=1}^t \left(\rho_s |\cdot|^{\frac{x_s-y_s}{2}} \otimes S_{x_s+y_s+1 }\oplus \rho_s^{\vee} |\cdot|^{-\frac{x_s-y_s}{2}} \otimes S_{x_s+y_s+1 } \right) \oplus \bigoplus_{j \in J} \rho_j \otimes S_{2z_j+1}, \]
which is the same as 
\[ \phi_{\psi}= \bigoplus_{ i\in I} \bigoplus_{k=0}^{b_i-1} \rho_i |\cdot|^{\frac{b_i-1}{2}-k} \otimes S_{a_i}\]
by assumption. Observe that the cardinality of $\Omega(\EE)$ (resp. $\Omega(\pi(\EE))$) are exactly the number of summands of $\phi_{\psi}$ (resp. $\phi'$). As $\phi_{\psi}=\phi'$, the containment $\Omega(\EE_{\rho}) \supseteq \Omega(\pi(\EE)_{\rho})$ in Lemma \ref{lem $L$-data} is indeed an equality for all $\rho$. Therefore, $\Omega(\EE)= \Omega(\pi(\EE))$. This completes the proof of Part (a). 

As a consequence of Part (a), Theorem \ref{thm non-vanishing}(i) shows that 
\[ \pi(sh^1(\EE))= L\left( \Delta_{\rho_1}[x_1+1,-y_1-1],\dots,\Delta_{\rho_t}[x_t+1,-y_t-1]; \pi\left(\sum_{j \in J} S_{2(z_j+1)+1},\varepsilon \right) \right).\]
The corresponding $L$-parameter is therefore just changing each $S_{a}$ in $\phi'$ by $S_{a+2}$, which is the same as how we derive $sh^{1}(\phi_{\psi})$ from $\phi_{\psi}$. This shows the case $t=1$ of Part (b). The general case follows from induction. This completes the proof of the proposition.
\end{proof}

The following is the main theorem of this section and describes those extended multi-segments which parameterize the local $L$-packet of the local Arthur packet.

\begin{thm}\label{thm in L-packet}
For any $\EE \in \Rep$, the representation $\pi(\EE)$ is in the local $L$-packet associated with $\psi_{\EE}$ if and only if $\EE$ satisfies (L). Moreover, Proposition \ref{prop $L$-data L-packet} holds without assuming $\EE$ is positive.
\end{thm}

\begin{proof}
Suppose $\EE$ is positive and write $\pi= \pi(\EE)$. If $\EE$ satisfies (L), then $\pi(\EE)\in \Pi_{\phi_{\psi_{\EE}}}$ by Proposition \ref{prop $L$-data L-packet}. Conversely, suppose $\pi \in \Pi_{\phi_{\psi_{\EE}}}$. Reversing the process of computing the $L$-parameter from the $L$-data of a representation, we see that the $L$-data of $\pi$ is described by Conditions (a) and (b) in Proposition \ref{prop $L$-data L-packet} for some $\EE'$ satisfying (L) and $\psi_{\EE'}=\psi_{\EE}$. Thus $\pi(\EE')=\pi=\pi(\EE)$, and hence $\EE=\EE'$ up to row exchanges by Lemma \ref{lem Moeglin}. Therefore, $\EE$ must satisfy (L).

In general, suppose $\pi(\EE) \in \Pi_{\phi_{\psi_{\EE}}}$. Take $d \in \Z_{\geq 0}$ such that $sh^d(\EE)$ is positive. Then Proposition \ref{prop L-packet shift}(b) shows that $\pi(sh^d(\EE)) \in \Pi_{sh^d(\phi_{\psi_{\EE}})}$, and hence $sh^d(\EE)$ must satisfy (L). It follows that $\EE$ also satisfies (L).

Conversely, suppose $\EE$ satisfies (L) and take $d \in \Z_{\geq 0}$ such that $sh^d(\EE)$ is positive. Proposition \ref{prop $L$-data L-packet} implies that $\pi(sh^d(\EE)) \in \Pi_{sh^d(\phi_{\psi_{\EE}})}$. Then we compute the $L$-data of $\pi(\EE)$ from $\pi(sh^d(\EE))$ by applying Theorem \ref{thm non-vanishing}(i), from which we conclude that $\pi(\EE)$ is nonzero and $\pi(\EE) \in \Pi_{\phi_{\psi_{\EE}}}$. Moreover, the $L$-data of $\pi(\EE)$ is given by the same description in Proposition \ref{prop $L$-data L-packet}. This completes the proof of the theorem.
\end{proof}

We demonstrate the above theorem on Example \ref{ex exhaustion}. The extended multi-segment $\EE_7$ satisfies (L), and hence $\pi(\EE_7) \in \Pi_{\phi_{\psi_7}}$. Note that $\EE_7$ is the only member in $\Psi(\pi(\EE_7))$ that satisfies (L). This reflects the fact that a representation lies in at most one $L$-packet of Arthur type.

Let $\pi$ be a tempered generic representation of the split groups $\Sp_{2n}(F)$, $\SO_{2n+1}(F)$. In the following proposition, applying Theorem \ref{thm in L-packet}, we show that the Aubert-Zelevinsky involution of $\pi$ lies in the $L$-packet associated with an anti-tempered local Arthur parameter, i.e., a local Arthur parameter of not good parity, whose restriction to the Deligne-$\SL_2(\BC)$ is trivial. In \cite[Lemma 6.13]{HLLZ22}, we gave another proof of this proposition using the results in \cite[Section 7]{Art13}, which works for quasi-split classical groups.

\begin{prop}\label{prop dual of generic}
    Suppose $\pi$ is a tempered generic representation of $G_n= \Sp_{2n}(F)$ or split $\SO_{2n+1}(F)$ that lies in a tempered local Arthur packet $\Pi_{\psi}$. Then $\widehat{\pi}$ lies in the $L$-packet associated with the anti-tempered local Arthur parameter $\widehat{\psi}$.
\end{prop}

\begin{proof}
    We assume $\pi$ is of good parity by Theorem \ref{thm reduction to gp}. Write 
    \[ \psi= \bigoplus_{\rho} \bigoplus_{i \in I_{\rho}} \rho \otimes S_{2\alpha_i+1} \otimes S_1.\]
    Then since generic representation corresponds to the trivial character (see the discussion before \cite[Theorem 4.1]{HLL24}), we have $\pi= \pi(\EE)$, where
    \[ \EE= \bigcup_{\rho} \{ ([\alpha_i,\alpha_i]_{\rho},0, 1) \}_{i \in (I_{\rho},>)}.\]
    Applying Theorem \ref{thm Aubert-Zelevinsky dual formula}, we obtain $\widehat{\pi}= \pi(dual(\EE))$, where
    \[ dual(\EE)= \bigcup_{\rho} \{( [\alpha_i,-\alpha_i]_{\rho},\lceil \alpha_i\rceil, \varepsilon_{\rho})\}_{i \in (I_{\rho},>')},\]
    and 
    \[\varepsilon_{\rho}=\begin{cases}
        1 & \text{ if }a_i \in \half{1}+\Z \text{ for any }i \in I_{\rho},\\
         (-1)^{|I_{\rho}|} & \text{ otherwise.}
    \end{cases}\]
    (Note that in the first case, the length of $[\alpha_i,\alpha_i]_{\rho}$ is equal to $2 \lceil \alpha_i\rceil$, and hence we can take the sign to be 1.)
    Since $dual(\EE)$ satisfies (L), Theorem \ref{thm in L-packet} implies that $\widehat{\pi}$ is in the $L$-packet associated with $\psi_{dual(\EE)}= \widehat{\psi}$. This completes the proof of the lemma.
\end{proof}


Finally, we remark that it is possible that $\Pi_{\psi}= \Pi_{\phi_{\psi}}$, where $\psi$ is non-tempered and $\Pi_{\psi}$ is not a singleton. We give an example below.

\begin{exmp}
    Let $\rho$ be the trivial representation and consider 
    \[ \psi= \rho \otimes S_{1} \otimes S_{3}+ \rho \otimes S_{3} \otimes S_1 + \rho \otimes S_{5} \otimes S_1,\]
    a local Arthur parameter of $\Sp_{10}(F)$. We have $\Pi_{\psi}= \{\pi(\EE_1),\pi(\EE_2),\pi(\EE_3),\pi(\EE_4)\}$, where
    \begin{align*}
        \EE_1= \bordermatrix{
        &-1&0&1&2\cr
        &\lhd & \oplus & \rhd& \cr
        &&&\oplus &\cr
        &&&&\oplus\cr         
        },&\ \  \EE_2= \bordermatrix{
        &-1&0&1&2\cr
        &\lhd & \ominus & \rhd& \cr
        &&&\ominus &\cr
        &&&&\oplus\cr         
        },\\
        \EE_3= \bordermatrix{
        &-1&0&1&2\cr
        &\lhd & \oplus & \rhd& \cr
        &&&\ominus &\cr
        &&&&\ominus\cr         
        },&\ \  \EE_4= \bordermatrix{
        &-1&0&1&2\cr
        &\lhd & \ominus & \rhd& \cr
        &&&\oplus &\cr
        &&&&\ominus\cr         
        }.
    \end{align*}
    These four extended multi-segments all satisfy (L). Thus  by Theorem \ref{thm in L-packet}
    \[\Pi_{\phi_{\psi}}=\{\pi(\EE_1),\pi(\EE_2),\pi(\EE_3),\pi(\EE_4)\}= \Pi_{\psi}. \]   
\end{exmp}

\section{``The" local Arthur parameter of a representation of Arthur type}\label{section the Arthur parameter}

In this section, for each representation $\pi=\pi(\EE)$ of Arthur type and of good parity, we construct an extended multi-segment $\EE^{|max|}$ such that $\pi(\EE^{|max|})= \pi$, that is, 
$\pi \in \Pi_{\psi_{\EE^{|max|}}}$,
and it satisfies the property that if $\pi \in \Pi_{\phi_{\psi}}$, then $\psi_{\EE^{|max|}}=\psi$ (see Theorem \ref{thm can max} below). Then, we give a way to characterize the property that $\pi \in \Pi_{\phi_{\psi}}$ in terms of the $L$-data of $\pi$ (see Theorem \ref{thm pi minus L packet}). This specifies a distinguished member $\psi^{max}(\pi):=\psi_{\EE^{|max|}}$ in the set \[ \Psi(\pi):= \{ \psi \ |\ \pi \in \Pi_{\psi}  \}.\]
We also define another distinguished member $\psi^{min}(\pi) \in \Psi(\pi)$ as an opposite of $\psi^{max}(\pi)$. We give two representation theoretic characterizations of these two distinguished members in the next section. Finally, we generalize these results for any $\pi \in \Pi_{\psi}$ where $\psi\in \Psi^+(G_n)$.

We have defined $\EE_{can}$ in Definition \ref{def can form}, which carries the most derivative information of $\pi(\EE)$ (see Theorem \ref{thm canonical form and derivatives}). However, if $\pi(\EE) \in \Pi_{\phi_{\psi}}$ for some $\psi$, $\psi_{\EE_{can}}$ is usually not equal to $\psi$. Let us explain this phenomenon based on Example \ref{ex Atobe}. Let
\[ \EE_1= \bordermatrix{ & 0&1&2 \cr 
& \ominus & \oplus & \ominus \cr }_{\rho}, \ \EE_4= \bordermatrix{ & 0&1&2 \cr 
& \ominus & &\cr 
& &\oplus & \cr 
&&& \ominus \cr }_{\rho}.  \]
Let $\psi_4$ be the local Arthur parameter associated with $\EE_4$. Then Theorem \ref{thm in L-packet} shows $\pi(\EE_1) \in \Pi_{\phi_{\psi_4}}$. On the other hand, we have computed that $\EE_1= \EE_{can}$. Therefore, $\psi_{\EE_{can}} \neq \psi_4$ in this example.

We may explain this phenomenon further under the assumption that $sh^{-1}(\EE)$ is non-negative. In this case, $\EE_{can}= \EE^{min}$ by construction, and it is the unique element in $\Psi(\EE)$ that is minimal, i.e., no $ui_{i,j}$ is applicable. (See Corollary \ref{cor abs min}(ii) below for the detail of this statement.)
A simple computation shows that if $\EE'\in \Psi(\EE)$ satisfies (L) (see Definition \ref{def (L)}), then $\EE'$ is necessarily maximal, i.e., no $ui^{-1}$ is applicable (see the proof of Proposition \ref{prop max elm in L-packet} below). Thus $\EE_{can}=\EE'$ if and only if $\EE_{can}$ is both minimal and maximal, which holds only if $\Psi(\EE)$ is a singleton.

Following the discussion above, it is natural to ask the following question. Is there a unique maximal element in the set $\Psi(\EE)$? The answer is negative without extra assumptions on $\EE$. Indeed, in Example \ref{ex Atobe}, $\EE_4,\EE_{7}, \EE_{8},\EE_9$ in $\Psi(\EE)$ are all maximal. 
In order to uniquely identify $\EE_4$, we give the following definition.

\begin{defn}\label{def abs max}
For an extended multi-segment $\EE=\cup_{\rho} \EE_{\rho}$, we say $\EE_{\rho}$ is absolutely maximal if the following holds.
\begin{enumerate}
    \item [(i)]$\EE_{\rho}$ is maximal, i.e., no $ui^{-1}$ is applicable on $\EE_{\rho}$.
    \item [(ii)] $dual(\EE_{\rho})$ is minimal, i.e., no $ ui_{i,j}$ is applicable on $dual(\EE_{\rho})$.
    \item [(iii)] For any $k \in I_{\rho}$ and admissible order $\gg$ of $I_{\rho}$ satisfying (P'), $dual_{k}^{-}$ is not applicable on $\EE_{\rho, \gg}$.
\end{enumerate}

We say $\EE$ is absolutely maximal if $\EE_{\rho}$ are absolutely maximal for all $\rho$.
\end{defn}

\begin{exmp}\ 
\begin{enumerate}
    \item [1.]Let us show that the extended multi-segments $\EE_7,\EE_8, \EE_9$ in Example \ref{ex Atobe} are not absolutely maximal. Indeed, we have \[  dual(\EE_7)=\EE_4, \ dual(\EE_8)=\EE_9, \ dual(\EE_9)=\EE_8. \]
None of above extended multi-segments are minimal, and hence $\EE_7,\EE_8, \EE_9$ do not satisfy Condition (ii) in the definition above. On the other hand, $dual(\EE_4)= \EE_{7}$, which is minimal. As it is in integer case, no partial dual is applicable. We conclude that $\EE_4$ is absolutely maximal.
\item [2.] Similarly, one can check that in Example \ref{ex exhaustion}, both $\EE_7$ and $\EE_8$ are maximal, but only $\EE_7$ is absolutely maximal. Furthermore, $\EE_7$ satisfies (L), so $\pi(\EE_7)$ is in the local $L$-packet associated with $\psi_{\EE_7}$.
\item [3.]To see why Condition (iii) in Definition \ref{def abs max} is necessary, we consider the following two extended multi-segments of $ \SO_{13}(F)$
\[ \EE_1= \bordermatrix{
& \half{-1} & \half{1} & \half{3}& \half{5} \cr 
&\oplus & \ominus && \cr 
&&& \oplus & \cr 
&&&& \ominus \cr
}_{\rho}, \ \EE_2= \bordermatrix{
 & \half{1} & \half{3}& \half{5} \cr 
 & \ominus && \cr 
&& \oplus & \cr 
&&& \ominus \cr
}_{\rho}, \]
where $\rho$ is the trivial representation. We have 
\[\pi(\EE_1) \cong \pi(\EE_2)= \pi( (1/2)^{-}, (3/2)^{+}, (5/2)^{-}).\]
$\EE_2$ satisfies (L), and hence $\pi(\EE_2)$ is in the local $L$-packet associated with $\psi_{\EE_2}$. On the other hand, both of them satisfy Conditions (i) and (ii) in Definition \ref{def abs max}, but $\EE_1$ does not satisfy the Condition (iii) since $ \EE_2= dual_1^{-}(\EE_1)$.
\end{enumerate}
\end{exmp}

In the next section, we give measurements of the ``temperedness" of local Arthur parameters, and show that if $\EE$ is absolutely maximal, then $\psi_{\EE}$ is the ``most tempered" parameter in the set $\{\psi \ | \ \pi(\EE) \in \Pi_{\psi}\}$ (see Theorems \ref{thm Jiang's partition}, \ref{thm Moeglin}).

Now we show that if $\pi(\EE)$ is in the local $L$-packet associated with $\psi_{\EE}$, then $\EE$ is absolutely maximal. 
\begin{prop}\label{prop max elm in L-packet}
If $\EE$ satisfies (L), then $\EE$ is absolutely maximal.
\end{prop}
\begin{proof}
It suffices to prove $\EE_{\rho}$ is absolutely maximal for each $\rho$. For simplicity, we write $\EE_{\rho}=\FF=\{([A_i,B_i]_{\rho},l_i,\eta_i)\}_{i \in (I_{\rho,>})}$.

First, we show that $\FF$ is maximal if $dual(\FF)$ is minimal. A simple computation shows that Corollary \ref{cor split} is never applicable on $\FF$, so there is no $ui\inv$ of type 3' applicable. Therefore, Corollary \ref{cor ui inverse} indicates that $\FF$ is maximal if $dual(\FF)$ is minimal. 

Next, we show that $ dual(\FF)$ is minimal. Since $dual(\FF)$ is minimal if and only if $sh^1(dual(\FF))$ is minimal, to check the minimality of $dual(\FF)$, we replace $\FF$ with $add^1(\FF)$ by Lemma \ref{lem identities}(ii), and hence assume $ A_k-B_k \geq 2$ for all $k \in I_{\rho}$.

A key observation from the non-vanishing result in Proposition \ref{prop $L$-data L-packet} is that for any $k \in I_{\rho}$, the extended multi-segment $\EE^{\rho} \cup add^{-1}_k(\FF)$ also satisfies (L), and hence $\pi( \EE^{\rho} \cup add^{-1}_k(\FF))$ is nonzero. Here, we use the assumption $ A_k -B_k \geq 2$. Then, by Lemma \ref{lem identities}(ii) and Theorem \ref{thm Aubert-Zelevinsky dual formula}, $\pi(dual(\EE^{\rho}) \cup sh^{-1}_k(dual(\FF))) \neq 0$ for all $k \in I_{\rho}$. Therefore, no $ui_{i,j}$ is applicable on $dual(\FF)$ by Lemma \ref{lem ui shift}, and hence $dual(\FF)$ is minimal.

Finally, we show that $dual_k^{-}$ is not applicable on $\FF_{\gg}$ for any admissible order $\gg$. By definition of $dual_k^{-}$, we may assume $B_k=-1/2$. We separate into two cases. 

Case (1): Suppose $A_k=1/2$. Then $0=A_k + B_k \leq A_i+B_i$ for any $i \in I_{\rho}$. The condition (L) implies that under any admissible order $\gg$ of $I_{\rho}$ satisfying (P'), the $k$-th row of $\FF_{\gg}$ is $([1/2,-1/2]_{\rho},1,1)$. Therefore, $dual_k^{-}$ is never applicable on $\FF_{\gg}$.

Case (2): Suppose $A_k>1/2$. Then the non-vanishing results in Proposition \ref{prop $L$-data L-packet} imply that $\EE^{\rho} \cup add_k^{-1}(\FF) \in \Rep$, and hence for any admissible order $\gg$ of $I_{\rho}$, the $k$-th row of $\FF_{\gg}$ is of the form $([A_k,-1/2]_{\rho},l,\eta)$ with $l\geq 1$. Therefore, $dual_k^{-}$ is never applicable on $\FF_{\gg}$.
This completes the proof of the proposition.
\end{proof}

The proof above is based on the argument that if $ui^{-1}$ of type 3' (see Definition \ref{def ui inverse}) is not applicable, then $\EE$ is absolutely maximal if and only if $ dual(\EE)$ is minimal and $dual_k^{+}$ is not applicable on $dual(\EE)$ under any admissible order. As explained in Remark \ref{rmk partial dual}, $dual_k^{-}$ is applicable on $\EE$ if and only if $ui_{0,k}$ is applicable on
\[ \{([-1/2,-1/2]_{\rho},0,1)\}+ dual(\EE_{\rho}), \]
where the index $0$ corresponds to the extra row $([-1/2,-1/2]_{\rho},0,1)$ added by the phantom operator of Atobe (\cite[Definition 3.4]{Ato22}). In the next lemma, we show that the applicability of $ui^{-1}$ of type 3' can also be detected by adding an extra phantom row of the form 
\[ ([x,-x-1]_{\rho}, \lfloor x+1 \rfloor ,1),\]
where $x > -1/2$. To give a motivation, recall that in Definition \ref{def ui} (also Definition \ref{ui def}), we delete the $j$-th row, which is of the form $([x,x+1]_{\rho},0,\eta)$, when the $ui_{i,j}$ is of type 3'. If we formally keep this row and then take dual, one may expect from Corollary \ref{cor ui inverse} that $ui_{j,i}$ is applicable on $dual(ui_{i,j}(\EE))$. A calculation indicates that if we raise the $j$-th row of $dual(ui_{i,j}(\EE))$ to the top of it by row exchanges, it becomes the phantom row described above.

\begin{lemma}\label{lem phantom and ui inverse}
Suppose $\EE= \in \Rep^{(P')}$ and 
\[ \FF:= \EE_{\rho}=\{([A_i,B_i]_{\rho},l_i,\eta_i)\}_{i \in (I_{\rho}, >)}.\]
Then the followings are equivalent:
\begin{enumerate}
    \item [(i)] We may apply $ui^{-1}$ of type 3' to split 
    the $k$-th row of $\FF$ into two rows with supports $ [x,B_k]_{\rho}, [A_k,x+1]_{\rho}$ (see Corollary \ref{cor split}).
    \item [(ii)] $ui_{0,k}$ is applicable on
    \[ \widetilde{\FF}= sh^{1}\left(\{([x,-x-1]_{\rho}, \lfloor x+1 \rfloor ,1) \}+ dual(\FF)\right),\]
    where we identify the total ordered set of index of $\widetilde{\FF}$ with $\{0\} \cup (I_{\rho},>)$ where $0$ is the minimal element.
\end{enumerate}
\end{lemma}
\begin{proof}
The idea of the proof is to perform the argument above rigorously by replacing $\FF$ with $add^1(\FF)$.

We first show Part (i) implies Part (ii). Let $\FF'$ be such that $ui_{k,j}(\FF')=\FF$ as described in Part (i). Then $ui_{k,j}$ is applicable on $add^1(\FF')$, which is of type 3 but not 3' (see Definition \ref{ui def}). Comparing with $ add^1(ui_{k,j}(\FF))$, $ ui_{k,j}(add^{1}(\FF'))$ has an extra row of the form $([x+1,x]_{\rho},1,1)$. Therefore, comparing with $dual(add^1(\FF))=sh^1(dual(\FF))$,  $dual\circ ui_{k,j} \circ add^1(\FF')$ contains an extra row of the form $([x+1,-x]_{\rho},\lfloor x+1 \rfloor,\ast)$. We may define a new admissible $\gg$ order on the index set of $dual\circ ui_{k,j} \circ add^1(\FF')$ to make the index of this extra row minimal. Then it can be deduced from the definition that  
\[ (dual\circ ui_{k,j} \circ add^1(\FF'))_{\gg}= \{([x+1,-x]_{\rho}, \lfloor x+1 \rfloor ,1)\} +sh^1(dual(\FF)).\]
Therefore, Part (ii) is a consequence of Corollary \ref{cor ui inverse}.

Conversely, suppose Part (ii) holds. From the non-vanishing condition in Theorem \ref{thm non-vanishing}(i) on $dual(\FF)$, $ui_{0,k}(\widetilde{\FF})$ is never of type 3'. Take any admissible order $\gg$ of $ \{0\}\cup I_{\rho}$ that satisfies ($P'$) and preserves the order of $I_{\rho}$. It can be deduced from the definition that $dual(\widetilde{\FF}_{\gg})$ can be obtained from $add^1(\FF)$ by inserting an extra row of the form $([x+1,x]_{\rho},1,1)$. Now Corollary \ref{cor ui inverse} implies 
\[ ui_{k,0} \circ dual\circ ui_{0,k}(\widetilde{\FF}_{\gg})= dual (\widetilde{ \FF}_{\gg}),  \]
where the $ui_{k,0}$ on the left hand side is necessarily of type 3. Then we conclude that
\[ ui_{k,0} ( add^{-1}\circ dual \circ ui_{0,k}(\widetilde{\FF}_{\gg}))= \FF  \]
up to row exchanges, and the $ui_{k,0}$ on left hand side is of type 3'. This shows Part (i). The proof of the lemma is now complete.
\end{proof}

\begin{remark}\label{rmk dual ui dual}
From the lemma above, we see $dual \circ ui_{i,j} \circ dual$ of type 3' is the same as an inverse of a composition of Atobe's phantom operator $P$ and $ui_{i,j}$ $($\cite[Definition 3.4]{Ato22}$)$. We give an example of this below. On the other hand, Corollary \ref{cor ui inverse} shows that $dual \circ ui_{i,j} \circ dual$ not of type 3' is indeed the same as inverse of $ui_{j,i}$. Therefore, all of the basic operators can be written as a composition of Atobe's $R_k, ui_k$, $P$ and their inverses. 

Remark \ref{rmk partial dual} also shows $dual_k$ is a composition of $R_k, ui_k$, P and their inverses. Therefore, Theorems \ref{thm integer} and \ref{thm half integer} imply Theorem \ref{thm Atobe's main thm} of \cite{Ato22}.
\end{remark}

\begin{exmp}\label{exmp phantom}
Consider 
\[ \EE_1= \bordermatrix{
& 0&1&2 \cr
& \ominus & \oplus & \ominus \cr
}_{\rho}\ ,\ \EE_2= \bordermatrix{
&-1& 0&1&2 \cr
& \lhd&\ominus & \oplus & \rhd \cr
&& \ominus &&\cr
}_{\rho}. \]
We have 
\[ dual (\EE_1)= \bordermatrix{
& 0&1&2 \cr
& \ominus & \oplus & \ominus \cr
}_{\rho}\ , \ dual (\EE_2)= \bordermatrix{
& 0&1&2 \cr
& \ominus &  & \cr
&&\oplus& \ominus \cr
}_{\rho},  \]
and hence $ dual \circ ui_{1,2} \circ dual(\EE_2)= \EE_1$. On the other hand, if we consider 
\[ P(\EE_1)= \bordermatrix{
&-1& 0&1&2 \cr
& \lhd&\rhd &  &  \cr
&& \ominus & \oplus &\ominus \cr
}_{\rho}.\]
Then $\EE_2= ui_{1,2} \circ P(\EE_1)$.


\end{exmp}

As a corollary, we give a characterization of absolute maximality, and a way to construct an absolutely maximal element $\EE^{|max|}$ from $\EE$. Then we show that there is a unique absolutely maximal element in the set $\Psi(\EE)$.

\begin{cor}\label{cor max}
Let $\EE \in \Rep^{(P')}$ and denote $\FF= \EE_{\rho}$. 
\begin{enumerate}
    \item [(i)] Write $\FF=\{ ([A_i,B_i]_{\rho}, l_i ,\eta_i) \}_{i \in (I_{\rho} ,>)}$ and $A= \max\{A_i\ | \ i \in I_{\rho}\}$. Then $\FF$ is absolutely maximal if and only if 
    \[\widetilde{\FF}:= sh^{\lceil A\rceil +1 } \left( \sum_{r=0}^{\lceil A\rceil } \{([A-r,-(A-r)-1]_{\rho},\left\lfloor A \right\rfloor-r+1,1)\}   + dual(\FF)  \right)\]
    is minimal.
    \item [(ii)] Let $M$ be a sufficiently large integer (e.g. the dimension of $\psi_{\EE}$). Denote
    \begin{align*}
        \widetilde{\FF}:=& \left( \sum_{r=0}^{\lceil A\rceil } \{([A-r,-(A-r)-1]_{\rho},\left\lfloor A \right\rfloor-r+1,1)^{M}\}+ dual(\FF)  \right)^{min}\\
        =& \{ ([A_i,B_i]_{\rho}, l_i, \eta_i)\}_{i \in (I_{\rho},>)}.
    \end{align*}
      Consider
      \[ I_{\rho,1}= \{i \in I_{\rho} \ | \ A_i +B_i <0 \}, \]
      and decompose $ I_{\rho}= I_{\rho,1} \sqcup I_{\rho,2}$. Define an admissible order $\gg$ on $I_{\rho}$ by the following: 
      \begin{enumerate}
          \item [$\oldbullet$] for $i \in I_{\rho,1}, j \in I_{\rho,2}$, we require $i \ll j,$
          \item [$\oldbullet$] for $i\neq j \in I_{\rho,1}$, we require $i \ll j$ if and only if $i < j$,
           \item [$\oldbullet$] for $i\neq j \in I_{\rho,2}$, we require $i \ll j$ if and only if $i < j$.
      \end{enumerate}
      Then we decompose $\widetilde{\FF}_{\gg}= \widetilde{\FF_1}+\widetilde{\FF_2}$ where $\widetilde{\FF_i}$ corresponds to the index set $I_{\rho,i}$. Finally we define $\FF^{|max|}:= dual(\widetilde{\FF_2})$.
     Then $\FF^{|max|}$ is absolutely maximal, and $\pi(\EE^{\rho} \cup \FF^{|max|})= \pi(\EE)$. Moreover, if $\FF$ is already absolutely maximal, then $\FF^{|max|}= \FF$ up to row exchanges.
     \item [(iii)] Suppose $\pi(\EE^{\rho} \cup \FF) \cong \pi(\EE^{\rho} \cup \FF')$. Then $\FF^{|max|}=(\FF')^{|max|}$. In particular, if $\EE \in \Rep$, there is a unique absolutely maximal element in the set $\Psi(\EE)$, which is equal to $ \EE^{|max|}:= \cup_{\rho} \EE_{\rho}^{|max|}$ up to row exchanges. We call $\EE^{|max|}$ the max form of $\EE$. 
\end{enumerate}
\end{cor}

\begin{proof}
Part (i) follows directly from the previous lemma, Corollary \ref{cor ui inverse} and Remark \ref{rmk partial dual}.

For Part (ii), we first show that $\widetilde{\FF}$ is well-defined. More explicitly, we claim that
\[\FF':=sh^{\lceil A\rceil+1} \left(\sum_{r=0}^{\lceil A\rceil } \{([A-r,-(A-r)-1]_{\rho},\left\lfloor A \right\rfloor-r+1,1)^{M}\}+ dual(\FF) \right)  \]
satisfies the non-vanishing conditions in Theorem \ref{thm non-vanishing}. As a consequence, there exists an $\EE' \in \Rep$ with $\EE'_{\rho}= \FF'$, which implies $(\FF')^{min}$ is well-defined. Then it is immediate from Part (i) that $\FF^{|max|}$ is absolutely maximal, and $\FF^{|max|}= \FF$ up to row exchanges if $\FF$ is already absolutely maximal.

Now we prove the claim. It suffices to show that for any $\FF$ satisfying Theorem \ref{thm non-vanishing} and any admissible order $\gg$ of $\FF_x:=\{([x+1,-x]_{\rho},\lfloor x+1 \rfloor ,1) \} + sh^1(\FF)$ satisfying ($P'$), $(\FF_x)_{\gg}$ also satisfies Theorem \ref{thm non-vanishing}. ($x$ is chosen such that the parity condition is correct.)

It is not hard to see that $(\FF_x)_{\gg}$ satisfies Theorem \ref{thm non-vanishing}(i), so it remains to check Part (ii), which is equivalent to check Part (ii) on $ dual((\FF_x)_{\gg})$ by Theorem \ref{thm Aubert-Zelevinsky dual formula}. Since $dual((\FF_x)_{\gg})$ can be obtained from $add^1(dual(\FF))$ by inserting a row of the form $([x+1,x]_{\rho},1,1)$, one can see that $dual((\FF_x)_{\gg})$ satisfies Theorem \ref{thm non-vanishing}(ii). This completes the proof of the claim and Part (ii).

To show Part (iii), we may assume $\FF'$ is one of $ui_{i,j}(\FF)$, $dual\circ ui_{i,j} \circ dual (\FF)$ or $ dual_k^{+}(\FF)$ by Theorems \ref{thm integer}, \ref{thm half integer}. If $\FF'= ui_{i,j}(\FF)$ is not of type 3', then Corollary \ref{cor ui inverse} shows $dual(\FF) \leq dual(\FF')$. If $\FF'=ui_{i,j}(\FF)$ is of type 3', then the proof of Lemma \ref{lem phantom and ui inverse} shows that 
\[dual(\FF) \leq \{([x,-x-1]_{\rho}, \lfloor x+1\rfloor,1)\} + \FF'\]
for some $x>-1/2$. If $\FF'=dual\circ ui_{i,j} \circ dual (\FF')$, then $dual(\FF')\leq dual(\FF')$. Finally, if $\FF'= dual_k^{+}(\FF)$, then
\[ \FF' \leq \{([-1/2,-1/2]_{\rho},0,1)\}+ \FF. \]
In any case, we have $\FF^{|max|}= (\FF')^{|max|}$ by the uniqueness of minimal element (Corollary \ref{cor minimal}). This completes the proof of the corollary.
\end{proof}

The following is the first main theorem of this section.

\begin{thm}\label{thm can max}
Suppose $\pi$ is an irreducible representation of good parity and of Arthur type, and $\EE$ is an extended multi-segment such that $\pi(\EE)=\pi$. Then $\pi$ is in the local $L$-packet associated with a local Arthur parameter $\psi$ if and only if $\EE^{|max|}$ satisfies (L), and in this case, $\psi_{\EE^{|max|}}=\psi$.
\end{thm}

\begin{proof}
The sufficient direction is done by Theorem \ref{thm in L-packet}. Now we show the necessary direction.

Suppose $\pi$ is in the local $L$-packet associated with $\psi$, then there exists an $\EE'=\cup_{\rho} \EE_{\rho}'$ such that $\supp(\EE')=\supp(\psi)$ and it satisfies (L) by Theorem \ref{thm in L-packet}. By Proposition \ref{prop max elm in L-packet}, $\EE'$ is absolutely maximal. Then $\EE'= \EE^{|max|}$ up to row exchanges by Corollary \ref{cor max}(iii), and hence $\EE^{|max|}$ also satisfies (L). 
This completes the proof of the theorem.
\end{proof}

Now we give the definition for the distinguished member $\psi^{max}(\pi)$ of $\Psi(\pi)$.

\begin{defn}\label{def psi max}
    Let $\pi$ be a representation of Arthur type of $G_n$.  As in Theorem \ref{thm red from nu to gp},  we write $\pi= \tau_{\psi_{nu,>0}} \times \tau_{np}  \rtimes \pi_{gp}$ where $\pi_{gp}$ is of good parity and write $\pi_{gp}= \pi(\EE)$ for some extended multi-segment $\EE$. Then we define
    \[ \psi^{max}(\pi):= \psi_{nu,>0} \oplus \psi_{np} \oplus \psi_{\EE^{|max|}} \oplus \psi_{np}^{\vee} \oplus \psi_{nu,>0}^{\vee}.\]
\end{defn}

In view of Theorem \ref{thm can max}, we may call $\psi^{max}(\pi)$ ``the" local Arthur parameter of $\pi$. Also, $\psi^{max}(\pi)$ gives a partition of representations of Arthur type.

\begin{cor}\label{cor psi max partition}
    Let $\psi \in \Psi^{+}(G_n)$. We define $\Pi_{\psi}^{max}:=\{\pi \in \Pi_{\psi}\ | \ \psi^{max}(\pi)=\psi\}$. Then we have 
    \[ \Pi_{\phi_{\psi}} \subseteq \Pi_{\psi}^{max} \subseteq \Pi_{\psi}.\]
    Moreover, the set of representations of Arthur type are partitioned by the sets $\Pi_{\psi}^{max}.$ That is,
    \[ \bigcup_{\psi \in \Psi^{+}(G_n)} \Pi_{\psi}= \bigsqcup_{\psi \in \Psi^{+}(G_n)} \Pi_{\psi}^{max}.\]
\end{cor}
\begin{proof}
    This follows directly from Theorem \ref{thm can max} and the well-definedness of $\psi^{max}(\pi)$ (see Corollary \ref{cor max}(iii)).
\end{proof}

Our next goal is to give a necessary and sufficient condition for $\pi \in \Pi_{\phi_{\psi^{max}(\pi)}}$ without involving extended multi-segments. We recall a property of $\psi^{max}(\pi)$ explored in a joint work with Zhang (\cite{HLLZ22}).  

Let $\pi$ be an irreducible representation of $G_n$. We write its $L$-data as
\[ \pi= L( \Delta_{\rho_1}[x_1,y_1],\dots, \Delta_{\rho_f}[x_f,y_f]; \pi(\phi,\varepsilon)).\]
Recall that we have already assumed $x_i+y_i$ is non-decreasing negative real numbers when writing down an $L$-data. Now we further assume that if  $i<j$, $\rho_i \cong \rho_j$, and $x_i+y_i=x_j+y_j$, then $x_i\leq x_j$. If $\pi$ is not tempered, i.e. $f\geq 1$, we define $\pi^{-}$ by removing the first Steinberg representation $\Delta_{\rho_1}[x_1,y_1]$ in the $L$-data of $\pi$. That is, we define
\[  \pi^{-}:= L( \Delta_{\rho_2}[x_2,y_2],\dots, \Delta_{\rho_f}[x_f,y_f]; \pi(\phi,\varepsilon)).\]
We may write $\pi^{-}=\pi^{\rho_1,-}$ to keep track of the $\rho_1$ in the Steinberg representation $\Delta_{\rho_1}[x_1,y_1]$ we removed.

\begin{prop}[{\cite[Proposition 5.3]{HLLZ22}}]\label{prop pi minus}
    If $\pi$ is of Arthur type, of good parity and non-tempered, then $\pi^{-}$ is also of Arthur type. 
    
    Moreover, let $\EE^{|max|}$ be the unique absolutely maximal extended multi-segment such that $\pi= \pi(\EE^{|max|})$. Write 
    \[\EE^{|max|}=\cup_{\rho'} \{([A_i,B_i]_{\rho'},l_i,\eta_i)\}_{i \in (I_{\rho'},>)},\]
    where we assume the admissible order $>$ on each $I_{\rho'}$ satisfies (P'), and further that $B_{i_1}=B_{i_2}$ and $A_{i_2}>A_{i_1}$ implies $i_1<i_2$. Suppose $\pi^{-}$ is obtained from $\pi$ by removing $\Delta_{\rho}[x,y]$. Then there exists an index $j \in I_{\rho}$ such that $[A_j,B_j]_{\rho}= [-y,x]_{\rho}$ and $\pi^{-}= \pi( add_{j}^{-1}(\EE^{|max|}))$.  If $A_j-B_j+1=2$, then $add_{j}^{-1}(\EE^{|max|})$ is understood as removing $([A_j,B_j]_{\rho},1,1)$ from $\EE^{|max|}$.  
\end{prop}

We remark that the proof of this proposition is based on the observation that the combinatorial definition of absolute maximality of extended multi-segments (Definition \ref{def abs max}) implies the non-vanishing of $ \pi( add_{j}^{-1}(\EE^{|max|}))$.

Now we state the second main theorem of this section. The theorem gives a way to check when $\pi\in\Pi_{\phi_{\psi^{max}(\pi)}}$ that depends only on the $L$-data of $\pi.$ In contrast, Theorem \ref{thm can max} depends on an extended multi-segment for $\pi.$

\begin{thm}\label{thm pi minus L packet}
    Suppose $\pi$ is a representation of $G_n$ of Arthur type and of good parity. Write its $L$-data as
    \[ \pi= L( \Delta_{\rho_1}[x_1,y_1],\dots, \Delta_{\rho_f}[x_f,y_f]; \pi(\phi,\varepsilon)),\]
where we assume that if  $i<j$, $\rho_i \cong \rho_j$, and $x_i+y_i=x_j+y_j$, then $x_i\leq x_j$. For $0 \leq i \leq f$, define
\[  \pi_i:= L( \Delta_{\rho_{i+1}}[x_{i+1},y_{i+1}],\dots, \Delta_{\rho_f}[x_f,y_f]; \pi(\phi,\varepsilon)).\]
Then the followings are equivalent.
\begin{enumerate}
    \item [(a)] $\pi$ is in the $L$-packet associated with $\psi^{max}(\pi)$.
    \item [(b)] For any $1 \leq  i  \leq f,$ $\psi^{max}(\pi_{i-1})$ contains $\rho_i\otimes S_{x_i-y_i+1}\otimes S_{-x_i-y_i+1}$, and
\begin{align}\label{eq psi minus}
    \psi^{max}(\pi_{i})=\psi^{max}(\pi_{i-1})-  \rho_i\otimes S_{x_i-y_i+1}\otimes S_{-x_i-y_i+1} +\rho_i\otimes S_{x_i-y_i+1}\otimes S_{-x_i-y_i-1}.
\end{align}
\end{enumerate}
\end{thm}

\begin{proof}
For $0 \leq i \leq f$, let $\EE_i$ denote the absolutely maximal extended multi-segment such that $\pi_i= \pi(\EE_i)$. Proposition \ref{prop pi minus} implies that for $1 \leq i \leq f$, we have $\EE_i= (add_{j_{(i-1)}}^{-1}(\EE_{i-1}))^{|max|}$, where $j_{(i-1)}$ is an index of the extended multi-segment $\EE_{i-1}$ such that $\psi_{ add_{j_{(i-1)}}^{-1}(\EE_{i-1})}$ is exactly the right hand side of \eqref{eq psi minus}. Therefore, Condition (b) is equivalent to 
\begin{enumerate}
    \item [(b')] $\EE_i=add_{j_{(i-1)}}^{-1}(\EE_{i-1})$ for any $1 \leq i \leq f$.
\end{enumerate}

Suppose $\pi$ is in the $L$-packet of $\psi^{max}(\pi)$, which implies that $\EE_0$ satisfies (L) by Theorem \ref{thm can max}. It is not hard to see from Definition \ref{def (L)} that if $\EE$ satisfies (L), then so does $add_j^{-1}(\EE)$ for any index $j$ of $\EE$ with $A_j-B_j+1\geq 2$. Therefore, we obtain that for any $1\leq i \leq f $, $add_{j_{(i-1)}}^{-1}(\EE_{i-1})$ satisfies (L), which implies that $add_{j_{(i-1)}}^{-1}(\EE_{i-1})$ is absolutely maximal by Proposition \ref{prop max elm in L-packet} and hence equal to $\EE_i$. This shows that Condition (a) implies Condition (b'), and hence implies Condition (b).

Conversely, suppose Condition (b) holds, and hence Condition (b') holds. We may regard the indices $j_{(i)}$ as indices of $\EE_0$ for $0 \leq i \leq f-1$, and
\[ \left(\sum_{i=0}^{f-1} add_{j_{(i)}}^{-1}\right)(\EE_0)= \EE_{f},\]
where $\psi_{\EE_{f}}= \psi^{max}(\pi_f)$ is tempered. This shows that $\EE_{0}$ satisfies (L) (see Remark \ref{rmk (L)}(2)). We conclude that Condition (b) implies Condition (a), which completes the proof of the theorem.
\end{proof}

Now we give examples for computing $\EE^{|max|}$ and $\psi^{max}(\pi)$.

\begin{exmp}\label{exam 11.8}
Consider 
\[ \EE= \bordermatrix{
&\half{-3} &\half{-1}&\half{1}&\half{3} \cr
& \lhd &\lhd&\rhd&\rhd \cr
&& \ominus & \oplus &\cr
&&&\oplus & \ominus \cr
}_{\rho},\]
where $\rho$ is symplectic of dimension $d$, and
\[ \pi= \pi(\EE)= L(\Delta_{\rho}[1/2,-3/2]; \pi( (1/2)^{-},(3/2)^{-}))\]
is an irreducible representation for $\SO_{12d+1}(F)$. We first compute $\EE^{|max|}$, and then give $\psi^{max}(\pi)$.

We consider
\[ dual(\EE)=  \bordermatrix{
 &\half{-1}&\half{1}&\half{3} \cr
&  \lhd&\ominus&\rhd \cr
& &\ominus  & \cr
&& & \oplus \cr
}_{\rho}\ ,\ \widetilde{\EE}:=  \bordermatrix{
&\half{-3} &\half{-1}&\half{1}&\half{3} \cr
& \lhd &\oplus & \rhd&\cr
&  &\oplus & &\cr
&&  \lhd&\ominus&\rhd \cr
&& &\ominus  & \cr
&&& & \oplus \cr
}_{\rho}. \]
Then we have 
\[(\widetilde{\EE})^{min}= \bordermatrix{
 &\half{-3}&\half{-1}&\half{1}&\half{3} \cr
& \lhd& \lhd& \rhd& \rhd \cr
&&  \ominus&\oplus& \cr
&& \oplus&   \ominus& \oplus \cr
}_{\rho}, \ dual((\widetilde{\EE})^{min})= \bordermatrix{
 &\half{1}&\half{3} \cr
 &\ominus& \oplus \cr
 &\oplus  & \cr
& & \ominus \cr
}_{\rho}. \]
One can check that $dual((\widetilde{\EE})^{min})$ satisfies Corollary \ref{cor max}(i), and hence we have $\EE^{|max|}=dual((\widetilde{\EE})^{min})$. Thus
\[\psi^{max}(\pi)= \rho \otimes S_{2} \otimes S_{1} + \rho \otimes S_{3} \otimes S_{2} + \rho \otimes S_{4} \otimes S_1. \]
After a row exchange,
\[ R_1(\EE^{|max|})=  \bordermatrix{
 &\half{1}&\half{3} \cr
  &\ominus& \cr
 &\lhd  & \rhd\cr
& & \ominus \cr
}_{\rho},  \]
so $\EE^{|max|}$ satisfies (L), and $\pi$ is in the local $L$-packet associated with $\psi^{max}(\pi)$.

In this example, one can check that 
\[ui_{2,3}(\EE^{|max|})= \EE_{can}= \bordermatrix{
&\half{1} & \half{3} \cr
& \ominus & \oplus \cr 
& \oplus & \ominus \cr
}_{\rho},\]
and hence $\EE_{can} \neq \EE^{|max|}$.

\end{exmp}

\begin{exmp}\label{exmp sp10 three packets}
 Let $\rho$ be the trivial representation. Consider the following three local Arthur parameters of $\Sp_{10}(F)$,
\begin{align*}
    \psi_1&= \rho \otimes S_1 \otimes S_7 + \rho \otimes S_2 \otimes S_2,\\
    \psi_2&= \rho \otimes S_1 \otimes S_7 + \rho \otimes S_1 \otimes S_1+ \rho \otimes S_3\otimes S_1,\\
    \psi_3&= \rho \otimes S_1 \otimes S_7 + \rho \otimes S_1 \otimes S_3 +\rho\otimes S_1\otimes S_1.
\end{align*} 
There are seven representations in $\Pi_{\psi_1} \cup \Pi_{\psi_2} \cup \Pi_{\psi_3}$. We denote them by $\pi_1,\dots ,\pi_7$. In this example, we list all extended multi-segments for each $\pi_i$, and specify $\psi^{max}(\pi_i)$.
 \begin{center}
\begin{tabular}{|c|c|c|}
\hline
     & $\Pi_{\psi}$ & $\Pi_{\phi_{\psi}}$  \\
     \hline
$\psi_1$  & $\{\pi_1,\pi_2,\pi_3\}$ & $\{\pi_3\}$\\
\hline 
$\psi_2$ & $\{ \pi_1, \pi_2, \pi_4,\pi_5\}$ & $\{ \pi_2, \pi_5 \}$\\
\hline 
$\psi_3$& $\{ \pi_1,\pi_2, \pi_6,\pi_7\}$ & $ \{\pi_7\}$\\
\hline
\end{tabular}     
 \end{center}
 The following computation also gives the set $ \{\psi \ | \ \pi_i \in \Pi_{\psi}\}$ for $i=1,\dots ,7$, and one can see that $ \Pi_{\psi_{1}},\Pi_{\psi_{2}},\Pi_{\psi_{3}}$ do not intersect with other local Arthur packets.
\begin{enumerate}
    \item [1.] Let $\pi_1= L( \Delta_{\rho}[-3,-3]; \pi( 0^{-},1^{+}, 2^{-}))$. We have $\pi_1= \pi(\EE_{1,i})$ for $i =1,2,3$, where
\begin{align*}
    \EE_{1,1}&= \bordermatrix{
    &-3&-2&-1&0&1&2&3 \cr
    &\lhd &\lhd & \lhd & \ominus& \rhd & \rhd &\rhd \cr
    & &&&\oplus&\ominus &&\cr
    }_{\rho},\\
    \EE_{1,2}&= \bordermatrix{
    &-3&-2&-1&0&1&2&3 \cr
    &\lhd &\lhd & \lhd & \ominus& \rhd & \rhd &\rhd \cr
    & &&&\oplus& &&\cr
    & &&&& \ominus&&\cr
    }_{\rho},\\
    \EE_{1,3}&= \bordermatrix{
    &-3&-2&-1&0&1&2&3 \cr
    &\lhd &\lhd & \lhd & \ominus& \rhd & \rhd &\rhd \cr
    & &&\lhd&\oplus&\rhd &&\cr
      & &&&\ominus& &&\cr
    }_{\rho}.
\end{align*}
Since $\EE_{1,2}$ is absolutely maximal, $\psi^{max}(\pi_1)= \psi_2$. However, $\EE_{1,2}$ does not satisfy (L), so $\pi_1$ is not in any local $L$-packet associated with a local Arthur parameter.
\item [2.] Let $\pi_2=L(\Delta_{\rho}[-3,-3],\Delta_{\rho}[-2,-2],\Delta_{\rho}[-1,-1]; \pi(0^{-},0^{-},1^{+}))$. We have $\pi_2=\pi(\EE_{2,i})$ for $i=1,2,3$, where
\begin{align*}
    \EE_{2,1}&= \bordermatrix{
    &-3&-2&-1&0&1&2&3 \cr
    &\lhd &\lhd & \lhd & \ominus& \rhd & \rhd &\rhd \cr
    & &&&\ominus&\oplus &&\cr
    }_{\rho},\\
    \EE_{2,2}&= \bordermatrix{
    &-3&-2&-1&0&1&2&3 \cr
    &\lhd &\lhd & \lhd & \ominus& \rhd & \rhd &\rhd \cr
    & &&&\ominus& &&\cr
    & &&&& \oplus&&\cr
    }_{\rho},\\
    \EE_{2,3}&= \bordermatrix{
    &-3&-2&-1&0&1&2&3 \cr
    &\lhd &\lhd & \lhd & \ominus& \rhd & \rhd &\rhd \cr
    & &&\lhd&\ominus&\rhd &&\cr
      & &&&\oplus& &&\cr
    }_{\rho}.
\end{align*}
Since $\EE_{2,2}$ is absolutely maximal, $ \psi^{max}(\pi_2)= \psi_2$. Also, $\pi_2\in \Pi_{\phi_{\psi_2}}$.
\item [3.] Let $\pi_3=L(\Delta_{\rho}[-3,-3],\Delta_{\rho}[-2,-2],\Delta_{\rho}[-1,-1], \Delta_{\rho}[0,-1]; \pi(0^{+}))$. We have $\pi_3=\pi(\EE_3)$, where
\[ \EE_{3}= \bordermatrix{
    &-3&-2&-1&0&1&2&3 \cr
    &\lhd &\lhd & \lhd & \oplus& \rhd & \rhd &\rhd \cr
    & &&&\lhd & \rhd && \cr }_{\rho}=\EE_3^{|max|}.\]
    Thus $\psi^{max}(\pi_3)=\psi_1$ and $\pi_3\in \Pi_{\phi_{\psi_1}}$.
\item [4.] Let $\pi_4=L(\Delta_{\rho}[-3,-3],\Delta_{\rho}[-2,-2]; \pi(0^{+},1^-,1^-))$. We have $\pi_4=\pi(\EE_4)$, where
\[ \EE_{4}= \bordermatrix{
    &-3&-2&-1&0&1&2&3 \cr
    &\lhd &\lhd & \lhd & \oplus& \rhd & \rhd &\rhd \cr
    & &&&\ominus & && \cr
        & &&& &\ominus && \cr}_{\rho}=\EE_4^{|max|}.\]
    Thus $\psi^{max}(\pi_4)=\psi_2$, but $\pi_4 \not\in \Pi_{\phi_{\psi_2}}$.
\item [5.] Let $\pi_5=L(\Delta_{\rho}[-3,-3],\Delta_{\rho}[-2,-2], \Delta_{\rho}[-1,-1]; \pi(0^{+},0^+,1^+))$. We have $\pi_5=\pi(\EE_5)$, where
\[ \EE_{5}= \bordermatrix{
    &-3&-2&-1&0&1&2&3 \cr
    &\lhd &\lhd & \lhd & \oplus& \rhd & \rhd &\rhd \cr
    & &&&\oplus & && \cr
        & &&& &\oplus && \cr}_{\rho}=\EE_5^{|max|}.\]
    Thus $\psi^{max}(\pi_5)=\psi_2$, and $\pi_5 \in \Pi_{\phi_{\psi_2}}$.
\item [6.] Let $\pi_6=L(\Delta_{\rho}[-3,-3],\Delta_{\rho}[-1,-2], \Delta_{\rho}[0,-1]; \pi(0^{+}))$. We have $\pi_6=\pi(\EE_6)$, where
\[ \EE_{6}= \bordermatrix{
    &-3&-2&-1&0&1&2&3 \cr
    &\lhd &\lhd & \lhd & \oplus& \rhd & \rhd &\rhd \cr
    & &&\lhd&\ominus &\rhd && \cr
        & &&&\ominus &&& \cr}_{\rho}=\EE_6^{|max|}.\]
    Thus $\psi^{max}(\pi_6)=\psi_3$, but $\pi_6 \not\in \Pi_{\phi_{\psi_3}}$.
\item [7.] Let $\pi_7=L(\Delta_{\rho}[-3,-3],\Delta_{\rho}[-2,-2], \Delta_{\rho}[-1,-1],\Delta_{\rho}[-1,-1]; \pi(0^{+},0^+,0^+))$. We have $\pi_7=\pi(\EE_7)$, where
\[ \EE_{7}= \bordermatrix{
    &-3&-2&-1&0&1&2&3 \cr
    &\lhd &\lhd & \lhd & \oplus& \rhd & \rhd &\rhd \cr
    & &&\lhd&\oplus &\rhd && \cr
        & &&&\oplus &&& \cr}_{\rho}=\EE_7^{|max|}.\]
    Thus $\psi^{max}(\pi_7)=\psi_3$, and $\pi_7 \in \Pi_{\phi_{\psi_3}}$.
\end{enumerate}
Finally, let $\phi_1, \phi_4,\phi_6$ be the $L$-parameters of $\pi_1,\pi_4,\pi_6$, which are not of Arthur type. The associated local $L$-packets are
\begin{align*}
    \Pi_{\phi_1}= \{& \pi_1, \pi_8=L(\Delta_\rho[-3,-3];\pi(0^{+},1^-,2^-)),\pi_9=L(\Delta_\rho[-3,-3];\pi(0^-,1^-,2^+)),\\ &\ \pi_{10}=L(\Delta_\rho[-3,-3];\pi(0^+,1^+,2^+))\},\\
    \Pi_{\phi_4}= \{& \pi_4, \pi_{11}=L(\Delta_\rho[-3,-3],\Delta_{\rho}[-2,-2];\pi(0^{+},1^+,1^+))\},\\
    \Pi_{\phi_6}= \{& \pi_6\}.
\end{align*}
One can check by Algorithm \ref{alg Arthur type} that $\pi_8,\pi_{9},\pi_{10}, \pi_{11}$ are not of Arthur type. 

We visualize the representations, local Arthur packets and local $L$-packets in this example in the following picture.
\begin{center}
\begin{tikzpicture}
\draw[rotate=330][red][thick] (0,48 pt) ellipse (40 pt and 92pt);
\draw[rotate=90][blue][thick] (0,48 pt) ellipse (40 pt and 92pt);
\draw[rotate=210][violet][thick] (0,48 pt) ellipse (40 pt and 92pt);
\draw[thick] (5 pt,-15pt) rectangle ( 135 pt, 15 pt);
\draw[blue][thick] (-5 pt,-15pt) rectangle ( -65 pt, 15 pt);
\draw[thick] (-110 pt,15 pt) rectangle ( -80 pt, -65 pt);
\draw[red][thick] (20 pt, 40 pt) rectangle ( 50 pt, 70 pt);
\draw [thick] (20 pt, -40 pt) rectangle ( 50 pt, -70 pt);
\draw[violet][thick] (40 pt, -80 pt) rectangle ( 70 pt, -110 pt);
\draw[blue] node at (20pt ,0) {$\pi_1$};
\draw node at (60pt ,0) {$\pi_8$};
\draw node at (90pt ,0) {$\pi_9$};
\draw node at (120pt ,0) {$\pi_{10}$};
\draw[blue] node at (-20pt ,0) {$\pi_2$};
\draw[blue] node at (-50pt ,0) {$\pi_5$};
\draw[blue] node at (-95pt ,0) {$\pi_4$};
\draw node at (-95pt ,-45pt) {$\pi_{11}$};
\draw[red] node at (35 pt ,55pt) {$\pi_{3}$};
\draw[violet] node at (35 pt ,-55pt) {$\pi_{6}$};
\draw[violet] node at (55 pt ,-95pt) {$\pi_{7}$};
\draw[red] node at (100 pt ,90pt) {$\Pi_{\psi_1}$};
\draw[blue] node at (-100 pt ,50pt) {$\Pi_{\psi_2}$};
\draw[violet] node at (100 pt ,-90pt) {$\Pi_{\psi_3}$};
\end{tikzpicture}    
\end{center}
 Each rectangle represents an local $L$-packet and each ellipse represents a local Arthur packet. For each local Arthur parameter $\psi$, we draw $\Pi_{\psi}$, $\Pi_{\phi_{\psi}}$ and $\pi_i$ such that $\psi^{max}(\pi_i)=\psi$ in the same color (not black), and draw those representations and local $L$-packets not of Arthur type in black.
\end{exmp}

We end this section by introducing an opposite idea, the absolute minimality.

\begin{defn}\label{def abs min}
For an extended multi-segment $\EE=\cup_{\rho} \EE_{\rho}$, we say $\EE_{\rho}$ is absolutely minimal if the following holds.
\begin{enumerate}
    \item [$\oldbullet$]$\EE_{\rho}$ is minimal, i.e., no $ui$ is applicable on $\EE_{\rho}$.
    \item [$\oldbullet$] $dual(\EE_{\rho})$ is maximal, i.e., no $ ui_{i,j}^{-1}$ is applicable on $dual(\EE_{\rho})$.
    \item [$\oldbullet$] For any $k \in I_{\rho}$ and admissible order $\gg$ of $I_{\rho}$ satisfying (P'), $dual_{k}^{+}$ is not applicable on $\EE_{\rho, \gg}$.
\end{enumerate}

We say $\EE$ is absolutely minimal if $\EE_{\rho}$ is absolutely minimal for any $\rho$.
\end{defn}

Absolute maximality and minimality are dual to each other in the following sense.

\begin{lemma}
Suppose $\EE\in \Rep^{(P')}$. Then $\EE$ is absolutely minimal if and only if $dual(\EE)$ is absolutely maximal.
\end{lemma}
\begin{proof}
It follows from the definition and Theorem \ref{thm applicability of dual_k}.
\end{proof}

In the following corollary, we show the existence and uniqueness of absolutely minimal member $\EE^{|min|}$ in $\Psi(\EE)$, and give a case that $\EE^{|min|}= \EE_{can}$.

\begin{cor}\label{cor abs min}
Suppose $\EE \in \Rep^{(P')}$. 
\begin{enumerate}
    \item[(i)] There exists a unique absolutely minimal element in $ \Psi(\EE)$, which is equal to $dual((dual(\EE))^{|max|})$ up to row exchanges. We denote this absolutely minimal element by $\EE^{|min|}$.
    \item[(ii)] Suppose $sh^{-1}(\EE)$ is non-negative. Then $\EE_{can}$ is the unique minimal element in $\Psi(\EE)$. In particular, we have $\EE^{|min|}= \EE_{can}$.
\end{enumerate}
\end{cor}
\begin{proof}
Part (i) follows directly from the previous lemma and Corollary \ref{cor max}(iii). 

For Part (ii), we first observe that the assumption implies $\EE_{\rho}= (\EE_{\rho})_{>1/2}$, and hence $\EE^{min}= \EE_{can}$ by Definition \ref{def can form}. Then it suffices to check that $ (\EE')^{min}= \EE^{min}$ for any $\EE' \in \Psi(\EE)$. We may assume both $\EE$ and $\EE'$ are already minimal.

By Corollary \ref{cor minimal}(i), $ (\EE_{\rho}')_{>1/2}= (\EE_{\rho})_{>1/2}= \EE_{\rho}$ up to row exchanges for any $\rho$. By comparing the dimension of the associated local Arthur parameters, we see that $\EE_{\rho}'= (\EE_{\rho}')_{>1/2} $. This completes the proof of Part (ii).
\end{proof}
We remark that $\EE^{|min|} \neq \EE_{can}$ in general. Indeed, for $\pi_1$ in Example \ref{exmp sp10 three packets}, we have $\EE_{1,1}=\EE_{can}$, $\EE_{1,2}=\EE^{|max|}$ and $\EE_{1,3}=\EE^{|min|}$. 

Analogously to absolutely maximal case, we have the following definition and corollary.

\begin{defn}\label{def psi min}
    Let $\pi$ be a representation of Arthur type of $G_n$.  As in Theorem \ref{thm red from nu to gp},  we write $\pi= \tau_{\psi_{nu,>0}} \times \tau_{\psi_{np}} \rtimes \pi_{gp}$ where $\pi_{gp}$ is of good parity and write $\pi_{gp}= \pi(\EE)$ for some extended multi-segment $\EE$. Then we define
    \[ \psi^{min}(\pi):= \psi_{nu,>0} \oplus \psi_{np} \oplus \psi_{\EE^{|min|}} \oplus \psi_{np}^{\vee} \oplus \psi_{nu,>0}^{\vee}.\]
\end{defn}

\begin{cor}
    Let $\psi \in \Psi^{+}(G_n)$. We define $\Pi_{\psi}^{min}:=\{\pi \in \Pi_{\psi}\ | \ \psi^{min}(\pi)=\psi\}$. Then the set of representations of Arthur type are partitioned by the sets $\Pi_{\psi}^{min}.$ That is,
    \[ \bigcup_{\psi \in \Psi^{+}(G_n)} \Pi_{\psi}= \bigsqcup_{\psi \in \Psi^{+}(G_n)} \Pi_{\psi}^{min}.\]
\end{cor}

\section{Characterization of \texorpdfstring{$\psi^{max}(\pi)$}{psi{max}(pi)} and \texorpdfstring{$\psi^{min}(\pi)$}{psi{min}(pi)}}\label{characterizations}

In this section, for each representation $\pi$ of Arthur type, we give three orderings $\geq_{O}$, $\geq_{A}$ and $\geq_{N}$ on the set \[\Psi(\pi):= \{ \psi \in \Psi^+(G_n) \ |\ \pi \in \Pi_{\psi}  \},\] 
and show that $\psi^{max}(\pi)$ and $\psi^{min}(\pi)$ are exactly the unique maximal and minimal elements under these orderings. The orderings $\geq_{A}, \geq_{N}$ give the representation theoretic characterizations of  $\psi^{max}(\pi)$ and $\psi^{min}(\pi)$ (see Theorems \ref{thm Jiang's partition}, \ref{thm Moeglin} below). Two more orderings and characterizations of $\psi^{max}(\pi)$ and $\psi^{min}(\pi)$ are introduced in \cite{HLLZ22}.

First, actions of the operators on extended multi-segments in previous sections naturally induce actions on local Arthur parameters as follows. 

\begin{defn}\label{def operators on parameters}
Suppose $\psi$ is a local Arthur parameter of $G_n$. Decompose $\psi= \psi_{nu,>0} \oplus \psi_{np} \oplus \psi_{gp} \oplus \psi_{np}^{\vee} \oplus \psi_{nu,>0}^{\vee}$ as in Theorem \ref{thm red from nu to gp} and write
\[ \psi_{gp}=\bigoplus_{\rho} \bigoplus_{i \in I_{\rho}} \rho \otimes S_{a_i} \otimes S_{b_i}.  \]
Then for $i,j,k \in I_{\rho}$, we define the actions of the operators $dual$, $ui_{i,j}$ and $dual_k^{-}$ on $\psi$ as follows. 
\begin{enumerate}
    \item $dual(\psi):=\widehat{\psi}$. We identify the index set $I_{\rho}(\psi_{gp})$ with $I_{\rho}(\widehat{\psi}_{gp})$ in the obvious way.
    \item For $r \in I_{\rho}$, let $A_r= \half{a_r+b_r}-1$ and $B_r= \half{a_r-b_r}$. Then we may rewrite the decomposition of $\psi_{gp}$ as 
    \[ \psi_{gp}= \bigoplus_{\rho} \bigoplus_{i \in I_\rho} \rho \otimes S_{A_i+B_i+1} \otimes S_{A_i-B_i+1}.\]
    The operator $ui_{i,j}$ is applicable on $\psi$ if the following conditions hold.
    \begin{enumerate}
    \item [$\oldbullet$] $A_j \geq A_i+1 \geq B_j >B_i.$
        \item  [$\oldbullet$] For any $r \in I_{\rho}$, if $B_i<B_r<B_j$, then $A_r \leq A_i $ or $A_r \geq A_j$.
    \end{enumerate}
    In this case, we define $ui_{i,j}(\psi_{gp})$ by replacing the summands 
    \[ \rho\otimes S_{A_i+B_i+1}\otimes S_{A_i-B_i+1} +\rho\otimes S_{A_j+B_j+1}\otimes S_{A_j-B_j+1} \]
of $\psi_{gp}$ with
    \[\rho\otimes S_{A_j+B_i+1}\otimes S_{A_j-B_i+1} +\rho\otimes S_{A_i+B_j+1}\otimes S_{A_i-B_j+1}. \]
    If $A_i+1-B_j=0$, then we omit the last summand, and say this $ui_{i,j}$ is of type 3'. Finally, we define $ui_{i,j}(\psi):= \psi_{nu,>0} \oplus \psi_{np} \oplus  ui_{i,j}(\psi_{gp}) \oplus \psi_{np}^{\vee} \oplus \psi_{nu,>0}^{\vee}$.
    \item The operator $dual_k^{-}$ is applicable on $\psi$ if $b_k=a_k+1$. In this case, we define $dual_k^{-}(\psi_{gp})$ by replacing the summand
    \[ \rho \otimes S_{a_k}\otimes S_{a_{k}+1}\]
    of $\psi_{gp}$ with 
    \[ \rho \otimes S_{a_k+1}\otimes S_{a_{k}},\]
    and we define $dual_k^{-}(\psi)= \psi_{nu,>0} \oplus \psi_{np} \oplus  dual_k^{-}(\psi_{gp}) \oplus \psi_{np}^{\vee} \oplus \psi_{nu,>0}^{\vee}$.
    \item The operator $dual_k^{+}$ is applicable on $\psi$ if $a_k=b_k+1$. In this case, we define $dual_k^{+}= dual \circ dual_{k}^{-} \circ dual(\psi)$.
    \item 
    Let $T$ be any of the operators above or their inverses. If $T$ is not applicable on $\psi$, we define $T(\psi)=\psi$.
\end{enumerate}
\end{defn}

We observe that among all the operators (including the inverses), the operators $ ui_{i,j}^{-1}$, $dual \circ ui_{j,i} \circ dual $ and $dual_k^{-}$ raise the ``temperedness" of local Arthur parameters under a certain measurement of temperedness (see Theorem \ref{thm Jiang's partition}(1) below). This idea leads us to the following definition. 

\begin{defn}\label{def raising operator}
We say that $T$ is a \emph{raising} operator if it is of the form  
$ ui_{i,j}^{-1}$, $dual \circ ui_{j,i} \circ dual,$ or $dual_k^{-}$.
\end{defn}

For a representation $\pi$ of Arthur type, the raising operators induce a partial order on $\Psi(\pi).$ 

\begin{defn}\label{def operator ordering}
We define a partial order $\geq_{O}$ on $\Psi^+(G_n)$ by $\psi_1 \geq_{O} \psi_2$ if $\psi_1=\psi_2$ or there exists a sequence of raising operators $\{T_l\}_{l=1}^m$ such that
\[ \psi_1= T_1 \circ \cdots \circ T_m(\psi_2).\]
\end{defn}
See Remark \ref{rmk O ordering partial order} below that $\geq_{O}$ is indeed a partial order. Then, we may rephrase the results in previous sections as follows.

\begin{thm}\label{thm structure of Psi(pi) intro}
Let $\mathrm{G}_n$ be $\Sp_{2n}$ or split $\SO_{2n+1}$ and $\pi$ is a representation of $G_n$ of Arthur type.
\begin{enumerate}
    \item If $\psi_1, \psi_2\in \Psi(\pi),$ then there exists a sequence of operators $\{T_l\}_{l=1}^m$ such that
\[ \psi_1= T_1 \circ \cdots \circ T_m (\psi_2),\]
where each $T_l$ is one of the operators $ui_{i,j}$,  $ dual \circ ui_{j,i}\circ dual$, $dual_k^{-}$, or their inverses.
\item  The distinguished members $\psi^{max}(\pi)$ and $ \psi^{min}(\pi)$ are the unique elements in $\Psi(\pi)$ satisfying the following inequality
    \[ \psi^{max}(\pi) \geq_O \psi \geq_O \psi^{min}(\pi),\]
    for any $\psi \in \Psi(\pi).$  
\end{enumerate}
\end{thm}

\begin{proof}
    Part (i) follows from Theorem \ref{main thm intro}. Part (ii) follows from Corollaries \ref{cor max}(iii), \ref{cor abs min}(i).
\end{proof}


\subsection{Partitions associated with local Arthur parameters}\label{sec: A ordering}
 In this subsection, we define an ordering $\geq_{A}$ on the set $\Psi(\pi)$ and show that $\psi^{max}(\pi)$ and $\psi^{min}(\pi)$ are exactly the unique maximal and minimal elements under this ordering.

Let $\widehat{G}_n(\BC) \hookrightarrow \GL_N(\BC)$ be the standard embedding. Recall that there is a one-to-one correspondence between the nilpotent orbits and a subset of partitions of $N$ (see \cite[\S 5.1]{CM93}). For each local Arthur parameter 
$$\psi:W_F \times \SL_2^{D}(\mathbb{C}) \times \SL_2^A(\mathbb{C}) \to \widehat{G}_n(\BC),$$
we consider the nilpotent orbit of $\widehat{G}_n(\BC)$ containing the following element 
\[d(\psi|_{\SL_2^{A}(\mathbb{C})}) \left( \begin{pmatrix}
0 &1 \\0 &0
\end{pmatrix} \right),\]
and denote $\underline{p}^A(\psi)$ the corresponding partition. To be more explicit, write 
\[ \psi= \bigoplus_{i=1}^l \rho_i|\cdot|^{x_i} \otimes S_{a_i} \otimes S_{b_i}, \]
and let $d_i= \dim (\rho_i \otimes S_{a_i})$. Then $\underline{p}^A(\psi)$ is given by $[ b_1^{d_1},\dots , b_n^{d_n} ]$. We remark that $\underline{p}^A(\psi)$ is a key ingredient considered in Jiang's Conjecture (\cite[Conjecture 4.2]{Jia14} and \cite[Conjecture 1.6]{LS22}).

If $\psi$ is tempered, then $\underline{p}^A(\psi)=[1^{N}]$, which gives the minimal partition of $N$ under the dominance order. Therefore, we consider the following definition.

\begin{defn}\label{A ordering}
We define a preorder $\geq_{A}$ on $\Psi(G_n)$ by $\psi_1 \geq_A \psi_2$ if $ \underline{p}^A(\psi_1) \leq \underline{p}^A(\psi_2)$ under the dominance order.
\end{defn}

We remark that $\geq_{A}$ is only a preorder but not a partial order since it is possible that $\underline{p}^A(\psi_1)=\underline{p}^A(\psi_2)$ but $\psi_1 \neq \psi_2$.
The following is the main theorem of this subsection.

\begin{thm}\label{thm Jiang's partition}
Let $\mathrm{G}_n$ be $\Sp_{2n}$ or split $\SO_{2n+1}$.
\begin{enumerate}
    \item If $T$ is a raising operator applicable on $\psi \in \Psi^{+}(G_n)$, then 
    \[T(\psi) \gneq_A \psi.\]
    In particular, if $\psi \geq_{O} \psi'$, then $\psi \geq_A \psi'$.
    \item Let $\pi$ be a representation of $G_n$ of Arthur type. The distinguished members $\psi^{max}(\pi)$ and $ \psi^{min}(\pi)$ are the unique elements in $\Psi(\pi)$ satisfying the following inequality
    \[ \psi^{max}(\pi) \geq_A \psi \geq_A \psi^{min}(\pi),\]
    for any $\psi \in \Psi(\pi).$
\end{enumerate}
\end{thm}
\begin{proof}
Part (2) follows immediately from Part (1) and Theorem \ref{thm structure of Psi(pi) intro}. Now we show Part (1) case by case. Note that given partitions $\underline{p}_1, \underline{p}_2$ and $\underline{p}_3$, if $\underline{p}_1 \geq \underline{p}_2$, then $\underline{p}_1 \sqcup \underline{p}_3 \geq \underline{p}_{2} \sqcup \underline{p}_3$. Therefore, we assume $\psi$ is of good parity.

\textbf{Case (i):} Assume that $T= ui_{i,j}^{-1}$. 

We first deal with the situation that $ui_{i,j}^{-1}$ is not of type 3'. Write 
\begin{align*}
     \psi&= \psi' +\rho \otimes S_{A_j+B_i+1} \otimes S_{A_j-B_i+1}+ \rho \otimes S_{A_i+B_j+1} \otimes S_{A_i-B_j+1}\\
     T(\psi)&= \psi'+ \rho \otimes S_{A_i+B_i+1} \otimes S_{A_i-B_i+1} +\rho \otimes S_{A_j+B_j+1} \otimes S_{A_j-B_j+1}.
\end{align*}  
where $ A_j > A_i \geq B_j>B_i$. Let $d= \dim(\rho)$. Then $T(\psi) \geq_A \psi$ if and only if 
\begin{align*}
     &[  (A_j-B_i+1)^{d (A_j+B_i+1)},(A_i-B_j+1)^{d (A_i+B_j+1)}]\\
     \geq & [ (A_i-B_i+1)^{d (A_i+B_i+1)}, (A_j-B_j+1)^{d (A_j+B_j+1)}],
\end{align*} 
which follows from $ A_j-B_i+1 > A_i-B_i+1$ and $ A_j-B_i+1 > A_j-B_j+1$.

Indeed, the same argument works for $ui_{i,j}^{-1}$ of type 3', except $A_i-B_j+1=0$, and we can ignore $(A_i-B_j+1)^{d (A_i+B_j+1)}$ in the partition.

\textbf{Case (ii).} Suppose that $T= dual \circ ui_{i,j} \circ dual$.

From Corollary \ref{cor ui inverse}, we know $dual\circ ui_{i,j} \circ dual=ui_{j,i}^{-1}$ if the $ui_{i,j}$ is not of type 3', and hence it is done in Case (i). Therefore, we may assume the $ui_{i,j}$ is of type 3'. Write
\begin{align*}
     \psi&= \psi'+ \rho \otimes S_{A_i+B_i+1} \otimes S_{A_i-B_i+1} +\rho \otimes S_{A_j+B_j+1} \otimes S_{A_j-B_j+1},\\
     T(\psi)&= \psi'+ \rho \otimes S_{A_j+B_i+1} \otimes S_{A_j-B_i+1},
\end{align*} 
where $ A_j> -B_j= A_i+1>A_i >-B_i$.
Let $d= \dim(\rho)$. Then $T(\psi) \geq_A \psi$ if and only if
\begin{align*}
     [ (A_i-B_i+1)^{d (A_i+B_i+1)}, (A_j-B_j+1)^{d (A_j+B_j+1)}]
     \geq [ (A_j-B_i+1)^{d(A_j+B_i+1)}] ,
\end{align*} 
which follows from $A_j-B_j+1> A_j-B_i+1$.

\textbf{Case (iii).} Assume that $T=dual_k^{-}$.

In this case, we write
\begin{align*}
     \psi&= \psi' +\rho \otimes S_{a} \otimes S_{a+1}\\
     T(\psi)&= \psi'+ \rho \otimes S_{a+1} \otimes S_{a}.
\end{align*}  
Let $d= \dim(\rho)$. Then $T(\psi) \geq_A \psi$ if and only if $[(a+1)^{a}]\geq [a^{a+1}]$, which is clear.

The proof of the theorem is now complete.
\end{proof}
\begin{remark}\label{rmk O ordering partial order}
Part (1) of above theorem shows that $\geq_{O}$ is indeed a partial order, i.e., $\psi_1 \geq_{O} \psi_2$ and $\psi_2 \geq_{O} \psi_1$ implies that $\psi_1=\psi_2$.
\end{remark}

We explain the theorem on Example \ref{ex Atobe}.
\begin{exmp}
 We redraw the relations among $\EE_1,\dots,\EE_9$ as follows.

$$
\begin{tikzcd}
 & & \EE_4& &\\
 &\EE_2\ar[ru, "ui^{-1}"] & &\EE_3 \ar[lu,"ui^{-1}"'] &\\
\EE_9\ar[ru,"dual \circ ui \circ dual"] & &\EE_1 \ar[lu,"ui^{-1}"'] \ar[ru,"ui^{-1}"] & &\EE_8\ar[lu,"dual \circ ui \circ dual"']  \\
 &\EE_6 \ar[ru,"dual \circ ui \circ dual"] \ar[lu,"ui^{-1}"] & &\EE_5 \ar[ru,"ui^{-1}"'] \ar[lu,"dual \circ ui \circ dual"'] & \\
 & & \EE_7\ar[ru,"dual \circ ui \circ dual"']\ar[lu,"dual \circ ui \circ dual"] & &
\end{tikzcd}
$$
The partitions of the associated local Arthur parameters are given by
$$
\begin{tikzcd}
 & & \left[1^9\right]& &\\
 &\left[2^4,1^1\right]\ar[ru, "ui^{-1}"] & &\left[2^2, 1^5\right] \ar[lu,"ui^{-1}"'] &\\
\left[5^1,1^4\right]\ar[ru,"dual \circ ui \circ dual"] & &\left[3^3\right] \ar[lu,"ui^{-1}"'] \ar[ru,"ui^{-1}"] & &\left[3^1,1^6\right]\ar[lu,"dual \circ ui \circ dual"']  \\
 &\left[5^1,2^2\right] \ar[ru,"dual \circ ui \circ dual"] \ar[lu,"ui^{-1}"] & &\left[4^2,1^1\right] \ar[ru,"ui^{-1}"'] \ar[lu,"dual \circ ui \circ dual"'] & \\
 & & \left[5^1,3^1,1^1\right]\ar[ru,"dual \circ ui \circ dual"']\ar[lu,"dual \circ ui \circ dual"] & &
\end{tikzcd}
$$
Note that $\EE_4=\EE^{|max|}$ and $\EE_7= \EE^{|min|}$. Also, $\underline{p}^{A}(\psi_{\EE_1}) \geq \underline{p}^A(\psi_{\EE_{8}})$, but $\psi_{\EE_8}$ and $\psi_{\EE_1}$ are not comparable under $\geq_{O}$.
\end{exmp}

\subsection{Order of zeros of Arthur normalized intertwining operators}
\label{sec: M ordering}

Let $\pi$ be a representation of $G_n$ of Arthur type. In this subsection, we consider another ordering $\geq_{N}$ on the set $\Psi(\pi)$ and show that $\psi^{max}(\pi)$ and $\psi^{min}(\pi)$ are also the unique maximal and minimal elements under this ordering. 

Denote
$$\sigma:=\St(\rho',a_0)=\Delta_{\rho'}[(a_0-1)/2,-(a_0-1)/2].$$
We consider the usual (non-normalized) intertwining operator
$$
M(s,\pi,\sigma): \St(\rho',a_0) | \cdot | ^s \rtimes\pi \rightarrow\St(\rho',a_0) | \cdot | ^{-s} \rtimes\pi,
$$
and the Langlands-Shahidi normalized intertwining operator given by
\[N^{LS}(s,\pi,\sigma):= M(s,\pi,\sigma) r(s,\pi,\sigma)^{-1},\]
where
\begin{align}\label{eq LS normalization}
    r(s,\pi, \sigma):=\frac{L(s, \sigma \times \pi)}{L(s+1,\sigma \times \pi)} \times \frac{L(\sigma,2s,r_{G_n})}{L(\sigma, 2s+1, r_{G_n})},
\end{align}
and $r_{G_n}=\mathrm{Sym}^2$  if $G_n=\SO_{2n+1}$ and $r_{G_n}=\bigwedge^2$ if $G_n=\Sp_{2n}$. For each $L$-parameter $\phi$ of $G_n$, we also denote $r(s,\phi, \sigma):=r(s,\pi,\sigma)$, where $\pi\in \Pi_{\phi}$. There is a more general definition of the Langlands-Shahidi normalization (\cite[p. 150]{Sha10}) in terms of the $L$- and $\varepsilon$-factors. Since the $\varepsilon$-factors do not provide any poles or zeros, for simplicity, we omit them in the above definition as in \cite{Moe10}.

Langlands-Shahidi normalized intertwining operators have been used to study automorphic $L$-functions. For global applications, it is advantageous to understand the holomorphicity and vanishing of these normalized intertwining operators. In general, it is expected that the Langlands-Shahidi normalized intertwining operators are holomorphic and nonzero in certain right half plane.

\begin{lemma}[{\cite[Lemma 4.2]{Kim05}}]\label{lem LS norm nonzero}
Suppose that $\sigma\otimes\pi$ is an irreducible tempered generic representation of a maximal Levi subgroup of $G_n$. Then $N^{LS}(s,\pi,\sigma)$ is holomorphic and nonzero on $\mathrm{Re}(s)\geq 0.$
\end{lemma}

\cite[Lemma 4.2]{Kim05} holds for groups beyond $G_n.$ A similar statement is contained in \cite[Theorem 4.11]{Kim05} which can be extended even further using the standard module conjecture (\cite[Corollary 1.2]{HO13}).

We recall another normalization associated with local Arthur parameters considered in \cite{Moe08, Moe10, Moe11b, Moe12, Art13}, called the Arthur normalization. Arthur used these normalized intertwining operators crucially in his proof of the existence of local Arthur packets (see  \cite[\S 2.3 and \S 2.4]{Art13}). Again, since the $\varepsilon$-factors do not provide any poles or zeros, for simplicity, we also omit them in the definition below following \cite{Moe10}.  Recall that $\sigma=\St(\rho',a_0).$
For each local Arthur parameter $\psi$, we define $r(s,\psi,\sigma):= r(s, \phi_{\psi},\sigma)$. More precisely, for
\[ \psi= \bigoplus_{i=1}^l \rho_i \otimes S_{a_i} \otimes S_{b_i}, \]
we have
\[ r(s, \psi,\sigma )= \prod_{i=1}^l \frac{L(s-\frac{b_i-1}{2},\sigma\times \St(\rho_i,a_i))}{L(s+\frac{b_i+1}{2},\sigma\times \St(\rho_i,a_i))}\times \frac{L(\sigma,2s,r_{G_n})}{L(\sigma,2s+1,r_{G_n})}. \]
For each $\psi \in \Psi(\pi)$, the Arthur normalized intertwining operator
is defined as
\[ N_{\psi}(s,\pi,\sigma):=  M(s, \pi, \sigma)r(s, \psi, \sigma)^{-1}.\]
Note that if $\pi \in \Pi_{\phi_{\psi}}$, then the Arthur normalized intertwining operator is the same as the Langlands-Shahidi normalized intertwining operator.

M{\oe}glin showed that the Arthur normalized intertwining operators are holomorphic in the right half plane.
\begin{thm}[{\cite[Theorem 3.2]{Moe10}}]\label{thm Moeglin holo}
$N_\psi(s,\pi,\sigma)$ is holomorphic for any real $s\geq \half{1}.$
\end{thm}
M{\oe}glin also showed that the image is either an irreducible representation or zero and calculated its image (\cite[Theorem A]{Moe11b}).

Let $f$ denote a meromorphic function. We write $\ord_{s=s_0} f(s)$ to denote the order of vanishing at $s_0$ of $f.$ That is, $\ord_{s=s_0} f(s)=m$ where $m$ is an integer such that $(s-s_0)^{-m} f(s)\neq 0$ in a neighborhood of $s_0.$ Note that $\ord_{s=s_0}f(s)<0$ if $f$ has a pole at $s_0$ and $\ord_{s=s_0}f(s)>0$ if $f$ has a zero at $s_0.$ Let $\pi$ be a representation of Arthur type. Theorem \ref{thm Moeglin holo} states that for any $\psi'\in\Psi(\pi)$ and $s_0\geq \half{1}$, $\ord_{s=s_0} N_{\psi'}(s,\pi, \sigma)\geq 0.$

Suppose that $\pi\in\Pi_{\phi_\psi}\cap \Pi_{\psi'}$ where $\psi$ and $\psi'$ are two local Arthur parameters and $\phi_\psi$ is the local Langlands parameter of Arthur type associated to $\psi.$ Then $N_\psi(s,\pi,\sigma)$ is the Langlands-Shahidi normalization and by \cite[Proposition 4.1]{Moe12}, we have
$$
 \ord_{s=s_0} N_\psi(s,\pi, \sigma)
 \leq \ord_{s=s_0} N_{\psi'}(s,\pi, \sigma)
$$
for any $s_0\in\mathbb{R}_{\geq \half{1}}.$ Note that $\psi^{max}(\pi)=\psi$ in this case by Theorem \ref{thm can max}. 

In general, it is natural to ask which member $\psi \in \Psi(\pi)$ gives the least vanishing for Arthur normalized intertwining operators. In this subsection, we verify that $\psi^{max}(\pi)$ is this member (Theorem \ref{thm Moeglin}).
Note that we do not show that $N_{\psi^{max}(\pi)}(s,\pi, \sigma)$ is non-vanishing. Indeed, if $\pi$ does not lie in any $L$-packet of Arthur type and its Langlands-Shahidi normalized intertwining operator is holomorphic,
then for some $s_0$ and $\sigma$, we have $\ord_{s=s_0}N_{\psi^{max}(\pi)}(s_0,\pi, \sigma)>0$ (by Corollary  \ref{cor van N ordering} below). 

In order to compare Arthur's normalization with Langlands-Shahidi's normalization we define the following ordering on $L$-parameters. This also induces an ordering on local Arthur parameters.

\begin{defn} \label{M ordering}\ 
\begin{enumerate}
    \item [(1)] We define an ordering $\geq_{N}$ on the set of $L$-parameters of $G_n$ by $\phi_1 \geq_{N} \phi_2$ if for any $\sigma= \St(\rho',a_0)$ and $s_0 \in \R_{\geq_{\half{1}}}$, the following inequality holds
$$
 \ord_{s=s_0} r(s,\phi_1, \sigma)
 \geq 
 \ord_{s=s_0} r(s,\phi_2, \sigma).
$$
In particular, $\phi_1 >_{N} \phi_2$ if $N(s,\phi_1,\sigma)$ has less vanishing than $N(s,\phi_2,\sigma)$ for $s \geq \half{1}$.

\item [(2)] Let $\pi$ be a representation of $G_n$ Arthur type. We define an ordering $\geq_{N}$ on $\Psi(\pi)= \{ \psi \ | \ \pi \in \Pi_{\psi} \}$ by $\psi_1 \geq_{N} \psi_2$ if $\phi_{\psi_1} \geq_{N} \phi_{\psi_{2}}$. This ordering is a partial order (see Proposition \ref{prop N partial order} below).
\end{enumerate}
\end{defn}

One can compute the order of poles and zeros of $r(s,\psi,\sigma)$ explicitly by \cite{JPSS83}, which was already done in \cite{Moe10}. The normalizing factor $r(s,\psi,\sigma)$ has no zeros for $s \geq \half{1}$. To describe the order of poles, we recall the following notation and computation.

Write $ \psi= \bigoplus_{i=1}^l \rho_i \otimes S_{a_i} \otimes S_{b_i}.$ Let $A_i=\frac{a_i+b_i}{2}-1$, $B_i=\frac{|a_i-b_i|}{2},$ $\zeta_i$ be the sign of $a_i-b_i$ if $a_i-b_i\neq 0$ and set $\zeta_i=+$ if $a_i-b_i=0.$ Recall that $\sigma=\St(\rho',a_0)$. Fix $b_0\in\mathbb{Z}$ and define $A_0, B_0, \zeta_0$ similarly. We denote the multi-sets 
\begin{align*}
    \Jord(\psi)&:=\{ (\rho_i,A_i,B_i,\zeta_i) \}_{i=1}^l,\\ \Jord(\psi,\rho)&:=\{ (\rho_i,A_i,B_i,\zeta_i) \ | \ 1 \leq i \leq l, \rho_i \cong \rho\},
\end{align*}
and call elements in $\Jord(\psi)$ Jordan blocks. The table below describes the cases where $(\rho_i, A_i,B_i,\zeta_i) \in \Jord(\psi,\rho')$ contributes a simple pole to $ r(s,\psi,\sigma)$ at $s_0:=\frac{b_0-1}{2}$.

\begin{center}
\begin{tabular}{|c|c|c|}
\hline
  $\zeta_i\backslash \zeta_0$ & + & -  \\
  \hline
   +  & $B_i\leq B_0 \leq A_0 \leq A_i$ & No contribution \\
   \hline
   - & $B_i\leq A_0 \leq A_i$ & $B_0\leq B_i\leq A_0 \leq A_i$ \\
   \hline
\end{tabular}
\end{center}

From the table above, we prove that $\geq_{N}$ is a partial order on $\Psi(\pi)$.
\begin{prop}\label{prop N partial order}
For any $\pi\in \Pi_{A}(G_n)$, $\geq_{N}$ is a partial order on $\Psi(\pi)$.
\end{prop}
\begin{proof}
Let $m(\rho\otimes S_{a}\otimes S_b, \sigma, s_0)$ denote the order of vanishing described by the table above. Here we set $\sigma= \St(\rho,a_0)$ and $s_0=\half{b_0-1}$.

First, we show that if $b>1$, then we can recover $\rho \otimes S_a \otimes S_b$ from the collection $(m(\rho\otimes S_{a}\otimes S_b, \sigma,s_0))_{ s_0 \in \R_{\geq \half{1}}} $. Note that when $b=1$, $m(\rho\otimes S_{a}\otimes S_b, \sigma, s_0)=0$ for any $(\sigma, s_0) $ with $s_0 \geq \half{1}$.

Let $A=\frac{a+b}{2}-1$, $B=\frac{|a-b|}{2},$ $\zeta$ be the sign of $a-b$ if $a-b\neq 0$ and set $\zeta=+$ if $a-b=0.$ Let $A_0,B_0, \zeta_0$ be defined similarly for the pair $(a_0,b_0)$ from $(\sigma, s_0)$. We also denote $ m(\rho\otimes S_{a}\otimes S_b, A_0, B_0, \zeta_0):= m(\rho\otimes S_{a}\otimes S_b, \sigma, s_0)$.

First, the sign $\zeta$ can be computed as follows. The set
\[\{ A_0 \in \half{1}\Z \ | \ m(\rho\otimes S_{a}\otimes S_b, A_0, B_0, -) \neq 0 \text{ for some }B_0 \in \half{1} \Z \} \]
is empty if and only if $\zeta=+$. 

Next, if $\zeta=+$, then (note that we assume $b>1$) 
\begin{align*}
    B&= \min\{B_0 \in \half{1}\Z \ | \ m(\rho\otimes S_{a}\otimes S_b, A_0, B_0, +) \neq 0 \text{ for some }A_0\in \half{1} \Z \},\\
    A&= \max\{A_0 \in \half{1}\Z \ | \ m(\rho\otimes S_{a}\otimes S_b, A_0, B, +) \neq 0  \}.
\end{align*}
If $\zeta=-$, then
\begin{align*}
    B&= \max\{B_0 \in \half{1}\Z \ | \ m(\rho\otimes S_{a}\otimes S_b, A_0, B_0, -) \neq 0 \text{ for some }A_0\in \half{1} \Z \},\\
    A&= \max\{A_0 \in \half{1}\Z \ | \ m(\rho\otimes S_{a}\otimes S_b, A_0, B, -) \neq 0  \}.
\end{align*}
This uniquely determines $(A,B)$. Together with $\zeta$, we recover $\rho \otimes S_a \otimes S_b$.

Now given $\psi \in \Psi(\pi)$, we may assume $\psi$ and $\pi$ are of good parity and write
\[\psi= \bigoplus_{\rho} \left(\bigoplus_{i \in I_{\rho,1}} \rho \otimes S_{a_i}\otimes S_{b_i}  \oplus \bigoplus_{i \in I_{\rho,2}} \rho \otimes S_{a_i}\otimes S_{1} \right),  \]
where $ b_i>1$ for $i \in I_{\rho,1}$. Again assign the triple $(A_i,B_i, \zeta_i)$ to each $(a_i,b_i)$ for $i \in I_{\rho,1}$ by the same recipe. Let $B^{-}:=\max\{B_j \ | \ j \in I_{\rho,1},\ \zeta_j=-\}$. By similar argument, we recover the multi-sets
\begin{align*}
    \{(B_j, A_j)\ | \ j \in I_{\rho,1},\ B_j=B^{-},\  \zeta_j=-\}.    
\end{align*}
Repeating this procedure, we determine the multi-sets
\begin{align*}
    \{(B_j, A_j)\ | \ j \in I_{\rho,1},\ \zeta_j=-\}.    
\end{align*}
Next, we set $B^{+}:=\min\{B_j \ | \ j \in I_{\rho,1},\ \zeta_j=+\}$. By similar argument, we recover the multi-sets
\begin{align*}
    \{(B_j, A_j)\ | \ j \in I_{\rho,1},\ B_j=B^{+},\  \zeta_j=+\}.    
\end{align*}
Repeating this procedure, we determine the multi-sets
\begin{align*}
    \{(B_j, A_j)\ | \ j \in I_{\rho,1},\ \zeta_j=+\}.    
\end{align*}
In summary, we have determined the subrepresentation
\[\bigoplus_{\rho} \bigoplus_{i \in I_{\rho,1}} \rho \otimes S_{a_i}\otimes S_{b_i}  \]
of $\psi$. Then comparing $\supp(\psi)$ and $\Omega(\pi)$, we determine $\psi$ completely by Lemma \ref{lem $L$-data}. 

In conclusion, given $\pi$, the map 
\[ \psi \mapsto \{(\ord_{s=s_0} r(s,\psi,\sigma),s_0,\sigma) \}_{\sigma, s_0 \in \R_{\geq \half{1}}}\]
is injective for $\psi \in \Psi(\pi)$. This completes the proof of the proposition.
\end{proof}

Remark that the last step of above proof used the information $\Omega(\pi)$ of $\pi$. In general, the map
\[ \psi \mapsto \{(\ord_{s=s_0} r(s,\psi,\sigma),s_0,\sigma) \}_{\sigma, s_0 \in \R_{\geq \half{1}}}\]
is not injective on $\Psi(G_n)$, but its restriction to $\Psi(\pi)$ is injective for any $\pi$ of Arthur type. For example, $r(s_0,\psi, \sigma)$ has no pole for any $s_0 \geq \half{1}$ for any tempered $\psi$.

The following is the main theorem of this subsection, which states that
$N_{\psi^{max}(\pi)}(s,\pi, \sigma)$ gives the least vanishing among all of the local Arthur parameters in $\Psi(\pi)$. 

\begin{thm}\label{thm Moeglin}
Let $\mathrm{G}_n$ be $\Sp_{2n}(F)$ or split $\SO_{2n+1}(F)$.
\begin{enumerate}
    \item If $T$ is a raising operator applicable on $\psi \in \Psi^{+}(G_n)$, then 
    \[T(\psi) \gneq_{N} \psi.\]
    In particular, if $\psi \geq_{O} \psi'$, then $\psi \geq_{N} \psi'$.
    \item Let $\pi$ be a representation of $G_n$ of Arthur type. The distinguished members $\psi^{max}(\pi)$ and $ \psi^{min}(\pi)$ are the unique elements in $\Psi(\pi)$ satisfying the following inequality
    \[ \psi^{max}(\pi) \geq_{N} \psi \geq_{N} \psi^{min}(\pi),\]
    for any $\psi \in \Psi(\pi).$
\end{enumerate}
\end{thm}
\begin{proof}
Part (b) follows immediately from Part (a) and Theorem \ref{thm structure of Psi(pi) intro}. Now we show Part (a) case by case. It suffices to show for any pair $(\sigma,s_0)$, we have 
\[ \ord_{s=s_0} r(s,\psi,\sigma) \leq \ord_{s=s_0} r(s, T(\psi), \sigma), \]
and there is a choice of $(\sigma,s_0)$ such that the inequality is strict. 

Say the indices involved in $T$ is in $I_{\rho}$. The inequality is indeed an equality if $\rho'$ is not isomorphic to $\rho$. Therefore, we assume $\rho\cong \rho'$, and then the choice of $\sigma$ is equivalent to the choice of $a_0$. We also omit $\rho$ in the quadruple $(\rho,A_i,B_i,\zeta_i)\in \Jord(\psi,\rho)$ in the rest of the proof. 

\textbf{Case (i).} Suppose that $T=ui_{i,j}^{-1}$.

We deal with the situation that the $ui_{i,j}^{-1}$ is not of type 3' explicitly here. The situation that $ui_{i,j}^{-1}$ is of type 3' follows from a similar argument, which we omit. In this situation, we have 
\begin{align*}
    \Jord(T(\psi)) \setminus (\Jord(T(\psi)) \cap \Jord(\psi))&= \{ (A_i,B_i,\zeta_i), (A_j,B_j,\zeta_j)\},\\
    \Jord(\psi) \setminus (\Jord(T(\psi)) \cap \Jord(\psi))&= \{ (A_i,B_j,\zeta_j), (A_j,B_i,\zeta_i)\},
\end{align*}
where $A_j>A_i \geq B_j \zeta_j >B_i\zeta_i$ as shown in the following picture.
\begin{center}
   \begin{tikzpicture}
    \draw (-80 pt , 17pt) rectangle (70pt, 2 pt);
    \draw (-20 pt , -2pt) rectangle (140 pt,-17 pt);
    \draw node at (-65pt, 9.5pt){$B_i\zeta_i$ };
    \draw node at (60pt, 9.5pt){$A_i$ };
    \draw node at (-5pt, -9.5pt){$B_j\zeta_j$ };
    \draw node at (130pt, -9.5pt){$A_j$ };
    \end{tikzpicture}
    \begin{tikzpicture}
    \draw (-80 pt , 17pt) rectangle (140 pt,2 pt);
    \draw (-20 pt , -2pt) rectangle (70pt, -17 pt) ;
    \draw node at (-65pt, 9.5pt){$B_i\zeta_i$ };
    \draw node at (130pt, 9.5pt){$A_j$ };
    \draw node at (-5pt, -9.5pt){$B_j\zeta_j$ };
    \draw node at (60pt, -9.5pt){$A_i$ };
    \end{tikzpicture}
\end{center}

It suffices to compare the contribution of these two pairs of Jordan blocks. Let $ \alpha_i$ (resp. $\alpha_j$, $\beta_i$, $\beta_j$) denote the order of the poles contributed from the Jordan block $(A_i,B_i,\zeta_i)$ (resp. $(A_j,B_j,\zeta_j)$, $(A_j,B_i,\zeta_i)$, $(A_i,B_j,\zeta_j)$). We are going to verify $\alpha_i+ \alpha_j \leq \beta_i + \beta_j$ in each case and show that there is a choice of $(a_0,s_0)$ which makes the inequality strict.

Case(i-1): Suppose $\zeta_i=\zeta_j=+$. Then $(\alpha_i,\alpha_j, \beta_i,\beta_j) \neq (0,0,0,0)$ only if $\zeta_0=+$. The Jordan block $(A,B,\zeta)$ contributes a pole if and only if $[A, B\zeta]_{\rho} \supseteq [A_0,B_0]_{\rho}$. Thus, the pair $(\alpha_i,\alpha_j, \beta_i,\beta_j)$ is given by the following table, where the slots marked with x violate the condition $A_0>B_0$.
\begin{center}
    \begin{tabular}{|c|c|c|c|c|c|}
    \hline
         &  $A_0 \leq B_i$ &$ B_i < A_0  \leq B_j$ & $B_j < A_0 \leq A_i$ & $A_i< A_0 \leq A_j$ & $A_j <A_0$\\
         \hline
        $B_0<B_i$ &$(0,0,0,0)$&$(0,0,0,0)$&$(0,0,0,0)$&$(0,0,0,0)$&$(0,0,0,0)$\\
         \hline
        $B_i \leq B_0< B_j$  &x&$(1,0,1,0)$&$(1,0,1,0)$&$(0,0,1,0)$&$(0,0,0,0)$\\
         \hline
        $B_j \leq B_0<A_i$ &x&x&$(1,1,1,1)$&$(0,1,1,0)$&$(0,0,0,0)$\\
         \hline
        $A_i \leq B_0< A_j$ &x&x&x&$(0,1,1,0)$&$(0,0,0,0)$\\
         \hline
        $A_j \leq B_0$ &x&x&x&x&$(0,0,0,0)$\\
        \hline
    \end{tabular}
\end{center}
In any case, we have $\alpha_i+\alpha_j \leq \beta_i+\beta_j$, and there is a case such that the inequality is strict. This completes the verification of this case.

Case (i-2): Suppose $\zeta_i=-, \zeta_j=+$. If $\zeta_0=-$, then $ \alpha_j=\beta_j=0$. In this case,  the Jordan block $ (A,B,-)$ contributes a pole if and only if $[A,-B]_{\rho} \supseteq [A_0,-B_0]_{\rho}$. Since $ [A_j,-B_i]_{\rho} \supsetneq [A_i,-B_i]_{\rho}$, we have $\alpha_i \leq \beta_i$, and choosing $(a_0,s_0)$ such that 
\[[A_j,-B_i]_{\rho} \supseteq [A_0,-B_0]_{\rho} \supsetneq [A_i,-B_i]_{\rho}\]
makes the inequality strict.

It remains to show if $\zeta_0=+$, the inequality holds. In this case, $\alpha_i \leq \alpha_j$. On the other hand, $(\alpha_j, \beta_j)=(1,0)$ if and only if 
\[ [A_j, B_j]_{\rho} \supseteq [A_0,B_0]_{\rho} \supsetneq [A_i,B_j]_{\rho},\]
and in this case, we have $ (\alpha_i,\beta_i)=(0,1)$. Therefore, $\alpha_i+\alpha_j \leq \beta_i+\beta_j$ in any case.

Case (i-3): Suppose $\zeta_i=\zeta_j=-$. If $ \zeta_0=-$, then the Jordan block $(A,B,-)$ contributes a pole if and only if $ [A, -B]_{\rho} \supseteq [A_0,-B_0]_{\rho}$, and the verification is identical with Case (i-1), which we omit. Therefore, it remains to check the inequality when $\zeta_0=+$. Indeed, we have
\[(\alpha_i,\alpha_j,\beta_i,\beta_j )= \begin{cases} (0,0,0,0) & \text{ if }A_0< B_j,\\
(0,1,0,1) & \text{ if }B_j \leq A_0< B_i,\\
(1,1,1,1) & \text{ if }B_j \leq A_0 \leq  A_i,\\
(0,1,1,0) & \text{ if }A_i<A_0\leq  A_j,\\
(0,0,0,0) & \text{ if }A_j<A_0.\\
\end{cases}\]
Thus $\alpha_i+\alpha_j \leq \beta_i+\beta_j$ in any case. This completes the verification of Case (i).

\textbf{Case (ii).} Assume that $T= dual \circ ui_{i,j} \circ dual$. 

If the $ui_{i,j}$ in the composition is not of type 3', then $dual \circ ui_{i,j} \circ dual=ui_{j,i}^{-1}$, and the conclusion follows from Case (i). Therefore, we assume $ui_{i,j}$ in the composition is of type 3'. In this case, we have 
\begin{align*}
   \Jord(T(\psi)) \setminus (\Jord(T(\psi)) \cap \Jord(\psi))&=\{  (A_j,B_i,\zeta_i)\},\\
    \Jord(\psi) \setminus (\Jord(T(\psi)) \cap \Jord(\psi))&=\{ (A_i,B_i,\zeta_i), (A_j,B_j,-)\},
\end{align*}
where $B_j=A_i+1$ as shown in the following picture.
\begin{center}
   \begin{tikzpicture}
    \draw [dashed] (-100 pt, 20pt)--( -100 pt ,-43 pt);
    \draw (-130 pt , -21pt) rectangle (0pt, -36 pt);
    \draw (-200 pt , -2pt) rectangle (140 pt,-17 pt);
    \draw node at (-70pt, 9.5pt){$B_i$ };
    \draw node at (-10pt,- 28.5pt){ $A_i$ };
    \draw node at (10pt, -9.5pt){$B_j$ };
    \draw node at (130pt, 9.5pt){$A_j$ };
    \draw node at (-115pt, 9.5pt){$B_i\zeta_i$ };
    \draw node at (-185pt, -9.5pt){$-B_j$ };
    \draw (-130 pt , 2 pt) rectangle (140 pt,17 pt);
    \draw node at (-70pt, -28.5pt){$B_i$ };
    \draw node at (-115pt, -28.5pt){$B_i\zeta_i$ };
    \draw node at (130pt, -9.5pt){$A_j$ };
    \draw node at (-100pt, 27pt) {$0$};
    \end{tikzpicture}
\end{center}
Let $\alpha_i$ (resp. $\beta_i, \beta_j$) be the order of poles contributed from the Jordan block $ (A_j,B_i,\zeta_i)$ (resp. $(A_i,B_i,\zeta_i)$, $(A_j,B_j,-)$). We are going to verify $\alpha_i \leq \beta_i+ \beta_j$ in any case, and there is a pair $(a_0,s_0)$ such that the inequality is strict.

Suppose $\zeta_i=+$. Then $\alpha_i=1$ if and only if $\zeta_0=+$ and $[A_j, B_i]_{\rho} \supseteq [A_0,B_0]_{\rho}$. We have 
\[ (\alpha_i,\beta_i,\beta_j)= \begin{cases} (1,1,0) & \text{ if } B_i \leq A_0 \leq A_i, \\
(1,0,1) & \text{ if } A_i < A_0 \leq A_j, \end{cases}\]
so $\alpha_i \leq \beta_i+\beta_j$ in any case. On the other hand, taking $ (a_0,s_0)$ such that $\zeta_0=-$ and $B_j \leq A_0 \leq A_j$, we have $(\alpha_i,\beta_i,\beta_j)=(0,0,1)$, and the inequality is strict. 

Suppose $\zeta_i=-$. Then if $\zeta_0=-$, we have $\alpha_i \leq \beta_j$, and the inequality is strict if we take $ (a_0,s_0)$ such that $[A_j,-B_j]_{\rho} \supseteq [A_0,-B_0]_{\rho} \supsetneq [A_j, -B_i]_{\rho}$. It remains to check $\alpha_i \leq \beta_i+\beta_j$ when $\zeta_0=+$. We have $\alpha_i=1$ if and only if $ B_i \leq A_0 \leq A_j$. In this case, we have 
\[ (\alpha_i,\beta_i,\beta_j)= \begin{cases} (1,1,0) & \text{ if } B_i \leq A_0 \leq A_i ,\\
(1,0,1) & \text{ if } A_i < A_0 \leq A_j, \end{cases}\]
so $\alpha_i \leq \beta_i+\beta_j$ in any case. This completes the verification of Case (ii).

\textbf{Case (iii).} Assume that $T=dual_k^{-}$.

In this case, we have  
\begin{align*}
    \Jord(T(\psi)) \setminus (\Jord(T(\psi)) \cap \Jord(\psi))&= \{ (A_k,1/2,+)\},\\
    \Jord(\psi) \setminus (\Jord(T(\psi)) \cap \Jord(\psi))&= \{  (A_k,1/2,-)\},
\end{align*}
Let $\alpha_k$ (resp. $\beta_k$) denote the order of pole contributed from the Jordan block $(A_k,1/2,+)$ (resp. $(A_k,1/2,-)$). We are going to verify $ \alpha_k \leq \beta_k$ in any case, and there is a choice of $(a_0,s_0)$ such that the inequality is strict.

Suppose $\zeta_0=-$. Then $\alpha_k$ is always zero, and the inequality follows. If we choose $ (a_0,s_0)$ such that $\zeta_0=-$, $B_0=1/2$ and $A_0 \leq A_k$, then we have $(\alpha_k, \beta_k)=(0,1)$, so the inequality is strict.

It remains to check the inequality when $\zeta_0=+$. One can check from the table that in this case, $\alpha_k=1$ implies $ 1/2 \leq B_0 < A_0 \leq A_k$, and hence $1/2 \leq A_0 \leq A_k$, which implies $\beta_k=1$. Therefore, $\alpha_k \leq \beta_k$ in any case. 

This completes the verification of Case (iii) and the proof of the theorem.
\end{proof}

The proof above admits the following corollary.

\begin{cor}\label{cor raising operator zeros}
If $T$ is a raising operator applicable on $\psi$, then for any $\sigma=\St(\rho',a_0)$, the quotient
\[ \frac{r(s,\phi_{T(\psi)},\sigma)}{r(s,\phi_{\psi},\sigma)}\]
has finitely many zeros on $\R_{\geq \half{1}}$.
\end{cor}

M{\oe}glin classified when the Langlands-Shahidi normalization and Arthur normalization differ at a fixed $s_0$ and $\sigma$ in terms of the vanishing or subquotient properties of the Arthur normalized intertwining operators (\cite[Proposition 4.1(2)]{Moe12}). The following theorem gives a sufficient condition which guarantees that there exists $s_0$ and $\sigma$ for which the normalizations differ.

\begin{thm}\label{thm M vanishing}
If $\pi$ be a representation of $G_n$ of Arthur type, then for any $\psi \in \Psi(\pi)$, we have $\phi_{\pi} \geq_{N} \phi_{\psi}$, where $\phi_{\pi}$ is the $L$-parameter of $\pi$. Moreover, if $ \pi \not\in \Pi_{\phi_{\psi}}$, then $\phi_{\pi} \gneq_{N} \phi_{\psi}$.
\end{thm}
\begin{proof}
 Write
\[\pi= L(\Delta_{\rho_1}[x_1,-y_1],\dots,\Delta_{\rho_f}[x_f,-y_f]; \pi_{temp}).\]
We apply induction on $f$ such that $\phi_{\pi} \geq_{N} \phi_{\psi}$, and $\phi_{\psi} \geq_{N} \phi_{\pi}$ only if $\pi \in \Pi_{\phi_{\psi}}$.

If $f=0$, then $\pi$ is tempered, and hence 
\[ \phi_{\pi}= \phi_{\psi^{max}(\pi)} \geq_{N} \phi_\psi \]
by Theorem \ref{thm Moeglin}. Moreover, $\phi_{\pi}=\phi_{\psi^{max}(\pi)}$ in this case, and hence $\phi_{\psi} \geq_N \phi_{\pi}$ only if $\psi=\psi^{max}(\pi)$ by Proposition \ref{prop N partial order}, which is equivalent to $\pi \in \Pi_{\phi_{\psi}}$.

If $f \geq 1$, then we assume $x_i-y_i> x_1-y_1$ or $x_i \geq  x_1$ for all $i >1$ and $\rho_i \cong \rho_1$. We define
\begin{align*}
    \pi^{-}&:=L(\Delta_{\rho_2}[x_2,-y_2],\dots,\Delta_{\rho_f}[x_f,-y_f]; \pi_{temp}),\\
    \psi^{-}&:= \psi^{max}(\pi)-\rho_1\otimes S_{y_1+x_1+1} \otimes S_{y_1-x_1+1}+\rho_1\otimes S_{y_1+x_1+1} \otimes S_{y_1-x_1-1}.
\end{align*}
By Proposition \ref{prop pi minus}, $\pi^{-}$ is of Arthur type and $\pi^{-}\in \Pi_{\psi^{-}}$. Note that
\begin{align}\label{eq pi minus}
    \phi_{\pi}&= \left(\rho_1|\cdot|^{\half{y_1-x_1}} \otimes S_{y_1+x_1+1}\oplus \rho_1|\cdot|^{-\half{y_1-x_1}} \otimes S_{y_1+x_1+1}\right)\oplus  \phi_{\pi^{-}},\\
    \nonumber
    \phi_{\psi^{max}(\pi)}&= \left(\rho_1|\cdot|^{\half{y_1-x_1}} \otimes S_{y_1+x_1+1}\oplus \rho_1|\cdot|^{-\half{y_1-x_1}} \otimes S_{y_1+x_1+1}\right)\oplus  \phi_{\psi^{-}}.
\end{align}
Then for any $s_0 \in \R_{\geq \half{1}}$, we have
\begin{align}\label{eq ineq Moeglin}
&\,\,\ord_{s=s_0} r(s,\phi_{\pi},\sigma )- \ord_{s=s_0} r(s,\phi_{\psi^{max}(\pi)},\sigma )\\
\nonumber =&\,\,\ord_{s=s_0} r(s,\phi_{\pi^{-}},\sigma )- \ord_{s=s_0} r(s,\phi_{\psi^{-}},\sigma )\\    \geq &\,\, \nonumber \ord_{s=s_0} r(s,\phi_{\pi^{-}},\sigma )- \ord_{s=s_0} r(s,\phi_{\psi^{max}(\pi^{-})},\sigma)\\
\nonumber \geq& \,\,0,
\end{align}
where the first equality follows from the Definition \eqref{eq LS normalization} (and \cite{JPSS83}), the first inequality follows from Theorem \ref{thm Moeglin}(2), and the last inequality follows from the induction hypothesis. This shows that
\[\phi_{\pi} \geq_{N} \phi_{\psi^{max}(\pi)} \geq_{N} \phi_{\psi}.\]
Moreover, if $\phi_{\psi} \geq_{N} \phi_{\pi}$, then \eqref{eq ineq Moeglin} implies that $ \psi=\psi^{max}(\pi)$, $\psi^{-}=\psi^{max}(\pi^{-})$, and $ \phi_{\psi^{max}(\pi^{-})} \geq_{N}  \phi_{\pi^{-}}$. Then by the induction hypothesis, we have $ \pi^{-} \in \Pi_{\phi_{\psi^{-}}}$, which implies $\pi \in \Pi_{\phi_{\psi}}$. This completes the proof of the theorem.
\end{proof}

As an immediate consequence, we have the following corollary.

\begin{cor}\label{cor van N ordering}
If $\pi$ is a representation of $G_n$ of Arthur type, but not in any $L$-packet of Arthur type, then for any $\psi \in \Psi(\pi)$, there exists a $\sigma$ and $s_0 \in \R_{\geq_{\half{1}}}$ such that $\ord_{s=s_0} N_{\psi}(s_0,\pi,\sigma)> \ord_{s=s_0} N^{LS}(s,\pi,\sigma)$. 
\end{cor}

\end{document}